\documentclass[intlimits]{amsart}
\usepackage[T1]{fontenc} % FONDAMENTALE
\usepackage[utf8]{inputenc}
\usepackage[english]{babel}

\usepackage{graphicx}
\usepackage{mathtools, amsmath, amssymb} % Carica gli ams in modo ordinato
\usepackage{mathrsfs, booktabs}
\usepackage{xcolor} % Rimosso dvipsnames per test

\usepackage{bm}
\usepackage{tikz}
\usepackage{pgfplots}
\pgfplotsset{compat=newest}
\usetikzlibrary{decorations.markings}
\usetikzlibrary{angles}
\usetikzlibrary{quotes}

\usepackage{cite}

\usepackage{amsthm}
\theoremstyle{plain}
\newtheorem{thm}{Theorem}[section]
\newtheorem{lem}[thm]{Lemma}
\newtheorem{prop}[thm]{Proposition}
\newtheorem{corr}[thm]{Corollary}
\newtheorem{ex}[thm]{Example}
\theoremstyle{definition}
\newtheorem{defn}{Definition}[section]
\theoremstyle{remark}
\newtheorem{rmk}{Remark}[section]

\def\faifig{1}
\definecolor{myred}{rgb}{0.9,0,0}
\definecolor{mygre}{rgb}{0,0.5,0}
\definecolor{myblu}{rgb}{0.1,0.2,0.8}
\definecolor{myvio}{rgb}{0.7,0.1,0.8}
\definecolor{mywhite}{rgb}{1,1,1}
\definecolor{myblack}{rgb}{0,0,0}
\definecolor{myoran}{rgb}{1,.4,.2}
\definecolor{mygray}{rgb}{.3,.3,.3}
\definecolor{mygrayl}{rgb}{.7,.7,.7}
\definecolor{mygreDark}{rgb}{0,0.33,0}
\definecolor{myyel}{rgb}{1,.95,.5}
\definecolor{myyell}{rgb}{1,.975,.75}
\definecolor{mylg}{rgb}{.85,.85,.85}
\definecolor{mygrea}{rgb}{0.4,0.6,0}
\definecolor{mygreb}{rgb}{0,0.6,0.4}
\definecolor{ForestGreen}{RGB}{34,139,34}

\newcommand{\gored}[1]{\textcolor{myred}{#1}}
\newcommand{\gogre}[1]{\textcolor{mygre}{#1}}
\newcommand{\gogrea}[1]{\textcolor{mygrea}{#1}}
\newcommand{\gogreb}[1]{\textcolor{mygreb}{#1}}

\newcommand{\gowhite}[1]{\textcolor{mywhite}{#1}}

\newcommand{\gogray}[1]{\textcolor{mygray}{#1}}
\newcommand{\gograyl}[1]{\textcolor{mygrayl}{#1}}
\newcommand{\goyel}[1]{\textcolor{myyel}{#1}}
\newcommand{\goyell}[1]{\textcolor{myyell}{#1}}
\newcommand{\golg}[1]{\textcolor{mylg}{#1}}

\newcommand{\goblue}[1]{\textcolor{blue}{#1}}

\newcommand{\sign}{\mathrm{sign}}
\newcommand{\dx}{\!\mathrm{d}}

\newcommand{\lam}{\lambda}
\newcommand{\arccosh}{\mathrm{arccosh}}

\newcommand{\edge}[2]{(#1\,#2)}

\newcommand{\rev}[1]{#1^{\rm (rev)}}
\newcommand{\pz}{\phantom{0}}
\newcommand{\walt}{w^{\pm}}
\newcommand{\bipg}[4]{#1=(#2,#3;#4)}

\date{\today}

\newcommand{\va}{\bm{\alpha}}
\newcommand{\vb}{\bm{\beta}}

\newcommand{\bR}{\mathbb{R}}
\newcommand{\ga}{\alpha}
\newcommand{\gb}{\beta}
\newcommand{\cT}{\mathcal{T}}
\newcommand{\cK}{\mathcal{K}}
\newcommand{\cW}{\mathcal{W}}
\newcommand{\cM}{\mathcal{M}}

\newcommand{\dotb}[2]{\put(#1.5,#2.5){{\color{blue}\circle*{0.65}}}}
\newcommand{\dotr}[2]{\put(#1.5,#2.5){{\color{red}\circle*{0.65}}}}
\newcommand{\dotbr}[3]{
\put(#1.5,#3.5){{\color{blue}\circle*{0.65}}}
\put(#2.5,#3.5){{\color{red}\circle*{0.65}}}
}
\newcommand{\dotx}[2]{\put(#1.5,#2.5){{\color{black}\circle*{0.33}}}}
\newcommand{\dotbrx}[3]{\put(#1.5,#3.5){{\color{black}\circle*{0.33}}}}

\newcommand{\dote}[2]{\put(#1.5,#2.5){{\circle{0.65}}}}

\newcommand{\eps}{\epsilon}
\newcommand{\w}{\omega}
\newcommand{\cO}{\mathcal{O}}
\newcommand{\argmin}{\mathop{\mathrm{argmin}}}

\newcommand*{\symdif}{\bigtriangleup}
\newcommand{\setminx}{\smallsetminus}

\renewcommand{\emptyset}{\varnothing}

\renewcommand{\geq}{\geqslant}
\renewcommand{\leq}{\leqslant}
\newcommand{\tr}{\mathrm{tr}}

\DeclareFontEncoding{LS1}{}{}
\DeclareFontSubstitution{LS1}{stix}{m}{n}
\DeclareSymbolFont{stixletters}{LS1}{stix}{m}{it}
\DeclareMathAccent{\cev}{\mathord}{stixletters}{"91}
\DeclareMathAccent{\vec}{\mathord}{stixletters}{"92}
\DeclareMathAccent{\vecev}{\mathord}{stixletters}{"95}
\usetikzlibrary{decorations}

\usepackage{hyperref}
\hypersetup{pdfauthor={CSS},pdftitle={Spanning trees and Assignment},%
            colorlinks, linktocpage=true, pdfstartpage=1, pdfstartview=FitV,%
    breaklinks=true, pdfpagemode=UseNone, pageanchor=true, pdfpagemode=UseOutlines,%
    plainpages=false, bookmarksnumbered, bookmarksopen=true, bookmarksopenlevel=1,%
    hypertexnames=true, pdfhighlight=/O,%
    urlcolor=orange, linkcolor=blue, citecolor=ForestGreen %pagecolor=RoyalBlue,%
}

\ifdefined\be
\renewcommand{\be}{\begin{equation}}
\else
\newcommand{\be}{\begin{equation}}
\fi
\newcommand{\ee}{\end{equation}}
\newcommand{\ef}[1]{\, #1}

\newcommand{\calR}{\mathcal{R}}
\newcommand{\calT}{\mathcal{T}}

\newcommand{\calD}{\mathcal{D}}

\newcommand{\calJ}{\mathcal{J}}
\newcommand{\calK}{\mathcal{K}}
\newcommand{\calL}{\mathcal{L}}

\newcommand{\calX}{\mathcal{X}}
\newcommand{\calC}{\mathcal{C}}
\newcommand{\calV}{\mathcal{V}}

\newcommand{\calM}{\mathcal{M}}

\newcommand{\bZ}{\mathbb{Z}}
\newcommand{\bN}{\mathbb{N}}

\newcommand{\bu}{{\bm u}}
\newcommand{\bv}{{\bm v}}

\begin{document}
\title[Spanning trees in the Assignment Problem]{Spanning trees in the Assignment Problem:\\two theorems and two conjectures}
\author{Sergio Caracciolo}
\email{\rule{0pt}{20pt}sergio.caracciolo@mi.infn.it}
\address{INFN, sezione di Milano, Milano, Italy}
\author{Gabriele Sicuro}
\email{gabriele.sicuro@unibo.it}
\address{Department of Mathematics, University of Bologna, Bologna, Italy}
\author{Andrea Sportiello}%
 \email{andrea.sportiello@gmail.com}
\address{CNRS \&  LIPN, Universit\'e Paris Nord, Villetaneuse, France}%
\date{\today}
\begin{abstract}
The \emph{Minimum Matching Problem} consists of finding an independent
edge set of minimum weight $M_{\star}(G)$ in a given edge-weighted
graph $G$. When $G$ is bipartite, this reduces to the \emph{Assignment
  Problem}. We consider a variant of this problem defined by taking
the union of optimal matchings across various slightly modified
versions of the base graph:
$H_{\calJ}(G)=\bigcup_{U \in \calJ} M_{\star}(G_{U})$.  We establish
two families of results: (1) In two distinct settings for the
Assignment Problem, we prove that the resulting graphs $H_{\calJ}$, as
well as certain associated graphs $\bar{H}_{\calJ}$, are spanning
trees on the relevant base graphs $G$ and~$\bar{G}$. (2) In these same
settings, assuming the edge weights are given by the $p$-th power of
Euclidean distances for point configurations in the plane, we show
that for $p=1$ the tree $H_{\calJ}$ is non-crossing (i.e., its planar
embedding has no crossing edges), whereas, remarkably, for $p=2$ the
associated tree $\bar{H}_{\calJ}$ is non-crossing.  Finally, we
introduce novel conjectures in Statistical Mechanics, to be explored
in future work: in the Random Euclidean Assignment Problem (where
points are i.i.d.\ on a planar domain), we conjecture that for $p=2$
the trees $\bar{H}_{\calJ}$ are asymptotically distributed as Uniform
Spanning Trees with free and wired boundary conditions in the two
respective settings. In particular, suitable paths on the tree in the
second setting, and on its planar dual in the first setting, are
asymptotically distributed as $\text{SLE}_{\kappa}$ with~$\kappa=2$.
\end{abstract}
\maketitle

\setcounter{tocdepth}{1}
\vspace*{-0.25cm}
\tableofcontents
% appendixtitleon
% appendixtitletocoff

\newpage

%% CHECK OUT THINGS....

%% \begin{verbatim}
%% https://en.wikipedia.org/wiki/Gradient-index_optics
%% https://www.physicsforums.com/threads/refraction-in-a-medium-with-a-gradient-of-refractive-index.594091/
%% https://physics.stackexchange.com/questions/122599/ray-tracing-in-a-inhomogeneous-media
%% https://en.wikipedia.org/wiki/Hamiltonian_optics
%% https://physics.stackexchange.com/questions/745161/a-conformally-flat-metric-is-also-ricci-flat
%% https://en.wikipedia.org/wiki/Conformally_flat_manifold
%% \end{verbatim}

%%%%%%%%%%%%%%%%%%%%%%%%%%%%%%%%%%%%%%%%%%%%%%%%%%%%%%%
\section{Introduction}
%%%%%%%%%%%%%%%%%%%%%%%%%%%%%%%%%%%%%%%%%%%%%%%%%%%%%%%

%-------------------------------------------------------
\subsection{A summary of results}
\label{sec.open}
%-------------------------------------------------------

\noindent
Let us start by giving a glimpse of our main results, leaving aside
variants and generalisations. Given an $n \times n$ real matrix $W=(W_{ij})$, the 
\emph{optimal assignment} $\pi^*$ is the 
permutation
$\pi \in \mathfrak{S}_n$ (unique under suitable generality hypothesis)
minimising the weight
$\cW(\pi)=\sum_{i=1}^n W_{i\,\pi(i)}$. Consider $\pi^*$ as a bijection
from $[n]$ to~$[n]$.

Now, let $W$ be an $(n+1) \times n$ matrix as above. Call $W^{(k)}$ the
matrix $W$ with the $k$-th row removed. Each of these square
matrices has an optimal assignment $\pi^{(k)}$, seen as a bijection from
$[n+1]\setminx k$ to $[n]$.
Introduce the subgraph of the complete bipartite graph, 
$H \subseteq \bipg{\calK_{n+1,n}}{A}{B}{A \times B}$, defined by
$\edge{a_i}{b_j}\in E(H)$ iff there exists a $k$ such that
$j=\pi^{(k)}(i)$.  We prove in Theorem~\ref{thm:Jtree} of
Section~\ref{sec.teoremaserio1} that, under suitable generality
hypothesis, this subgraph is a spanning tree of $\calK_{n+1,n}$, and
all vertices $b_j$ have degree 2. This property suggests to define a
``projected'' graph $\bar{H}$, subgraph of $\calK_{n+1}$, such that
$\edge{a_{i}}{a_{i'}}$ is in $E(\bar{H})$ iff there exists $j$ such
that both $\edge{a_i}{b_j}$ and $\edge{a_{i'}}{b_j}$ are in $E(H)$
(if $j$ exists, it is unique), and it is easy to conclude that also
$\bar{H}$ is a tree, now spanning on~$\calK_{n+1}$.

If the weights $W_{ij}$ are the $p$-th power of the Euclidean distance
among points $x_i$ and $y_j$ in the plane, associated to the vertices
$a_i$ and $b_j$, respectively, the graphs $H$ and $\bar{H}$ have a
natural embedding in the plane, such that, if $\edge{a_i}{b_j} \in E(H)$,
then the segment $[x_i,y_j]$
%
% $\{ t x_i + (1-t) y_j\}_{t\in[0,1]}$ 
is part of the embedded graph.  We prove (relatively easily) in
Corollary \ref{corr.1pnoncross} that, at $p=1$, the embedded graph $H$
is non-crossing, that is, no pair of segments do cross, and (with more
effort) in Theorem~\ref{th:treenoncross} of Section~\ref{sec.caso2dp2}
that, at $p=2$, the projected embedded graph $\bar{H}$ is
non-crossing.  As shown in Appendix~\ref{app.crossingPneq2}, the
hypotheses of this last theorem are quite tight: not only $\bar{H}$ is
in general crossing for all values $p\neq 2$, but also for $p=2$ it is
in general crossing on any two-dimensional manifold that has some
(positive or negative) curvature anywhere (including isolated conical
singularities). Furthermore, the statement in which strict
inequalities are replaced by weak inequalities does not hold either,
as shown again through counterexamples.  Conversely, the theorem at
$p=2$ extends from the plane to other surfaces with no curvature, like
the torus, if the segment is chosen appropriately among the several
geodesics connecting two points.

\begin{figure}[t]
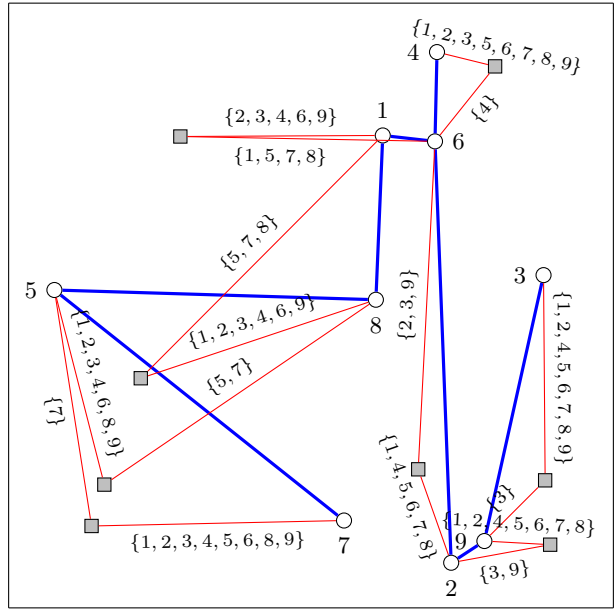

\begin{center}
\if\faifig1

\else [...TikZ\ code...] \fi
\end{center}
\caption{\label{fig.esempioapertura}% 
  Example of matching subgraphs $H$ (red) and $\bar{H}$ (blue).  Here
  $n=8$ and $W_{ij}=\|x_i-y_j\|^2$.  We show the labels of the
  vertices $b \in B$, and the sets $\calC(e)$.  Note that both $H$ and
  $\bar{H}$ are spanning trees (respectively on $G=\calK_{n+1,n}$ and
  on $\bar{G}=\calK_{n+1}$), according to Theorem \ref{thm:Jtree}, and
  that, while the tree $H$ is not guaranteed to be non-crossing
  (because $p=2$, instead of $p=1$ for Corollary \ref{corr.1pnoncross}
  to apply), and indeed it is crossing, the tree $\bar{H}$ is
  non-crossing, in agreement with Theorem~\ref{th:treenoncross}.}
\end{figure}

Theorems \ref{thm:Jtree} and \ref{th:treenoncross} are illustrated by
an example in Figure~\ref{fig.esempioapertura}, where the sets
$\calC(\edge{a_i}{b_j})=\{k\,|\,j=\pi^{(k)}(i)\}$, playing a crucial
role in our proofs, are also shown. The (two) sets $\calC(e)$
associated to the edges $e$ of $H$ incident to a vertex $y_j$ are a
partition of $[n+1]$, while the sets associated to the vertex $x_i$
are a partition of~$[n+1]\setminx i$.

We also prove, in Section~\ref{sec.AssCapa}, that the $(n+1)\times n$
matrix $T$, defined as
$T_{ij}=\frac{1}{n(n+1)}|\calC(\edge{a_i}{b_j})|$ if
$\edge{a_i}{b_j}\in H$ and $T_{ij}=0$ if $\edge{a_i}{b_j}\not\in H$,
is the optimal transportation plan
% for the square Wasserstein distance, 
in the Monge--Kantorovich Optimal Transport Problem with
costs $W(x_i,y_j)=W_{ij}$, among the measures
%
%% $\mu(z)=\frac{1}{n+1} \sum_{i=1}^{n+1} \delta(z-x_i)$ and
%% $\nu(z)=\frac{1}{n} \sum_{j=1}^n \delta(z-y_j)$, 
\begin{align}
\mu(z)&=\frac{1}{n+1} \sum_{i=1}^{n+1} \delta(z-x_i)
\ef;
&
\nu(z)&=\frac{1}{n} \sum_{j=1}^n \delta(z-y_j)
\ef;
\end{align}
and this gives a second interpretation for the graph $H$ (and, in
turn, for $\bar{H}$). We consider this fact as not obvious, as it
relies on a proposition (Corollary~\ref{corr.HnewIsHold_ass})
% Appendix~\ref{app.AssCapa}) 
%
that is rather subtle: the statement holds in our context, but it
doesn't (as some counterexamples easily show) if the hypotheses are
relaxed even mildly, for example in the natural variant of the setting
on $\cK_{n+2,n}$ rather than~$\cK_{n+1,n}$.

%% We also prove that the $(n+1)\times n$ matrix $T$, defined as
%% $T_{ij}=\frac{1}{n(n+1)}|\calC(\edge{a_i}{b_j})|$ if
%% $\edge{a_i}{b_j}\in H$ and $T_{ij}=0$ if $\edge{a_i}{b_j}\not\in H$,
%% is the optimal transportation plan for the square Wasserstein
%% distance, in the Monge--Kantorovich Optimal Transport Problem, among
%% the measures
%% %
%% %% $\mu(z)=\frac{1}{n+1} \sum_{i=1}^{n+1} \delta(z-x_i)$ and
%% %% $\nu(z)=\frac{1}{n} \sum_{j=1}^n \delta(z-y_j)$, 
%% \begin{align}
%% \mu(z)&=\frac{1}{n+1} \sum_{i=1}^{n+1} \delta(z-x_i)
%% \ef;
%% &
%% \nu(z)&=\frac{1}{n} \sum_{j=1}^n \delta(z-y_j)
%% \ef;
%% \end{align}
%% %
%% and this gives a second interpretation for the graph $H$ (and, in
%% turn, for $\bar{H}$). We consider this fact as not obvious, as it
%% relies on a proposition (Corollary~\ref{corr.HnewIsHold_ass} of
%% Appendix~\ref{app.AssCapa}) that is rather subtle: the statement
%% holds in our context, but it doesn't if the hypotheses are relaxed
%% even mildly, for example in the natural variant of the setting on
%% $\cK_{n+2,n}$ instead that $\cK_{n+1,n}$ (as some counterexamples
%% easily show).

\medskip
\noindent
In future work we will investigate the statistical properties of the
graph $\bar{H}$, in the limit of large $n$, where the $2n+1$ points
$x_i$ and $y_j$ are chosen i.i.d.\ uniformly on some domain $\Omega$
in the Euclidean space $\bR^d$. Inspired by previous (and future) work
on a field-theoretical approach to the Euclidean Random Assignment
Problem \cite{Caracciolo2014,Caracciolo2015,Caracciolo:162}, we will
be led to conjecture that $\bar{H}$ is asymptotically distributed as a
Uniform Spanning Tree on $\Omega$, with free boundary conditions, and
that, in a ``symmetric'' variant of the setting above (where adapted
versions of our theorems still hold), that we describe later on in
this paper in Sections~\ref{sec.teoremaserio1b} and
\ref{ssec.noncrossreservoir}, $\bar{H}$ is asymptotically distributed
as a Uniform Spanning Tree on $\Omega$, with wired boundary
conditions.

The latter setting is of special interest for two main reasons.
First, from the classical work of Ajtai, Komlós, and Tusnády
\cite{Ajtai1984} (see also \cite{Caracciolo2014,Ambrosio2019}), it is
well known that, in the two-dimensional Euclidean Random Assignment
Problem, because of the logarithmic divergence of soft modes in the
random density distributions, the edges appearing in $M_\star(G)$, and
in the matching subgraph $H$,
% $H_\calJ$, 
have an `anomalous' factor 
$\sqrt{\ln n}$ w.r.t.\ the typical first-neighbour distances. This
implies non-local fluctuations due to global conservations, preventing
the existence of a well-defined thermodynamic local limit in the sense
of Benjamini--Schramm. Even for $p=1$, where the graph $H$
% $H_\calJ$ 
is non-crossing, the edges form a microscopic peculiar ``zig-zag''
pattern, with combs of almost parallel edges connecting vertices of
opposite vertex classes, oriented along a certain weighted gradient of
the density distributions.  Nonetheless, passing to the forementioned
symmetric setting, and focusing on the graph $\bar{H}$,
% $\bar{H}_\calJ$, 
allows to strip out this behavior, pathological at the aims of local
limits, and may unveil a regular, short-range effective structure, and
it is reasonable to expect that, under such randomized geometric
ensembles, the projected spanning trees 
$\bar{H}$
% $\bar{H}_\calJ$ 
converge in the thermodynamic limit to a canonical tree-valued random
process, and hopefully to the Uniform Spanning Tree (UST).

The second reason is that the boundary conditions on our random tree
$\bar{H}$ in the symmetric setting are analogous to those that are
appropriate for the definition of $\kappa=2$ Schramm--Loewner
evolution
\cite{Wilson1996Generating, Schramm1999ScalingLO, LawSchWerUST}, so
that, under the assumption that our trees $\bar{H}$ converge to the
UST, we can identify suitable paths on $\bar{H}$ as being
% (conjecturally) 
asymptotically distributed as a radial SLE${}_{\kappa=2}$, that is, a
loop-erased random walk (LERW) connecting a point inside $\Omega$ to a
point on~$\partial\Omega$.  This is shown in
Figure~\ref{fig.exLerwVsNoi} on an example.  Furthermore, the planar
dual of the UST with free boundary conditions is the UST with wired
boundary conditions, thus, provided that our conjecture holds, also in
the first setting, up to taking the dual, we can identify a radial
SLE${}_{\kappa=2}$ random curve. (The planar dual of a polygonal
region of the plane can be defined canonically through the notion of
``constrained Delaunay triangulation'' \cite{chewCDT}).  This is shown
in Figure~\ref{fig.exLerwVsNoi2} on an example.

\begin{figure}[t]
\[
\includegraphics[width=.48\textwidth]{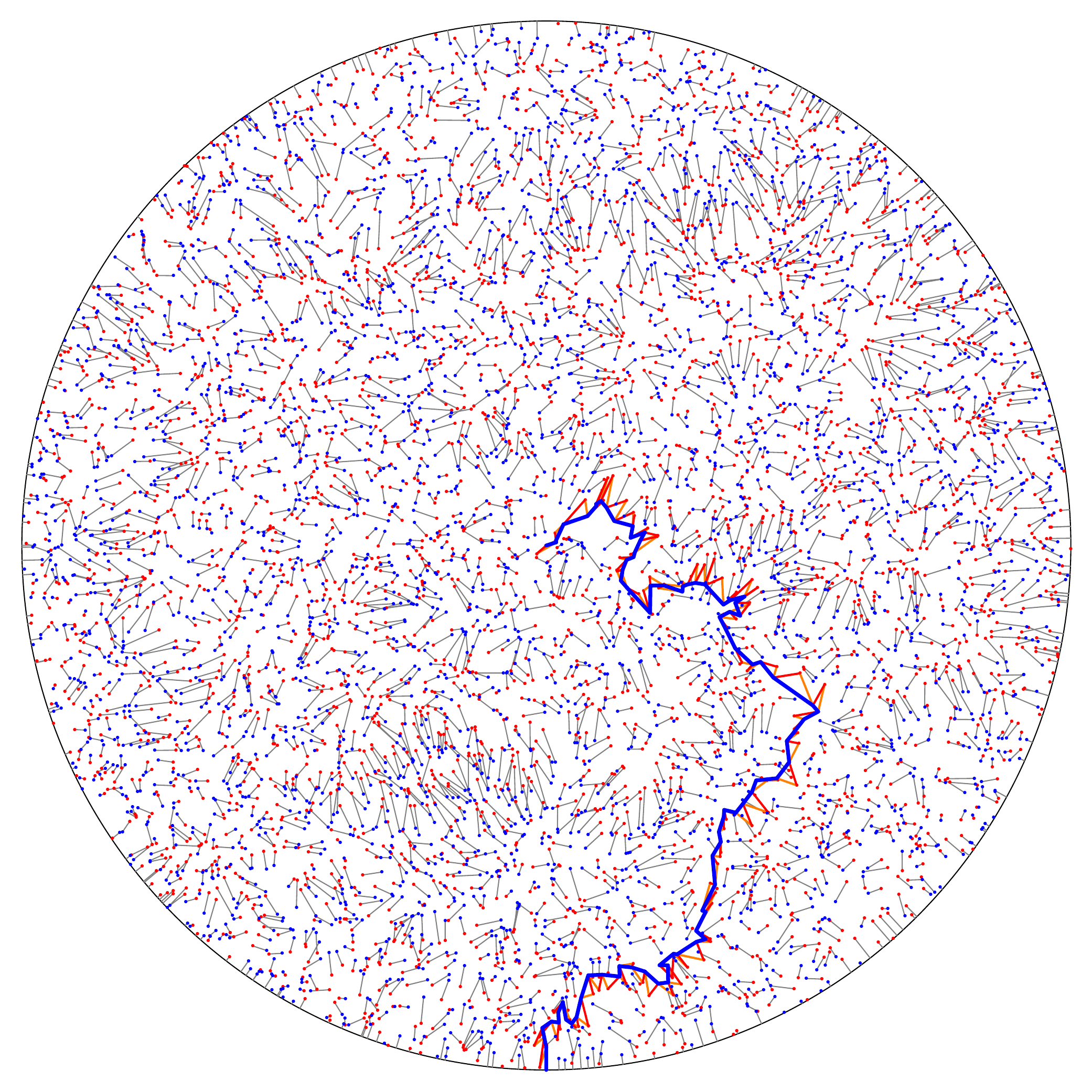}
% fig_exSLEnostro1_n1200.pdf}
\quad
\includegraphics[width=.48\textwidth]{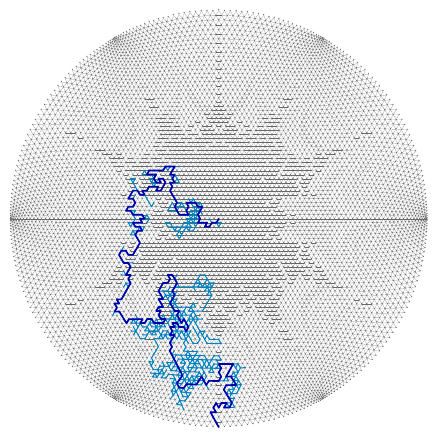}
% fig_lerwHex_n20seed1.pdf}
\]
\caption{\label{fig.exLerwVsNoi}% 
  Left: the path on $\bar{H}$ in the ``symmetric'' setting of
  Section~\ref{sec.teoremaserio1b}, connecting the center of the disk
  to a point on the boundary. Here $n=7650$. Right: a LERW on a
  hexagonal portion ($\ell=50$) of the triangular lattice, from the
  center of the grid, conditioned to reach the midpoint of a side,
  conformally mapped to the disk. The sizes have been chosen so to
  have as many vertices of a given class on the left image, as grid
  points on the right image.  We conjecture that the continuum limit
  of these two random objects do coincide.}
\end{figure}

While a comprehensive discussion of our conjectures and their physical
backing from Statistical Field Theory is deferred to a companion
paper, it is reassuring that the combinatorial theorems proven here
establish a rigorous microscopic foundation for them: they guarantee
that $\bar{H}$ is always a spanning tree (rather than an arbitrary
spanning subgraph), and that in the Euclidean setting at $p=2$ it is
non-crossing (making it entirely plausible that it fills the space
uniformly in a planar way, much like a UST).

In the Statistical Mechanics of Critical Phenomena, SLE curves, in one
of their two main incarnations (radial and chordal), have emerged as
the continuum limit of suitable paths on a regular lattice, associated
to random configurations at the critical temperature
\cite{KagerNien,cardySLErew}.  In the Statistical Mechanics of
Disordered Systems, there has been some activity in trying to identify
SLE curves, now for the ground state of the system, that is, random
configurations from the Gibbs ensemble, within random instances, in
the limit of zero temperature.  However, this line of research has
proven much more challenging than the one for ordered systems, even at
the level of conjectures and numerical results.  To our knowledge,
there is to date no clear construction of a disordered system which is
critical at zero temperature, and for which certain lattice paths
associated to the ground state may converge in the continuum limit to
a radial or chordal SLE$_{\kappa}$, satisfying Schramm's left-passage
probability formula, not even at a conjectural level.

It is precisely this long-standing gap in the study of
zero-temperature disordered systems that underscores the significance
of our findings: in constrast with the limited theoretical
understanding of the other disordered models considered so far --- the
random-bond Ising model, the Minimal Spanning Tree, or the standard
(non-bipartite) matching problem --- our construction may provide a
setting with an unexpected analytical control, which may elevate the
curves of the Assignment Problem from heuristic (and mostly numerical)
observations of random fractal objects to a definite setting for
genuine radial SLE, where exact predictions can be rigorously tested.

\begin{figure}[t]
\[
\includegraphics[width=.48\textwidth]{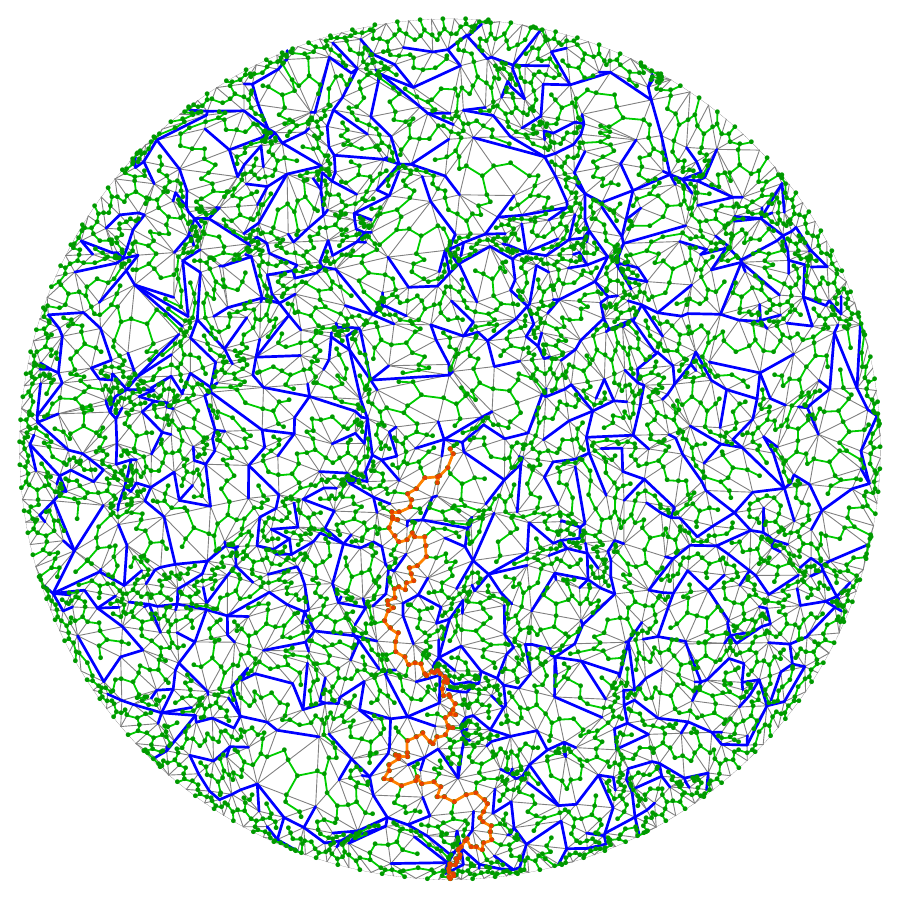}
% albdual_n1140s1L128_lato_clean2.pdf
\quad
\includegraphics[width=.48\textwidth]{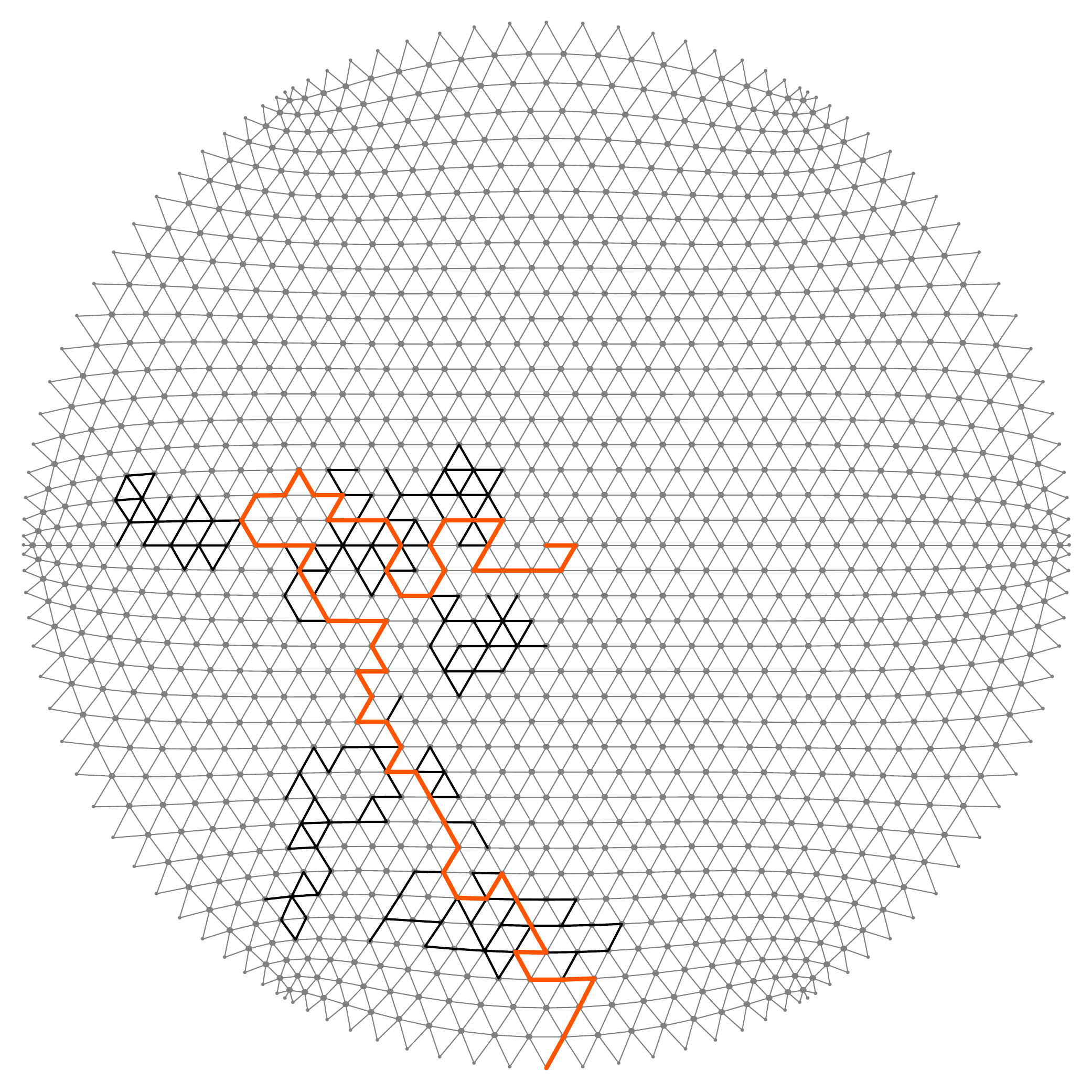}
% fig_lerwHex_n20seed1.pdf}
\]
\caption{\label{fig.exLerwVsNoi2}% 
  Left: the path on the dual of $\bar{H}$ in the setting of
  Section~\ref{sec.teoremaserio1}, connecting the center of the disk
  to a point on the boundary. The dual triangulation is the
  constrained Delaunay triangulation associated to 
  $\mathbb{D} \setminx \bar{H}$, with vertices given by the vertices
  of $\bar{H}$, plus $\sim \sqrt{n}$ extra points on the boundary of
  the disk.  Here $n=1140$. Right: a LERW on a hexagonal portion
  ($\ell=20$) of the triangular lattice, from the center of the grid,
  conditioned to reach the midpoint of a side, conformally mapped to
  the disk. The sizes have been chosen so to have as many vertices of
  each class on the left image, as grid points on the right image.  We
  conjecture that the continuum limit of these two random objects do
  coincide.}
\end{figure}

%%%%%%%%%%%%%%%%%%%%%%%%%%%%%%%%%%%%%%%%%%%%%%%%%%%%%%%
\subsection{A perspective from Statistical Mechanics}
\label{sec.statmechintro}
%%%%%%%%%%%%%%%%%%%%%%%%%%%%%%%%%%%%%%%%%%%%%%%%%%%%%%%

\noindent
In this section, we describe in more detail the bibliographic framework
for our main conjectures (which, in turn, are supported by the theorems
proved in this paper). Readers who are not particularly interested in
Statistical Mechanics, or who wish to dive straight into the mathematical
core, may safely skip this section without loss of continuity.

Identifying the ground state of a disordered system can be a
challenging task. In particular, this is the case, even in
average-case complexity, when the underlying Statistical Mechanics of
the model displays \emph{spontaneous symmetry breaking}
\cite{mezard1987spin,bookMontMez}, making the problem notoriously
intractable (despite recent astonishing breakthroughs, such as
Montanari's algorithm finding near-optimal solutions for the
Sherrington--Kirkpatrick Hamiltonian in $\sim C(\eps)\, n^2$ time under
the full RSB assumption~\cite{montanari_optSK}).

This problem may become more tractable (and by less sophisticated
means) for disordered systems that, unlike the
Sherrington--Kirkpatrick model, are non-critical at every positive
temperature, but become critical (and thus interesting) in the limit
of zero temperature. And indeed,
% Luckily enough, 
there are a few disordered
systems, even in finite dimension, that are conjectured to have such a
property. On top of this, for such systems the study of the ground
state implicitly becomes equivalent to the study of the Gibbs ensemble
at the critical temperature, adding a further motivation to this
investigation.

However, typically the ground state in itself is not very illustrative
of the thermodynamic behaviour of the system, and what really reveals
scale invariance and long-range correlations is the study of
minimal-energy excitations of the ground state, or, in other words,
the comparison between the ground state of the original instance and
that of an instance perturbed by some elementary modification
\cite{Houdayer98,seppala_PhysRevE.58.R5217,middleton95_PhysRevE.52.R3337}.

This study of \emph{excited states} has been applied also to the
Minimum Matching Problem \cite{Houdayer98,Caracciolo2021,Kahlke_2023}.
These papers perform a slight variant of what we are doing here: while
in the existing literature the elementary perturbation consists of
forbidding an edge belonging to the ground state of the original
instance, in our case a perturbation of the instance consists of
modifying the costs of the edges incident to one given vertex
(possibly all of them). While this small variation has a major impact
at the level of exact combinatorial results, it does not significantly
change the underlying Statistical Mechanics of the system. 
In particular, it remains true that the symmetric difference of the
two optimal solutions is composed of a single cycle or path
graph,\footnote{Geometrically, replacing the $i$-th row with the
  $j$-th row corresponds to removing a vertex $v_i$ and introducing
  $v_j$, yielding an open path connecting $v_i$ to $v_j$. Abstractly,
  since $v_i$ and $v_j$ are temporarily identified in $K_{n,n}$, the
  symmetric difference forms a single closed cycle.}
%% \footnote{To interpret the costs geometrically as distances in a
%%   geometric setting, one may view the perturbation as removing a
%%   vertex $v_i$ and introducing a new vertex $v_j$. Within the abstract
%%   framework of the linear Assignment Problem, this corresponds to
%%   replacing the $i$-th row of the cost matrix with the $j$-th
%%   row. Under this \emph{geometric} persepective, the symmetric
%%   difference of the two optimal solutions is composed of a single open
%%   path connecting $v_i$ to $v_j$. However, the vertices $v_i$ and
%%   $v_j$ are (temporarily) identified in the underlying abstract
%%   complete bipartite graph $K_{n,n}$, so that, in the \emph{abstract}
%%   perspective, the symmetric difference between the two optimal
%%   matchings forms a single closed cycle rather than an open path.}
whose asymptotic geometric structure reveals the critical properties
of the model.

For models on planar graphs, this structural property makes the study of
excitations in the Minimum Matching Problem remarkably similar to that of
interfaces and domain walls in standard Statistical Mechanics models.

Starting from 1999, through the seminal work of Oded Schramm
\cite{Schramm1999ScalingLO}, interfaces have started to play a
prominent role in the rigorous understanding of criticality in
Statistical Mechanics. Namely, \emph{Schramm--Loewner Evolution}
curves (SLE), in one of their two main incarnations (radial and
chordal), have emerged as the scaling limit of suitable paths on a
regular lattice associated with random configurations at the critical
temperature \cite{KagerNien,cardySLErew}. This has paved the way for
rigorous proofs of conformal invariance in the continuum limit of
two-dimensional Statistical Mechanics models undergoing second-order
phase transitions --- most notably the families corresponding to the
$q$-state Potts model in the range $0\leq q \leq 4$, and the $O(n)$
loop model in the range $-2\leq n \leq 2$, which contain the
$M_{p,p+1}$ minimal models of conformal field theory as discrete
subsets. To date, this approach has led to rigorous convergence proofs
to SLE in three landmark cases, listed in increasing order of
technical difficulty: the uniform spanning trees / loop-erased random
walks, site percolation on
the triangular lattice, and the critical Ising model.\footnote{These
  three cases correspond exactly to the first examples of minimal
  models in the $A$-series, associated with the values $q = 4
  \cos^2\left(\frac{\pi}{m}\right)$ for $m=2, 3, 4, \dots$ in the
  Fortuin--Kasteleyn representation of the Potts model. More
  precisely:
\begin{itemize}
    \item $q \to 0$ (Uniform Spanning Trees / LERW): corresponds to
      the first building block ($m=2$, central charge $c=-2$
\cite{Gaberdiel1999-sx,Pearce2006-vq}, and
      value of SLE parameter $\kappa=2$ \cite{Schramm1999ScalingLO,LawSchWerUST}).
    \item $q = 1$ (Percolation): corresponds to the second case ($m=3$, $c=0$, and $\kappa = 8/3$).
    \item $q = 2$ (Ising Model): corresponds to the third case ($m=4$, $c=1/2$, and $\kappa = 3$).
\end{itemize}
Thus, this progression is not merely a scaling of technical
difficulty, but precisely reflects the natural sequence of CFT minimal
models (the $M_{p,p+1}$ series) when exploring the ADE
classification.

We apologize for providing references only for the UST case, more
pertinent to the present paper, in order to avoid a too lengthy
digression.}

As anticipated above, there have been attempts in the recent
literature to establish a similar connection within the Statistical
Mechanics of Disordered Systems. In this case the natural idea is of
trying to identify SLE curves associated to the ground state of the
system (or, rather, to elementary excitations of it), that is, random
configurations in the limit of zero temperature, within random
instances (indeed, various disordered systems are known or conjectured
to be critical at zero temperature in dimension $d=2$)
\cite{Amoruso2006Conformal,BernLed,PhysRevB97064410,
% PhysRevB97064410,
Caracciolo2021,Kahlke_2023}.
% slides here
% https://www.physik.uni-leipzig.de/~janke/CompPhys12/Folien/khoshbakht.pdf
%

So, in short, the natural line of thought, that has emerged in the
literature, is as follows. First, one identifies a model that is
``critical at zero temperature'', i.e.\ its ground state has
long-range correlation functions (note that even this first step is
most often conjectural).  Then, from general arguments of scale
invariance in two-dimensional systems, one conjectures that the model,
in its thermodynamic limit (for diverging system size, or equivalently
for the lattice spacing going to zero) has a form of covariance under
conformal transformations.  Then, one tries to construct a simple
(i.e., non self-intersecting) curve in the discrete model, determined
unambiguously from a configuration, that, in the limit, may converge
to a SLE random curve, for some value of $\kappa$ that is either
conjectured \emph{a priori} from the known properties of the critical
model, or fitted from the data. Such a possibility is then tested
numerically against various statistical properties that SLEs must
have.

The heuristic tests to determine $\kappa$ for a random curve that is
conjectured to be a SLE in some continuum limit are not all on the
same ground: some tests (like the basic test of conformal covariance,
and the determination of the Hausdorff dimension
$d_{\rm f}=1+\frac{\kappa}{8}$) only address bulk properties of the
model, and thus may predict the ``good'' value of $\kappa$, even for
curves which are not genuinely SLE's, because the precise description
fails in the prescription of boundary conditions that satisfy the
crucial ``domain Markov property'' for Schramm's setting to work. A
very solid test is one based on the left-passage probability. This
quantity is both quite constraining (as, for radial SLE, it predicts a
two-dimensional surface, while for chordal SLE it predicts a
probability distribution on an interval, in contrast with $d_{\rm f}$,
that provides just one numerical value), and extremely sensitive to
boundary conditions. For example, chordal SLEs associated to USTs with
wired boundary conditions have left-passage probability appropriately
given by Schramm's formula at $\kappa=2$, namely
$p_{\rm left}(\theta)= \frac{\pi + 2 \theta + \sin(2 \theta)}{2 \pi}$ for 
$\theta \in [-\pi/2, \pi/2]$, in the upper half-plane $\mathbb{H}$,
while the analogous quantity for the path associated to USTs with free
boundary conditions gives the different rather simple expression
$p_{\rm left}(\theta)=\frac{\pi + 2 \theta}{2 \pi}$ (which,
accidentally, is the Schramm formula for $\kappa=4$).

To our knowledge, there is to date no clear construction of a
disordered system which is critical at zero temperature, and for which
certain lattice paths associated to the ground state may converge in
the continuum limit to a radial or chordal SLE$_{\kappa}$, satisfying
Schramm's left-passage probability formula, not even at a conjectural
level.  The best result in this direction is in \cite{BernLed}, where
it is shown that, indeed, boundary conditions with fixed endpoints
lead to heuristic fitted values of $\kappa$ that are inconsistent
across different tests, while one boundary condition that leads to
possibly consistent values involves ``floating'' endpoints, on
distinct boundaries of a doubly-connected domain (i.e., a cylinder),
that is a considerably more subtle setting, both because the domain is
not simply-connected \cite{Hagendorf2008-lc}, and because the
endpoints are not prescribed \cite{BBHdipolarSLE}. The value of
$\kappa$ determined numerically from $p_{\rm left}$ is in the range
$2.24 < \kappa < 2.40$, consistent with the value $d_{\rm f}=1.28(1)$
measured separately in the paper (and thus implying $2.16 < \kappa <
2.32$), but not consistent with the recent more precise estimates
\cite{PhysRevB97064410} $d_{\rm f}=1.27319(9)$,
implying $2.1848 < \kappa < 2.1862$ (that is, the two values
determined by two distinct statistical tests differ by roughly 2
standard deviations). Furthermore, there is little hope to identify a
minimal model that may correspond to this theory. Given the formulas
for the central charge,
% $c(\kappa)=\frac{(3\kappa-8)(6-\kappa)}{2\kappa}$ and
$c(\kappa)=(3\kappa-8)(6-\kappa)/(2\kappa)$ and
% $c_{p,q}=1-\frac{6(p-q)^2}{pq}$, 
$c_{p,q}=1-6(p-q)^2/(pq)$, the only possibility is the exotic choice
$(p,q)=(3,5)$, that gives
$\kappa(c)=\kappa(-\frac{3}{5})=\frac{12}{5}=2.4$, barely at the limit
of the fit from $p_{\rm left}$ in \cite{BernLed}, and largely ruled
out by the fit for $d_{\rm f}$ in \cite{PhysRevB97064410}.  On top of
this, as elucidated in the conclusions of \cite{PhysRevB97064410},
previous attempts to identify the scaling dimension of the energy
operator with an entry in the Kac table, while consistent with early
low-precision data, are now ruled out by the more precise
measurements.

%% However, this line of research has proven much more challenging than
%% the one for ordered systems, even at the level of conjectures and
%% numerical results, with the numerical values of $\kappa$ found so far
%% being ``exotic'' (instead that rationals associated to the ADE
%% classification), and even inconsistent for different numerical tests
%% applied to the same given system, while the geometry of the curve is
%% mostly forced to fall out of the standard setting.

Conversely, in our model we provide two distinct constructions for
curves that we expect to be genuine \emph{radial SLE} on a simply
connected domain, e.g.\ on the unit disk.  On top of this, we do not
perform any numerical fit for the parameter $\kappa$, unlike common
practices in the study of disordered systems (such as those in
Refs.~\cite{Amoruso2006Conformal, BernLed, PhysRevB97064410,
  Kahlke_2023}),
%   Caracciolo2021, 
but rather test the \emph{a priori} analytical prediction that
$\kappa=2$, which, as mentioned above, is the value that is rigorously
established for Uniform Spanning Trees.

%%   This value is rigorously established for Uniform Spanning Trees
%% (e.g., through the exact scaling limit results by Lawler, Schramm,
%% and Werner %\cite{Wilson1996Generating, LawSchWerUST}),
%% \cite{LawSchWerUST}), corresponding to the well-studied non-unitary
%% Logarithmic Conformal Field Theory at $c=-2$ associated with
%% $(p,q)=(1,2)$ (see e.g.~\cite{Gaberdiel1999-sx,Pearce2006-vq}).

Having at one's disposal a natural candidate value of $\kappa$ is a
major advantage that the former papers did not have, of course not
because of any methodological shortcomings in the analysis, but rather
in light of the complexity of the underlying physical models
considered there (like the random-bond Ising model, and analogously to
what is found in \cite{Caracciolo2021} for the Random Dimer Model in
the non-bipartite case),
% standard non-bipartite Assignment Problem), 
that lack an exact algebraic backbone.

This bipartite\,/\,non-bipartite structural gap in the Minimum
Matching Problem comes as no surprise from a Combinatorial
Optimization perspective.  Note for example how the extensive
monography of Schrijver \cite{schrijver2003combinatorial} contains a
``Part II: Bipartite Matching and Covering'', and a ``Part III:
Nonbipartite Matching and Covering''. Furthermore, from the point of
view of Statistical Physics, it is also inspired by an analogous
phenomenon in Dimer Models (although presumably due to different
microscopic reasons), where, in the frameworks of Kenyon, Okounkov and
Sheffield
\cite{Kenyon2001Dominos,KenyonOkounkovSheffield2006,KenyonOkounkov2007},
while combinatorial Pfaffian techniques formally survive in the planar
non-bipartite setting, the rich geometric structure that allows clean
conformal invariance to emerge, through the introduction of an
auxiliary height function, is entirely unlocked only when one
focalizes on periodic bipartite lattices.

%% Our ability to conjecturally hit the exact entry in the ADE/CFT
%% classification grid is thus rooted in an analogous fortunate
%% structural choice.

On top of the fact that we can conjecturally identify `some' candidate
limit of CFT minimal model, in the attempt of a theoretical
understanding of interfaces in ground states of disordered systems,
the specific conjectured value of $\kappa=2$ has crucial advantages.
Indeed, establishing the scaling limit of discrete objects in
Statistical Mechanics toward continuous objects like the
Schramm-Loewner Evolution (SLE) under the rigid constraints of pure
Gromov--Hausdorff (or Gromov--Hausdorff--Prokhorov) convergence is
typically too complex, and requires a control of $n$-point observables
that is beyond current possibilities.  To bypass these topological
obstructions, specifically for the case of Uniform Spanning Trees and
$\kappa=2$, a more viable pathway relies on \emph{effective resistance
  metrics}, which naturally mirror the Laplacian structure of Uniform
Spanning Trees and behave more flexibly under renormalization.

Nonetheless, moving from mean-field approximations or single-point
field-\!\! theoretic limits to a fully controlled metric convergence
requires bridging a formidable analytical gap. While optimal transport
tools, such as Brenier potentials, guarantee robust \emph{a priori}
regularity, analyzing infinitesimal variations between instances leads
to nonlinear PDEs whose sensitivity analysis directly mirrors the
oriented paths of our projected trees $\bar{H}$. While technically
demanding, this framework may provide the appropriate tools needed to
understand our candidate SLE curves as objects in metric spaces.

We conclude this section by observing that the model of Uniform
Spanning Trees has proven to be quite transversal in Statistical
Mechanics, with relations not only with the $O(n)$ dense loop model in
the limit $n \to 0$ and in the Potts Model in the limit $q \to 0$
% [ ALAN REVIEW ON POTTS ]
\cite{Sokal2005-yp}, and the related field theory of ``free scalar complex Grassmann
fermions'' 
\cite{Caracciolo2007-mc},
% [ NOI JPHYSA ], 
in ordinary (ordered, at equilibrium)
Statistical Mechanics, but also with an important model from
\emph{non-equilibrium} Statistical Mechanics, namely the Abelian Sandpile
Model 
\cite{Dhar1999-ys}.
%[ DHAR REVIEW ]. 
This would be the first case in which it is also connected to the
(equilibrium) Statistical Mechanics of a \emph{disordered system}.

The facts illustrated above altogether motivate the importance of the
conjectures presented in this paper.

%%%%%%%%%%%%%%%%%%%%%%%%%%%%%%%%%%%%%%%%%%%%%%%%%%%%%%%
\section{Matchings on modified graphs}
\label{sec.setting}
%%%%%%%%%%%%%%%%%%%%%%%%%%%%%%%%%%%%%%%%%%%%%%%%%%%%%%%

%\noindent
\subsection{Matchings on generic-weighted graphs}
Let $G=(V,E)$ denote a (finite, simple, loopless, undirected)
graph\footnote{In this paper, unless clear from the context, we use
  the notation $V_G$ and $E_G$, or $V(G)$ and $E(G)$, to denote the
  vertex- and edge-sets of a graph $G$, then $\edge{u}{v}$ is the edge
  with endpoints $u$ and $v$ (unique if $G$ is simple), and
  $\deg_G(v)$ denotes the degree of the vertex $v$ in $G$.}  with
vertex set $V$ and edge set $E\subseteq\{\edge{u}{v}\mid u,v\in
V,\ u\neq v\}$.
% , intended as a collection of unordered pairs of $V$, so that
% $e=\{u,v\}\in E$ is said to have $u,v\in V$ as endpoints. 
We say that $G$ is \emph{bipartite} if there exist $A,B\subset V$ that
partition $V$, i.e., $V=A\cup B$ and $A\cap B=\emptyset$, such that
$E\subseteq\{\edge{u}{v}\mid u\in A,\ v\in B\}$: in this case we will
adopt the notation $\bipg{G}{A}{B}{E}$. The graph is equipped with a
set of real \emph{edge-weights} $w$, i.e., to each edge
$e=\edge{u}{v}\in E$ is associated a \emph{weight}
$w_e\equiv w_{u,v}\equiv w_{v,u} \in \bR$. We will consider spanning
subgraphs $H \subseteq G$, i.e., $H=(V,E_H)$ for some 
$E_H\subseteq E$. For a subgraph $H$, its \emph{weight}, or
\emph{cost}, is the sum of the weights assigned to the edges in its
edge set,
$w(H)=\sum_{e\in E_H} w_e$.\footnote{In Statistical Mechanics,
  normally the weights are multiplicative, not additive. However, the
  nomenclature has been inherited from the literature on Combinatorial
  Optimization.}
This setting arises in several problems within Combinatorial
Optimization, where the goal is to find the subgraph $H$ in a given
family that has the minimum weight. Variants of this paradigm include
a number of well-known problems with a wide range of algorithmic
complexities, such as, in order of increasing complexity, the Minimum
Spanning Tree Problem, the Minimum-Weight Perfect Matching Problem or
its generalization, the Minimum-Weight $k$-factor Problem, and the
Travelling Salesman Problem (that is, the Minimum-Weight Hamiltonian
Cycle Problem)
\cite{lawlere:76, citeulike:472316, lovasz2009matching,
  schrijver2003combinatorial}.  Of these problems, only the latter is
\textsf{NP}-complete. Somewhat surprisingly, even the apparently
similar Chinese Postman Problem can be solved in polynomial time. In
fact, a crucial step in solving the latter relies on the solvability
of the Minimum-Weight Perfect Matching Problem.

In this paper, we focus on the Minimum-Weight Perfect Matching
Problem, that we now define.

A \emph{perfect matching} $M=(V,E_M)$ in a graph $G=(V,E)$ is a
spanning subgraph $M\subseteq G$ such that $\deg_M(v)=1$
% \coloneqq|\{e\in\ E_M\colon v\in e\}|=1$ 
for all $v \in V$. Such a configuration is called an \emph{assignment}
if $G$ is bipartite, a \emph{1-factor}, within the theory of Graph
Factorization, and a \emph{dimer covering} in the context of
Statistical Mechanics.\footnote{Note that, in Statistical Mechanics,
  it is natural to consider \emph{multiplicative} weights in the form
  $\omega_{e}=\exp(-\beta w_{e})$, where $\beta$ is the inverse of the
  temperature, so that the Gibbs measure of a subgraph is proportional
  to the \emph{product} of the edge weights $\omega_{e}$, while it is
  the \emph{Hamiltonian} which is defined as the sum of the parameters
  $w_e$, that would be called the \emph{energies} of the edges. The
  optimal configuration would be called the \emph{ground state} of the
  system.}  We will denote by $\mathcal{M}(G)$ the (possibly empty)
set of perfect matchings on $G$. According to the paradigm above, we
say that a perfect matching $M \in \mathcal{M}(G)$ has weight
\be
w(M)\coloneqq \sum_{e\in E_M}w_e
\ef.
\ee
The set $\calM_\star(G)$ of
\emph{minimum-weight perfect matchings} on $G$ is given by the set
of $M\in \mathcal{M}(G)$ realising the minimum of the function $w$.
% such that $w(M)=\min_{\hat M\in\calM(G)}w(\hat M)$. 
If $\calM_\star(G)$ has cardinality 1, we denote by $M_\star(G)$
its unique element.

A complete analysis of this problem is provided in
\cite[Ch.\ 24--26]{schrijver2003combinatorial}, and, in the
structurally simpler (bipartite) case of Assignment, in
\cite[Ch.\ 17--18]{schrijver2003combinatorial}.

When the weighted graph $G$ is the balanced complete bipartite graph
$\calK_{n,n}$, its structure is completely encoded by the 
$n \times n$ matrix of weights $W$, and the set $\calM(G)$ is
identified with $\mathfrak{S}_n$, that is, $w(M(\pi))=\sum_{i=1}^n
W_{i\,\pi(i)}$. The two notations (the graph one and the matrix one)
will be both used in our forthcoming study of the Assignment
Problem.\footnote{Along this paper we will use the following customary
  notation: given a matrix $W \in \mathbb{K}(n,m)$, and a set $I
  \subseteq [n]$ of rows, and $J \subseteq [m]$ of columns, we will
  denote by $W^{I|J}$ the minor of $W$ with the rows with indices in
  $I$, and the columns with indices in $J$, removed.}

A recurring notion is the following:
\begin{defn}[Alternating sum of weights on a cycle]
\label{def.genwei1pre}
Let $C=(e_1,e_2,\dots,e_{2\ell})$ be a (rooted oriented even-length)
cycle of $G$ (with edges labeled in cyclic order). The
\emph{alternating sum} $\walt(C)$ is the expression: 
\be 
\walt(C) =
\sum_{i=1}^{2\ell} (-1)^i w_{e_i}
\ef.
\label{eq:altsum}
\ee
\end{defn}
\noindent 
An elementary fact, used repeatedly along this paper, is the following
% that has already been used implicitly in Proposition
% \ref{prop.38765874} 
\begin{prop}
\label{prop.swapcyc}
Let $G$ be as above, $\cM(G)\neq\emptyset$ and $M\in \cM_{\star}(G)$.
Then there exists no cycle $C=(e_1,e_2,\dots,e_{2\ell})$
such that all the $\ell$ edges $e_{2i-1}$ are in $M$, and
$\walt(C)>0$.
\end{prop}
\begin{proof}
Otherwise the subgraph 
$M'=M \setminx \{e_{2i-1}\}_{1\leq i \leq \ell} \cup \{e_{2i}\}_{1\leq i \leq \ell}$,
that is also a matching, would have a smaller cost, namely
$w(M')=w(M)-\walt(C)$, thus contradicting the optimality of $M$.
\end{proof}
\noindent 
Arguments of this sort will appear repeatedly along this paper, and
will be complemented by more subtle variants starting from
Section~\ref{sec.teoremaserio1}.  

% \noindent
This paper proves a number of theorems that hold under a suitable
generality hypothesis on the graph weighting, that in particular
ensures the uniqueness of the minimum-weight perfect matching
(provided that it exists). We state this hypothesis precisely here:
\begin{defn}[Generic weights] 
\label{def.genwei1}
The weights of a weighted graph $G$ are
  \emph{generic} if, for every even-length cycle
  $C=(e_1,e_2,\dots,e_{2\ell})$ in $G$ (with edges labeled in cyclic
  order), we have that
% the alternating sum $\walt(C)$ is not vanishing:
% \emph{alternating sum} $\walt(C)$ is not vanishing:
\be
\walt(C)
% = \sum_{i=1}^{2\ell} (-1)^i w_{e_i} 
\neq 0
\ef.
\label{eq:altsumrepeat}
\ee
\end{defn}
\noindent 
We call this condition ``generic weights'' because it holds almost
surely in several randomized versions of the problem, when the weights
are random real variables (and in particular when they are
i.i.d.\ variables sampled from a measure with no point masses), so,
again with an eye on applications in Statistical Mechanics, it is not
an unnatural condition to require.

Furthermore, weight functions which are not generic are the union of
hyperplanes within $\bR^{E}$, thus they are a submanifold with
codimension 1, and several results for the general situation can be
deduced by continuity from the generic-weight case (often with strict
inequalities that become weak inequalities). In most of our proofs,
the generic-weight assumption is just adopted in order to ``break the
tie'' among equally valid choices, in a way that remains consistent
all along some construction procedure.
In particular, as anticipated, generic weights guarantee that the
minimum is non-degenerate:
\begin{prop}
\label{prop.38765874}
If $G$ has generic weights and $\calM(G)\neq\emptyset$, then $|\calM_\star(G)|=1$.
\end{prop}
\begin{proof}
Suppose that two distinct matchings $M_1,M_2\in\calM_\star(G)$
exist. Their symmetric difference 
%
%% $M_1\symdif M_2=\bigcup_{1 \leq j \leq k}C_j$ 
$M_1\symdif M_2=C_1 \cup \ldots \cup C_k$
is a non-empty subgraph
with $k>0$ connected components, in which each connected component
$C_j$ is a cycle of even length, with edges alternately in
$M_1$ and $M_2$.  If the weights are generic, the alternating sum on
each of these cycles $C_j$ is non-zero. However, none of the $C_j$'s
can be present, in light of the optimality of $M_1$ or $M_2$ (depending on
the sign of $\walt(C_j)$), and Proposition~\ref{prop.swapcyc}, so that
$k$ must be zero.
\end{proof}

%-------------------------------------------------------
\subsection{Admissible removals and matching subgraphs}
%-------------------------------------------------------
Let us now introduce the customary
notion of \emph{restriction} of a graph to a subset of its vertices.

\begin{defn}[Restriction]Let $G=(V,E)$ be a graph. We denote by
  $G_U=(V\setminx U; E_{G_U})$ the \emph{restriction} of $G$ to the
  complement of $U$ in $V$, 
  that is, the edges of $G_U$ are the edges of $G$ such that none of the endpoints is in $U$.
%   $E_{G_U}\coloneqq\{e\in E_G\ \vert\ e\cap U=\emptyset\}$.
\end{defn}

\begin{defn}[Admissible removal]
  Let $G=(V,E)$ be a graph. A subset $U\subseteq V$ is an
  \emph{admissible removal} if the restriction $G_U$ allows for a
  perfect matching. We denote by $\calV(G)$ the set of admissible
  removals for $G$, that is
\be
\calV(G)\coloneqq\{U\subseteq V\ \vert\ 
\calM(G_U)\neq\emptyset\}.
\ee
\end{defn}
\noindent
As the condition of having generic weights is clearly hereditary,
(i.e., if $G$ has generic weights, also all $G_U$, with weights given
by the obvious restriction, have generic weights), under the sole
assumption of generic weights for $G$, we are entitled to call
$M_U=M_{\star}(G_U)$ the (unique) optimal matching in $G_U$ for an
admissible removal
$U \in \calV(G)$. This is the central notion of this paper, as in the
following we will deal with the collections $\{M_U\}_{U \in \calJ}$
for subsets of admissible removals $\mathcal J\subseteq\calV(G)$. In
particular, we shall introduce:
\begin{defn}[Matching subgraph] 
\label{def.matchsubg}
The datum of a graph $G=(V,E)$ with
  generic weights $w$ and a set $\calJ\subseteq \calV(G)$ of
  admissible removals, determines a subgraph
\be
H_{\calJ} \coloneqq \bigcup_{U\in\calJ}M_U.
\ee
We call such a subgraph the \emph{matching subgraph} determined by the triple
$(G,w,\calJ)$.
\end{defn}

\noindent The subgraph $H_{\calJ}$ is therefore the union of a
collection of perfect matchings obtained by altering $G$ via
admissible removals. 

Whenever $\bigcap_{U\in\calJ}U=\emptyset$, $H_{\calJ}$ is a spanning
subgraph of~$G$.
%, and we say that $\calJ$ is \emph{spanning} for $G$ in this case.
This is the most relevant situation in the applications which
are of interest to us (but not necessarily in the intermediate
situations introduced in the proofs).

Later on in Theorem \ref{thm:Jtree} on page \pageref{thm:Jtree} (and
in Theorem \ref{thm:Jtreesym} on page \pageref{thm:Jtreesym}) we will
prove that, in two certain interesting situations, the matching
subgraph is a tree. When this will be established, for brevity we may
just call $H_\calJ$ the \emph{matching tree} associated to the
instance.

\label{pag.matrcosts}
Up to completing the weight function with sufficiently high values, we
can always imagine to ``extend'' a graph $G=(V,E)$ to the complete
graph $\cK_n$, or the complete bipartite graph $\cK_{n,m}$.  In the
complete bipartite case, the weights $w_{ij}$ are conveniently encoded
in the $n \times m$ \emph{matrix of costs} $W$, with
$W_{ij}=w_{\edge{a_i}{b_j}}$.  In this case, the check that each
$\calM(G_U)$, for $U\in {\calJ}$, is non-empty, reduces to some
trivial conditions
% ($|U|\equiv n\ \textrm{(mod 2)}$ 
($n-|U|$ even in the first case, and $|U\cap A|-|U\cap B|=n-m$ in the
second case).  On the other hand, as we have anticipated, the
genericity of the weights holds almost surely in various randomized
settings. Thus in most concrete applications the verification of
the hypotheses in Definition~\ref{def.matchsubg} (that 
$\calJ \subseteq \calV(G)$ and that the weights are generic)
is not problematic.

Note that, while $H_{\calJ}$ is the union of optimal configurations
for several slightly different instances of the Matching Problem, it
is not clear that it can also be characterized as the optimal
configuration of one single instance of some (other) problem in
Combinatorial Optimization.  As we will see in
Corollary~\ref{corr.HnewIsHold_ass} of 
Section~\ref{sec.AssCapa}, 
% Appendix~\ref{app.AssCapa}, 
in the main setting for a matching subgraph studied in this paper this
is indeed the case (but we are not aware of a positive answer in the
most general setting).

\begin{ex}[Square lattice]
\label{ex:square}
Let $G$ be a square portion of side $2\ell+1$ of the square
lattice. This graph is bipartite and we say that the vertices in the
same class of the corners have `even parity', whereas the remaining
vertices have `odd parity'. We have therefore $(2\ell+1)\times
(2\ell+1)$ vertices, of which $2\ell^2+2\ell$ with odd parity and
$2\ell^2+2\ell+1$ with even parity. A set $U$ is an admissible removal
only if the even vertices in $U$ are one more than the odd ones. In
particular, while $G$ itself does not allow for a perfect matching
(i.e., $\varnothing \not\in \calV(G)$), the set $\calJ=\{ \{v\}
\;\vert\; v \in V \textrm{~even} \}$ is a set of admissible
removals. 
An example of configuration is in Figure~\ref{fig:triangle}, left.
\end{ex}

\begin{figure}[t]
\includegraphics[scale=2.2]{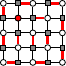}
%[height=0.2\columnwidth]
\qquad
\includegraphics[scale=2.2]{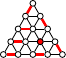} 
%[height=0.2\columnwidth]
\caption{Examples of graphs that have admissible removals consisting
  of a single vertex. (\emph{Left}) A square portion of the square
  lattice of side $5$. Even\,/\,odd, vertices are shown as circles\,/\,squares. Here, all circles are admissible
  removals. (\emph{Right}) A triangular portion of the triangular
  lattice of side $5$. This graph is in fact tripartite, 
%(with circles, squares and diamonds corresponding to different vertex classes), 
although this property does not play an important role here, and each
singleton set in each class is an admissible removal.}\label{fig:triangle}
\end{figure}

\begin{ex}[Triangular lattice]
\label{ex:triangular}
Let $G$ be a triangular portion of the triangular lattice. This graph
is not bipartite and has an even or odd number of vertices depending
on the length of its side, modulo 4. In particular, when the number of
vertices is odd, $\varnothing \not\in \calV(G)$, while the set 
$\calJ=\{ \{v\} \;\vert\; v \in V \}$ is a set of admissible removals.
An example of configuration is in Figure~\ref{fig:triangle}, right.
\end{ex}

\noindent
If we call
\begin{align}
\label{eq.2876587653}
\calC(e)&
\coloneqq\{U\in\calJ\ \vert\ e\in M_U\},
&
\calC_v&
\coloneqq
\{U\in\calJ\ \vert\ v\in U\},
\ef
\end{align}
we have that $E_{H_\mathcal{J}}=\{e\in E\ \vert\  \calC(e)\neq\emptyset\}$. 
For any vertex $v$, the sets $\calC_v$ and $\{\calC(e)\}_{e \sim v}$
constitute a partition of $\calJ$. In particular, for any two incident
edges $e$, $e'$ in $H_{\calJ}$ the sets $\calC(e)$ and $\calC(e')$ are
disjoint, as the edges of a matching are all isolated.

In the following we shall often investigate closely sets $\calJ$
having elements of small cardinality, and illustrate our arguments by
means of suitable figures. For this practical reason, we will refer to
$\calC(e)$ as the set of \emph{colors} of the edge $e$, and, if
$U\in\calC(e)$, we say that $e$ \emph{has color} $U$. Note that an
edge may have more than one color. As has been already the case in the
proof of 
Propositions \ref{prop.swapcyc} and
% Proposition 
\ref{prop.38765874}, 
a crucial tool in our arguments
will be the investigation of suitable alternating sums of weights. At
this aim we need some definitions:
\begin{defn}
    Given a path $P=(e_1,\dots,e_\ell) \in G$ of length $\ell$ (with
    the obvious edge labeling), and two sets $U$ and $U'$ in $\calJ$,
    we say that $P$ is \emph{alternating} with respect to colors $U$
    and $U'$ if the edges along the path have alternately colors $U$
    and $U'$. The associated \emph{alternating sum} along $P$ is the
    sum of the weights along $P$ taken with alternating sign, in
    order, i.e.\ $\walt(P)=\sum_{k=1}^\ell(-1)^{k-1}w_{e_k}$.  Note that,
    calling $\rev{P}$ the reverse path of $P$, we have that $\rev{P}$
    is still alternating w.r.t.\ colors $U$ and $U'$, and
    $\walt(\rev{P})=(-1)^{\ell-1} \walt(P)$.

    Similarly, given a (rooted oriented) cycle of even length
    $C=(e_1,\dots,e_{2\ell})\in G$ and two sets $U$ and $U'$ in
    $\calJ$, we say that $C$ is \emph{alternating} w.r.t.\ colors $U$
    and $U'$ if the edges along the cycle have alternately colors $U$
    and $U'$, and the associated \emph{alternating sum} along $C$ is
    $\walt(C)$ as given in Eq.~\eqref{eq:altsum} in
    Definition~\ref{def.genwei1pre}.  Similarly as above, changing
    root and orientation on $C$ does not change the fact that $C$ is
    alternating w.r.t.\ colors $U$ and $U'$, and the weight $\walt(C)$
    is unchanged up to possibly an overall sign.

    Furthermore, given $C=(e_1,\dots,e_{2\ell})\in G$ as above, and a
    set $U$ in $\calJ$, we say that $C$ is \emph{semi-alternating}
    w.r.t.\ color $U$ if the edges of $C$ that have color $U$ are
    exactly those with odd index, and the associated alternating sum
    along $C$ is still $\walt(C)=\sum_{k=1}^{2\ell}
    (-1)^{k-1}w_{e_k}$. Now, the group of transformations for the
    choice of root and orientations, for $C$ to remain
    semi-alternating, has only half the cardinality of that for an
    alternating cycle, and $\walt(C)$ remains unchanged (not up to a
    sign).
\end{defn}
An elementary fact on matching subgraphs is the following:
\begin{rmk}
\label{rmk.28976954}
    In the graph $H_\calJ$ determined by a triple $(G,w,\calJ)$ there
    are no semi-alternating cycles $C$ with $\walt(C)>0$.  Furthermore, if
    the weights $w$ are generic, there are no alternating cycles.
\end{rmk}

\noindent
Indeed, a semi-alternating cycle in color $U$ as above is nothing but
a cycle with positive alternating sum as in
Proposition~\ref{prop.swapcyc}, for $M_U \subseteq G_U$.  For the
second part, note that a cycle alternating in $U$ and $U'$ and with
non-zero alternating sum is semi-alternating with positive
alternating sum in $U$ or in~$U'$.

With this setting and definitions in mind, in the following sections
we will establish some non-trivial facts about the subgraph
$H_{\calJ}$.  Let us now analyse the structure of $H_{\calJ}$, by
isolating a special part of it. Let us denote 
$V_{\calJ} \coloneqq \bigcup_{U\in\calJ}U\subseteq V$ the set of
vertices appearing in at least one element of $\calJ$.

\begin{defn}[Nontrivial part]
%, nontrivial component]
Let $(G,w,\calJ)$ be as above.
% a graph with generic weights and let $\calJ\subseteq\calV(G)$. 
We call the \emph{nontrivial part} of
$H_{\calJ}$, and denote it by $H_{\calJ}^*$, the union of the
connected components of $H_{\calJ}$ whose vertex set has non-empty
intersection with~$V_{\calJ}$.
% We say that $H_{\calJ}^*$ is the \emph{nontrivial component} of $H_{\calJ}$ if it is connected.
\end{defn}
\noindent 
Before introducing the most relevant properties of
$H_{\calJ}^*$, let us establish (or remind) some simpler facts.
\begin{lem}
\label{lem:diff}
Let $G$ be a graph with generic weights and
$\calJ=\{U_1,U_2\}\in\calV(G)$ be a pair of non-empty vertex
sets. Given two generic matchings $M_1\in\calM(G_{U_1})$ and
$M_2\in\calM(G_{U_2})$ (not necessarily of minimal weight), the graph
$K=M_1\cup M_2$ consists of isolated dimers, cycles of alternating
colors $U_1$ and $U_2$, alternating (w.r.t.\ $U_1$ and $U_2$) paths of
even length from vertices in $U_1$ to vertices in $U_2$, and
alternating paths of odd length connecting two vertices in the same
set, $U_1$ or $U_2$.
\end{lem}
\begin{proof}
Let us denote by $K^*$ the graph obtained from $K$ by dropping the
isolated dimers.  Observe that the vertex set of $K$ is
$V_K=V\setminx(U_1\cap U_2)$. Each vertex 
$v\in V_K\setminx (U_1\cup U_2)$ is either connected to the same
neighbour $u$ by both $M_1$ and $M_2$ (so that the edge $\edge{u}{v}$
in an isolated dimer and thus is in $K \setminx K^*$), or connected to
two distinct vertices, thus
$\mathrm{deg}_K(v)=\mathrm{deg}_{K^*}(v)=2$. Similarly, a vertex $v\in
U_1\symdif U_2$ is connected to a single vertex $u$ (say, if $v \in
U_1$, it is connected to $u$ iff $\edge{u}{v} \in M_2$), thus
$\deg_K(v)=\deg_{K^*}(v)=1$. The global properties of $K$ and $K^*$
follow trivially from this list of possible local structures.
\end{proof}
\noindent Note that in the Lemma above the overall number of paths
%joining vertices in $U_1$ to vertices in $U_2$ is bounded by 
is $\frac{1}{2}|U_1\symdif U_2|$. Combining this observation with
Remark \ref{rmk.28976954}, we also have the following
% If $U_1=\{u\}$ and $U_2=\{u'\}$, with $u\neq u'$, then the symmetric
% difference consists of exactly one (alternating) path, and possibly
% some cycles.
\begin{corr}
\label{cor.coppia}
Let $G$ be a graph with generic weights $w$, and
$\calJ=\{U_1,U_2\}\subseteq\calV(G)$.  The graph $H_{\calJ}\setminx
H_{\calJ}^*$ consists of a collection of isolated dimers.  The graph
$H_{\calJ}^*$ contains no cycles and consists of alternating paths
only. In particular, if $U_1 \symdif U_2 = \{u,u'\}$ (that is,
$U_1=\{u\}$ and $U_2=\{u'\}$ with $u\neq u'$, or $U_1=\emptyset$ and
$U_2=\{u,u'\}$), then it consists of a single alternating path
connecting $u$ to~$u'$.
\end{corr}
\noindent
Some aspects of Corollary \ref{cor.coppia}
% Lemma \ref{lem:diff}
generalise easily to arbitrary sets $\calJ$:
\begin{lem}Given a graph $G$ with generic weights, and a non-empty
  subset $\calJ \subseteq\calV(G)$, the subgraph $H_{\calJ}\setminx
  H_{\calJ}^*$ consists of a collection of isolated dimers.
\label{lem:T}
\end{lem}
\begin{proof}
When $\calJ=\{U\}$, $H_{\calJ}$ consists of the sole matching $M_{U}$,
% while $V_{H_{\calJ}}=V\setminx U$ 
and the proposition is trivial. The case $|\calJ|=2$
% restriction of the proposition to the case $\calJ=\{U_1,U_2\}$ 
is already established in Corollary~\ref{cor.coppia}. When $\calJ$ has
larger cardinality, let $e$ be an edge in $H_{\calJ}$, but not in
$H_{\calJ}^*$. If $\calC(e)=\calJ$, it is an isolated dimer. Assume
that this is not the case. So there is a ``color'' of $\calJ$, say
$U_1$, which is in $\calC(e)$, and a color, say $U_2$, which is not in
$\calC(e)$. By applying Corollary~\ref{cor.coppia} to the pair
$\calJ'=\{U_1,U_2\}$, we conclude that $e$ is part of an alternating
path, whose endpoints are in $U_1 \cup U_2$ (in fact, they are in the
smaller set $U_1 \symdif U_2$). As adding more colors to $\calJ$ makes
the partition of the vertices into connected components of $H_\calJ$ even
coarser, we deduce that $e$ must be in $H_{\calJ}^*$.
\end{proof}
\noindent
Corollary~\ref{cor.coppia} also gives a sufficient conditions for the
symmetric difference of two optimal matchings to consist of a unique
path. As we will see, this property is crucial for the matching
subgraph to have remarkable combinatorial properties. For this
reasons, we shall introduce a definition
\begin{defn}
\label{def.singlepath}
We say that a set $\calJ$ is \emph{single-path} if, for all
$\{U_1,U_2\}\subseteq \calJ$, we have $\frac{1}{2}|U_1\symdif U_2|=1$,
and thus, for what we said above, 
%
% $M_{\star}(G_{U_1}) \symdif M_{\star}(G_{U_2})=
$M_{U_1}\symdif M_{U_2}$
consists of a unique alternating path.
\end{defn}

\noindent
A typical situation in which $\calJ$ is single-path is when it
consists of singleton sets, 
i.e.\ $\calJ=\{\{v\}\}_{v \in V'}$ for some $V' \subseteq V_G$. In
this case the matching subgraph $H_\calJ$ has a further property.
Recall that an edge $e=\edge{u}{v} \in G$ is a \emph{bridge} of $G$ if
$u$ and $v$ are in distinct connected components of $G \setminx e$.
A graph $G$ is said to be \emph{bridgeless} if it contains no
bridges. Then we have
\begin{lem}
\label{lem:nobridge}
If $\calJ=\{\{v\}\}_{v \in V'}$, and $u, v \in V'$, then the edge
$\edge{u}{v}$ is not a bridge of $H_{\calJ}$.
\end{lem}
\begin{proof}
The proof is by contradiction. Suppose that $e=\edge{u}{v}$ is a bridge of $H_\calJ$,
and call $H'$ the component of $H_\calJ\setminx \edge{u}{v}$ containing
$u$ (and thus it does not contain $v$). Clearly, $e$ is neither in 
$M_{\{u\}}$, nor in $M_{\{v\}}$.
% neither $\{u\}$ nor $\{v\}$ are in $\calC(e)$, 
Also, all the edges of
% the matching 
$M_{\{u\}}$ have endpoints either both in $H'$ or both not in $H'$,
and the same holds for $M_{\{v\}}$. However, while the restriction of
$M_{\{u\}}$ to $H'$ covers all vertices in
$V_{H'} \setminx u$, so that $|V_{H'}|$ must be odd, the restriction
of $M_{\{v\}}$ to $H'$ covers all vertices in $V_{H'}$, so that
$|V_{H'}|$ must be even. This contradiction allows to conclude.
\end{proof}
\noindent
As a result, we have the following fact:
\begin{corr}
\label{cor.settingmono}
In a triple $(G,w,\calJ)$, where 
%
% $|V_G|$ is odd,
$G=(V,E)=\calK_{2n+1}$, 
$w$ is a generic weight function, and
$\calJ=\{ \{v\} \}_{v \in V}$, we have that $H_\calJ=H_\calJ^*$ is a
connected spanning subgraph of 
% $G$,
$\calK_{2n+1}$, 
and is bridgeless.
\end{corr}
\noindent
We shall call the \emph{monopartite standard setting} the type of
triple $(G,w,\calJ)$ as in the corollary above, when $G=\calK_{2n+1}$
(either from the beginning, or, as explained above, in the situation
where we ``extend'' a graph $G=(V,E)$, with $|V|$ odd, to the complete
graph by adding edges with sufficiently large weights). In fact the
corollary holds also if $G$ is not the complete graph,
provided that $|V|$ is odd, and that $\calJ \subseteq \calV$.

Let us now prove two last, slightly more subtle facts. The first one
extends to our setting the well-known fact that chains of length $>2$
have a trivial role in Matching Theory.  That is, if $G$ has a chain
of length 3 (i.e., two adjacent vertices of degree 2), the triple
$(G,w,\calJ)$ is essentially equivalent to a second triple
$(G',w',\calJ')$, where $G'$ is the graph in which the chain has been
replaced by a single edge, and $w'$, $\calJ'$ are suitably
constructed.  More precisely, let $(v_0,v_1,v_2,v_3)$ be a chain of
$G$, that is $\deg_G(v_1)=\deg_G(v_2)=2$.
% and in particular $\deg_G(v_1)=\deg_G(v_2)=2$. 
We shall set
\begin{itemize}
\item
$V_{G'}=V_G \setminx \{v_1,v_2\}$ and
$E_{G'}=E_G \setminx \{
\edge{v_0}{v_1},
\edge{v_1}{v_2},
\edge{v_2}{v_3}\} \cup \edge{v_0}{v_3}$
\item
$w'_{v_0,v_3}=w_{v_0,v_1}-w_{v_1,v_2}+w_{v_2,v_3}$
\item
$\calJ'=\{ \phi(U) \}_{U \in \calJ}$, where, calling 
$\tilde{U}=U\cap \{v_0,v_1,v_2,v_3\}$ and $\hat{U}=U\setminx \tilde{U}$,
% the restriction of $U$ to the complement of $\{v_0,v_1,v_2,v_3\}$,
we have:
\be
\phi(U)=\left\{
\begin{array}{ll}
U & v_1,v_2 \not\in \tilde{U} \\
\hat{U}  & \tilde{U} = \{v_1,v_2\} \\
\hat{U} \cup \{v_3\} & \tilde{U} = \{v_1\} \textrm{~or~} 
%\tilde{U} = 
\{v_1,v_2,v_3\}\\
\hat{U} \cup \{v_0\} & \tilde{U} = \{v_2\} \textrm{~or~} 
%\tilde{U} = 
\{v_0,v_1,v_2\}\\
\hat{U} \cup \{v_0,v_3\} & \tilde{U} = \{v_0,v_1\},
% \textrm{~or~} 
% \tilde{U} = 
\{v_2,v_3\}
% \textrm{~or~} 
% \{v_0,v_3\} 
\textrm{~or~} 
\{v_0,v_1,v_2,v_3\}
\end{array}
\right.
\ee
All the other cases are
%  deduced by symmetry, or 
excluded because $U$
would not be an admissible removal.
\end{itemize}
Note that $w'$ is generic if and only if $w$ is generic.  
This leads to the following remark
\begin{rmk}
\label{rmk.reservoiriseasy}
In the construction above, if $U \cap \{v_1,v_2\}=\varnothing$ for all
$U \in \calJ$, the matching subgraph $H_{\calJ}$ for the triple
$(G',w',\calJ)$, modified by contracting the edge $\edge{v_1}{v_2}$,
coincides with the matching subgraph $H_\calJ$ for the triple
$(G,w,\calJ)$, modified by subdividing the edge $\edge{v_0}{v_3}$ by
adding one new vertex.
\end{rmk}
Indeed, in our setting,
the restriction of $M_{\star}(G_U)$ (for weights $w'$) to the
complement of the chain $(v_0,v_1,v_2,v_3)$ coincides with the
restriction of $M_{\star}(G'_{\phi(U)})$ to the complement of the edge
$\edge{v_0}{v_3}$. This already implies that the matching subgraph
$H_\calJ$ for the triple $(G,w,\calJ)$ coincides with the matching
subgraph $H_{\calJ'}$ for the triple $(G',w',\calJ')$, when the former
is restricted to the complement of the chain, and the latter to the
complement of the edge $\edge{v_0}{v_3}$.

If, furthermore, we have that $U \cap \{v_1,v_2\}=\varnothing$ for all
$U \in \calJ$, then we just have $\phi(U)=U$ for all $U \in \calJ$,
and $\calJ'=\calJ$. If
$\edge{v_0}{v_3} \in M_{\star}(G'_{U})$, then 
$\edge{v_0}{v_1}, \edge{v_2}{v_3} \in M_{\star}(G_{U})$, while if
$\edge{v_0}{v_3} \not\in M_{\star}(G'_{U})$, then 
$\edge{v_1}{v_2} \in M_{\star}(G_{U})$. From this, the claim of our
remark follows easily.

Now we go to the second subtle fact of this section. Its relevance
will be clear only in Section \ref{sec.caso2d}, but it has a level of
generality that makes it appropriate at this point.
\begin{defn}
\label{def.forbpair}
Given a graph $G=(V,E)$, with generic weights $w$, a 
\emph{forbidden pair} $\{e_1,e_2\}$ is a pair of edges of $E$ with
distinct endpoints,
$e_1=\edge{u_1}{v_1}$ and $e_2=\edge{u_2}{v_2}$, such that also
$e'_1=\edge{u_1}{v_2}$, $e'_2=\edge{u_2}{v_1}$, $e''_1=\edge{u_1}{u_2}$ and
$e''_2=\edge{v_1}{v_2}$ are in $E$, and
the two cycles $C'=(e_1,e'_1,e_2,e'_2)$ and
$C''=(e_1,e''_1,e_2,e''_2)$ have
$\walt(C'),\walt(C'')>0$, that is
\begin{align}
w_{e_1}+w_{e_2}-w_{e'_1}-w_{e'_2} &>0
\ef;
&
w_{e_1}+w_{e_2}-w_{e''_1}-w_{e''_2} &>0
\ef.
\end{align}
If $G$ is bipartite, then a forbidden pair 
$\{e_1,e_2\}$
is a pair such that
also
$e'_1=\edge{u_1}{v_2}$ and $e'_2=\edge{u_2}{v_1}$ 
are in $E$, and
$\walt(C')>0$, that is
\be
w_{e_1}+w_{e_2}-w_{e'_1}-w_{e'_2}>0
\ef.
\ee
Let us denote by $F=F(G;w)$ the set of forbidden pairs of~$G$.
\end{defn}
%% (in other words, the matching
%% $\{\{u_1,v_2\},\{u_2,v_1\}\}$ has smaller weight than
%% $\{\{u_1,v_1\},\{u_2,v_2\}\}$).
\begin{lem}
\label{lem.nopairF}
Let us consider our problem, with graph $G$ and generic weights $w$.
% $G$ being a complete bipartite graph $\calK_{n,m}$, or a complete graph $\calK_n$. 
Then, for any single-path
% \footnote{I.e., such that $\frac{1}{2}|U_1\symdif U_2|=1$
% for all $\{U_1,U_2\}\subseteq \calJ$, as in Definition~\ref{def.singlepath}.}
%
$\calJ \subseteq \calV(G)$,
% such that for all $U_1,U_2 \in \calJ$ the subgraph $H^*_{\{U_1,U_2\}}$ is a single path,
the matching subgraph $H_{\calJ}$ does not contain any pair of edges
in the list $F$ of forbidden pairs.
\end{lem}
\begin{proof}
Suppose, for a contradiction, that the two edges $e_1=\edge{u_1}{v_1}$
and $e_2=\edge{u_2}{v_2}$ are in $H_{\calJ}$, and the pair
$\{e_1,e_2\}$ is in~$F$.

If $\calC(e_1)\cap\calC(e_2)\neq \varnothing$, let $U\in\calJ$ be an
element of $\calC(e_1)\cap\calC(e_2)$. Then, the matching
$(M_U \setminx \{e_1,e_2\}) \cup \{e'_1,e'_2\}$ would have a smaller
weight than $M_U$, reaching a contradiction.
So there are
$U_1 \in \calC(e_1)$ and $U_2 \in \calC(e_2)$ distinct, with 
$U_2 \not\in \calC(e_1)$ and $U_1 \not\in \calC(e_2)$.  Now, consider
$M_{U_1} \symdif M_{U_2}$.  The edges $e_1$ and $e_2$ are contained in
the unique path of the configuration, which is alternating in colors
$U_1$ and $U_2$, and are at even distance. Up to renaming who's who
among $u_i$ and $v_i$, we can set that this path, in some orientation,
has the form
\be
P_1 \circ \edge{u_1}{v_1} \circ P \circ \edge{v_2}{u_2} \circ P_2
\ef.
\ee
Note that, if $G$ is bipartite, the fact that $P$ has even length
indeed implies that its endpoints are in the same vertex class of $G$,
hence the choice of names for the vertices.

The stability of $M_{U_1}$ w.r.t.\ swapping the cycle 
$C_1=(u_1,v_1) \circ P \circ (v_2,u_1)$ takes the form 
\be
w_{e_1} + W -w_{e'_1}<0
\ef,
\ee
with $W=\walt(P)$, the alternating sum of $P$, with its
appropriate sign.  Similarly, the stability of $M_{U_2}$
w.r.t.\ swapping the cycle 
%
% $Q_2=(u_2,v_1) \circ P \circ (v_2,u_2)$ 
$C_2=(u_2,v_2) \circ \rev{P} \circ (v_1,u_2)$ 
takes the form
\be
w_{e_2} - W -w_{e'_2}<0
\ef.
\ee
Note that, indeed, here $W$ comes with opposite sign, because $P$ is
of even length.
Summing the two inequalities gives
\be
w_{e_1}+w_{e_2}-w_{e'_1}-w_{e'_2}<0
,
\ee
in contradiction with the condition for $\{e_1,e_2\}$ to be in~$F$.
\end{proof}
\noindent
The proof above is illustrated in Figure~\ref{fig.forlem.nopairF}.

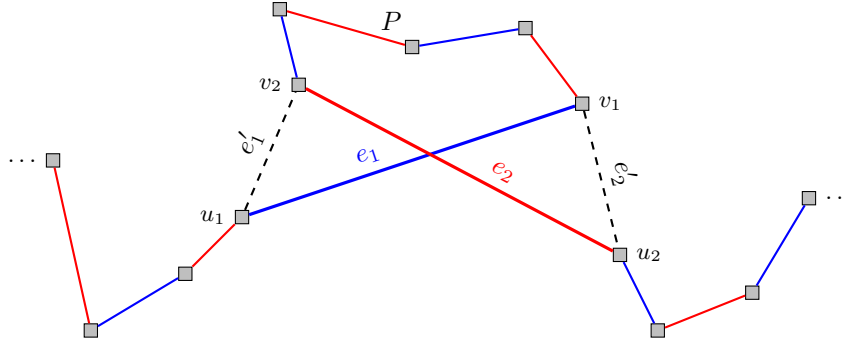
\begin{figure}[t]
\begin{center}
\if\faifig1  
\begin{tikzpicture}[scale=2.5]
\node[label={180:{\small $u_1$}},inner sep=2.5pt,rectangle,fill=gray!50,draw] (u1) at (-1,.2) {}; 
\node[label={180:{\small $v_2$}},inner sep=2.5pt,rectangle,fill=gray!50,draw] (v2) at (-.7,.9) {}; 
\node[label={0:{\small $u_2$}},inner sep=2.5pt,rectangle,fill=gray!50,draw] (u2) at (1,0) {}; 
\node[label={0:{\small $v_1$}},inner sep=2.5pt,rectangle,fill=gray!50,draw] (v1) at (.8,.8) {}; 
\node[inner sep=2.5pt,rectangle,fill=gray!50,draw] (x1) at (.5,1.2) {}; 
\node[label={100:{\rule{20pt}{0pt}$P$}},inner sep=2.5pt,rectangle,fill=gray!50,draw] (x2) at (-.1,1.1) {}; 
\node[inner sep=2.5pt,rectangle,fill=gray!50,draw] (x3) at (-.8,1.3) {}; 
\node[inner sep=2.5pt,rectangle,fill=gray!50,draw] (a1) at (1.2,-.4) {}; 
\node[inner sep=2.5pt,rectangle,fill=gray!50,draw] (a2) at (1.7,-.2) {}; 
\node[label={0:{\small $\ldots$}},inner sep=2.5pt,rectangle,fill=gray!50,draw] (a3) at (2.0,.3) {}; 
\node[inner sep=2.5pt,rectangle,fill=gray!50,draw] (b1) at (-1.3,-.1) {}; 
\node[inner sep=2.5pt,rectangle,fill=gray!50,draw] (b2) at (-1.8,-.4) {}; 
\node[label={180:{\small $\ldots$}},inner sep=2.5pt,rectangle,fill=gray!50,draw] (b3) at (-2.0,.5) {}; 
\draw[very thick,blue] (u1) -- node[sloped,above]{$e_1$\rule{30pt}{0pt}} (v1); 
\draw[very thick,red] (u2) -- node[sloped,above]{\rule{30pt}{0pt}$e_2$} (v2); 
\draw[thick,dashed] (v1) -- node[sloped,above]{$e'_2$} (u2); 
\draw[thick,dashed] (u1) -- node[sloped,above]{$e'_1$} (v2); 
%% \draw[thick,dashed] (u2) -- node[sloped,above]{$Q_2$} (v1); 
%% \draw[thick,dashed] (u1) -- node[sloped,above]{$Q_1$} (v2); 
\draw[thick,red] (v1) -- (x1); 
\draw[thick,blue] (x2) -- (x1);
\draw[thick,red] (x2) -- (x3); 
\draw[thick,blue] (v2) -- (x3); 
\draw[thick,blue] (u2) -- (a1);
\draw[thick,red] (a2) -- (a1); 
\draw[thick,blue] (a2) -- (a3); 
\draw[thick,red] (u1) -- (b1);
\draw[thick,blue] (b2) -- (b1); 
\draw[thick,red] (b2) -- (b3); 
\end{tikzpicture}
\else [...TikZ\ code...] \fi
\end{center}
\caption{\label{fig.forlem.nopairF}%
Illustration of the alternating path and cycles appearing in
  the proof of Lemma~\ref{lem.nopairF}.}
\end{figure}

%-------------------------------------------------------
\subsection{A variant: matchings with a reservoir}
\label{sec.reservintro}
%-------------------------------------------------------

A variant of the setting above is a situation in which the graph
$G=(V\cup \{s\},E)$ has a special vertex, here denoted by $s$, that
can participate in a covering with an unconstrained degree, i.e.,
valid configurations $M\subseteq G$ are not anymore just perfect
matchings of $G$, but rather subgraphs such that $\deg_M(v)=1$ for all
$v \in V$, whereas no constraint is imposed on $\deg_M(s)$. We call
$s$ the \emph{reservoir vertex}, or simply the reservoir, of the
graph.

If $G$ has the form $G=(A \cup B \cup \{s\},E)$ with 
$A\cap B=\emptyset$ and 
$E\subseteq (A \times B) \cup ((A \cup B) \times \{s\})$, we say that
we are in the
\emph{bipartite setting with reservoir}, even though $G$ is not
bipartite. (Of course, it can be made a genuine bipartite graph by
adding \emph{two} reservoir sites, $s_A$ and $s_B$, and stating that
all edges are either of the form $\edge{a_i}{b_j}$, or
$\edge{a_i}{s_B}$, or $\edge{s_A}{b_j}$).

We will consider removals $U$ that only concern ``ordinary''
vertices, and always keep the reservoir in the graph $G_U$.

\begin{ex}[Square lattice with reservoir on the boundary]
\label{ex:squareres}
Let $G$ be a square portion of side $\ell$ of the square lattice, plus
the extra vertex $s$, that is not drawn. The vertices connected to the
reservoir are those on the boundary.  This graph is bipartite in the
sense above.  Now there are no parity constraints on $\ell$, as well
as on the sets of admissible removals.  Examples of configurations in
which $U$ is the empty set, or consists of a single vertex, are
presented in Figure~\ref{fig:sqres}.
\end{ex}

\begin{figure}[t]
\includegraphics[scale=2.2]{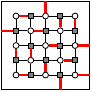}
%[height=0.2\columnwidth]
\qquad
\includegraphics[scale=2.2]{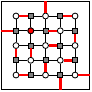}
\caption{\label{fig:sqres}% 
  Examples of matchings on graphs with a reservoir, where the vertices
  $v_i$ are those of a portion of the square lattice, and the
  reservoir $s$ can be identified with the boundary of the
  domain. Notations are as in Figure \ref{fig:triangle}.
  (\emph{Left}) A matching with $U=\emptyset$.  (\emph{Right}) A
  matching with $|U|=1$.}
\end{figure}

\noindent
Under our generality hypotheses, this generalised setting can be
implemented within the ordinary setting for matchings, up to
performing a limit. In the monopartite case, with triple
$(G,w,\calJ)$, with $G=(V\cup\{s\},E)$ and $|V|=n$, let us
construct a new (weighted) graph $G^{\rm res}$ as follows:
\begin{itemize}
\item $G^{\rm res}=(V\cup S,E^{\rm res})$, with $|S|=|V|=n$;
\item 
$E^{\rm res}$ contains the edges of $E$ not incident to the reservoir,
  the edges $\edge{v_i}{s_i}$ for all $v_i\in V$ such that
  $\edge{v_i}{s}\in E$, and all the edges $\edge{s_i}{s_j}$ for $1
  \leq i < j \leq n$
%% $E^{\rm res}=
%% \{\edge{v_i}{v_j}\}_{\substack{v_i,v_j\in V,\\ \edge{v_i}{v_j}\in E}}
%% \cup 
%% \{\edge{v_i}{s_i}\}_{\substack{v_i\in V,\\ \edge{v_i}{s}\in E}}
%% \cup 
%% \{\edge{s_i}{s_j}\}_{1 \leq i < j \leq n}$
\item $w^{\rm res}_{v_i,v_j}=w_{v_i,v_j}$
\item $w^{\rm res}_{v_i,s_i}=w_{v_i,s}$
\item $w^{\rm res}_{s_i,s_j}$ are i.i.d.\ random numbers in the
  range $[0,\eps]$, for some $\eps>0$. 
\end{itemize}
Now, call $\calR$ the operator that shrinks the reservoir auxiliary
part of the graph $S$ to the single vertex $s$, that is, sends the
edges $\edge{v_i}{s_i}$ of $E^{\rm res}$ to the edges $\edge{v_i}{s}$ of $E$, and
drops the edges $\edge{s_i}{s_j}$. We have that
\begin{align}
\label{eq.9837628}
\lim_{\eps \to 0} w(M_\star(G^{\rm res}_U))
&=w(M_\star(G_U))
;
&
\lim_{\eps \to 0} \calR M_\star(G^{\rm res}_U)
&= M_\star(G_U)
.
\end{align}
In fact,
$w(M_\star(G^{\rm res}_U))-w(M_\star(G_U)) \in [0,\eps\, n/2]$.

In the bipartite setting, the limit can be performed in a way that
preserves the bipartition. For a triple
$(G,w,\calJ)$, with $G=(A\cup B\cup\{s\},E)$,  $A\cap B=\emptyset$  and $|A|=n_1$, 
$|B|=n_2$, 
we construct a new weighted bipartite graph $G^{\rm res}$ as follows:
\begin{itemize}
\item $\bipg{G^{\rm res}}{A\cup C}{B \cup D}{E^{\rm res}}$, with 
$|C|=|B|=n_2$ and $|D|=|A|=n_1$;
\item 
$E^{\rm res}$ contains all the edges of $E$ of the form
  $\edge{a_i}{b_j}$, and edges 
$\edge{a_i}{d_i}$ and $\edge{b_i}{c_i}$ for all 
$a_i\in A$ such that $\edge{a_i}{s}\in E$, and all
$b_i\in B$ such that $\edge{b_i}{s}\in E$, respectively,
plus all the edges
$\edge{c_i}{d_j}$ for $1 \leq i \leq n_2$ and $1\leq j \leq n_1$
%% $E^{\rm res}=
%% \{\edge{a_i}{b_j}\}_{\substack{a_i \in A,\\ b_j\in B,\\ \edge{a_i}{b_j}\in E}}
%% \cup \{\edge{a_i}{d_i}\}_{\substack{a_i\in A,\\ \edge{a_i}{s}\in E}}
%% \cup \{\edge{b_i}{c_i}\}_{\substack{b_i\in B,\\ \edge{b_i}{s}\in E}}
%% \cup \{\edge{c_i}{d_j}\}_{\substack{1 \leq i \leq n_2,\\ 1\leq j \leq n_1}}$
\item $w^{\rm res}_{a_i,b_j}=w_{a_i,b_j}$
\item $w^{\rm res}_{a_i,d_i}=w_{a_i,s}$, 
$w^{\rm res}_{b_i,c_i}=w_{b_i,s}$, 
\item $w^{\rm res}_{c_i,d_j}$ are i.i.d.\ random numbers in the
  range $[0,\eps]$, for some $\eps>0$. 
\end{itemize}
Then equations (\ref{eq.9837628}) still hold, and
$w(M_\star(G^{\rm res}_U))-w(M_\star(G_U)) \in [0,\eps \min(n_1,n_2)]$.

A slightly more subtle fact is that the reservoir setting can be
implemented within the ordinary setting, without performing a limit,
at the price of introducing a larger auxiliary graph (essentially, the
size gets multipled by four, instead of doubling). When the graph is generic,
we introduce a new weighted graph $G^{\rm res}$ as
follows:
\begin{itemize}
\item $G^{\rm res}=(V\cup S \cup S' \cup V',E^{\rm res})$, with $|S|=|V|=|S'|=|V'|=n$;
\item $E^{\rm res}$ contains an edge $\edge{v_i}{v_j}$
and an edge $\edge{v'_i}{v'_j}$ for each $\edge{v_i}{v_j}\in E$, an
edge $\edge{v_i}{s_i}$ and an edge $\edge{v'_i}{s'_i}$ for each
$\edge{v_i}{s}\in E$, and an edge
$\edge{s_i}{s'_i}$ for each $1 \leq i \leq n$
%% \item $E^{\rm res}=
%% \{\edge{v_i}{v_j}\}_{\substack{v_i,v_j\in V,\\ \edge{v_i}{v_j}\in E}}
%% \cup
%% \{\edge{v'_i}{v'_j}\}_{\substack{v_i,v_j\in V,\\ \edge{v_i}{v_j}\in E}}
%% \cup 
%% \{\edge{v_i}{s_i}\}_{\substack{v_i\in V,\\ \edge{v_i}{s}\in E}}
%% \cup 
%% \{\edge{v'_i}{s'_i}\}_{\substack{v_i\in V,\\ \edge{v_i}{s}\in E}}
%% \cup 
%% \{\edge{s_i}{s'_i}\}_{1 \leq i \leq n}$
\item $w^{\rm res}_{v_i,v_j}=w^{\rm res}_{v'_i,v'_j}=w_{v_i,v_j}$
\item $w^{\rm res}_{v_i,s_i}=w^{\rm res}_{v'_i,s'_i}=w_{v_i,s}$
\item $w^{\rm res}_{s_i,s'_i}=0$
\item $\calJ^{\rm res}=\{U \cup U'\}_{U \in \calJ}$, where $U'$ is the
  counterpart of $U$ inside $V'$ (that is, $v'_j \in U'$ iff $v_j
  \in U$).
\end{itemize}
Under this construction, the bipartite case is just a special case of
the non-bipartite case.
\label{pg.reservoir}
Then, calling $\calL$ the operator that sends 
$H^{\rm res} \subseteq G^{\rm res}$ to $H \subseteq G$, with the rule
that $\edge{v_i}{v_j}\in H$ iff $\edge{v_i}{v_j}\in H^{\rm res}$, and
$\edge{v_i}{s}\in H$ iff $\edge{v_i}{s_i}\in H^{\rm res}$, we have
\begin{align}
\label{eq.9837628x}
w(M_\star(G^{\rm res}_{U\cup U'}))
&=2 w(M_\star(G_U))
\ef;
&
\calL M_\star(G^{\rm res}_{U\cup U'})
&= M_\star(G_U)
\ef.
\end{align}
We will explain these facts later on, after that a more general
setting of ``symmetric graphs'' is introduced in
Section~\ref{sec.teoremaserio1b}.

Remark that in this case, if $G\setminx s$ is bipartite, then 
$G^{\rm res}$ is bipartite, with $v_j$ and $v'_j$ in different vertex
classes, as well as $s_j$ and~$s'_j$.

Recall
% from Section \ref{sec.setting} 
(e.g.\ from page \pageref{pag.matrcosts}) that, up to introducing
further edges with large weights, an instance on a bipartite graph
$G=(A,B;E)$ is conveniently encoded in a
$n \times m$ matrix of costs $W$, where the rows and columns are
associated to the $n=|A|$ and $m=|B|$ vertices of the two
classes. Then, the weights for a bipartite instance with reservoir is
similarly encoded in a
$(n+1) \times (m+1)$ matrix of costs,
with 
$W_{ij}=w_{\edge{a_i}{b_j}}$ if $i\leq n$ and $j\leq m$,
$W_{n+1\,m+1}$ is a special symbol (a yellow cell in the following
notation) that denotes that the entry $(n+1,m+1)$ cannot be taken, 
and
$W_{n+1\,j}=w_{\edge{s}{b_j}}$ and
$W_{i\,m+1}=w_{\edge{a_i}{s}}$ otherwise. Then, the construction of
the auxiliary problem involves a larger (square) matrix of costs,
having as one of the four blocks the transpose of the original matrix
of costs, and two of the blocks allowing only diagonal entries, as in
the following example:
\[
\setlength{\unitlength}{10pt}
\thicklines
\raisebox{48pt}{
}
\]
A dot on the diagonal blocks denotes that, in the chain
$(v_i,s_i,s'_i,v'_i)$ of $G^{\rm res}$, the edges $\edge{v_i}{s_i}$ and
$\edge{s'_i}{v'_i}$ have been taken in $M$, and the absence of the dot
means that the edge $\edge{s_i}{s'_i}$ has been taken in $M$.

%%%%%%%%%%%%%%%%%%%%%%%%%%%%%%%%%%%%%%%%%%%%%%%%%%%%%%%
\section{A first setting where the matching subgraph is a spanning tree}
\label{sec.teoremaserio1}
%%%%%%%%%%%%%%%%%%%%%%%%%%%%%%%%%%%%%%%%%%%%%%%%%%%%%%%

\noindent
We have seen above with relatively small effort that, if 
$\calJ=\{U_1, U_2\}$ and $|U_1 \symdif U_2|=2$ (so that $\calJ$ is
``single-path''), the matching subgraph $H_{\calJ}$ indeed consists of
isolated dimers and one path, so in particular it is acyclic. This
section and Section \ref{sec.teoremaserio1b} are devoted to the proof of the first
surprising (and considerably harder) fact, that, in two general
settings of the Assignment Problem, with $|\calJ|$ scaling as the size
of the graph, the matching subgraph $H_{\calJ}$ is indeed a spanning
tree. This fact is also at striking difference w.r.t.\ the case
considered in Lemma \ref{lem:nobridge}, where we provide a setting in
which no edge of $H_\calJ$ is a bridge, while in a tree all the edges
are bridges.

Let us start by investigating under which circumstances it is
conceivable that, for $|\calJ|=3$, the graph $H_{\calJ}$ is guaranteed
to be connected and acyclic. There does not seem to be much room for
generality here, in light of Lemma \ref{lem:nobridge}, and as already
the case in which $G$ is a triangle, and $\calJ$ is the set of the
three singleton subsets, independently from the weight function, we
have $H_\calJ=G$, and thus it is not acyclic. 

These facts suggest to restrict our attention to the case in which
$\bipg{G}{A}{B}{E}$ is bipartite.  From now onward, it will be
understood that vertices named $a_1$, $a_2$,\dots\,are in $A$, and
vertices $b_1$, $b_2$,\dots\,are in $B$, where $A$ and $B$ are the two
classes of vertices of the bipartite graph $G$, and that vertices (and
edges) with distinct names are distinct.

For reasons which are not clear at this point, in performing our
manipulations we should keep track of a certain property, namely that
all $B$-vertices of $H^*_{\calJ}$ not in a set $U$ have degree~2.

Let us provide one last preparatory observation:
\begin{prop}
\label{prop.tradeaforb}
Let $(G,w,\calJ)$ be a triple, with $\bipg{G}{A}{B}{E}$ bipartite.
Let $a \in A$ be an element of $V_{\calJ}$ (that is, there exists
$U\ni a$). 
Construct a new triple $(\tilde{G},\tilde{w},\tilde{\calJ})$ as
follows:
\begin{itemize}
\item $V(\tilde{G})=V(G) \cup \{b\}$, where $b$ is a new $B$-vertex;
\item $E(\tilde{G})=E(G) \cup \{\edge{a}{b}\}$;
\item $\tilde{w}_e=w_e$ if $e\in E(G)$, and $\tilde{w}_{a,b}=0$;
\item for all $U \in \calJ$, $\tilde{U}=U \setminx a$ if $a \in U$
  and $\tilde{U}=U\cup\{b\}$ otherwise.
\end{itemize}
Then the graph $H_{\calJ}^* \subseteq G$ is a tree, and all vertices in 
$(B \setminx V_\calJ)\cap V_{H_{\calJ}^*}$ have
degree~$2$, if and only if the same facts hold for 
$H_{\tilde{\calJ}}^* \subseteq \tilde{G}$.
\end{prop}
\begin{proof}
Indeed, there is an obvious bijection $\phi$ between matchings 
$M \in \calM(G_U)$ and matchings 
$\tilde{M} \in \calM(\tilde{G}_{\tilde{U}})$.  If $a\in U$, then
$\tilde{M}=\phi(M)=M \cup \edge{a}{b}$, while if $a\not\in U$, then
$\tilde{M}=\phi(M)=M$.  In both cases, $w(M)=w(\tilde{M})$, so that
the image of the optimal matching in $G_U$ is the optimal matching in
$G_{\tilde{U}}$.  This modification leaves $H_{\calJ}$ essentially
unchanged, up to the addition of an edge $\edge{a}{b}$, which is
attached to $H_{\calJ}$ on $a$, and has $b$ as a leaf (the presence of
this edge is implied by the fact that we asked that $a$ is in at least
one of the sets $U$). In this way, the properties of interest of
$H_{\calJ}$ (to be connected, acyclic, and having $B$-vertices of
degree $2$ if they are not in $V_\calJ$) are left unchanged.
\end{proof}

\noindent
Let us introduce a definition that will appear in the most subtle part
of the following analysis:
\begin{defn}
\label{def.3ham}
Let $\bipg{G}{A}{B}{E}$ be a connected bipartite graph, with three leaves
(all $B$-vertices) and all other vertices of degree 3, equipped with a
proper 3-coloration of the edges\footnote{A proper edge-coloration of
  a graph $G=(V,E)$ is a map $\phi:E \to C$ such that
  $\phi(e_1)\neq\phi(e_2)$ if $e_1$ and $e_2$ are incident.}, such
that each leaf is incident to an edge of a different color. We say
that $G$ is \emph{3-Hamiltonian} if there are no alternating
cycles. That is, the subgraph alternating in each pair of colors is a
single path connecting the two leaves of those colors, and visiting
all the vertices of $G$ except for the third leaf (i.e., a Hamiltonian
path on $G$ minus the third leaf).
\end{defn}
\noindent
If $G$ is 3-Hamiltonian, the graph $G'$ obtained by joining the three
leaf vertices into a single vertex $b$ is a ($B$-vertex-rooted)
properly 3-colored graph such that the subgraph alternating in each
pair of colors is a Hamiltonian cycle of $G'$.  Some examples of
3-Hamiltonian graphs are provided in Figure~\ref{fig:3HamN5}.

The fact that the family of 3-Hamiltonian graphs is infinite and
contains arbitrarily large graphs is not completely trivial, but
indeed true.  As a curiosity, these graphs can also be enumerated,
thanks to a connection with a well-known quantity in Representation
Theory, namely the number of factorizations of a long cycle into two
long cycles. There are several ways of seeing this, and one goes as
follows.  Call $Z_n$ the number of 3-Hamiltonian graphs $G$ with $n$
$A$-vertices (i.e., graphs $G'$ with $2n$ vertices).  Say that the
colors are blue, red and green.  Label the $B$-vertices of $G'$ from
$1$ to $n$,
%, and the $A$-vertices from $a_1$ to $a_n$, 
in the order given by the cycle in blue and red, rooted at the vertex
obtained by connecting the three leaves, and oriented as to start with
the blue edge. Label the blue, red and green edges with integers from
$1$ to $n$, according to the incident $B$-vertex. Thus each label $i$
is associated to one $B$-vertex, and to one edge per color (while we
can leave the $A$-vertices unlabeled). 

Call $\sigma$, $\tau$ and $\rho$ the permutations such that, if the
blue, red and green edges incident to the same $A$-vertex have indices
$i$, $j$ and $k$, respectively, then $j=\sigma(i)$, $k=\tau(j)$ and
$i=\rho(k)$.  An example of labeling is
\[
\if\faifig1  

\else [...TikZ\ code...] \fi
\]
It is easily seen that:
\begin{itemize}
\item $\rho\tau\sigma$ is the identity permutation, that is
  $\rho(\tau(\sigma(i)))=i$ for all $1\leq i \leq n$, because all
  $A$-vertices have degree 3;
\item $\sigma$, $\tau$ and $\rho$ are all of cycle type $\lambda=(n)$,
  as we asked that $G$ is 3-Hamiltonian;
\item $\sigma$ is always the long-cycle permutation
  $(n\,n-1\,\ldots\,2\,1)$, as we have labeled the vertices according
  to the blue-red path;
\item the choices of $(\tau,\rho)$ such that the properties above do
  hold are in bijection with 3-Hamiltonian graphs $G$ (in particular,
  graphs as $G$, once that they are rooted at a $B$-vertex and
  3-edge-colored, have no non-trivial automorphisms).
\end{itemize}
Thus $Z_n$ is given by the number of ways in which the canonical long
cycle $\sigma^{-1}=(1\,2\,3\,\ldots\,n)$ in $\mathfrak{S}_n$ can be
decomposed as the product of two long cycles, $\tau$ and $\rho$.  As
the long cycle in $\mathfrak{S}_n$ has signature $(-1)^{n-1}$, $Z_n$
can be non-zero only for odd values of $n$, and in this case it can be
evaluated through the theory of characters of the symmetric group (see
for example~\cite[Ex.~4.9]{BOCCARA1980105} or \cite{STANLEY1981255}, or 
{\tt https://oeis.org/A060593} and references therein):
%% Counting Cycles in Permutations by Group Characters, With an Application to a Topological Problem
%% D.M. Jackson
%% Transactions of the American Mathematical Society
%% Vol. 299, No. 2 (Feb., 1987), pp. 785-801 (17 pages)
%
\be
Z_{2n+1}=\frac{(2n)!}{n+1}=(1,1,8,180, 8064, 604800,\ldots) 
\ee

\begin{figure}[t]
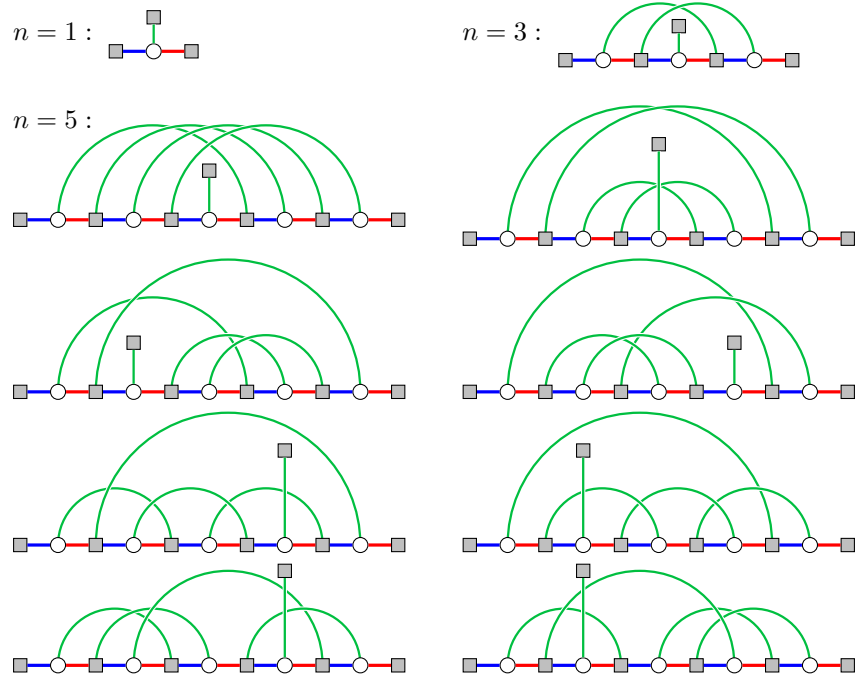

\if\faifig1  
\begin{align*}
% {1}
&
n=1:\ 

\end{align*}
\else [...TikZ\ code...] \fi
\caption{\label{fig:3HamN5}%
All 3-Hamiltonian graphs with up to $n=5$ $A$-vertices (circles), and
$7$ $B$-vertices (squares).}
\end{figure}

\noindent
Now we are ready to state and prove one of the main tools of this paper:
\begin{lem}
\label{lem:tripla}
Let $\bipg{G}{A}{B}{E}$ be a generic-weighted bipartite graph and let
$\calJ=\{U_1,U_2,U_3\}\subseteq\calV(G)$. Call $U\coloneqq U_1 \cap
U_2 \cap U_3$, and $U'_i\coloneqq U_i\setminx U$. Suppose that we are
in one of the four cases below:\footnote{Which are all the cases
  obtained from the first row of the table by applications of
  Proposition~\ref{prop.tradeaforb}, and also, for $U=\varnothing$,
  all the single-path cases with three sets up to swapping $A$
  and~$B$.}
\begin{center}
\begin{tabular}{ccc}
\toprule
$U'_1$ & $U'_2$ & $U'_3$ \\
\midrule
$\{a_1\}$ & $\{a_2\}$ & $\{a_3\}$ \\
$\{a_1, b_3\}$ & $\{a_2, b_3\}$ & $\varnothing$ \\
$\{a_1, b_2, b_3\}$ & $\{b_3\}$ & $\{b_2\}$ \\
$\{b_2, b_3\}$ & $\{b_1, b_3\}$ & $\{b_1, b_2\}$ \\
\bottomrule
\end{tabular}\end{center}
\iffalse
\[
\begin{array}{|ccc|}
\hline
\rule{0pt}{11pt}%
U'_1 & U'_2 & U'_3 
\raisebox{-6pt}{\rule{0pt}{11pt}}%
\\
\hline
\rule{0pt}{11pt}%
\{a_1\} & \{a_2\} & \{a_3\} \\
\{a_1,b\} & \{a_2,b\} & \varnothing \\
\{a,b_2,b_3\} & \{b_3\} & \{b_2\} \\
\{b_2,b_3\} & \{b_1,b_3\} & \{b_1,b_2\}
\raisebox{-6pt}{\rule{0pt}{11pt}}%
\\
\hline
\end{array}
\]
\fi
Then the graph $H_{\calJ}^*$ is a tree, and all vertices in 
$(B \setminx(U_1 \cup U_2 \cup U_3))\cap V_{H_{\calJ}^*}$ have
degree~$2$.
\end{lem}
\begin{proof}
  Observe that, up to (possibly multiple) applications of Proposition
  \ref{prop.tradeaforb} above, it is sufficient to consider the last
  case of the table above.

For an edge $e\in H_{\calJ}^*$, we say that $e$ is
\emph{bicolored} if $|\calC(e)|=2$ and \emph{monocolored} if
$|\calC(e)|=1$.  With our labeling, $\calC_{b_1}=\{U_2,U_3\}$ (and so
on), thus $b_1$ must be a leaf of $H_{\calJ}^*$, and the only edge
$e=\edge{a}{b_1}$ in $H_{\calJ}^*$ incident to $b_1$ must have
$\calC(e)=U_1$.  In other words, the three $b_i$'s are leaf vertices,
incident to monocolored edges.  All other vertices of $H_{\calJ}^*$
are either of degree $2$, and incident to one monocolored and one
bicolored edge, or of degree $3$, and incident to monocolored edges
only.

We will perform a series of manipulations on $H_{\calJ}^*$ in order to
simplify its analysis. Let us consider a bicolored edge 
$e \in H_{\calJ}^*$.  
% $e=\{a,b\} \in H_{\calJ}^*$.  
For what we said above,
% its endpoints are not in a $U_j$, and 
such an edge will be necessarily adjacent to two monocolored edges,
say $e_1$ and $e_2$, one per
endpoint,
and 
having the remaining color,
$\calC(e_1)=\calC(e_2)= \calJ\setminx \calC(e)$. We construct a new,
smaller graph
% $\tilde{G}$ and subgraph 
$\tilde{H}_{\calJ}^*$, that coincides with $H_{\calJ}$ everywhere
outside the path $(e_1,e,e_2)$, and has this path replaced by a single
new edge $\tilde{e}$, and state that its weight is given by the
alternating sum of the path, $w_{\tilde{e}}\coloneqq
w_{e_1}-w_e+w_{e_2}$.  Also, we set
$\calC(\tilde{e})=\calC(e_1)=\calC(e_2)$. Pictorially\footnote{Here
  and in the following we will pictorially represent as shaded the
  vertices that might have other incident edges within $H_{\calJ}$,
  and use squares for vertices in $B$, and circles for vertices in
  $A$.}
\begin{equation*}
\if\faifig1  
\begin{tikzpicture}%[baseline={([yshift=-.5ex]current bounding box.center)}]
    \node[circle,inner sep=2pt,fill=white,draw=black!20] (1) at (0,0)   {};
    \node[inner sep=2.5pt,rectangle,fill=gray!50,draw] (2) at (1,0)   {};
    \node[circle,inner sep=2pt,fill=white,draw] (3) at (2,0)   {};
    \node[inner sep=2.5pt,rectangle,fill=gray!30,draw=none] (4) at (3,0)   {};
    \draw[very thick,red] (2) to[out=15, in=165,looseness=0.2,edge node={node [above] {\small\color{black} $w_e$}}] (3);
    \draw[very thick,blue] (2) to[out=-15, in=-165,looseness=0.2] (3);
    \draw[very thick,green!75!blue] (1) to[edge node={node [above] {\small\color{black} $w_{e_1}$}}] (2);
    \draw[very thick,green!75!blue] (3) to[edge node={node [above] {\small\color{black} $w_{e_2}$}}] (4);
  \end{tikzpicture}
\else [...TikZ\ code...] \fi
\mapsto
\if\faifig1  
\begin{tikzpicture}%[baseline={([yshift=-.5ex]current bounding box.center)}]
    \node[circle,inner sep=2pt,,draw=black!20] (1) at (0,0)   {};
    \node[inner sep=2.5pt,rectangle,fill=gray!30,draw=none] (4) at (1,0)   {};
    \draw[very thick,green!75!blue] (1) to[edge node={node [above] {\small\color{black} $w_{e_1}-w_e+w_{e_2}$}}] (4);
  \end{tikzpicture}
\else [...TikZ\ code...] \fi
\end{equation*}
These manipulations are performed in such a way that various
properties of our graph are preserved: alternating sums along paths
and cycles are left unchanged, as well as the property of being
alternating (or semi-alternating) w.r.t.\ certain colors, and the
vertices that have been removed have degree 2, so that 
the pertinent required property (that all vertices in 
$(B \setminx(U_1 \cup U_2 \cup U_3))\cap V_{H_{\calJ}^*}$ have
degree~$2$) holds for $H_{\calJ}^*$ if and only if it holds for
$\tilde{H}_{\calJ}^*$.  Of course, connectivity and bipartition are
also unaffected.

By performing these operations as long as possible, in any order, we
are left with a graph in which no edge is bicolored. And thus, no
vertex is of degree~$2$.  With abuse of language, we will denote this
graph by $H_\calJ^*$ as well.  Thus $H_\calJ^*$ has three leaves,
namely the vertices $\{b_1,b_2,b_3\}$, and all other vertices of
degree~3. Furthermore, the graph $H_\calJ^*$ is bipartite, the edges
are properly 3-colored, and (for this reason, and also because it is
the union of three paths sharing their endpoints) it is connected.

By equating the number of half-edges attached to vertices in $A$ and
in $B$, we get a linear relation between $n=|A|$ and $|B|$, namely
$|B|=n+2$. The content of our statement is that, in fact, $n=1$ (and
$|B|=3$, that is, in the original graph $H^*_\calJ$, there are no
vertices $b \in B$ with degree larger than 2).  In other words, the
graph obtained through our sequence of manipulations is exactly the
following one:
\begin{equation}
\label{gr:triplo}
 H_\calJ^*=
\if\faifig1  
\begin{tikzpicture}[baseline={([yshift=-.5ex]current bounding box.center)}]
    \node[circle,inner sep=2pt,fill=white,draw] (1) at (0,0)   {};
    \node[inner sep=2.5pt,rectangle,fill=gray!50,draw] (2) at (0.25, 0.433013)   {};
    \node[inner sep=2.5pt,rectangle,fill=gray!50,draw] (3) at (0.25, -0.433013)   {};
    \node[inner sep=2.5pt,rectangle,fill=gray!50,draw] (4) at (-0.5,0)   {};
    \draw[very thick,blue] (1) to (2);
    \draw[very thick,red] (1) to (3);
    \draw[very thick,green!75!blue] (1) to (4);
  \end{tikzpicture}
\else [...TikZ\ code...] \fi
\end{equation}
We know that $H_\calJ^*$ cannot have any alternating cycle. Thus, it
must be a ``3-Hamiltonian graph'', in the sense introduced above in
Definition~\ref{def.3ham}.

Let us denote by
\be
P
=
\left(
b_1,\alpha_1,\beta_1,\alpha_2,\dots,\beta_{n-1},\alpha_{n},b_2\right)
% b_1,\alpha_1,\beta_1,\alpha_2,\dots,\beta_{\ell-1},\alpha_{\ell},b_2\right)
%% \left(\{
%% b_1,\alpha_1\},\{\alpha_1,\beta_1\},\{\beta_1,\alpha_2\},\dots,\{\beta_{\ell-1},\alpha_{\ell}\},\{\alpha_\ell,b_2\}\right)
\ee
the unique path in $H_\calJ^*$, alternating in colors $U_1$ and $U_2$
and connecting $b_1$ to $b_2$, that is, we have named
$\{\alpha_i\}_{i=1}^{n}$ and $\{\beta_i\}_{i=1}^{n-1}$ the
$A$-vertices and $B$-vertices of degree 3 of our graph, in the order
induced by the path.

The rest of the graph consists of edges of color $U_3$. One of this
edges connects an $A$-vertex (say, $\alpha_\kappa$, for 
$1\leq \kappa \leq n$) to the leaf $b_3$, while the $n-1$
remaining ones pair the remaining $A$-vertices along the path to the
non-leaf $B$-vertices along the path.  Let us denote by
$\pi\colon[n-1]\to[n]\setminx\{\kappa\}$ the map such that
$\edge{\alpha_{\pi(i)}}{\beta_i}$ is an edge of color $U_3$ in
$H_\calJ^*$.

Representing colors $U_1$, $U_2$ and $U_3$ as blue, red and green,
respectively, an example is the following:
\be
\label{eq:spine}
\if\faifig1  
\begin{split}
&

\\
&
\kappa=3\,;\qquad \pi\,:\,(1,2,3,4,5,6) \to (5,7,6,1,2,4)
\end{split}
\else [...TikZ\ code...] \fi
\ee
where we have adopted a useful graphical notation, to draw an arc of
color $U_3$ (in green) above or below the spine, if the $A$-vertex
endpoint is on the right or on the left of the $B$-vertex endpoint,
respectively.

At this point we have exploited at the maximal possible extent the
consequences of the fact that $H_\calJ^*$ has no alternating cycles,
reducing the problem from all viable (connected 3-colored bipartite)
graphs with three leaves to the relatively smaller family of
3-Hamiltonian graphs.  And this has not been sufficient to rule out
all other possibilities besides the small Y-shaped graph with $|A|=1$
in equation (\ref{gr:triplo}), that is the claim of our lemma (for
example, the relatively large graph in the picture above is indeed a
large example of 3-Hamiltonian graph, as it has no alternating cycle
whatsoever, because the alternating path from $b_1$ to $b_3$ visits
all vertices except for $b_2$, and the alternating path from $b_2$ to
$b_3$ visits all vertices except for $b_1$). We shall then use the
information that comes from semi-alternating cycles. Recall (from
Remark \ref{rmk.28976954} on page \pageref{rmk.28976954}) that each
semi-alternating cycle provides a (strict) linear inequality among the
edge weights. A remarkable fact is that, although the family of graphs
of the form above is infinite, we have a general recipe, valid for
each graph except the one in equation (\ref{gr:triplo}), to combine
these linear inequalities (all with coefficient $1$) in order to
provide a contradiction.

We shall now describe this recipe, and start by
listing our choice of semi-alternating cycles.  First, we have
$n-1$ cycles $L_k$, which are exactly the cycles that contain a
single green edge $e^{(k)}=\edge{\alpha_{\pi(k)}}{\beta_k}$ of color $U_3$.
% besides the edge $e_3$ (green in our drawing). 
These cycles contain $e^{(k)}$, and the portion of the path $P$
between $\alpha_{\pi(k)}$ and $\beta_k$.  The cycle $L_k$ is
semi-alternating of color $U_1$ (blue in the drawing) or $U_2$ (red),
if $k>\pi(k)$ or $k\leq \pi(k)$, respectively (that is, if the arc is
drawn above or below the spine, in our graphical notation).

Now let us consider the subgraph consisting of all the $n-1$ edges
$e^{(k)}$, the $\kappa-1$ edges of color $U_2$ (red) on the portion of
path $P$ between $b_1$ and $\alpha_\kappa$, and the $n-\kappa$ edges
of color $U_1$ (blue) on the portion of $P$ between $\alpha_\kappa$
and $b_2$.  All the vertices of this graph have degree 2, and are
incident to exactly one edge of color $U_3$ (green), so this graph is
composed of a collection $C_j$ of cycles, all semi-alternating
on color $U_3$. These cycles are illustrated here for our running
example:
\begin{equation}
\label{eq:spinecycles}
\if\faifig1  

\else [...TikZ\ code...] \fi
\end{equation}
(there are two cycles in our example).

Now, let us consider the inequality $I[w]>0$ derived from the sum of
the inequalities associated to all the semi-alternating cycles $L_k$
and $C_j$.
% those associated to the semi-alternating cycles 
Our goal is to prove that this inequality cannot be satisfied, because
the expression $I[w]=\sum_e c_e w_e$, with $c_e \in \mathbb{Z}$, is in
fact identically zero, i.e.\ all the $c_e$'s are zero.

An easy observation is that the weights of the edges incident to the
three leaves $b_i$ 
%$w_{e_3}$
do not appear in any of our inequalities, as these edges are just in
no cycle whatsoever.
%  $e_3$ is adjacent to a leaf so it is not in any cycle.

Each edge $e^{(k)}$ of color $U_3$ (other than $e_3$) appears in the
inequality for $L_k$, and in exactly one inequality for the $C_j$'s,
and with opposite sign (because all the $C_j$'s are semi-alternating
in color $U_3$, and the $L_k$'s are semi-alternating in colors $U_1$
or $U_2$), so that the corresponding coefficient cancels out in the
expression~$I[w]$.

Then, for an edge $e \in P$, let us call $n_e^{\pm}$ the number of
times that $w_e$ occurs in an inequality of type $L_k$, with sign
$\pm$, and $m_e$ the number of times that $w_e$ occurs in an
inequality of type $C_j$ (when it appears, it always does with
coefficient $-1$). We have 
$m_{\edge{\alpha_i}{\beta_i}}=1$ if $i<\kappa$,
$m_{\edge{\beta_i}{\alpha_{i+1}}}=1$ if $i\geq \kappa$, and $m_e=0$ for
all other edges $e$ in $P$ (that is, edges $e$ of color $U_2$ left of
$\alpha_\kappa$ and of color $U_1$ right of $\alpha_\kappa$ have
$m_e=1$, while edges of color $U_1$ left of $\alpha_\kappa$ and of
color $U_2$ right of $\alpha_\kappa$ have $m_e=0$). Then, if $e$ and
$e'$ are consecutive along $P$, with $e$ on the left, and not the two
edges incident to $\alpha_\kappa$, we have that they appear in the
same list of cycles $L_k$ (and with opposite sign), except for the
unique cycle associated to the edge $e^{(k)}$ incident to both $e$ and
$e'$.  This implies the relations
\begin{enumerate}
\item $n_{e'}^+=n_e^-$ and $n_{e'}^-=n_e^+$ if $e$ and $e'$ are the edges incident to $\alpha_\kappa$;
\item $n_{e'}^+=n_e^-+1$ and $n_{e'}^-=n_e^+$ if $e$, $e'$ are incident to $\alpha_i$ for $i\neq \kappa$, and $\pi(i)\geq i$;
\item $n_{e'}^+=n_e^-$ and $n_{e'}^-=n_e^+-1$ if $e$, $e'$ are incident to $\alpha_i$ for $i\neq \kappa$, and $\pi(i)<i$;
\item $n_{e'}^+=n_e^-$ and $n_{e'}^-=n_e^+-1$ if $e$, $e'$ are incident to $\beta_i$ and $\pi^{-1}(i)>i$;
\item $n_{e'}^+=n_e^-+1$ and $n_{e'}^-=n_e^+$ if $e$, $e'$ are incident to $\beta_i$ and $\pi^{-1}(i)\leq i$.
\end{enumerate}
This pattern is illustrated in the picture below, where the number of
cyan/orange dots on the same column of the edge $e$ correspond to the
value of $n_e^+$ and $n_e^-$, respectively, and the cyan dots above
the line correspond to $m_e=1$:
\begin{equation}\label{eq:spinesign}
\if\faifig1

\else [...TikZ\ code...] \fi
\end{equation}
Combining (2) and (3) gives that
$n_{e'}^+-n_{e'}^-=-(n_{e}^+-n_{e}^-)+1$, if $e$, $e'$ are incident to
$\alpha_i$ for $i\neq \kappa$, regardless from $\pi$, while combining
(4) and (5) gives that $n_{e'}^+-n_{e'}^-=-(n_{e}^+-n_{e}^-)+1$, if
$e$, $e'$ are incident to $\beta_i$, regardless from $\pi$. But the
sequence $f_n + f_{n+1} = 1$ is the string
$(\ldots,0,1,0,1,0,1,\ldots)$ provided that a single value $f_0$ is
zero. Using the initial condition that $n_{e}^+-n_{e}^-=0$ on the
leftmost and rightmost edges of the path, gives that
$n_{e}^+-n_{e}^-=m_e$ for all edges of $P$, that is restated in the
fact that the coefficient of the corresponding weight $w_e$ cancels
out in the expression $I[w]$. As a result, as a function of $w$, we
have $I[w]=0$, which is in contradiction with the fact that $I[w]$ is
a non-trivial positive sum of strict inequalities.

Having discarded all possibilities with $n>1$, we are left with $n=1$,
i.e.\ there are no $B$-vertices of degree 3, as was to be proven.
\end{proof}

\begin{figure}[t]
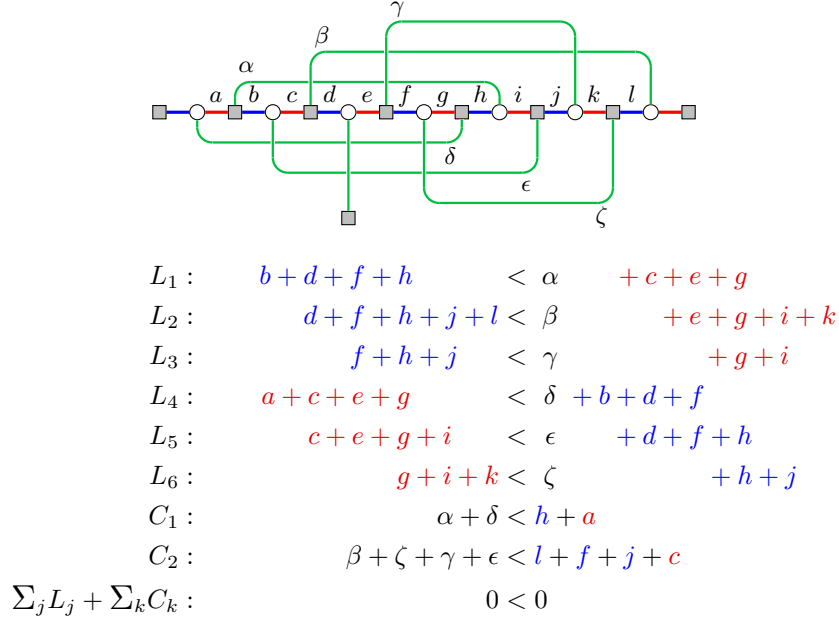

\[
% \label{eq:spine}
\if\faifig1  

\else [...TikZ\ code...] \fi
\]
% OSSERVAZIONE DI LATEX
% il tex distanzia i + in maniera diversa in espressioni tipo "+y" e
% "x+y", quindi nei phantom BISOGNA aggingere dei simboli di larghezza
% nulla {} per forzare l'altro tipo di spacing.
\begin{align*}
L_1:&& 
\goblue{b+d+f+h\phantom{{}+j+l}}&<\mathmakebox[12pt][c]{\alpha}\gored{\phantom{{}+a}+c+e+g}
\\
L_2:&& 
\goblue{\phantom{b+{}}d+f+h+j+l}&<\mathmakebox[12pt][c]{\beta}\gored{\phantom{{}+a+c}+e+g+i+k}
\\
L_3:&&
\goblue{\phantom{b+d+{}}f+h+j\phantom{{}+l}}&<\mathmakebox[12pt][c]{\gamma}\gored{\phantom{{}+a+c+e}+g+i}
\\
L_4:&& 
\gored{a+c+e+g\phantom{{}+i+k}}&<\mathmakebox[12pt][c]{\delta} \goblue{{}+b+d+f}
\\
L_5:&&
\gored{\phantom{a+{}}c+e+g+i\phantom{{}+k}}&<\mathmakebox[12pt][c]{\epsilon}
\goblue{\phantom{{}+b}+d+f+h}
\\
L_6:&&
\gored{\phantom{a+c+e+{}}g+i+k}&<\mathmakebox[12pt][c]{\zeta} \goblue{\phantom{{}+b+d+f}+h+j}
\\
C_1:&&
\alpha+\delta&<\goblue{h}+\gored{a}
\\
C_2:&&
\beta+\zeta+\gamma+\epsilon&<
\goblue{l}+\goblue{f}+\goblue{j}+\gored{c}
\\
\text{\raisebox{-.5pt}{\scalebox{1.3}{$\Sigma$}}}_j L_j
+
\text{\raisebox{-.5pt}{\scalebox{1.3}{$\Sigma$}}}_k C_k
% \textrm{\textstyle{$\sum_j C_j + \sum_k L_k$}} 
:&&
0&<0
\end{align*}
\caption{Illustration of the expression $I[w]$ for our running
  example.}
% The sum of the inequalities above produces the contra\-diction~$0<0$.}
\end{figure}

\noindent
Once the fact above for sets $\calJ$ of cardinality 3 has been
established, we can deduce similar consequences on settings of more
practical interest, where $|\calJ|$ is arbitrarily large. We start by
proving a structure theorem on sets $\calJ$:
\begin{lem}
\label{lem.formaJ}
Let $\bipg{G}{A}{B}{E}$ be a bipartite graph, and $\calJ$ a
single-path\footnote{I.e., for all $U_i$, $U_j \in \calJ$ we have
  $|U_i \symdif U_j|=2$, as in Definition~\ref{def.singlepath}.}
collection of cardinality at least 2 such that $\bigcap_{U \in \calJ}
U=\varnothing$, and for all $U_i$, $U_j$, $U_k \in \calJ$
% and, for all subsets $\calJ' \subseteq \calJ$ of cardinality 3, 
the hypotheses of Lemma \ref{lem:tripla} hold. Then there exist a set
$A_0 \subseteq A$ and a set $B_0 \subseteq B$, such that
$|\calJ|=|A_0|+|B_0|$, and the collection $\calJ$ is of the form
\begin{equation}
\label{eq.3478687645}
\calJ = \{ \{a\} \cup B_0 \}_{a \in A_0} \cup
\{ B_0 \setminx b \}_{b \in B_0}
.
\end{equation}
\end{lem}
\begin{proof}
First of all, observe that the construction given above in
(\ref{eq.3478687645}) is easily verified to be compatible with the
requirements.
% table. 
We shall prove that no other structures are possible.

The case $|\calJ|=2$ is steadly verified.
%% Remark again that, for the hypotheses of Lemma \ref{lem:tripla} to
%% hold (that is, such that all triples of removals have the structure of
%% one of the rows of the table), we shall have that the ``Hamming
%% distance'' among any two distinct removals is exactly 2, that is,
%% $|U_i\symdif U_j|=2$ for all
%% %
%% $U_i$, $U_j \in \calJ$ with $i\neq j$. Then, 

Then, for $|\calJ|\geq 3$, observe that, in all four cases of the
table in Lemma \ref{lem:tripla}, an $A$-vertex appears in a single set
(or in all three sets, or none), and a $B$-vertex appears in two sets
(or in all three sets, or none).

Let us show that $|U_i \cap A| \leq 1$ for all $i$.  Assume, for the
sake of contradiction, that there exists a set $U \in \calJ$, and two
distinct elements $a_1$, $a_2$ in $U$. As $\bigcap_{U \in J}
U=\varnothing$, there must be a set $U'$ that contains only $a_1$ and
a set $U''$ that contains only $a_2$, or a set $U'$ that does not
contain neither $a_1$ nor $a_2$, and some other set $U''$ (because
$|\calJ|\geq 3$). In both cases, the triple $\{U,U',U''\}$ has a
structure incompatible with the table.

% We can tentatively set $A_0=(\bigcup_{U \in \calJ} U) \cap A$, and
% $B_0=(\bigcup_{U \in \calJ} U) \cap B$.

Analogously, let us define $\tilde{B}_0=(\bigcup_{U \in \calJ} U) \cap
B$, and show that $|\tilde{B}_0 \setminx (U_i \cap B)| \leq 1$ for all
$i$.
%  (for $B_0$ defined as above). 
Suppose, for a contradiction, that
there exists a set $U \in \calJ$, and two distinct
elements $b_1$, $b_2$ in $\tilde{B}_0$ but not in $U$. The definition
of $\tilde{B}_0$ implies that there must be a set $U'$ that contains
only $b_1$ and a set $U''$ that contains only $b_2$, or a set $U'$
that contains both $b_1$ and $b_2$, and some other set $U''$ (because
$|\calJ|\geq 3$). In both cases, the triple $\{U,U',U''\}$ has a
structure incompatible with the table.

Finally, let us show that, for all $U$, the pair of integers 
$(|U_i \cap A|,|\tilde{B}_0 \setminx (U_i \cap B)|)$ can only be
$(1,0)$ or $(0,1)$. As we know that $|U_i \symdif U_j|=2$ for all
$i\neq j$, either all sets have pair $(1,0)$ or $(0,1)$, or all sets
have pair $(0,0)$ or $(1,1)$.

According to our definitions, there is at most a single distinct set
$U$ with pair $(0,0)$. So there must exist two sets $U'$ and $U''$
with pair $(1,1)$, and such that $U' \symdif U''$ is either of the
form $\{a_1,a_2\}$, or of the form $\{b_1,b_2\}$, thus the triple
would be either of the form $(\{b\},\{a_1\},\{a_2\})$, or of the form
$(\{b_1,b_2\},\{a,b_1\},\{a,b_2\})$ none of which is compatible with
the table.

So we are only left with the case in which all the pairs are $(1,0)$
or $(0,1)$. In this case we see that $\calJ$ must be a subset of the
list given in equation (\ref{eq.3478687645}) with
$B_0=\tilde{B_0}$. Also, if any element of the list is missing, we
would have a contradiction either with the definition of $B_0$, or
with the fact that $\bigcap_{U \in J} U=\varnothing$. This allows us
to conclude.
\end{proof}
\noindent
% Now we are ready to establish the first theorem of this paper:
%
% Having established the fact above for sets $\calJ$ of cardinality 3
% has an immediate consequence on settings of more practical interest:
We can then establish
\begin{lem}
\label{lem:co2}
Let $\bipg{G}{A}{B}{E}$ be a generic-weighted bipartite graph, and $\calJ$
as in Lemma \ref{lem.formaJ} (in particular, as in equation
(\ref{eq.3478687645})).  Then the graph $H_{\calJ}^*$ is 
connected,
$A_0 \cup B_0 \subseteq V_{H_{\calJ}^*}$,
% a tree, 
and all vertices in $(B \setminx B_0) \cap V_{H_{\calJ}^*}$
% $(B \setminx \bigcup_{U \in \calJ} U)\cap V_{H_{\calJ}^*}$
have degree $2$, while all vertices in $B_0$ are leaves.
% $J\subseteq C(G)\subseteq A$ with $|J|\geq 2$. The $B$-vertices in
% $H_{\calJ}^*$ have coordination $2$.
\end{lem}
\begin{proof} 
The case $|\calJ|=2$ is obvious, and the case $|\calJ|=3$ is the
statement of Lemma~\ref{lem:tripla}. Now, if $|\calJ|\geq 4$, we shall
prove connectivity and the fact that all $B$-vertices in $H_\calJ^*$
not in a $U_j$ are of degree $2$. 
% We shall again work, with no loss of generality, in the setting in
% which the $U_j$'s contain only $B$-vertices.  
By contradiction, if $H_\calJ^*$ has a $B$-vertex $b$ of
degree 3 or larger, let $e_1$, $e_2$ and $e_3$ be distinct edges
incident to $b$, and (up to renaming the $U_j$'s) $U_i \in \calC(e_i)$
for $i=1,2,3$. Then, calling $\calJ'=\{U_1,U_2,U_3\}$, we have that
$H_{\calJ'}^* \subseteq H_\calJ^*$, and that $b$ has degree $3$
already in $H_{\calJ'}^*$, which is not possible in light of
Lemma~\ref{lem:tripla}, and the fact that adding more colors increases
the quantities $\deg_{H_\calJ}(v)$.

Let us determine which vertices have degree 1, and which have degree
2, in light of the structure theorem above, Lemma~\ref{lem.formaJ}. 
A vertex $b \in B_0$ is present in $G_U$ only for the corresponding
set $U=B_0\setminx b$, so $\calC_b=\calJ\setminx U$ and there exists
a single edge $e$, incident to $b$, with $\calC(e)=\{U\}$, so $b$ is a
leaf. A vertex $b \not\in B_0$ is present in all $G_U$'s, so
$\calC_b=\varnothing$. If it has degree 1, then there is an incident
edge $e$ such that $\calC(e)=\calJ$, so $e$ is an isolated edge of
$H_\calJ$, and thus it is not in $H_\calJ^*$. So it must have degree 2.

Now, for connectivity, let us apply again the structure theorem, Lemma
\ref{lem.formaJ}. As we know by definition that any vertex of
$H^*_\calJ$ is connected to some vertex in $V_{\calJ}$,
it suffices to prove that all vertices in 
$A_0 \cup B_0$ are in the same component. Let $a_1$ and $a_2$ be
distinct elements of $A_0$. Then, calling $U_1=\{a_1\} \cup B_0$ and
$U_2=\{a_2\} \cup B_0$, we have that $M_{U_1} \symdif M_{U_2}$
consists of a path (of even length) connecting $a_1$ and $a_2$, and of
course $M_{U_1} \symdif M_{U_2} \subseteq M_{U_1} \cup M_{U_2}
\subseteq H_{\calJ}$, while the path must be inside $H_{\calJ}^*$ by
definition of $H_{\calJ}^*$. Similarly, if $b_1$ and $b_2$ are
distinct elements of $B_0$, calling $U'_1=B_0 \setminx b_1$ and
$U'_2=B_0 \setminx b_2$, we have that $M_{U'_1} \symdif M_{U'_2}$
consists of a path (of even length) connecting $b_2$ and $b_1$, while
if $a_1$ is in $A_0$ and $b_1$ is in $B_0$, we have that $M_{U_1}
\symdif M_{U'_1}$ consists of a path (of odd length) connecting $a_1$
and $b_1$.
\end{proof}
\noindent
Now we are ready to establish the first of the `theorems' mentioned in
the title of this paper:
\begin{thm} 
In the conditions of Lemma \ref{lem:co2}, namely, for a triple
$(G,w,\calJ)$, with $\bipg{G}{A}{B}{E}$ bipartite, generic weights, and
$\calJ$ of the form $\calJ = \{ \{a\} \cup B_0 \}_{a \in A_0} \cup \{
B_0 \setminx b \}_{b \in B_0}$ for some $A_0 \subseteq A$ and $B_0
\subseteq B$,
% given by equation (\ref{eq.3478687645}))
% Assume $ J\subseteq C$ with $| J|\geq 2$. 
the subgraph $H_{\calJ}^*$ is a tree. If $A_0=A$, then
$H_{\calJ}=H_{\calJ}^*$ is a spanning tree of $G$.
\label{thm:Jtree}
\label{th:spantree}
\end{thm}
\noindent
We will give two different proofs.\footnote{One proof is quite short
  and is based on the classical Euler formula. The second one,
  instead, is based on a structural argument that is purely local. In
  this sense it may be considered more satisfactory, as local
  arguments are intrinsecally more robust than proofs based on a
  global numerical balance.}
\begin{proof}[First proof]
Note that we must have that $\calM(G_U) \neq \varnothing$ for all
$U \in \calJ$, and in particular, as the $G_U$'s are all bipartite,
these graphs must have as many $A$-vertices as $B$-vertices, that implies
$|A|-1=|B|-|B_0|=n$ for some $n$. Let
% Let assume as above $|B|=|A|-1=n$ and let 
$n_d$ be the number of isolated dimers in $H_{\calJ}$. It follows that
the number $n_B$ of $B$-vertices of degree $2$ in $H_{\calJ}^*$ is
$n_B=n-n_d$, while the vertices in $B_0$ are all in $H_{\calJ}^*$, and
have degree 1. Then, the number of $A$-vertices in $H_{\calJ}^*$ is
$n_A=n+1-n_d=n_B+1$.  Counting the number of edges as 
\be
|E_{H_{\calJ}^*}|=\sum_{b \in V_{H_{\calJ}^*} \cap B}
\deg_{H_{\calJ}^*}(b)
\ee
we get $|E_{H_{\calJ}^*}|=|B_0|+2 n_B$.
% (n-n_d)$.
On the other side $|V_{H_{\calJ}^*} \cap B| = |B_0|+n_B$ and
$|V_{H_{\calJ}^*} \cap A| = n_B+1$, that is 
$|V_{H_{\calJ}^*}| = |B_0|+2n_B+1$. As $H_{\calJ}^*$ is connected, by
Euler formula $H_{\calJ}^*$ is a tree. If $A_0=A$, we have 
$A \subseteq V_{H_{\calJ}^*}$. As $G$ is bipartite, there cannot be
any isolated dimer in $H_{\calJ}\setminx H_{\calJ}^*$ (because one of
its endpoints would be an $A$-vertex), thus $H_{\calJ}=H_{\calJ}^*$
and $V_{H_{\calJ}}=V_G$.
\end{proof}

\begin{proof}[Second proof]
Let us now provide a second proof, not relying on the Euler
formula. We will only use that $H_\calJ^*$ is connected, and all the
$B$-vertices have degree 1 or 2.
%  In light of Proposition \ref{prop.tradeaforb}, we can assume that,
%  for some set $B_0$, we have $\calJ=\{ B_0 \seminx b \}_{b \in B_0}$.
What is left to prove is that $H_\calJ^*$ has no cycles. If
$H_\calJ^*$ consists of a Hamiltonian cycle, there is some partition
$\calJ=J_1 \cup J_2$ such that the sets $\calC(e)$ along the cycle are
alternately $J_1$ and $J_2$, so $H_\calJ^*$ would be an alternating
cycle for some pair of colors $U_1 \in J_1$ and $U_2 \in J_2$, which
is not allowed. Thus, $H_\calJ^*$ has a cycle of length $2\ell$, with
edges $(e_1,e_2,\ldots,e_{2\ell})$, containing some $A$-vertex $a$ of
degree larger than 2, say between edges $e_{2\ell}$ and $e_1$.  The
fact that $a$ has degree larger than 2 implies that there is some
color $U \in \calJ$ which is contained in a set $\calC(e)$ for some
$e$ incident to $a$, and not in the cycle. Now, this color is neither
in $\calC(e_{2\ell})$ nor in $\calC(e_1)$, but it is in
$\calC(e_{2\ell-1})$ and in $\calC(e_2)$, because the vertices between
$e_{2j-1}$ and $e_{2j}$ are $B$-vertices, and have degree
2. Similarly, it is not in $\calC(e_{2\ell-2})$ nor in $\calC(e_3)$,
but it is in $\calC(e_{2\ell-3})$ and in $\calC(e_4)$, and so
on. Given the parity of the cycle length, this gives a contradiction.
As we have found a contradiction, we have determined that $H_\calJ^*$ has no
cycles, thus it is a tree.
\end{proof}

\begin{rmk}
\label{rmk.reconstr}
In the conditions of the theorem above, the datum of $G$, $\calJ$ and
$H_\calJ$ allows to determine the sets $\{\calC(e)\}_{e\in E}$ and
$\{\calC_v\}_{v\in V}$ (even without knowing~$w$).

This is similar to the fact that we can prune completely a tree by
removing the leaves in any order. Indeed, let us initialize a set of
auxiliary quantities $\calD_v = \calJ$. The sets $\calC_v$ are all
known from $\calJ$ alone. Then, at any step of the algorithm, if
$e=\edge{u}{v}$ is an edge attached to the leaf $u$, we can set
$\calC(e)=\calD_u \setminx \calC_u$, update 
$\calD_v \to \calD_v \setminx \calC(e)$, and remove the edge $e$ from
the tree.
\end{rmk}
\noindent
Let us introduce a notion of \emph{projection} for bipartite graphs:
  \footnote{In network analysis and complex systems, the operation of
    collapsing one partition of a bipartite graph to induce a
    structure on the remaining nodes via shared neighbors is commonly
    known as \emph{one-mode projection} (or \emph{bipartite
      projection}). When the shared nodes generate groups of size
    greater than two, as will be the case in the generalization of
    Section \ref{sec.ktuples}, the resulting structure naturally
    generalizes from a projected graph to a hypergraph projection
    (sometimes related to the incidence hypergraph or hypergraph
    representation of two-mode networks). In this paper, for short, we
    just use the name \emph{projection}, as there is no risk of
    confusion.}
\begin{defn}[Projection]
\label{def.proj}
Let $\bipg{G}{A}{B}{E}$ be a bipartite graph.  The \emph{projection} of
$G$ is the graph $\bar{G}=(A,E')$ (possibly not simple), where in $E'$
we have as many copies of the edge $\edge{a_1}{a_2}$ as many vertices
$b \in B$ exist such that $\edge{a_1}{b}$ and $\edge{a_2}{b}$ are
edges of $E$. In other words, the neighborhood of a vertex $b$ in $G$
is some star-graph with leaves $a_1$, \ldots, $a_\ell$, and is
replaced in $\bar{G}$ by the $\ell$-clique on $a_1$, \ldots,
$a_\ell$. Then, the union of these cliques has to be taken by allowing
multiple edges.
% on $U\subseteq A$ is the graph $\bar G_U=(U,\bar E_U)$ such that
% $\bar E_U=\{\{u_1,u_2\}\ \vert\ \exists b\in B\text{ with }
% \{u_1,b\},\{u_2,b\}\in E\}$.
\end{defn}
\noindent
Note that $\bar{G}$ is simple iff $G$ has girth larger than~4 (and
more generally the girth of $\bar{G}$ is half of that of $G$), and
that the number of connected components of $\bar{G}$ is equal to the
number of components of $G$, minus the number of isolated
$B$-vertices.

\begin{rmk}
As we will be mainly interested in the projection of matching
subgraphs, it is worth noting that the projection of an isolated dimer
is an isolated vertex, thus $\overline{H^*_{\calJ}}$ coincides with
$\overline{H_{\calJ}}$ with isolated vertices dropped out.
\end{rmk}

\begin{rmk}
\label{rmk.projForIsFor}
The projection of a forest $F=(A\cup B,E)$ where all $B$-vertices have
degree 1 or 2 is a forest $\bar{F}$ with the same number of components
as~$F$.  In particular, in the conditions of Theorem \ref{thm:Jtree},
with $A_0=A$, the tree $\bar{H}_\calJ=\overline{(H_\calJ)}$ is
spanning on $\bar{G}$.
\end{rmk}

\noindent
% Similarly to what we have discussed for non-bipartite graphs, 
Recall that, up to completing the weight function with sufficiently
high values, we can always imagine to ``complete'' a bipartite graph
$\bipg{G}{A}{B}{E}$, with $|A|=|B|+1$, to the complete bipartite
graph $\calK_{|A|,|B|}=\calK_{n+1,n}$. In this case, we do not need to
check for the fact that $\calM(G_U)$ is non-empty, and we only need to
ensure the genericity of the weights (which holds almost-surely in
various randomized settings).
% and that $|A|=|B|+1$. 
Thus, as a result of the analysis above, we have determined another
interesting setting:
\begin{corr}
\label{cor.setting1}
In a triple $(G,w,\calJ)$, where 
$\bipg{G}{A}{B}{E}=\calK_{n+1,n}$, $w$ is a generic weight
function, and $\calJ=\{ \{a\} \}_{a \in A}$, we have that
$H_\calJ=H_\calJ^*$ is a spanning tree on $\calK_{n+1,n}$, with all
$B$-vertices of degree 2, and $\bar{H}_\calJ$ is a spanning tree on
$\calK_{n+1}$.
\end{corr}
\noindent
We shall call the \emph{first bipartite standard setting} the type of
triple in the corollary above.

The precise information on the sets $\calC(e)$ that we have
accumulated along the proofs of this section allows us to establish
also the following interesting fact:
\begin{lem}
\label{lem.subtreeIsHJ}
Let $(G,w)$ be a graph with generic weights, in the first bipartite
standard setting, and let $H$ be the resulting matching tree. Thus
$|V(G)|=|V(H)|=2n+1$. Let $H'$ be a subtree of $H$, such that all
$B$-vertices have degree 2 and the restriction of $(G,w)$ to $V(H')$
is a valid instance (in particular, $H'$ has $m+1$ $A$-vertices and
$m$ $B$-vertices). Then the matching tree associated to the
restriction of $(G,w)$ to $V(H')$ coincides with~$H'$.
\end{lem}
\begin{proof}
Call $C\subseteq \{\{a_i\}\}$ the set of ``colors'' associated to the
$A$-vertices of $H'$. We claim that $H'$ is characterized as the
subgraph of $H$ such that $e \in E(H')$ iff 
$\calC(e) \cap C \neq \varnothing, C$.
Suppose by contradiction that this is not the case. The first
possibility is that there exists $e=\edge{a}{b} \in E(H')$ with 
$\calC(e) \cap C = C$. But this is excluded immediately, because 
$a \in C$, and $a\not\in \calC(e)$.  Another possibility is that
$\calC(e) \cap C = \varnothing$. As $b$ has degree 2 in $H'$, there
must be a second edge $e'=\edge{a'}{b}$ in $H$, and also $a'$ is in
$C$. As $\calC(e)\cup\calC(e')=\calJ$, and $a' \not\in \calC(e')$, it
must be that $a' \in \calC(e)$, and this produces a contradiction.
Then there is a third possibility, namely that there exists
$e\in E(H) \setminx E(H')$, and $a,a' \in C$ such that 
$a \in \calC(e)$ but $a' \not\in \calC(e)$. However 
$M(a) \symdif M(a')$ consists of a single open path on $H$, with
endpoints in $a$ and $a'$, and as both $a$ and $a'$ are in $H'$, that
is a subtree of $H$, this path must be fully contained in $H'$. All
other edges $e'$ of $G$ not in this path are such that
$\calC(e')\cap\{a,a'\}$ is either $\varnothing$ or $\{a,a'\}$, so we
reached a contradiction. What is left is our claim.

This implies that the restriction of $(G,w,\calJ)$ to
$(G,w,\{\{a\}\}_{a \in V(H')\cap A})$ has matching subgraph consisting
of $H'$ plus a collection of isolated dimers, all contained in $H$.

Now, portions of the graph $G$ that are the same in all $M_U$'s, for 
$U \in \calJ$, have a trivial role in the problem, and can be removed
consistently. In particular, if the instance $(G,w,\calJ)$ has a
matching subgraph $H$ containing an isolated dimer $e=\edge{a}{b}$,
and $a$ and $b$ do not appear in any of the $U$'s, calling $(G',w')$
the graph obtained by removing $e$ and all of its incident edges, and
restricting $w$ accordingly, the matching subgraph of $(G',w',\calJ)$
must coincide with $H\setminx e$. The iterated application of this
reasoning allows us to conclude our claim on the matching
subgraph~$H'$.
\end{proof}

\subsection{A digression: a simultaneous Hungarian gauge}
\label{sec.hunggauge}
%-------------------------------------------------------

\noindent
The Assignment Problem, in its full generality, has a ``gauge
invariance'':
\begin{rmk}[Gauge Invariance]
\label{rmk.gaugeinv}
Let $W$ be a $n \times n$ real matrix, and
$\bm{\lambda}=(\lambda_1,\ldots,\lambda_n)$,
$\bm{\mu}=(\mu_1,\ldots,\mu_n)$ two vectors. Define
$W'_{ij}=W_{ij}-\lambda_i-\mu_j$. Then, calling $\cW(\pi)$ the cost of
the assignment $\pi$ for the matrix of costs $W$, 
$\cW'(\pi)$ the cost for the matrix of costs $W'$, 
and
$c=\sum_{i=1}^n(\lambda_i+\mu_i)$, we have
\be
\cW'(\pi) = \cW(\pi)-c
\ee
and in particular $\pi_\star(W)=\pi_\star(W')$.
\end{rmk}
\noindent
This notion plays a prominent role in the design of the
\emph{Hungarian Algorithm}, which identifies $\pi_{\star}(\{W_{ij}\})$
in polynomial time. In particular, the algorithm works also because it
can certificate that the matching $\pi_\star$ is indeed optimal. The
certification is in the form of a \emph{Hungarian gauge}, that is, a
choice of $\bm{\lambda}$ and $\bm{\mu}$ such that the new weights
$W'_{ij}=W_{ij}-\lam_i-\mu_j$ satisfy $W'_{i\,\pi(i)}=0$ for all
$1\leq i\leq n$ and $W'_{ij}\geq 0$ for all $1\leq i,j\leq n$.

For a general matrix of costs, the Hungarian Algorithm performs a set
of ``moves'' on a certain auxiliary graph. As a result, the set of
positions $(ij)$ such that $W'_{ij}=0$ at the end of the algorithm is
in general larger than the set of the $(i\,\pi(i))$'s, and thus
constitutes some subgraph 
$\calK_{n,n} \supseteq H_{\rm gauge} \supseteq M_{\star}$. If $W$ is
generic, $H_{\rm gauge}$ is a spanning forest. The properties of this
graph have been investigated in detail in~\cite[Chapter 7]{Fichera}.
In particular, there are several ways to perform some further gauge
transformations, staying within the class of gauges which are
Hungarian, and transforming $H_{\rm gauge}$ into a spanning tree of
$\calK_{n,n}$. However, none of these choices seem to be canonical in
a natural sense. Furthermore, even the forest obtained at the end of
the algorithm is not canonical, and depends on the details of the
implementation. For example, it is not even invariant under
permutation of rows and columns, or transposition, of the cost matrix
$W$. One natural way of choosing the gauge parameters is in terms of
the asymptotic values of the ``cavity fields'' for the parallel
implementation of the Cavity Equations \cite{bookMontMez} (see again
\cite{Fichera}).  This choice is invariant under permutation of rows
or columns, while still not being invariant under transposition. As a
further drawback, it is defined in terms of quantities whose
definition is rather subtle (one should first prove that, for generic
instances, the cavity fields in parallel implementation, which are a
function of discrete time values, become ultimately periodic up to a
linear shift, and then define the gauge parameters in terms of the
fields within this periodic regime, averaged over one period).

Thus, the following theorem, which is relatively compact and simple to
state, should come as a surprise:
\begin{thm}
In the first bipartite standard setting at size $n$, with matrix of
costs $W \in \mathbb{R}(n+1,n)$, 
there exists a choice of the $2n+1$ gauge
parameters $\lam_i$ and $\mu_j$ such that its restriction
to each of the $n+1$ instances 
$G\setminx \{a_i\}$
% (that is, for the minor $W^{\{k\}|\varnothing}$)
%  that is, the matrix $W'$ with $W'_{ij}=W_{ij}-\lam_i-\mu_j$ is such
%  that
is a Hungarian gauge.\footnote{That is, the gauge transformation
with parameters
$(\lam_1,\ldots,\lam_{k-1},\lam_{k+1},\ldots,\lam_{n+1})$ and
$(\mu_1,\ldots,\mu_n)$ is Hungarian for the minor
$W^{\{k\}|\varnothing}$ of the matrix of costs $W$.}

On top of this, if $W$ encodes generic weights, among the possible
choices of gauges with the property above, there is a canonical one
\begin{align}
\lam_i &= -\cW(M_\star(G\setminx \{a_i\}))
\ef;
&
\mu_j &= \cW(M_\star(G\cup \{b'_j\}))
\ef;
&
\end{align}
and in this case
$H_{\rm gauge}=H_\calJ$.
\end{thm}
\noindent
Above, by $b'_j$ we denote a second copy of the vertex $b_j$.
In other words, by a slight abuse of language, $\calM(G\cup \{b'_j\})$
is defined as the set of spanning subgraphs of $G$ such that all
vertices have degree 1, except for $b_j$, that has degree 2, and
$M_\star(G\cup \{b'_j\}) \in \calM(G\cup \{b'_j\})$ is defined as the
subgraph of minimal cost in this set.

\begin{figure}
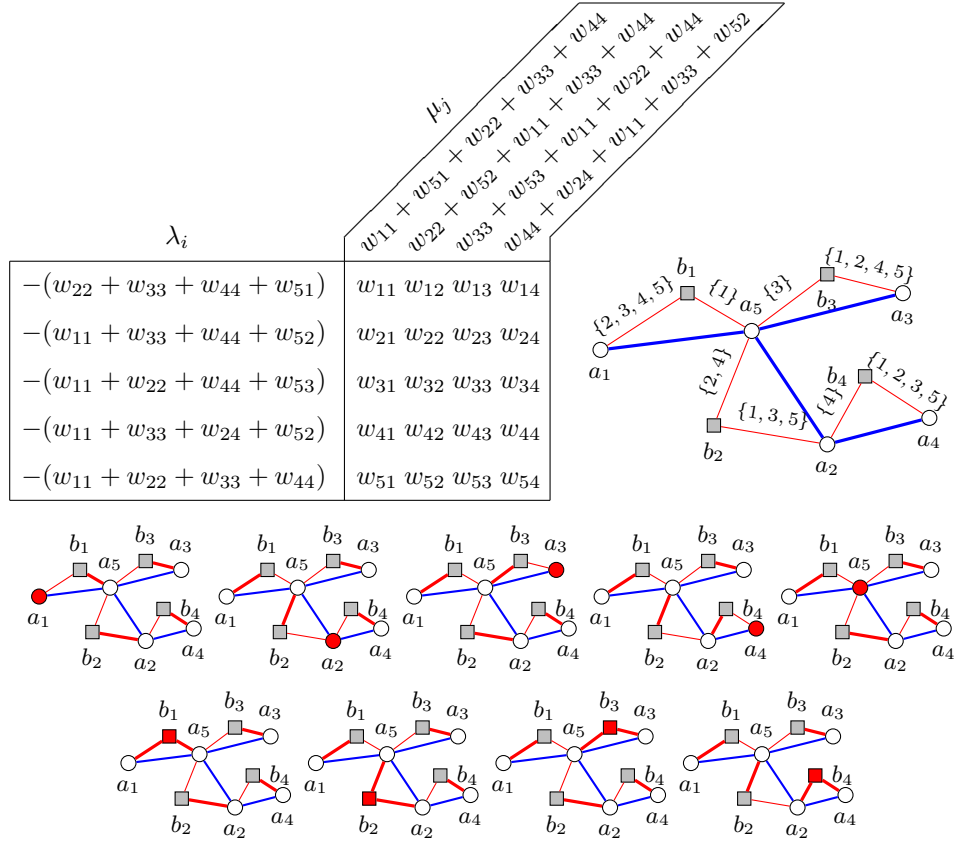

% \begin{gather}
\begin{center}
\setlength{\unitlength}{18pt}
%
% ------------------------------------
\if\faifig1  
\hspace*{-152pt}%
\setlength{\unitlength}{18pt}
%
\else [...TikZ\ code...] \fi
\end{center}
%\end{gather}
\caption{\label{fig.exGauge}%
Top left: cost matrix and gauge parameters for an example of size 4,
shown on top right. Bottom: optimal configurations $M_i$ and $M^{(j)}$.
}
\end{figure}

\begin{proof}
It is sufficient to establish the statements above concerning matrices
$W$ which are generic, as the statements for $W$ not necessarily
generic are deduced by continuity.

Let us start by proving that $W'_{ij}=0$ if 
$\edge{a_i}{b_j}\in H_\calJ$.  Indeed, the tree $H_\calJ$ has an
interesting property. First, for all vertices $a_i$, the forest
$H_\calJ \setminx a_i$ admits a dimer covering (this by the very
definition of $H_\calJ$, namely the unique dimer covering is the
subset of $e \in H_\calJ$ such that
$\{a_i\}\in \calC(e)$, or, in other words, the optimal matching 
$M_i:=M_\star(G\setminx \{a_i\})$).
% is the unique dimer covering of the forest $H_\calJ \setminx a_i$,
Then, for all vertices $b_j$, call $a'(j)$ and $a''(j)$ the two
neighbors of $b_j$ (in any order), and call $H'(j)$ the portion of
$H_\calJ$ containing the edge $\edge{b_j}{a'(j)}$ and all its
associated branch (as seen from $b_j$), and similarly for
$H''(j)$. Then both $H'(j)$ and $H''(j)$ admit a dimer covering. The
dimer covering $M'(j)$ of $H'(j)$ is the set of $e \in H'(j)$ such
that $\{a''(j)\}\in \calC(e)$, while the dimer covering $M''(j)$ of
$H''(j)$ is the set of $e \in H''(j)$ such that
$\{a'(j)\}\in \calC(e)$. In fact, all edges $e$ in $M'(j)\cup M''(j)$
have $\{a'(j)\},\{a''(j)\}\in \calC(e)$, except for the two edges
incident to $b_j$ and $b'_j$, which have only one of the two colors.

We claim that the optimal subgraph
$M^{(j)}:=M_\star(G\cup \{b'_j\})$
is the union of the unique dimer coverings $M'(j)$ and
$M''(j)$ of the trees $H'(j)$ and $H''(j)$.
This is seen, again, by excluding the existence of certain
semi-alternating cycles.

Given any two 
$M_1, M_2 \in \calM(G\cup \{b'_j\})$, their symmetric difference is a
collection of disjoint cycles, except possibly for a 8-shaped subgraph
(i.e., a graph with the topology of a ``bouquet graph'' $B_2$), in
which the vertex of degree 4 is $b_j$.  This is a simple variant of
Lemma \ref{lem:diff} adapted to this case, and its proof is a minor
variation of the one for the lemma.

For $M_1$, $M_2$ as above, say that 
$M_1 \symdif M_2=C_1 \cup C_2 \cup \ldots \cup C_k$, where a $C_j$ is
either a cycle disjoint from the rest of $M_1 \symdif M_2$, or one of
the two handles of the 8-shape subgraph (because these are cycles of
even length, thus the colors of the two edges of the handle incident
to the hub of the bouquet must be distinct).  Each $C_\alpha$ can be
swapped separately (say from $M_1$), to obtain another configuration
$M_1 \symdif C_\alpha \in \calM(G\cup \{b'_j\})$.  So, for each
$C_\alpha$ that is semi-alternating in some color
$U \in \calJ$, the sign of $\cW(M_1 \symdif C_\alpha)-\cW(M_1)$ is
determined by the parity of the semi-alternating sum (i.e., if the
edges $e$ with
$U\in \calC(e)$ are those in $M_1$ or those in $M_2$).

In particular, let us assume, for the sake of contradiction, that
$M'(j) \cup M''(j)$ and $M^{(j)}$ are distinct. Thus their symmetric
difference is a non-empty union of cycles in the sense above.
However, for any $M$, each cycle on $\big( M'(j) \cup M''(j)
\big)\symdif M$ must be semi-alternating in at least one of the two
colors $\{a'(j)\}$ and $\{a''(j)\}$, because the only two edges $e$ in
$M'(j) \cup M''(j)$ such that
$\big\{\{a'(j)\},\{a''(j)\} \big\} \not\subseteq \calC(e)$ are the two
edges incident to $b_j$, that cannot be simultaneously in the same
alternating cycle, and with the same parity.  In turn, this would
give a contradiction with the optimality of 
$M_\star(G\setminx \{a'(j)\})$ or of
$M_\star(G\setminx \{a''(j)\})$.

Once that this fact has been established, it is easy to see
that, for any $j$, if $i$ is the index such that
$a_i=a'(j)$, then
$M^{(j)}=M'(j) \cup M''(j)=M_i \cup \edge{a_i}{b_j}$. As a result,
\be
\begin{split}
W'_{ij}
&=w_{ij}-\lam_i-\mu_j
=w_{ij}+\cW(M_\star(G\setminx \{a_i\}))-\cW(M_\star(G\cup \{b'_j\}))
\\
&=w_{ij}+\cW(M_i)-\cW(M^{(j)})
%(M'(j) \cup M''(j))
=0
\ef.
\end{split}
\ee
Now we shall prove that, if $\edge{a_i}{b_j} \not\in H_\calJ$, then
$W'_{ij}$ is strictly positive.  Call $P$ the path on $H_\calJ$
connecting $a_i$ and $b_j$, and call $\tilde{a}$ the neighbour of
$b_j$ on $H_\calJ$ not in $P$.\footnote{Note that, at this point, in
  order to know that $\tilde{a}$ exists, it is crucial that we have
  established in the previous section that all $B$-vertices have
  degree 2, or at least the fact that the number of $A$-vertices is
  larger than the number of $B$-vertices, so that there must be an
  $A$-vertex not contained in $P$. This is one of the ingredients
  that, surprisingly, make the structure of simultaneous Hungarian
  gauges more rigid than the structure of Hungarian gauges for a
  single instance.}
It is easily seen that $P=(e_1,e_2,\ldots,e_{2\ell+1})$ has odd
length, and that $\{\tilde{a}\}\in \calC(e_{2j+1})$ for all
$j=0,1,\ldots,\ell$. As $w'_{e_k}=0$ for all the edges $e_k$ of $P$,
and, by our contradictory hypothesis, $w'_{\edge{a_i}{b_j}}\leq 0$, we
would have a semi-alternating cycle with non-positive alternating sum,
that is, in the matching
$M=M_\star(G \setminx \tilde{a})$, the matching 
$\tilde{M}=M \symdif (P \cup \edge{a_i}{b_j})$ obtained by swapping the
cycle $P \cup \edge{a_i}{b_j}$ would have $\cW(\tilde{M}) \leq \cW(M)$, thus
contradicting the fact that $M$ is the unique optimal configuration in
$\calM(G \setminx \tilde{a})$.
\end{proof}
\noindent
An example of simultaneous Hungarian gauges for one given instance is
provided in Figure~\ref{fig.exGauge}.

%%%%%%%%%%%%%%%%%%%%%%%%%%%%%%%%%%%%%%%%%%%%%%%%%%%%%%%
\section{A second setting where the matching subgraph is a spanning tree}
\label{sec.teoremaserio1b}
%%%%%%%%%%%%%%%%%%%%%%%%%%%%%%%%%%%%%%%%%%%%%%%%%%%%%%%

\noindent
Now, let us introduce a special family of graphs, that we
shall call \emph{symmetric graphs}:
\begin{defn}
\label{def.symgraph}
A graph $G=(V,E)$, equipped with a weight function
$w$, is \emph{symmetric} if there exist
a partition of the vertex sets
$V=V_{\rm l} \cup V_{\rm r}$ 
(`l' and `r' stand for `left' and `right')
and an involution 
$\iota: V_{\rm l} \leftrightarrow V_{\rm r}$ 
such that:
\begin{itemize}
\item for all $\edge{u}{v} \in E$, 
either both $u$ and $v$ are in $V_{\rm l}$ (or both in $V_{\rm r}$),
or they belong to distinct sets and are images of each other under the
involution $\iota$.
%% either both $u$ and $v$ are in 
%% %
%% $V_{\rm l}$ (or they are both in $V_{\rm r}$), or
%% $v=\iota u$.
\item
if an edge $e=\edge{u}{v}$ is in $E$, then also 
$\iota e =\edge{\iota v}{\iota u}$ is in $E$, and has the same weight,
$w_e=w_{\iota e}$.
\end{itemize}
Notions of graph theory, like edges, paths, cycles,\ldots, will have
the prefix `sym-' if they are stable under the involution. For
example, a \emph{sym-edge} is an edge of the form $\edge{v}{\iota v}$,
a \emph{sym-path} is a path $P=(e_1,\ldots,e_k)$ such that 
$\iota e_i=e_{k+1-i}$, and so~on.
% We call \emph{bridge edges} the edges of the form $\edge{v}{\iota v}$.
\end{defn}
\noindent
Note that, although in symmetric graphs there exist edges of the same
weight, 
this does not make non-generic weights any more likely.
% this does not make it easier to have non-generic weights. 
Indeed, if a cycle is fully contained (say) on the left part
of $G$, there are no repeated variables $w_e$, and if a cycle visits
both the left and the right part of $G$, then it also uses some
sym-edges,
% of the form $\{v,\iota v\}$, 
whose edge-weights appear only once as independent variables.

\begin{figure}
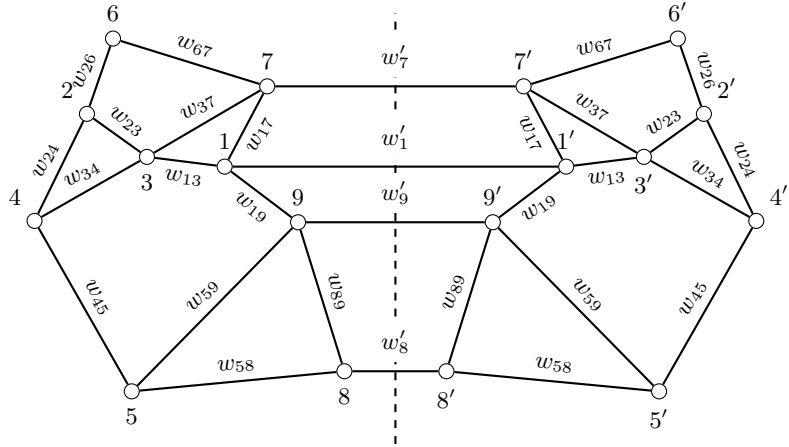

\[
\if\faifig1  

\else [...TikZ\ code...] \fi
\]
\caption{\label{fig.exsymgr}%
Example of symmetric weighted graph, in the sense of
  Definition~\ref{def.symgraph}. The dashed line denotes an axis of
  reflection implementing the involution~$\iota$.}
\end{figure}

Note also that the exact representation of reservoir settings,
introduced in Section~\ref{sec.reservintro}, is in this framework,
with $V_{\rm l}=V\cup S$ and $V_{\rm r}=V'\cup S'$, and with $\iota
v_j = v'_j$ and $\iota s_j = s'_j$.  In that case, the weights of
sym-edges are zero. Nonetheless, it remains true that the hypothesis
of generic weights is plausible. Indeed, 
%% in order to have a cycle with zero alternating sum identically in
%% terms of the $w_{v_i,v_j}$'s and $w_{v_i,s}$'s, seen as independent
%% variables, it should use all these variables either zero or two times
%% (and with opposite sign),
in order to have a cycle with a zero alternating sum identically in
terms of the independent variables $w_{v_i,v_j}$ and $w_{v_i,s}$, the
cycle must use each such variable either zero or two times (with
opposite signs)
so it must be a cycle that uses only sym-edges (which are of zero
weight) and pairs of edges $e$ and $\iota e$ (which use the same
variable). Then, it is a symmetric cycle using two sym-edges. However,
in this case the pairs $e$ and $\iota e$ appear in the same parity
class along the cycle, and thus with the same sign in the alternating
sum.

We shall produce variants of the statements of the previous sections,
adapted to the case of symmetric weighted bipartite graphs. We start
with a definition of symmetric admissible removal:
\begin{defn}[Symmetric admissible removal]
  Let $G=(V,E)$ be a symmetric (weighted) graph. A subset $U\subseteq
  V$ is a \emph{symmetric admissible removal} if the restriction $G_U$
  allows for a perfect matching, and is a symmetric weighted graph. In
  other words, $U$ is a sym-subset of $V$, that is $\iota U=U$. We
  denote by $\calV_{\rm sym}(G)$ the set of symmetric admissible
  removals for~$G$.
%% , that is
%% \[
%% \calV(G)\coloneqq\{U\subseteq V\ \vert\ 
%% \calM(G_U)\neq\emptyset\}.
%% \]
\end{defn}
\noindent
Remark that, in a symmetric graph $G$ with generic weights, the
(unique) perfect matching of minimal weight $M_{\star}$ is symmetric, i.e.,
if $e \in M_{\star}$, then also $\iota e \in M_{\star}$. 
%% if $\{u,v\} \in M_{\star}$, then also 
%% %
%% $\{\iota v,\iota u\} \in M_{\star}$. 
Similarly, in a setting with symmetric removals, also the quantities
$\calC(e)$ and $\calC_v$ are symmetric (that is, 
$\calC(e)=\calC(\iota e)$ and $\calC_v=\calC_{\iota v}$).  Note that
these facts imply as a corollary the statement on the exact
realization of systems with a reservoir, announced in equation
(\ref{eq.9837628x}). Indeed, in those notations we called $U$ and $U'$
the restriction of the symmetric admissible removal to the left and
right parts, respectively, and we have that 
$M_\star(G^{\rm res}_{U\cup U'})$ is a symmetric graph, so that, if
the edge $\edge{v_i}{v_j}$ contributes with its weight $w_{v_i\,v_j}$
to $w(M_\star(G_U))$, then the edges $\edge{v_i}{v_j}$ and
$\edge{v'_i}{v'_j}$ contribute with their weights
$w_{v_i,v_j}=w_{v'_i,v'_j}$ to
$w(M_\star(G^{\rm res}_{U\cup U'}))$, 
and similarly if
$\edge{v_i}{s}$
contributes with $w_{v_i,s}$ to
$w(M_\star(G_U))$, then 
$\edge{v_i}{s_i}$ and $\edge{v'_i}{s'_i}$ contribute with 
$w_{v_i,s_i}$ and $w_{v'_i,s'_i}$
to $w(M_\star(G^{\rm res}_{U\cup U'}))$, while if the sym-edge
$\edge{s_i}{s'_i}$ is in $M_\star(G^{\rm res}_{U\cup U'})$, then it does
not contribute to the weight, as $w_{s_i,s'_i}=0$.

Keeping in mind the statement of
Lemma~\ref{lem:diff} (for arbitrary matchings, thus not necessarily
symmetric), we shall comment on the pertinent variant of
Corollary~\ref{cor.coppia}:
\begin{corr}
\label{cor.coppiasym}
Let $G$ be a symmetric graph and 
$\calJ=\{U_1,U_2\}\in\calV_{\rm sym}(G)$.  The graph $H_{\calJ}^*$
contains no cycles and consists of alternating paths only. The paths
are either all contained in the left or right part of $G$ (and come in
symmetric pairs, i.e.\ there are paths 
$P_{\rm l}^{(i)}$ and $P_{\rm r}^{(i)}$, with 
$\iota P_{\rm l}^{(i)} = P_{\rm r}^{(i)}$), or are fixed points of the
involution (i.e., are paths $P_{\rm b}^{(j)}$, and 
$\iota P_{\rm b}^{(j)}$ is the path $P_{\rm b}^{(j)}$, traversed in
reverse order). The former contain no sym-edges, and the latter
(that thus are \emph{sym-paths})
are of odd length, and contain a unique sym-edge in the middle.

In particular, if $U_1=\emptyset$ and $U_2=\{u,\iota u\}$,
% with $u \in V_{\rm l}$, 
then $H_{\calJ}^*$ consists of a single alternating sym-path
connecting $u$ to $\iota u$, and if $U_1=\{v,\iota v\}$, and
$U_2=\{u,\iota u\}$, with $u,v \in V_{\rm l}$, then $H_{\calJ}^*$
consists either of two alternating sym-paths, connecting $u$ to 
$\iota u$ and $v$ to $\iota v$, or of a pair of (non-sym) paths,
connecting $u$ to $v$ and $\iota u$ to $\iota v$.
\end{corr}
\noindent
The statement of the corollary is rather straightforward, the only
point worth emphasizing is that sym-paths contain a \emph{unique}
sym-edge (located at the middle, since sym-edges must be fixed under
the reversal of the path).  We will exploit this in an important way
in what follows.

The point of studying symmetric graphs and removals is that it will
provide a second, different setting in which (in a suitable sense) the
matching subgraph $H_{\calJ}$ is a spanning tree. Again, we will
require that $G$ is bipartite (and connected, which implies the
existence of at least one sym-edge, which in turn implies that
$\iota A = B$). We will write $A_{\rm l}$ for $A \cap V_{\rm l}$, etc.
Also, in the rest of this section (and only here) we adopt the
notational convention that when naming the vertices $a_1$, $a_2$,\ldots
and $b_1$, $b_2$,\ldots, it is understood that $b_j = \iota a_j$.
Also, for a set $U=\{u_1,\ldots,u_k\}$, we write 
$U=\{u_1,\ldots,u_j;u_{j+1},\ldots,u_k\}$ to state concisely that 
$U\cap V_{\rm l}=\{u_1,\ldots,u_j\}$ and
$U\cap V_{\rm r}=\{u_{j+1},\ldots,u_k\}$.
Then, the analogue of the crucial Lemma~\ref{lem:tripla}
% (of which, for simplicity, we consider a smaller set of cases ???????????? VERO??)
% only the case in which all $U_j$'s are in $A$) 
is
\begin{lem}
\label{lem:triplasym}
Let $\bipg{G}{A}{B}{E}$ be a generic-weighted symmetric bipartite graph and let
$\calJ=\{U_1,U_2,U_3\}\subseteq\calV_{\rm sym}(G)$. Call $U\coloneqq U_1
\cap U_2 \cap U_3$, and $U'_i\coloneqq U_i\setminx U$. 
Suppose that we are in one of the seven cases below:
\begin{center}
\begin{tabular}{ccc}
\toprule
$U'_1$ & $U'_2$ & $U'_3$ \\
\midrule
$\{a_1;b_1\}$ & $\{a_2;b_2\}$ & $\varnothing$ \\
$\{a_1,b_2;b_1,a_2\}$ & $\varnothing$ & $\{b_2;a_2\}$ \\
$\{b_2;a_2\}$ & $\{b_1;a_1\}$ & $\{b_1,b_2;a_1,a_2\}$ \\
\midrule
$\{a_1;b_1\}$ & $\{a_2;b_2\}$ & $\{a_3;b_3\}$ \\
$\{a_1,b_3;b_1,a_3\}$ & $\{a_2,b_3;b_2,a_3\}$ & $\varnothing$ \\
$\{a_1,b_2,b_3;b_1,a_2,a_3\}$ & $\{b_3;a_3\}$ & $\{b_2;a_2\}$ \\
$\{b_2,b_3;a_2,a_3\}$ & $\{b_1,b_3;a_1,a_3\}$ & $\{b_1,b_2;a_1,a_2\}$ \\
\bottomrule
\end{tabular}\end{center}
Then the graph $H_{\calJ}^*$ is a forest
% a tree, or a forest with two components related by $\iota$, 
and all vertices in 
%
% $B_{\rm l}\cap V_{H_{\calJ}^*}$
$(B_{\rm l} \setminx(U_1 \cup U_2 \cup U_3))\cap V_{H_{\calJ}^*}$
have degree~$2$.
\end{lem}
\noindent
The following proof shares its general methodology with the proof of
Lemma~\ref{lem:tripla}, although it requires a more elaborate case
analysis due to the symmetries involved. However, as illustrated by
the upcoming Figures~\ref{fig.3Usym_2a} and \ref{fig.3Usym_2b} and
further discussed in Remark~\ref{rmk.symispivot}, the exhaustive list
of cases can be conceptually understood as the possible
``duplications'' or pivoting constructions originating from the base
case of Lemma~\ref{lem:tripla}. With this unifying structural picture
in mind, the technical verification of the cases should appear less
formidable and much more natural.

\begin{proof} 
Observe that, as was already the case for Lemma~\ref{lem:tripla}, and
in light of Proposition~\ref{prop.tradeaforb}, 
the first two rows of the table can be reduced to the third row, 
and the following three rows can be reduced to the last one, 
by adding extra vertices to $G$, with suitable weights, i.e., 
for a given $\calJ \subseteq\calV_{\rm sym}(G)$,
we can trade a pair of vertices $a$, $b$ in some $U$ (with
$a \in V_{\rm l}$ and $b \in V_{\rm r}$) by introducing two new
vertices $b' \in V_{\rm l}$ and $a' \in V_{\rm r}$, two new edges
$\edge{a}{b'}$ and $\edge{a'}{b}$ in $E$ with weight
$w_{a,b'}=w_{a',b}=0$, and letting $U'=U \setminx \{a,b\}$
(and $M_{U'} \symdif M_U=\{\edge{a}{b'},\edge{a'}{b}\}$) if 
$\{a,b\} \subseteq U$, and $U'=U \cup \{b',a'\}$ (and $M_{U'}=M_U$) otherwise.
This modification leaves $H_{\calJ}$ essentially unchanged, up to the
addition of edges $\edge{a}{b'}$, which are attached to $H_{\calJ}$
on $a$, and have $b'$ as a leaf, on the left part of the graph
(and their analogous transformation on the right part), so the
properties of $H_{\calJ}$ relevant to us (to be connected, acyclic,
and having $\deg(b)=2$ for all $b \in V_{\rm l}$ not in a $U_j$) are
left unchanged. So it is sufficient to consider the last case of each
of the two portions of the table above, which are the cases in which
the elements of $U'_j$ on the left part are all $B$-vertices. In these
cases, and with our labeling, $\calC_{b_1}=\{U_2,U_3\}$ (and so on),
thus $b_1$ must be a leaf of $H_{\calJ}^*$, and the edge 
$e=\{a,b_1\} \in H_{\calJ}^*$ must have $\calC(e)=\{U_1\}$, and so
on. Again, we will use blue, red and green for colors $U_1$, $U_2$ and
$U_3$, respectively.

We will repeat the manipulations of Lemma~\ref{lem:tripla}, in order
to simplify chains of length 3. For sym-edges, the removal is just as
before. For non-sym-edges, of course, the manipulations must be
implemented simultaneously on the left and right part of the graph, so
as to preserve the symmetry.

Performing these operations as long as possible, in any order, will
leave us again with a graph
% $\tilde{H}_\calJ^*$ 
in which no edge is bicoloured, that, again with abuse of language, we
will just call~$H_\calJ^*$ as before.

Now the graph $H_\calJ^*$ has two or three $B$-vertex leaves on the
left side, namely the vertices $\{b_1,b_2\}$, or $\{b_1,b_2,b_3\}$, in
the two cases of the table, the symmetric $A$-vertex leaves on the
right side, and all other vertices of degree 3. Furthermore,
$H_\calJ^*$ is bipartite, and the edges are properly
3-colored. However, unlike Lemma~\ref{lem:tripla}, now $H_\calJ^*$ is
not necessarily connected, as it may have up to three connected
components (all the components of $H_\calJ^*$ must contain at least
two leaves, as a component with no leaves would contain some
alternating cycle, which is not allowed, and this bounds the possible
number of components).

In particular, in the two cases of the table under analysis, the graph
obtained through our sequence of manipulations can be one of the
following (we use a dashed line to describe which portions of the
graph are on the left):
\if\faifig1  
\begin{align}
% H_\calJ^*
&
% \in
\left\{\rule{0pt}{15pt}\right.
% nessuno

\left.\rule{0pt}{20pt}\right\}
\label{gr:triplo1r4}
\end{align}
\else \be \label{gr:triplo1r4} [...TikZ\ code...] \ee \fi
Our goal is to prove that the cases above are all that can occur, and
in particular there are no $B$-vertices of degree $3$ on the left
side. In other words, up to permutations of colors, all components $H$
of $H_\calJ^*$ must occur either as a pair $(H,\iota H)$, with $H$ all
contained in the left part of $G$, and consisting of three edges only,
% ``shaped like a Y'' 
as in equation (\ref{gr:triplo}), or as a
symmetric component $H=\iota H$, consisting of a single sym-edge,
% (thus ``shaped like an I''), 
or of one sym-edge and its four
incident leaf edges.
%  (thus ``shaped like a H''), 
The following picture provides our claimed list of all the
possibilities:
\if\faifig1  
\begin{align}
\label{eq.claimlistH}
&

\end{align}
\else \be \label{eq.claimlistH} [...TikZ\ code...]\ee \fi
Let us start to prove this, and call $H$ a component of $H_\calJ^*$.
As $H_\calJ^*$ is symmetric, if $H$ has no sym-edges, then it must
be all contained in the left part of $G$, and have a symmetric
counterpart $\iota H$ all contained in the right part of $G$. We shall
prove that, in this case, the only possibility is the first element of
the list in (\ref{eq.claimlistH}).
But in this case the fact that $G$ is symmetric plays no role, as we
only use the vertices and edges on the left part, and we can conclude
in light of Lemma~\ref{lem:tripla}.  

So let us assume that $H$ is a connected symmetric graph with a
non-zero number of sym-edges.  Call $m_i^+$ the number of sym-edges of
color $U_i$ with the $B$-vertex endpoint on the left, and $m_i^-$
those with $B$-vertex endpoint on the right.  Call $\nu_i$ the
indicator function $\nu_i=1$ if $b_i \in H$ and $\nu_i=0$ if $b_i
\not\in H$.

The fact
that $H$ is properly 3-colored implies that there exists a constant
$c \in \mathbb{Z}$ such that $m_i^+-m_i^--\nu_i=c$ for all three~$i$'s.
However, it must also be the case that $m_i^+,m_i^- \leq 1$ for all
$i$, and, for all $i \neq j$, it cannot be that both $m_i^+$ and
$m_j^-$ are non-zero. Indeed, let $e$ be a sym-edge of color $U_i$,
and with $B$-vertex endpoint on the left. Consider the alternating
component $P_{e,i,j}$ in colors $U_i$ and $U_j$ going through $e$. As
there are no alternating cycles in $H$, the graph $P_{e,i,j}$ must be
a path connecting $b_i$ to $a_i$. Similarly, if $e$ is a sym-edge
of color $U_i$, and with $B$-vertex endpoint on the right, the
alternating component $P_{e,i,j}$ in colors $U_i$ and $U_j$ going
through $e$ must be a path connecting $b_j$ to $a_j$. As these paths,
when they exist, are unique and have a unique sym-edge, in their
middle step, the claims above follow.  More precisely, if 
$m_i^+ \geq 2$, let $e'$ and $e''$ be two distinct such sym-edges, and
$j \neq i$. Then a contradiction comes from the fact that $P_{e',i,j}$
and $P_{e'',i,j}$ must both exist, be disjoint, and both have $b_i$ as
endpoint.  If $m_i^- \geq 2$, let $e'$ and $e''$ be two distinct such
sym-edges, and
$j \neq i$. Then a contradiction comes from the fact that $P_{e',i,j}$
and $P_{e'',i,j}$ must both exist, be disjoint, and both have $b_j$ as
endpoint.  Finally, for $i \neq j$, if $m_i^+>0$ and $m_j^->0$, let
$e'$ and $e''$ be one such sym-edge of each of these types,
respectively. Then a contradiction comes from the fact that
$P_{e',i,j}$ and $P_{e'',j,i}$ must both exist, be disjoint, and both
have $b_i$ as endpoint.

So, from the analysis of the consequences of the symmetry, and of not
having alternating cycles, there are relatively few possibilities left
over for the parameters $m_i^{\pm}$ of a connected component $H$ (up to
permutation of the indices of the $b_j$'s, and of the colors). Namely,
there are three possibilities for the list $(\nu_1,\nu_2,\nu_3)$, each
declined in multiple variants. When the number of $b_j$'s in the
component is $1$ or $2$, there are two cases:
$(m_i^+=\nu_i\,;\,m_i^-=0)$ and
$(m_i^+=0\,;\,m_i^-=1-\nu_i)$, while when all $\nu_i$'s are 1, then either
$(m_i^+=1\,;\,m_i^-=0)$ for all $i$, or $(m_i^+=0\,;\,m_i^-=1)$ for all $i$,
or $m_i^+=m_i^-=1$ for a single color $i$, and $0$ elsewhere.  Thus,
the lists of integers
\[
\big(
(\nu_1,\nu_2,\nu_3)\,;\,(m_1^+,m_2^+,m_3^+)\,;\,(m_1^-,m_2^-,m_3^-)
\big)
\]
that can
occur are summarised in the following table, where the right-most
column states what we shall prove in the following:
\be
\label{eq.tab7cases}
\begin{array}{rcccc}
\toprule
\textrm{name} & 
(\nu_1,\nu_2,\nu_3) & (m_1^+,m_2^+,m_3^+) & (m_1^-,m_2^-,m_3^-)
&
\textrm{possible graphs $H$}
\\
\midrule
\textrm{(1a)} & (1,0,0) & (1,0,0) & (0,0,0) & 
\if\faifig1  
\begin{tikzpicture}[baseline={([yshift=-.5ex]current bounding
      box.center)}]
\draw[dashed] (0,-.15) -- (0,.15); 
    \node[inner sep=2.5pt,rectangle,fill=gray!50,draw] (4) at (-.75,0)   {};
    \node[circle,inner sep=2pt,fill=white,draw] (8) at (.75,0)   {};
    \draw[very thick,blue] (4) to (8);
  \end{tikzpicture}
\else [...TikZ\ code...] \fi
 \\
\textrm{(1b)} & (1,0,0) & (0,0,0) & (0,1,1) & \textrm{none} \\
\textrm{(2a)} & (1,1,0) & (1,1,0) & (0,0,0) & \textrm{none} \\
\textrm{(2b)} & (1,1,0) & (0,0,0) & (0,0,1) & 
\if\faifig1  
\begin{tikzpicture}[baseline={([yshift=-.5ex]current bounding
      box.center)}]
\draw[dashed] (0,-.58) -- (0,.15); 
    \node[circle,inner sep=2pt,fill=white,draw] (1) at (-.25,-.2165)   {};
    \node[inner sep=2.5pt,rectangle,fill=gray!50,draw] (3) at (-.75, -0.433013)   {};
    \node[inner sep=2.5pt,rectangle,fill=gray!50,draw] (4) at (-.75,0)   {};
    \draw[very thick,red] (1) to (3);
    \draw[very thick,blue] (1) to (4);
    \node[inner sep=2.5pt,rectangle,fill=gray!50,draw] (5) at (.25,-.2165)   {};
    \node[circle,inner sep=2pt,fill=white,draw] (7) at (.75, -0.433013)   {};
    \node[circle,inner sep=2pt,fill=white,draw] (8) at (.75,0)   {};
    \draw[very thick,red] (5) to (7);
    \draw[very thick,blue] (5) to (8);
\draw[very thick,green!75!blue] (1) to (5);
  \end{tikzpicture}
\else [...TikZ\ code...] \fi
 \\
\textrm{(3a)} & (1,1,1) & (1,1,1) & (0,0,0) & \textrm{none} \\
\textrm{(3b)} & (1,1,1) & (0,0,0) & (1,1,1) & \textrm{none} \\
\textrm{(3c)} & (1,1,1) & (0,0,1) & (0,0,1) & \textrm{none} \\
\bottomrule
\end{array}
\ee
The first situation, in (1a), is easy. Suppose that $H$ does not
consist of a single edge. Thus the vertex adjacent to $b_1$ is some $a
\neq a_1$, of degree 3. The alternating component in color $U_2$ and
$U_3$ (i.e., red and green) passing through $a$ cannot be a path, as
there are no leaves whose incident edge has these colors, thus it must
be a cycle, which is not allowed.

Similarly, for (1b), the alternating component in color $U_2$ and
$U_3$ passing through (for example) the sym-edge of color $U_2$
cannot be a path, as there are no leaves of corresponding color, thus
it must be a cycle, which is not allowed.

In the other five cases, we cannot conclude solely through the
analysis of alternating cycles, and we must once again rely on
semi-alternating cycles.

The cases (2a), (3a) and (3b) are treated in a similar way. In all
these cases, there must be a symmetric path $P_1$, alternating in
colors $U_1$ and $U_2$, connecting $b_1$ to $a_1$, and a similar path
$P_2$, disjoint from $P_1$, connecting $b_2$ to $a_2$.  There cannot
be any other vertex $v$, or any other edge $e$ of color $U_1$ or
$U_2$, besides those contained in $P_1$ or $P_2$, as the component
alternating in colors $U_1$ and $U_2$ passing through $v$ or $e$ would
be a cycle. Thus, all other edges besides the paths $P_1$ and $P_2$
are of color $U_3$.  Note that, as $H$ is connected, in particular
$P_1$ is connected to $P_2$, so there must exist at least one edge of
color $U_3$, contained in the left part, with one endpoint in $P_1$
and one in $P_2$.

As illustrated in Figure~\ref{fig.3Usym_2a}, for the case (2a), the
graph $H'$ consisting of the left part of $H$, the two sym-edges
connected (on the right) to a new $A$-vertex $a$ of degree 3, and a
new leaf $B$-vertex $b_3$ connected to $a$ (with $\edge{a}{b_3}$ of
color $U_3$) is a 3-Hamiltonian graph (with no prescribed symmetries),
and it has at least two edges of color $U_3$ (the one incident to
$b_3$, and one of those connecting the two paths, and contained on the
left part).  Thus, it has the same form as the graphs already analysed
in Lemma~\ref{lem:tripla}, in the subtle analysis of semi-alternating
cycles.  Similarly, for the case (3a), the graph $H'$ consisting of
the left part of $H$, and the three sym-edges connected (on the right)
to a new $A$-vertex $a$ of degree 3, is a 3-Hamiltonian graph
% of the form already analysed in the lemma, 
and has at least two edges of color $U_3$ (again, the one incident to
$b_3$, and one of those connecting the two paths, and contained in the
left part).  Finally, for the case (3b), the graph $H'$ consisting of
the left part of $H$, and the three sym-edges connected (on the right)
to a new $B$-vertex $b$ of degree 3, is a 3-Hamiltonian graph
% is of the form already analysed in the lemma, 
and has at least two edges of color $U_3$.  So, in all of these three
cases, we can conclude in light of the analysis in
Lemma~\ref{lem:tripla}.
% These situations are illustrated in Figure~\ref{fig.3Usym_2a}.

\begin{figure}[t]
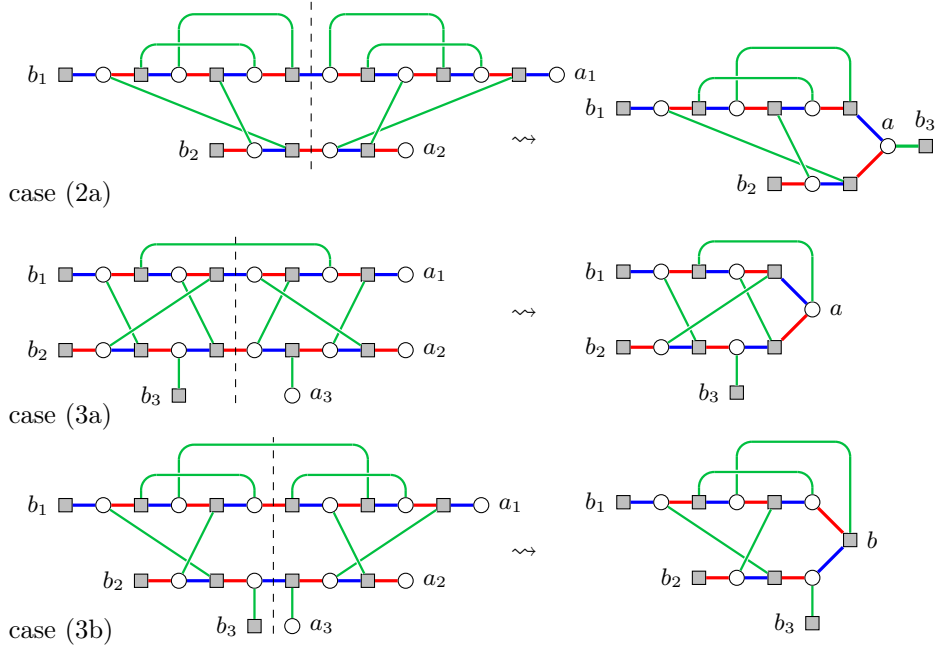

\[
% caso 2a
\setlength{\unitlength}{30pt}

%%%%%
\]
\caption{\label{fig.3Usym_2a}Graphs $H$ and $H'$ in the three cases of
  the table in which $b_1 \leftrightarrow a_1$ and $b_2
  \leftrightarrow a_2$, that is (2a), on top, (3a), in the middle, and
  (3b), at the bottom. We have deliberately chosen examples of graphs
  $H$ such that the corresponding graph $H'$ is the same 3-Hamiltonian
  graph, which, in all three cases, has been ``doubled'' at a pivot
  point of the following types: the $A$-vertex adjacent to $b_3$, for
  case (2a), an $A$-vertex not adjacent to $b_3$, for case (3a), and a
  non-leaf $B$-vertex, for case (3b).}
\end{figure}

Also the cases (2b) and (3c) are treated in a similar way.  Now, the
component alternating in colors $U_1$ and $U_2$ starting from $b_1$
must be a path $P$ ending in $b_2$ and all contained in the left part,
while $\iota P$ connects $a_1$ and $a_2$ and is contained in the right
part. There cannot be any other vertex $v$, or any other edge $e$ of
color $U_1$ or $U_2$, besides those contained in $P$ and in $\iota P$,
as the component alternating in colors $U_1$ and $U_2$ passing through
$v$ or $e$ would be a cycle. Thus, all other edges besides the paths
$P$ and $\iota P$ are of color~$U_3$.

As illustrated in Figure~\ref{fig.3Usym_2b}, for the case (2b), the
graph $H'$ obtained by taking the left part of $H$ plus the sym-edge
is a 3-Hamiltonian graph.  Note that, in this case, we only know that
it has at least one green edge, not two as in the cases analysed
above, and indeed the corresponding row in the table of
(\ref{eq.tab7cases}) allows for the small $H$-shaped diagram.
%% is of the same form of the graphs
%% already analysed in Lemma~\ref{lem:tripla}.
Finally, for the case (3c), the graph $H'$ obtained by taking the left
part of $H$ plus a new edge of color $U_3$, formed by connecting the
two left endpoints of the sym-edges, is a 3-Hamiltonian graph,
and has at least two edges of color $U_3$ (the newly introduced one,
and the one incident to $b_3$).  So also in these cases we can
conclude in light of the results of Lemma~\ref{lem:tripla}.
%  These situations are illustrated in Figure~\ref{fig.3Usym_2b}.
This completes our proof.
\end{proof}

\begin{figure}[t]
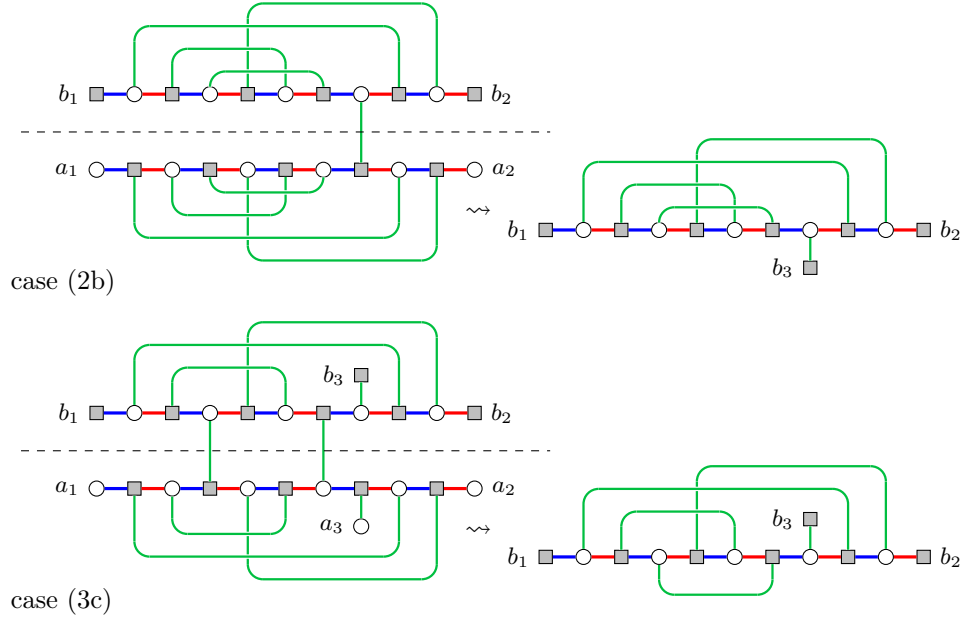

\[
% caso 3b
\setlength{\unitlength}{30pt}

%%%%%
\]
\caption{\label{fig.3Usym_2b}Graphs $H$ and $H'$ in the two cases of
  the table in which $b_1 \leftrightarrow b_2$ and 
  $a_1 \leftrightarrow a_2$, that is (2b), on top, and (3c), at the
  bottom. Now, for reasons of compactness, the `left' part of $H$ is
  above the dashed line (not on the left). We have deliberately chosen
  examples of graphs $H$ such that the corresponding graph $H'$ is the
  same 3-Hamiltonian graph (coinciding with the one in
  Figure~\ref{fig.3Usym_2a}) which, in both cases, has been
  ``doubled'' at a pivot point of the type: the vertex $b_3$, for case
  (2b), and an edge of color $U_3$ not incident to $b_3$, for case
  (3c).}
\end{figure}

\begin{rmk}
\label{rmk.symispivot}
As we have seen, the ``symmetric 3-Hamiltonian graphs'' $H$, that is,
the symmetric graphs $H$ that cannot be excluded solely on the basis
of alternating cycles, and require an analysis of semi-alternating
cycles, are close relatives of the (generic) 3-Hamiltonian graphs. The
connection is more clear in the opposite direction: for any
3-Hamiltonian graph $H'$, with $n$ $A$-vertices and $n+2$
$B$-vertices, we can obtain several symmetric 3-Hamiltonian graphs
$H$, namely, one graph of type (2a), with $2n$ vertices of each class,
one of type (2b), with $2n+1$ vertices of each class, (these two
graphs, by `pivoting' $H'$ at the leaf $b_3$, in two different ways),
$n$ graphs of type (3b) and $n$ graphs of type (3a), all with $2n$
vertices of each class, one for each non-leaf green edge of $H'$ (by
pivoting at the $B$-vertex or at the $A$-vertex endpoint of the green
edge), and $n$ graphs of type (3c), with $2n+2$ vertices of each class
(by pivoting at the green edge in a symmetric way).  These
constructions are illustrated in the Figures~\ref{fig.3Usym_2a}
and~\ref{fig.3Usym_2b}, by following the arrows in the inverse
direction.

These facts have immediate consequences on the structural
understanding of symmetric 3-Hamiltonian graphs, and on their
enumeration (in light of the enumeration of generic 3-Hamiltonian
graphs obtained above).
\end{rmk}

\noindent
Again, once the above fact for sets $\calJ$ of cardinality 3 is established,
we can deduce similar consequences on settings of more practical
interest, where $|\calJ|$ is arbitrarily large. If we follow the same
scheme as in the section for generic graphs, we should now prove a
structure theorem on sets~$\calJ$.

Before doing this, as our main interest is in the restriction of
$H_\calJ$ to the left part of the graph, let us define this notion:
\begin{defn}[Left-restriction]
\label{def.leftrestr}
Let $G=(V_{\rm l} \cup V_{\rm r},E)$ be a symmetric graph. We define
its \emph{left-restriction} $G^{\rm (l)}$ as the graph with vertex-set
$V_{\rm l} \cup \{\sigma \}$, and edges $e=\edge{u}{v}$, if
$\edge{u}{v}$ is in $E(G)$ and both $u$ and $v$ are in $V_{\rm l}$,
and $e=\edge{v}{\sigma}$, if
$\edge{v}{\iota v}$ is in~$E(G)$.
\end{defn}
\noindent
Note that, deliberately, we use the symbol $\sigma$ for the extra
vertex emerging from this procedure. Indeed, this vertex has a role
analogous, but slightly different from the extra reservoir vertex $s$
in the setting of Section~\ref{sec.reservintro}.

We also say that $\calJ$ is \emph{left-single-path} if, for all
$U_1,U_2 \in \calJ$, the left-restriction of 
$M_{U_1} \symdif M_{U_2}$ consists of a single path. That is, if
$M_{U_1} \symdif M_{U_2}$ consists either of a single sym-path 
$P\equiv\iota P$, or of two sym-paths, $P_1\equiv\iota P_1$ and
$P_2\equiv\iota P_2$ (that are connected through $\sigma$ after the
restriction), or of two symmetric paths, $P$ and $\iota P$, contained
in the left and in the right parts, respectively, and containing no
sym-edges. This occurs when
$|(U_i \symdif U_j) \cap V_{\rm l}|\in \{1,2\}$ for all $U_i,U_j \in \calJ$.

Similarly, as the sets $U \in \calV_{\rm sym}(G)$
are determined by their left part, $U \cap V_{\rm l}$, we shall
more concisely describe these subsets.
\begin{lem}
\label{lem.formaJsym}
Let $\bipg{G}{A}{B}{E}$ be a symmetric bipartite graph, and 
$\calJ \subseteq \calV_{\rm sym}(G)$ of cardinality at least 2 such
that $\bigcap_{U \in J} U=\varnothing$, for all $U_i$, $U_j \in \calJ$
we have $d_{ij} :=|(U_i \symdif U_j) \cap V_{\rm l}|=2$, and for all $U_i$,
$U_j$, $U_k \in \calJ$
% and, for all subsets $\calJ' \subseteq \calJ$ of cardinality 3, 
the hypotheses of Lemma \ref{lem:triplasym} hold.  Then there exist a
set $A_0 \subseteq A_{\rm l}$ and a set $B_0 \subseteq B_{\rm l}$,
such that $|\calJ|=|A_0|+|B_0|$, and the collection $\calJ$ is of the
form
\begin{equation}
\label{eq.3478687645sym}
\calJ \cap V_{\rm l} = \{ \{a\} \cup B_0 \}_{a \in A_0} \cup
\{ B_0 \setminx b \}_{b \in B_0}
%% \calJ = \{ \{a\} \cup B_0; \{\iota a\} \cup \iota B_0 \}_{a \in A_0} \cup
%% \{ B_0 \setminx b; \iota B_0 \setminx \iota b \}_{b \in B_0}
\ef.
\end{equation}
If we also allow
$d_{ij}=|(U_i \symdif U_j) \cap V_{\rm l}|=1$
% , and $|\calJ|\geq 4$,
then there is a second possibility,
\begin{equation}
\label{eq.3478687645sym2}
\calJ \cap V_{\rm l} = \{ \{a\} \cup B_0 \}_{a \in A_0} \cup
\{ B_0 \setminx b \}_{b \in B_0} \cup \{ B_0 \}
%% \calJ = \{ \{a\} \cup B_0; \{\iota a\} \cup \iota B_0 \}_{a \in A_0} \cup
%% \{ B_0 \setminx b; \iota B_0 \setminx \iota b \}_{b \in B_0}
\ef.
\end{equation}
\end{lem}
\begin{proof}
First of all, observe that the construction given above in
(\ref{eq.3478687645sym}) and (\ref{eq.3478687645sym2}) is easily
verified to be compatible with the corresponding requirements, namely,
(\ref{eq.3478687645sym2}) is compatible with the whole table in the
lemma, while (\ref{eq.3478687645sym}) is compatible with the bottom
part of the table (i.e., the last four rows of seven).
% table. 
We shall prove that no other structures are possible.

The case of equation (\ref{eq.3478687645sym}), under the condition
$|(U_i \symdif U_j) \cap V_{\rm l}|=2$, gives a set of equations on
the left part of the $U$'s which is identical to the equations for the
$U$'s in Lemma~\ref{lem.formaJ}, so this case is proven in light of
the previous results. We have to deal with the case of equation
(\ref{eq.3478687645sym2}), under the condition 
$d_{ij} \in \{1,2\}$.
% $|(U_i \symdif U_j) \cap V_{\rm l}| \in \{1,2\}$.

Clearly, sets $U_i$ and $U_j$ with $d_{ij}=2$ must
have cardinality of the same parity, while sets with $d_{ij}=1$ 
must have opposite parity. Thus, all collections $\calJ$ with 
$d_{ij} \in \{1,2\}$ must be of the form $\calJ=\calJ_+ \cup \calJ_-$,
with elements $U$ in the two sets having opposite parity of
cardinality, and $\calJ_+$, $\calJ_-$ being as in
Lemma~\ref{lem.formaJ}, for some sets $A_0^{\pm}$ and $B_0^{\pm}$,
whenever of cardinality at least 2.

If both $\calJ_\pm$ have cardinality 1, then their single sets $U_+$ and
$U_-$ differ by a single vertex (of type $A$ or $B$), so we are in the
conditions of equation (\ref{eq.3478687645sym2}) with
$(A_0,B_0)=(\{a\},\varnothing)$ or
$(A_0,B_0)=(\varnothing,\{b\})$. Otherwise, say that $\calJ_+$ has
cardinality at least 2.
% We shall now determine the conditions for $(A_0^+,B_0^+)$ to be
% compatible with $(A_0^-,B_0^-)$.
If $A_0^+$ and $B_0^+$ are both non-empty, there exists a single
set $U$ which has parameter $d=1$ for all sets $U' \in \calJ_+$,
namely $U=B_0^+$, and again we are in the conditions of equation
(\ref{eq.3478687645sym2}). If one of the two sets is empty, then the
only possibility is that one of the two following facts hold:
\begin{align}
\label{eq.3478687645sym2spora}
\calJ \cap V_{\rm l} &= \{ \{a_1\}, \{a_2\}, \{a_1,a_2\} \}
\ef;
&
\calJ \cap V_{\rm l} &= \{ \{b_1\}, \{b_2\}, \varnothing \}
\ef.
\end{align}
However, none of the two possibilities above is compatible with the
table in the lemma. As the case analysis is complete, we can conclude.
\end{proof}
\noindent
\begin{rmk}
\label{rmk.sigmapassacol32}
Note that, while in the setting of equation (\ref{eq.3478687645sym})
it is possible that $H_\calJ$ does not contain any sym-edge (and thus,
that after the left-restriction $\sigma$ constitutes a singleton), in
the setting of equation (\ref{eq.3478687645sym2}) every connected
component of the graph $H_\calJ^*$ must contain at least one sym-edge,
as $M_{U_1}\symdif M_{U_2}$ contains one sym-edge if
$|U_1 \symdif U_2|=1$.
\end{rmk}

\noindent
We can then establish
\begin{lem}
\label{lem:co2sym}
Let $\bipg{G}{A}{B}{E}$ be a generic-weighted symmetric bipartite graph,
and $\calJ$
as in equation
(\ref{eq.3478687645sym2}) of
Lemma \ref{lem.formaJsym}.
Then the graph $H_{\calJ}^*$ consists only of symmetric components
with a unique sym-edge,
and all vertices in $(B \setminx B_0) \cap V_{\rm l}$
have degree $2$, while all the vertices in $B_0$ are leaves.
\end{lem}
\begin{proof} 
The case $|\calJ|=2$ is obvious, and the case $|\calJ|=3$ is the
statement of Lemma \ref{lem:triplasym}. Also, the fact that each
component contains at most one sym-edge is already established, and
the fact that it contains at least one sym-edge is implied by the
remark above.  Now, if $|\calJ|\geq 4$, we shall prove that all left
$B$-vertices in $H_\calJ^*$ not in a $U_j$ are of degree $2$.
This is proven essentially by the same argument as in the proof of
Lemma~\ref{lem:co2}. By contradiction, if $H_\calJ^*$ has a left $B$-vertex
$b$ of degree 3 or larger, let $e_1$, $e_2$ and $e_3$ be distinct
edges incident to $b$ (note that at most one of them can be a
sym-edge), and (up to renaming the $U_j$'s) $U_i \in \calC(e_i)$ for
$i=1,2,3$. Then, calling $\calJ'=\{U_1,U_2,U_3\}$, we have that
$H_{\calJ'}^* \subseteq H_\calJ^*$, and that $b$ has degree $3$
already in $H_{\calJ'}^*$, which is not possible in light of
Lemma~\ref{lem:triplasym}.

Let us determine which left $B$-vertices have degree 1, and which have
degree 2, in light of the structure theorem,
Lemma~\ref{lem.formaJsym}.  A vertex $b \in B_0$ is present in $G_U$
only for the corresponding set $U=B_0\setminx b$, so
$\calC_b=\calJ\setminx U$ and there exists a single edge $e$, incident
to $b$, with $\calC(e)=\{U\}$, so $b$ is a leaf. A vertex $b \not\in
B_0$ is present in all $G_U$'s, so $\calC_b=\varnothing$. If it has
degree 1, then there is an incident edge $e$ such that
$\calC(e)=\calJ$, so $e$ is an isolated edge of $H_\calJ$, and thus it
is not in $H_\calJ^*$. So it must have degree 2.
\end{proof}
\noindent
Now we are ready to establish the second theorem of this paper:
\begin{thm} 
\label{thm:Jtreesym}
\label{th:spantreesym}
In the conditions of Lemma \ref{lem:co2sym}, namely, for a triple
$(G,w,\calJ)$, with $\bipg{G}{A}{B}{E}$ bipartite, generic weights, and
$\calJ$ of the form 
\begin{align*}
\calJ_{\rm l} 
&= \{ \{a\} \cup B_0 \}_{a \in A_0} \cup \{
B_0 \setminx b \}_{b \in B_0}
\cup \{B_0\}
\\
\calJ
&=\{U \cup \iota U\}_{U \in \calJ_{\rm l}}
\end{align*}
for some $A_0 \subseteq A_{\rm l}$ and $B_0\subseteq B_{\rm l}$,
% given by equation (\ref{eq.3478687645}))
% Assume $ J\subseteq C$ with $| J|\geq 2$. 
the subgraph $H_{\calJ}^*$ is a forest where each component is
symmetric and contains a unique sym-edge. If $A_0=A_{\rm l}$, then
$H_{\calJ}=H_{\calJ}^*$ is a spanning forest of $G$, while its
left-restriction $H_{\calJ}^{\rm (l)}=(H_{\calJ}^*)^{\rm (l)}$ is a
spanning tree of $G^{\rm (l)}$, the left-restriction of $G$.
\end{thm}
\begin{proof}
Curiously, in this case we could not devise a proof that is just a
minor variant of the first proof presented for the general case
Theorem \ref{thm:Jtree}. We will instead produce a proof that is
similar in spirit to the second proof of that theorem.

We know already most of the claimed properties of our graph
$H_\calJ^*$, and what we need to establish is that each component $H$
is acyclic. We already know from Theorem \ref{thm:Jtree} that there
are no cycles completely contained in the left part. So $H$ is acyclic
iff it has exactly one sym-edge.  If there is a cycle, it contains at
least two distinct sym-edges.  If there is such a cycle
$C=(e_0,e_1,e_2,\ldots) \subseteq H$, let $e_0$ be a sym-edge, and let
$e_\ell$ be the first subsequent sym-edge along the cycle.  Then there
exists also a sym-cycle, that uses exactly the two sym-edges $e_0$ and
$e_{\ell}$, namely
$C'=(e_0,e_1,e_2,\ldots,e_{\ell-1},e_\ell, \iota
e_{\ell-1},\ldots,\iota e_2,\iota e_1) \subseteq H$. This cycle has
length at least 4, and, if the left endpoints of the sym-edges are in
the same vertex class, it has length at least 6.
% where the left part of the cycle is the path $P=(e_0,e_1,\ldots,e_{\ell})$.  
The left endpoint of each of the sym-edges may be an $A$-vertex or a
$B$-vertex. In the latter case, we would have a symmetric chain of
length 3, that we can contract into a single sym-edge with $A$-vertex
left endpoint:
\begin{equation*}
\if\faifig1  
\begin{tikzpicture}%[baseline={([yshift=-.5ex]current bounding box.center)}]
    \node[circle,inner sep=2pt,fill=white,draw=black!20] (1) at (0,0)   {};
    \node[inner sep=2.5pt,rectangle,fill=gray!50,draw] (2) at (1,0)   {};
    \node[circle,inner sep=2pt,fill=white,draw] (3) at (2,0)   {};
    \node[inner sep=2.5pt,rectangle,fill=gray!30,draw=none] (4) at (3,0)   {};
\draw[thick,dashed] (1.5,-.9) -- (1.5,.9); 
    \draw[very thick] (2) to[edge node={node [above] {\colorbox{white}{\small\color{black} $w'$}}}] (3);
    \draw[very thick] (1) to[edge node={node [above] {\small\color{black} $w$}}] (2);
    \draw[very thick] (3) to[edge node={node [above] {\small\color{black} $w$}}] (4);
    \draw[very thick] (2) to[edge node={node [below] {\colorbox{white}{\small\color{black} $\calJ \setminx \calC$}}}] (3);        
    \draw[very thick] (1) to[edge node={node [below] {\small\color{black} $\calC$}}] (2);
    \draw[very thick] (3) to[edge node={node [below] {\small\color{black} $\calC$}}] (4);
  \end{tikzpicture}
\else [...TikZ\ code...] \fi
\qquad
  \raisebox{26pt}{$\to$}
  \quad
\if\faifig1  
\begin{tikzpicture}%[baseline={([yshift=-.5ex]current bounding box.center)}]
    \node[circle,inner sep=2pt,,draw=black!20] (1) at (0,0)   {};
    \node[inner sep=2.5pt,rectangle,fill=gray!30,draw=none] (4) at (1,0)   {};
\draw[thick,dashed] (.5,-.9) -- (.5,.9); 
    \draw[very thick] (1) to[edge node={node [above] {\colorbox{white}{\small\color{black} $2w-w'$}}}] (4);
    \draw[very thick] (1) to[edge node={node [below] {\colorbox{white}{\small\color{black} $\calC$}}}] (4);
  \end{tikzpicture}
\else [...TikZ\ code...] \fi
\end{equation*}
After this operation, the cycle has length at least 2.  So, w.l.o.g.,
we can think that both our sym-edges have an $A$-vertex left endpoint,
and thus that $\ell\geq 1$ is an odd integer. As the sets $\calC(e)$
are non-empty, we can call $U$ one element of $\calC(e_0)$, and $U'$
one element of $\calC(e_\ell)$. If $\calC(e_0) \cap \calC(e_\ell) \neq
\varnothing$, then we can take $U'=U$. However, $\calC(e_1) \subseteq
\calJ \setminx \calC(e_0)$, and $\calC(e_2) = \calJ \setminx
\calC(e_1)$ (because $B$-vertices have degree 2), thus
%\be
%
$\calC(e_2) = \calJ \setminx \calC(e_1) \supseteq \calC(e_0)$
%
%\ee
and $U \in \calC(e_2)$, and similarly $U \in \calC(e_4)$,
and so on, which is in contradiction with the fact that $\ell$ is odd.
If instead $\calC(e_0) \cap \calC(e_\ell) = \varnothing$, then 
$U$ and $U'$ are distinct. However, as we have seen, 
$U \in \calC(e_{j})$ for all $0 \leq j < \ell$ even, and, by a
symmetric argument,
$U' \in \calC(e_{j})$ for all $0 < j \leq \ell$ odd, thus the cycle is
alternating in colors $U$ and $U'$, which is not allowed.  This
completes our proof.
\end{proof}

\begin{rmk}
In the conditions of Theorem \ref{thm:Jtreesym}, with $A_0=A_{\rm (l)}$, the
projection of the left part 
$\bar{H}^{\rm (l)}_\calJ=\overline{(H_\calJ)^{\rm (l)}}$ is spanning
on $\overline{G^{\rm (l)}}$.
\end{rmk}

\begin{rmk}
\label{rmk.bridgerootC}
A careful analysis of the proof shows a further interesting property
of $H_\calJ$ in the setting of Theorem \ref{thm:Jtreesym}.  Calling
$C_{\alpha}$ the set of colors $U\in \calJ$ such that 
$a \in A_0 \cap V(H_{\alpha})$ or $b \in B_0 \cap V(H_{\alpha})$, and
calling $e_\alpha$ the only sym-edge of the component $H_\alpha$,
then $\calC(e_\alpha)=C_{\alpha}$ if the endpoint of $e_\alpha$ on the
left is a $B$-vertex, and $\calC(e_\alpha)=\calJ \setminx C_{\alpha}$
if the endpoint of $e_\alpha$ on the left is an $A$-vertex.
\end{rmk}

\noindent
Up to completing the weight function with sufficiently high values, we
can always imagine to ``complete'' a symmetric bipartite graph 
to a graph, that we shall call $\calK_{n,m}^{\rm (sym)}$, with
$n=|A_{\rm l}|$ and $m=|B_{\rm l}|$, composed of the left part
$G_{\rm l}=(A_{\rm l}\cup B_{\rm l},A_{\rm l}\times B_{\rm l})=\calK_{n,m}$, 
the right part 
$G_{\rm r}=(A_{\rm r}\cup B_{\rm r},A_{\rm r}\times B_{\rm r})=\calK_{m,n}$, 
and all the $m+n$ edges between $a_i$ and $b_i = \iota a_i$.

For example, the graph $\calK_{3,5}^{\rm (sym)}$ is the following:
\[
\if\faifig1  

\else [...TikZ\ code...] \fi
\]
In this case, we do not need to check for the fact that the
$\calM(G_U)$'s are non-empty, and we only need to ensure the
genericity of the weights (which holds almost-surely in various
randomized settings). Thus, in summary, we have determined a second
interesting bipartite setting:
\begin{corr}
\label{cor.setting1sym}
In a triple $(G,w,\calJ)$, where 
$G=\calK_{n,m}^{\rm (sym)}$ is the bipartite graph
$(A\cup B, A_{\rm l} \times B_{\rm l} \cup 
A_{\rm r} \times B_{\rm r} \cup \big\{ \edge{a_i}{\iota a_i}
\big\})$,
% \cup \big\{ \{b_j,\iota b_j\} \big\} 
$w$ is a generic symmetric weight function, and
$\calJ=\{ \{a,\iota a\} \}_{a \in A_{\rm l}} \cup \{ \varnothing\}$,
we have that $H_\calJ=H_\calJ^*$ is a spanning forest on
$\calK_{n,m}^{\rm (sym)}$, with all left $B$-vertices and right
$A$-vertices of degree 2, and the projection of the left-restriction,
$\bar{H}^{\rm (l)}_\calJ$, is a spanning tree on $\calK_{n+1}$.
\end{corr}
\noindent
We shall call the \emph{second bipartite standard setting} the type of
triple in the corollary above.

The case of symmetric graphs obtained when implementing the setting
with a reservoir differs only slightly. Now, in the monopartite
setting, we ``complete'' the symmetric graph, with vertices $v_i$,
$v'_i$, $s_i$ and $s'_i$, to a graph consisting of two copies of
$\calK_{n}$, connected by $n$ chains of length 3 between homologous
vertices, $(v_i,s_i,s'_i,v'_i)$, that we may call
$\calK_{n}^{\rm (res)}$.  For example, the graph
$\calK_{5}^{\rm (res)}$, shown alongside its original reservoir formulation, is the following:
\[
\if\faifig1
\begin{array}{l}
\begin{tikzpicture}[baseline={([yshift=-.5ex]current bounding box.center)}]
\draw[double=black,white,double distance=0.3mm] (0,1) arc (180:360:.5);
\draw[double=black,white,double distance=0.3mm] (1,1) arc (180:360:.5);
\draw[double=black,white,double distance=0.3mm] (2,1) arc (180:360:.5);
\draw[double=black,white,double distance=0.3mm] (3,1) arc (180:360:.5);
\draw[double=black,white,double distance=0.3mm] (0,1) arc (180:360:1);
\draw[double=black,white,double distance=0.3mm] (1,1) arc (180:360:1);
\draw[double=black,white,double distance=0.3mm] (2,1) arc (180:360:1);
\draw[double=black,white,double distance=0.3mm] (0,1) arc (180:360:1.5);
\draw[double=black,white,double distance=0.3mm] (1,1) arc (180:360:1.5);
\draw[double=black,white,double distance=0.3mm] (0,1) arc (180:360:2);
    \node[circle,inner sep=2pt,fill=white,draw] (1r) at (0,1)   {};
    \node[circle,inner sep=2pt,fill=white,draw] (2r) at (1,1)   {};
    \node[circle,inner sep=2pt,fill=white,draw] (3r) at (2,1)   {};
    \node[circle,inner sep=2pt,fill=white,draw] (4r) at (3,1)   {};
    \node[circle,inner sep=2pt,fill=white,draw] (5r) at (4,1)   {};
    \draw[double=black,white,double distance=0.3mm,thick] (1r) -- (0,2.1);
    \draw[double=black,white,double distance=0.3mm,thick] (2r) -- (1,1.9);
    \draw[double=black,white,double distance=0.3mm,thick] (3r) -- (2,1.7);
    \draw[double=black,white,double distance=0.3mm,thick] (4r) -- (3,1.5);
    \draw[double=black,white,double distance=0.3mm,thick] (5r) -- (4,1.3);
\draw[double=black,white,double distance=0.3mm,thick] (0,2.1) arc (180:90:.2);
\draw[double=black,white,double distance=0.3mm,thick] (1,1.9) arc (180:90:.2);
\draw[double=black,white,double distance=0.3mm,thick] (2,1.7) arc (180:90:.2);
\draw[double=black,white,double distance=0.3mm,thick] (3,1.5) arc (180:90:.2);
\draw[double=black,white,double distance=0.3mm,thick] (4,1.3) arc (180:90:.2);
    \draw[double=black,white,double distance=0.3mm,thick] (0.2,2.3) -- (5.0,2.3);
    \draw[double=black,white,double distance=0.3mm,thick] (1.2,2.1) -- (5.0,2.1);
    \draw[double=black,white,double distance=0.3mm,thick] (2.2,1.9) -- (5.0,1.9);
    \draw[double=black,white,double distance=0.3mm,thick] (3.2,1.7) -- (5.0,1.7);
    \draw[double=black,white,double distance=0.3mm,thick] (4.2,1.5) -- (5.0,1.5);
%
%% \draw[dashed] (5,-1.4) -- (5,3); 
\node[inner sep=1pt,rectangle,draw,fill=white] (ang1) at (5,1.9) {\rule{0pt}{15pt}$s$\raisebox{-12pt}{\rule{0pt}{20pt}}}; 
  \end{tikzpicture}
\\
\end{array}
\else [...TikZ\ code...] \fi
\]
The case of the reservoir bipartite setting allows for a similar construction.
Now the graph $G$ consists of a
$\calK_{n,m}=(A \cup B, A \times B)$, a
$\calK_{m,n}=(B' \cup A', B' \times A')$, and chains of length 3
between homologous vertices, $(a_i,s_i,s'_i,b'_i)$ and
$(b_j,s_j,s'_j,a'_j)$. 
We can call $\calK_{n,m}^{\rm (res)}$ this graph.
Thus $\calK_{n,m}^{\rm (res)}$ has overall $4(n+m)$
vertices.
For example, the graph $\calK_{3,5}^{\rm (res)}$, shown alongside its
original reservoir formulation, is the following:
\[
\if\faifig1
\begin{array}{l}
\begin{tikzpicture}[baseline={([yshift=-.5ex]current bounding box.center)}]
    \node[inner sep=2.5pt,rectangle,fill=gray!50,draw] (1r) at (0,1)   {};
    \node[inner sep=2.5pt,rectangle,fill=gray!50,draw] (2r) at (1,1)   {};
    \node[inner sep=2.5pt,rectangle,fill=gray!50,draw] (3r) at (2,1)   {};
    \node[inner sep=2.5pt,rectangle,fill=gray!50,draw] (4r) at (3,1)   {};
    \node[inner sep=2.5pt,rectangle,fill=gray!50,draw] (5r) at (4,1)   {};
    \node[circle,inner sep=2pt,fill=white,draw] (1b) at (0.5,0)   {};
    \node[circle,inner sep=2pt,fill=white,draw] (2b) at (2,0)   {};
    \node[circle,inner sep=2pt,fill=white,draw] (3b) at (3.5,0)   {};
    \draw[thick] (1r) -- (1b);
    \draw[thick] (2r) -- (1b);
    \draw[thick] (3r) -- (1b);
    \draw[thick] (4r) -- (1b);
    \draw[thick] (5r) -- (1b);
    \draw[thick] (1r) -- (2b);
    \draw[thick] (2r) -- (2b);
    \draw[thick] (3r) -- (2b);
    \draw[thick] (4r) -- (2b);
    \draw[thick] (5r) -- (2b);
    \draw[thick] (1r) -- (3b);
    \draw[thick] (2r) -- (3b);
    \draw[thick] (3r) -- (3b);
    \draw[thick] (4r) -- (3b);
    \draw[thick] (5r) -- (3b);
    \draw[double=black,white,double distance=0.3mm,thick] (1b) -- (0.5,-.7);
    \draw[double=black,white,double distance=0.3mm,thick] (2b) -- (2,-.5);
    \draw[double=black,white,double distance=0.3mm,thick] (3b) -- (3.5,-.3);
\draw[double=black,white,double distance=0.3mm,thick] (0.5,-.7) arc (180:270:.2);
\draw[double=black,white,double distance=0.3mm,thick] (2,-.5) arc (180:270:.2);
\draw[double=black,white,double distance=0.3mm,thick] (3.5,-.3) arc (180:270:.2);
    \draw[double=black,white,double distance=0.3mm,thick] (5.0,-.9)  -- (0.7,-.9);
    \draw[double=black,white,double distance=0.3mm,thick] (5.0,-.7)  -- (2.2,-.7);
    \draw[double=black,white,double distance=0.3mm,thick] (5.0,-.5)  -- (3.7,-.5);
    \draw[double=black,white,double distance=0.3mm,thick] (1r) -- (0,2.1);
    \draw[double=black,white,double distance=0.3mm,thick] (2r) -- (1,1.9);
    \draw[double=black,white,double distance=0.3mm,thick] (3r) -- (2,1.7);
    \draw[double=black,white,double distance=0.3mm,thick] (4r) -- (3,1.5);
    \draw[double=black,white,double distance=0.3mm,thick] (5r) -- (4,1.3);
\draw[double=black,white,double distance=0.3mm,thick] (0,2.1) arc (180:90:.2);
\draw[double=black,white,double distance=0.3mm,thick] (1,1.9) arc (180:90:.2);
\draw[double=black,white,double distance=0.3mm,thick] (2,1.7) arc (180:90:.2);
\draw[double=black,white,double distance=0.3mm,thick] (3,1.5) arc (180:90:.2);
\draw[double=black,white,double distance=0.3mm,thick] (4,1.3) arc (180:90:.2);
    \draw[double=black,white,double distance=0.3mm,thick] (0.2,2.3) -- (5.0,2.3);
    \draw[double=black,white,double distance=0.3mm,thick] (1.2,2.1) -- (5.0,2.1);
    \draw[double=black,white,double distance=0.3mm,thick] (2.2,1.9) -- (5.0,1.9);
    \draw[double=black,white,double distance=0.3mm,thick] (3.2,1.7) -- (5.0,1.7);
    \draw[double=black,white,double distance=0.3mm,thick] (4.2,1.5) -- (5.0,1.5);
%
% \draw[dashed] (5,-1.4) -- (5,3); 
\node[inner sep=1pt,rectangle,draw,fill=white] (ang1) at (5,.7) {\rule{0pt}{50pt}$s$\raisebox{-47pt}{\rule{0pt}{20pt}}}; 
  \end{tikzpicture}
\\
\end{array}
\else [...TikZ\ code...] \fi
\]
For future convenience, it is useful to formulate a specialization of 
Corollary~\ref{cor.setting1sym}:
\begin{corr}
\label{cor.setting1symres}
In a triple $(G,w,\calJ)$, where 
$G=\calK_{n,m}^{\rm (res)}=((A\cup C) \cup (B \cup D),E)$, with 
$E$ containing the edges $\edge{a_i}{b_j}$,
$\edge{a_i}{d_i}$
and 
$\edge{b_i}{c_i}$ with endpoints on the same side, and the sym-edges
$\edge{c_i}{d_i}$ (with $d_i = \iota c_i$),
%% =\{(a_i\,b_j)\}_{\substack{a_i \in A,\\ b_j\in B,\\ (a_i\,b_j)\in E}}
%% \cup \{(a_i\,d_i)\}_{\substack{a_i\in A,\\ (a_i\,s)\in E}}
%% \cup \{(b_i\,c_i)\}_{\substack{b_i\in B,\\ (b_i\,s)\in E}}
%% \cup \{
%% \}_{\substack{1 \leq i \leq n_2,\\ 1\leq j \leq n_1}}$,
%
$w$ is a generic symmetric weight function with 
$w_{c_i,d_i}=0$,
%$w_{c_i,\iota c_i}=w_{d_i,\iota d_i}=0$,
and
$\calJ=\{ \{a,\iota a\} \}_{a \in A_{\rm l}} \cup \{ \varnothing\}$,
we have that $H_\calJ=H_\calJ^*$ is a spanning forest on
$\calK_{n,m}^{\rm (res)}$, with all left $B$-vertices and $D$-vertices
and right $A$-vertices and $C$-vertices of degree 2, and the
projection of the left-restriction, $\bar{H}^{\rm (l)}_\calJ$, is a
spanning tree on $\calK_{n+1}$.
\end{corr}
\noindent
%% We shall call the \emph{reservoir monopartite standard setting} the type
%% of triple above.
We shall call the \emph{reservoir bipartite standard setting} the type
of triple above.

%%%%%%%%%%%%%%%%%%%%%%%%%%%%%%%%%%%%%%%%%%%%%%%%%%%%%%%
\section{The hyper-Assignment Problem}
\label{sec.ktuples}
%%%%%%%%%%%%%%%%%%%%%%%%%%%%%%%%%%%%%%%%%%%%%%%%%%%%%%%

\noindent
In this section, we discuss how the properties established above
extend to a variant of the problem that we refer to as the
\emph{hyper-Assignment Problem}.  This provides a natural stepping
stone between the ordinary Assignment Problem and Hitchcock's full
transportation model (which we discuss at length in 
Section~\ref{sec.AssCapa}): 
% Appendix~\ref{app.AssCapa}): 
%
while being a particular intermediate restriction, it turns out to be
precisely the most general setting in which all the structural
theorems established above continue to hold.
%% This is an intermediate case between
%% ordinary Assignment and Hitchcock's generalization of the
%% Assignment Problem (that we discuss at length in Appendix
%% \ref{app.AssCapa}).

In this model, for integers $n \geq m$,
we have a matrix of costs $W$, of dimension $n \times m$, describing
the weights of the edges of the complete bipartite graph $G=\cK_{n,m}$
(say, with $n$ vertices in set $A$ and $m$ vertices in set $B$), and a
set of parameters $\{\gb_j\}_{1\leq j \leq m}$, valued in $\bN^+$ and
summing up to $n$.

A hyper-assignment $F$ is a forest on $G$ in which each component is a
star graph with a $B$-vertex $b_j$ as a hub, incident to exactly $\gb_j$
$A$-vertices. Then, the optimal hyperassigment $F_{\star}(W)$ is the
one minimizing the cost function
\be
\cW(F)=\sum_{\edge{a}{b}\in F}W_{ab}
\ef.
\ee
%
%% In other words, introducing the \emph{transportation matrix} $T$ such
%% that $T_{ij}=1$ if $\edge{a_i}{b_j}\in E(F)$ and $T_{ij}=0$ otherwise,
%% the optimal hyper-assignment 
%% corresponds to the matrix $T$
%% % , valued in $\{0,1\}$, 
%% minimizing the cost function
%% $\sum_{i,j}T_{ij}W_{ij}$ subject to the marginal constraints $\sum_j
%% T_{ij}=1$ for all $i$ and $\sum_i T_{ij}=\gb_j$ for all $j$
In other words, introducing the \emph{transportation matrices} $T$ as
matrices with non-negative entries subject to the marginal constraints
$\sum_j T_{ij}=1$ for all $i$ and $\sum_i T_{ij}=\gb_j$ for all $j$,
and defining the optimal transportation matrix as the one minimizing
the cost function $\sum_{i,j}T_{ij}W_{ij}$, it turns out that the
optimal matrix $T$ is valued in $\{0,1\}$, yielding $T_{ij}=1$ if
$\edge{a_i}{b_j}\in E(F)$ and $T_{ij}=0$ otherwise.

Following the same paradigm as in the previous sections, in particular
adapting the notion of first bipartite standard setting introduced in
Corollary~\ref{cor.setting1}, we establish the following
\begin{thm}
\label{thm.setting1hyper}
Given a triple $(G,w,\calJ)$, where 
$\bipg{G}{A}{B}{E}=\calK_{n+1,m}$, $w$ is a generic weight function,
and $\calJ=\{ \{a\} \}_{a \in A}$, 
and a vector $\vb \in (\bN^+)^m$ with
$\sum_j \gb_j=n$,
we have that $H_\calJ=H_\calJ^*$ is
a spanning tree on $\calK_{n+1,m}$, with $B$-vertices $b_j$ of degree
$\gb_j+1$, and $\bar{H}_\calJ$ is a 
% $(k+1)$-uniform 
spanning hypertree on
$\calK_{n+1}$.
\end{thm}
\noindent
We shall call the \emph{hyper-assignment standard setting} the type of
triple in the theorem above.

At first glance, one might be tempted to think that a generic instance
of the hyper-Assignment Problem of size $n \times m$ is simply related
to a (non-generic) instance of the ordinary bipartite Matching Problem
(i.e., the Assignment Problem) of size $n$, obtained by replacing each
$B$-vertex $b_j$ with $\gb_j$ identical copies sharing the exact same
costs (meaning each column of the cost matrix is replicated $\gb_j$
times). The ground states
$\{ \pi_{\star} \}$ of this instance, seen as permutations from the
row- to the column-indices, are the image of a single ground state
under the action of the subgroup $\otimes_j \mathfrak{S}_{\gb_j}$ (on the
column index). 

Furthermore, one could attempt to handle this non-generic instance
through a standard perturbation and limiting procedure, for instance,
by introducing a new cost matrix $P$ with i.i.d.\ entries in $[0,1]$
and considering the perturbed instances with costs $W+\eps P$ for
$\eps \in \bR^+$. For some small enough open interval 
$\eps \in \;]0,\eps_{\rm min}[$, 
the ground state is unique and independent of $\eps$.

%% And, yet again, we can treat this non-generic instance
%% by performing a perturbation and a limit procedure (for example, we
%% can take a new matrix of costs $P$ with entries i.i.d.\ in $[0,1]$,
%% and consider the instances with costs $W+\eps P$, for 
%% %
%% $\eps \in \bR^+$. For some interval $\eps \in \;]0,\eps_{\rm min}[$ 
%% %
%% the ground state is unique and independent of $\eps$).

We will perform these operations back and forth in the forthcoming
analysis (calling \emph{splitting} the procedure of introducing
multiple copies and perturbing the weights, and \emph{merging} the
inverse process).

So, one could argue that the theorem above is a direct consequence of
all that we said in the previous sections on the ordinary bipartite
matching.  This is not the case, as in fact a crucial ingredient is
still missing, and is provided in the following lemma. Now an
admissible removal $U$ is a subset of $A$, and a vector 
$\vb' \preceq \vb$, such that $\sum_j (\gb_j-\gb'_j) = n-|A|$, and the
associated instance admits at least one hyper-assignment. Then we have
\begin{lem}
\label{lem.hypermatchNOiji}
Let $\bipg{G}{A}{B}{E} \subseteq \cK_{n,m}$ be a weighted bipartite
graph, with generic weight function $w$, and let $U_1$, $U_2$ be two
admissible removals.
% (that is, the sets of hyper-assignments $\cM(G_{U_1})$ and
% $\cM(G_{U_2})$ are non-empty).
Let $M_1$ and $M_2$ be the two optimal hyper-assignments for $G_{U_1}$
and $G_{U_2}$, respectively. Then $H=M_1 \symdif M_2$ does not contain
any cycle that visits two or more distinct $B$-vertices.
\end{lem}
\begin{proof}
Let us first translate the statement above in the setting where we go
back to the ordinary assignment, and we break the degeneracy of the
ground state through an infinitesimal perturbation. Let us call
$b_i^{\alpha}$, for $\alpha=1,2,\ldots,\gb_i$, the $\gb_i$ copies of the
$B$-vertex $b_i$. If some $\gb'_i$ is non-zero in $U_1$ or $U_2$,
choose a whatever definite choice on which vertices are kept (say, the
first $\gb_i-\gb'_i$).

For each value of $\eps$ except for a finite subset of $\bR^+$, the
optimal matchings $M_1(\eps)$ and $M_2(\eps)$ are unique, and the
lemmas of the previous sections apply directly. Any component of
$H(\eps)=M_1(\eps) \symdif M_2(\eps)$ is an open path with only its
endpoints in $U_1 \symdif U_2$.
%either an open path or a cycle.
This is all we know for arbitrary instances of the assignment
problem. The claim of the lemma is that, if 
$\eps \in [0,\eps_{\rm min}[$,
none of these paths visits a $B$-vertex $b_i^{\alpha}$, then a
$b_j^{\beta}$, and then a $b_i^{\gamma}$, in this order along the
path.

Let us prove this by contradiction. Any total ordering of $U_1 \symdif
U_2$ induces an ordering and orientation of the paths composing
$H(\eps)$, so that, under our contradictory hypothesis, there exists a first
counterexample, where we choose to order the counterexamples by the
position of the $b_i^{\gamma}$ vertex. In this counterexample, there
are no vertices $b_i^{\delta}$ within the path from $b_i^{\alpha}$ to
$b_i^{\gamma}$, and no other $B$-vertices within this path have the
same column index (otherwise it would not be the first one in the
ordering). Thus, a generic such path has the form
$P=(b_{j_0}^{\alpha},a_{i_0},b_{j_1}^{\beta_1},a_{i_1},\ldots,a_{i_{\ell-1}},b_{j_\ell}^{\beta_\ell},a_{i_\ell},b_{j_0}^{\gamma})$
with $\ell \geq 1$ and all $i_h$'s and $j_h$'s distinct. The
alternating cost of this path, $W(P,\eps)$, has the obvious
Lipshitzianity property
\be
|W(P,\eps)-W(P,\eps')| \leq 2 (\ell+1) |\eps-\eps'|
< 2n |\eps-\eps'|
\ef.
\ee
% is given by the alternating cost of the path at $\eps
From the generality hypothesis on the hyper-assignment instance, we
know that $W(P,0)\neq 0$, because it corresponds to the alternating
cost of the non-trivial cycle
$\tilde{P}=(b_{j_0},a_{i_0},b_{j_1},a_{i_1},\ldots,a_{i_{\ell-1}},b_{j_\ell},a_{i_\ell})$
of the original instance. So, for $\eps$ small enough, $W(P,\eps)$ and
$W(P,0)$ have the same (non-zero) sign. Whichever this sign, it would
lead to a contradiction either on the optimality of $M_1=M_1(0)$, or on the
optimality of~$M_2=M_2(0)$.
\end{proof}
\noindent
With this lemma at hand, we are now ready to prove our theorem above.
\begin{proof}[Proof of Theorem \ref{thm.setting1hyper}]
Let us work again in the ``splitting'' setting, in which we go back to
the ordinary assignment by taking $\gb_j$ copies of the $B$-vertex
$b_j$, and we break the degeneracy of the groud state through an
infinitesimal perturbation.

For each value of $\eps$ except for a finite subset of $\bR^+$, the
optimal matchings $\{M_U(\eps)\}_{U\in \calJ}$ are unique, and the
matching subgraph $H_\calJ(\eps)$ is defined.  What we know to start
with is that this graph is a spanning tree, and each $B$-vertex
$b_j^{\alpha}$ has degree 2. Thus, if we perform the merging operation
towards the original graph, each $B$-vertex $b_j$ has degree in the
range $\{2,\ldots,2\gb_j\}$. Part of our claim is that, for $\eps$
small enough, this degree is always $\gb_j+1$.

First of all, it is easily seen that the degree is at least $\gb_j+1$.
Indeed, call $A_\alpha$ the set of vertices in $A$ which are
neighbours of $b_j^{\alpha}$ (in particular $|A_\alpha|=2$). If we
have $|A_\alpha \cap A_\beta|=2$, for some $\alpha \neq \beta$, then
we have a cycle of length 2, which is not allowed. Similarly, if we
had an ordered sublist $(\alpha_1,\alpha_2,\ldots,\alpha_h)$ in $[k]$
% $\{1,\ldots,k\}$ 
such that 
$a_{i_1}\in A_{\alpha_1}\cap A_{\alpha_2}$,
$a_{i_2}\in A_{\alpha_2}\cap A_{\alpha_3}$, \ldots,
$a_{i_h}\in A_{\alpha_h}\cap A_{\alpha_1}$, 
then again we would have a cycle in our tree.

Let us call $A_{(j)}$
% $\partial b_j$ 
the union of the $A_\alpha$'s, and $B_{(j)}$ the set
$\{b_j^{\alpha}\}$. Then we can construct an auxiliary bipartite graph
$F=(A_{(j)}\cup B_{(j)},E)$ by stating that an edge $\edge{a}{\ga}$ is present if the
corresponding $A$-vertex $a$ is in the set $A_{\alpha}$ of the neighbours
of $b_j^{\alpha}$. What we determined above is that
this graph must be a forest, as otherwise we would have a cycle of the
form above. Calling $c(F)$ the number of connected components in $F$,
we have that $\deg_{H_\calJ}(b_j)=\gb_j+c(F)$ and our claim on the degree
is equivalent to the fact that $F$ is indeed a spanning tree.

Let us prove this by contradiction. Say that $b_j^\alpha$ and
$b_j^\beta$ are vertices of $B_{(j)}$ in distinct components of $F$,
and call $a_1^{\pm}$ and $a_2^{\pm}$ the vertices such that
$A_\alpha=\{a_1^+,a_1^-\}$ and $A_\beta=\{a_2^+,a_2^-\}$. Now consider
the four paths $P_{\nu_1,\nu_2}=M_{a_1^{\nu_1}} \symdif
M_{a_2^{\nu_2}}$, for $\nu_1,\nu_2\in\{+1,-1\}$.  One of these paths
must be of the form
$P_{\nu_1,\nu_2}=(a_1^{\nu_1},b_j^\alpha,a_{1}^{-\nu_1},\ldots,a_{2}^{-\nu_2},b_j^\beta,a_2^{\nu_2})$
where in the intermediate dots there is at least one vertex
$b_s^\gamma$ with $s \neq j$ (as otherwise $b_j^\alpha$ and
$b_j^\beta$ would be in the same component of $F$). But this is
exactly the circumstance that is excluded by
Lemma~\ref{lem.hypermatchNOiji}.

So we conclude that $|A_{(j)}|=\gb_j+1$, and that, when performing the
projection $\bar{H}_\calJ$, the resulting hyperedges, which are given
by the partitions of $A_{(j)}$ according to the connected component of
$F$, are in fact composed of the sets $A_{(j)}$ \emph{tout court}
(because $c(F)=1$). 
As the properties of being a forest and of being connected are
preserved by the operation $\bar{\cdot}$ (when the starting graph has
no isolated vertices), we deduce that $\bar{H}_\calJ$ is a
% $(k+1)$-uniform 
spanning hypertree on $A$, as claimed.
\end{proof}

\noindent
In this framework, the sets $\calC(e)$ satisfy a property analogous to
that of the ordinary Assignment case, which now requires shifting from
standard sets to \emph{multisets}, and from ordinary unions to
\emph{multiset unions} (denoted by $\uplus$), which add element
multiplicities (e.g., $\{1,2,2,4\}\uplus\{2,3\}=\{1,2,2,2,3,4\}$). We
also adopt the notation $\times$ for scalar multiset replication, as
in $2 \times \{1,3,4\}=\{1,1,3,3,4,4\}$.

%% In this framework, the sets $\calC(e)$ satisfy properties analogous to
%% the one of the ordinary Assignment case, which however requires to
%% pass from the use of sets to the use of \emph{multisets}, and from the
%% notion of union to the one of \emph{multiset union} (that we denote by
%% $\uplus$), which sums element multiplicities, as in
%% $\{1,2,2,4\}\uplus\{2,3\}=\{1,2,2,2,3,4\}$. We also use the symbol
%% $\times$ for a form of exterior product, as in 
%% %
%% $2 \times \{1,3,4\}=\{1,1,3,3,4,4\}$.

Specifically, the following hold: (1)~A set $\calC(e)$ is non-empty if
and only if $e \in H_\calJ$; (2)~all sets $\calC(e)$ are standard sets
without repetitions; (3)~for every $1 \leq i \leq n$, the sets
$\calC(\edge{a_i}{b_j})$ form an ordinary partition of
$[n]\setminus\{i\}$; (4)~for every $1 \leq j \leq m$, the sets
$\calC(\edge{a_i}{b_j})$ form a partition with respect to multiset
union of the set $\gb_j \times [n]$, namely
$\biguplus_i \calC(\edge{a_i}{b_j})=\gb_j \times [n]$.
%% We have that: (1) A set $\calC(e)$ is non-empty if and only if $e \in
%% H_\calJ$; (2) all sets $\calC(e)$ are ordinary sets, with no
%% repetitions; (3) for all $1 \leq i \leq n$, the sets
%% $\calC(\edge{a_i}{b_j})$ form an (ordinary) partition of the set
%% $[n]\setminx i$; (4) for all $1 \leq j \leq m$, the sets
%% $\calC(\edge{a_i}{b_j})$ form a partition w.r.t.\ multiset union of
%% the set $\gb_j \times [n]$, that is 
%% %
%% $\biguplus_i \calC(\edge{a_i}{b_j})=\gb_j \times [n]$.  
All of the properties above, except for the last one, mirror those of
the ordinary case; the final property differs solely through the
substitutions $\cup \to \uplus$ and $[n] \to \gb_j \times [n]$. This
is an immediate consequence of the corresponding statements for the
sets $\calC(\edge{a_i}{b_j^{\alpha}})$ derived after splitting the
instance and computing the matching subgraph, combined with their
behavior under the merging procedure.

%% All of the properties above except for the last one are just as in the
%% ordinary case, and the last one differs only by the substitution
%% %
%% $\cup \to \uplus$ and $[n] \to \gb_j \times [n]$. This is an immediate
%% consequence of the ordinary statement on the sets
%% $\calC(\edge{a_i}{b_j^{\alpha}})$ obtained after splitting the
%% instance and calculating the matching subgraph, and its implications
%% on the sets $\calC(\edge{a_i}{b_j})$ obtained after the merging
%% procedure.

% ---------- SEPARAZIONE OVVIA --------- %

%%%%%%%%%%%%%%%%%%%%%%%%%%%%%%%%%%%%%%%%%%%%%%%%%%%%%%%
\section{Geometry of matching trees on finite-dimensional domains}
\label{sec.caso2d}
%%%%%%%%%%%%%%%%%%%%%%%%%%%%%%%%%%%%%%%%%%%%%%%%%%%%%%%

\noindent
From this section onward, we will consider the case in which the
weighted graphs under analysis are associated to point configurations
on a manifold.

Let $\Omega$ be a $d$-dimensional connected differentiable manifold
with a connection. Then, geodesic curves and their lengths are
defined. Let $|\gamma|$ denote the length of the curve $\gamma$.  We
will denote by $\gamma_{x,y}$ the shortest arc of geodesic from the
point $x$ to $y$ (which is unique up to a subset of zero measure of
$\Omega \times \Omega$ --- when this is not the case one says that $x$
and $y$ \emph{form a conjugate pair}).  If $\Omega$ has a boundary, we
call $\gamma_x$ the shortest arc of geodesic from the point $x$ to
some point $u(x)$ on $\partial\Omega$ (which is unique up to a subset
of zero measure of $\Omega$ --- when this is not the case one says
that $x$ \emph{lies on the cut locus of $\partial \Omega$}). Then we have:
% Our definition of embedding goes as follows:
\begin{defn}[Geodesic embedding]
Given, on one side, the pair $(G,w)$, with $G=(V,E)$ a graph and
$w:E\to \bR$ its associated weight function, and, on the other side,
the triple $(\Omega,X,f)$, where $\Omega$ is a $d$-dimensional
differentiable manifold with a connection, 
$X=\{x_v\}_{v \in V}$ is
a set of points in $\Omega$ in bijection with $V$, and 
$f:\bR^+ \to \bR$ is a strictly monotone function, we say that
$(\Omega,X,f)$ \emph{provides a geometric realisation} of $(G,w)$ if,
for all $e=\edge{u}{v} \in E$, the length of the shortest geodesic
% $\gamma_e:\bx_u \to \bx_v$ is such that $f(|\gamma_e|)=w_e$.
$\gamma_e:x_u \to x_v$ is such that $f(|\gamma_e|)=w_e$.  
If $G=(V\cup \{s\},E)$ is a graph in the setting with a reservoir, and
$\Omega$ is a manifold with boundary $\partial \Omega$, we shall
implement geometrically the graphical setting with reservoir, by
considering the boundary as the location of the reservoir, that is, by
further requiring that the length of $\gamma_{x_v}$ is such that
$f(|\gamma_{x_v}|)=w_{v,s}$.
\end{defn}
\noindent
When $\Omega$ is a convex subset of $\bR^d$ (with the flat metric), we
will just write $\|x-y\|$ for $|\gamma_{x,y}|$.  In this situation, a
common choice of function $f$ is to set $f(|\gamma|) = |\gamma|^p
=\|x-y\|^p$ for $p>0$. When this is the case, we say that we deal with
the problem \emph{with exponent $p$}. As affine transformations of the
weights do not alter the resulting matching subgraphs, i.e.\ the
weight functions $w_e$ and $w'_e=a w_e+b$ (with $a \in \bR^+$ and
$b\in \bR$) give the same graph $H_\calJ$, choosing 
$f(|\gamma|) = |\gamma|^p$ is equivalent to choose
\be
\label{eq.3867875}
f(|\gamma|) = \frac{|\gamma|^p-1}{p}
\ef.
\ee
This alternate prescription has a sensible limit also for $p\to 0$,
that is just $f(|\gamma|) = \ln |\gamma|$, and has the standard
monotonicity (shorter paths have smaller costs) also for $p<0$.
Indeed, also the more general class of weights (\ref{eq.3867875}) for
$p \in \bR$, (that is well-defined for $p \leq 0$ provided that all
adjacent vertices are distinct), consists of functions $f$ preserving
the covariance of the problem under rescaling.  In fact it is
essentially the largest family that preserves this covariance.  So we
have a family of problems, interesting at the aim of taking a
thermodynamic limit, where the exponent parameter $p$ is an arbitrary
real number.

The datum of $(G,\calJ,\Omega,X,f)$ denotes a setting $(G,w,\calJ)$,
together with a geometric realization.  When we are in the
``monopartite standard setting'' (i.e., when $G=\calK_{2n+1}$ and
$\calJ$ is the set of singletons), we will use the shorter notation
$(\calK_{2n+1},\{\{v\}\}_{v \in V(G)},\Omega,X,f)\equiv(\Omega,X,f)$
(where implicitly $2n+1=|X|$),
while when we are in the ``first bipartite standard setting''
(i.e., $G=\calK_{n+1,n}$ and $\calJ$ is the set of singletons of
the first vertex-set), we will use the shorter notation
$(\calK_{n+1,n},\{\{a\}\}_{a \in A},\Omega,(X,Y),f)\equiv(\Omega,(X,Y),f)$, 
with the understanding that $x_j=x_{a_j}$ and $y_j=y_{b_j}$, and that $n=|X|-1=|Y|$.

In order to perform our analysis, it is convenient to work under a
genericity hypothesis also at the level of the embedding, that we now
introduce:
\begin{defn}[Generic geodesic embedding]
In the monopartite setting, we say that the embedding is
\emph{generic} if the resulting weight function $w$ is generic, and
the point configuration is geometrically generic, meaning that:
\begin{itemize}
    \item the points in $X$ are all distinct;
    \item no pair of adjacent points forms a conjugate pair;
    \item pairs of geodesic arcs $\gamma_{e'}$ and $\gamma_{e''}$ may
      only intersect transversally.
\end{itemize}
In the reservoir setting, we additionally require that no point lies
on the boundary cut locus, and that the pairs $(u(x_v),\theta(x_v))$
--- consisting of the boundary endpoint and its arrival angle in a
local chart --- are all distinct. Finally, all relevant geodesics
(that is, both the internal arcs $\gamma_{x_v, x_u}$ and the
reservoir/boundary arcs $\gamma_{x_v}$) may only intersect
transversally.
%% %% In the bipartite setting we shall only require that no pair of points
%% %% in distinct classes forms a conjugate pair, while 
%% In the reservoir setting (where the boundary of the domain plays the
%% role of the reservoir) we shall also require that none of the points
%% lie on the cut locus of $\partial \Omega$, and that, calling
%% $(u(x),\theta(x))$ the endpoint on the boundary of the geodesic
%% $\gamma_{x}$ and its arrival angle (in some local chart around $u$),
%% % (this is well-defined because of the condition above), 
%% then the pairs $(u(x_v),\theta(x_v))$ are all distinct for all
%% boundary geodesics $\gamma_{x_{v}}$. Analogously to the previous case,
%% we require that the set of all geodesics arcs 
%% %
%% (both the $\gamma_{x_v, x_u}$'s and the $\gamma_{x_v}$'s), if they
%% intersect pairwise, they do so transversally.
\end{defn}
\noindent
A possible exception to the genericity requirement above is that we
may allow that $x_{v_1}$ and $x_{v_2}$ do coincide, whenever the
number of common neighbours of $v_1$ and $v_2$ is at most 1. Indeed,
if $v_1$ and $v_2$ had two common neighbours $u_1$ and $u_2$ we would
have a cycle $C$, passing through the vertices $(v_1,u_1,v_2,u_2)$ in
this order, with $\walt(C)=0$, obviously in contradiction with our
requirement in Definition~\ref{def.genwei1}. But when this situation
does not occur there is no obvious obstruction to the properties of
genericity of the weights that we need in our proofs. In particular,
in the bipartite or bipartite with reservoir settings, it is not
forbidden that vertices in different classes may be placed in the same
position. It is a possibly counterintuitive fact of Optimal Assignment
that if $w_{a,b}=0$, and $w_{a',b'}>0$ for all other edges, it may be
the case that $\edge{a}{b}\not\in M_{\star}(G)$ (for example, this is
the case when $W=\begin{psmallmatrix} 0 & 1 \\ 1 &
3\end{psmallmatrix}$), as the optimal matching is the result of a
complicated global optimization, thus coinciding vertices as above do
not have in general a trivial role.

% the nearest points on the boundary $u(x_{v})$ are all distinct.
%% We will say that the embedding is \emph{generic} if the resulting
%% weight function $w$ is generic, and the embedding is generic in the
%% ordinary geometric sense, namely the $x_v$'s are all distinct, the
%% % the $\bx_v$'s are all distinct, and the
%% $\gamma_e$'s intersect transversally, the geodesic arc of minimal
%% length is unique for all edges, and there are no triples of distinct
%% vertices $(u_1,u_2,v)$ with $x_v \in \gamma_{\{u_1,u_2\}}$.
%% ---
%% We assume the set of points $(X,Y)$ is generic, in the sense that no
%% $x_i$ or $y_j$ lies on the cut locus of $\partial \Omega$, and no
%% $y_j$ is conjugate to any $x_i$ (nor do they lie on the corresponding
%% focal loci).
%% ---
% \noindent

Now, we need a notion that is the core of the main results of the
forthcoming sections.
\begin{defn}[Non-crossing embedding of a graph]
Let $\Omega$, $X$ and $G=(V,E)$ be as above (that is, $x_v$ is a
bijection between $V$ and $X$, and the embedding is generic).  We say
that the embedding of $G$ in $\Omega$ is \emph{non-crossing} if there
are no pairs of edges $\edge{u}{v}$ and $\edge{u'}{v'}$ such that
$\gamma_{x_u,x_v}$ and $\gamma_{x_{u'},x_{v'}}$ have an internal point
in common.
\end{defn}
\noindent
This notion extends naturally to the hypergraph setting.
Recall that a set $S \subseteq \Omega$ is said to be
\emph{geodetically convex} if, for all pairs of points $x, y \in S$,
% there is a unique 
and all shortest geodesic $\gamma$ connecting them, the curve $\gamma$
is contained in $S$.  The \emph{geodesic convex hull} of a set $X$
% (when it exists) 
is defined as the intersection of all geodetically convex sets $S$
containing~$X$.
% \footnote{This set may not exist
\begin{defn}[Non-crossing embedding of a hypergraph]
Let $\Omega$, $X$ and the hypergraph $G=(V,\mathcal{E})$ be as above. 
%% Let $\Omega$ be a manifold as above, let $G=(V,\mathcal{E})$ a hypergraph,
%% and let $\{x_v\}$ be a bijection between the vertices of $G$ and a
%% collection of points in $\Omega$.
% 
%% Under suitable conditions on the hyperedges $A$ of $G$ (in particular,
%% that for all ordered lists $(x_1,\ldots,x_k)$ of distinct vertices in
%% a hyperedge $A$, the concatenation of $\gamma_{x_1,x_2}$,
%% $\gamma_{x_2,x_3}$, \ldots, $\gamma_{x_k,x_1}$ is contractible), we
%% can define $\Gamma_A \subseteq \Omega$, the \emph{geodesic convex hull
%%   of $A$},
We define $\Gamma_A \subseteq \Omega$, the embedding of the hyperedge
$A$, as the geodesic convex hull of the set $\{x_v\}_{v \in A}$. In
particular, if $G$ is an ordinary graph then $\Gamma_{\edge{u}{v}}$ is
$\gamma_{x_u,x_v}$, while if $\Omega$ is a portion of $\bR^d$
then $\Gamma_A$ is the ordinary convex hull of the set $\{x_v\}_{v \in A}$.
We say that the embedding of $G$ is \emph{non-crossing} if there are
no $A \neq A' \in E$ such that $\Gamma_A$ and $\Gamma_{A'}$ have an
internal point in common.  
\end{defn}
\noindent
Generic embeddings have $\dim(\Gamma_A)=\min(\dim(\Omega),|A|-1)$, and
generic sets of dimension $d'$ and $d''$ in a manifold of dimension
$d$ are in general non-crossing whenever $d>d'+d''$, so that the
embedding of $G$ is in general non-crossing provided that $d$ is
larger than the sum of the cardinalities of the two largest
hyperedges, minus 2, while, when this condition is not met (as for
example for the embedding of graphs in dimension 2, or of 3-uniform
hypergraphs in dimension up to 4), establishing that a family of
(generic) embedded graphs is non-crossing is a non-trivial statement.

Recall that an important ingredient of the previous sections is the
Definition~\ref{def.proj} of the notion of ``projection'' (that we
apply to the matching subgraph $H$ when working in a bipartite
setting). We have to provide some supplementary prescriptions for how
to deal with the embedding. If $\Omega$ is simply connected (in the
case of flat metrics), the obvious definition for the embedding of
graphs (or even of hypergraphs) works as expected. However, when the
manifold $\Omega$ has a non-trivial fundamental group $\pi_1(\Omega)$
(e.g., when $\Omega$ is a cylinder, or a torus), we need some
precisions.

Whenever we refer to oriented geodesic arcs, if we write
$\gamma_{x_a,x_b}$, it is intended that the arc is oriented from $x_a$
to $x_b$.  Consequently, for example, the concatenation
$\gamma_{x_a,x_b} \circ \gamma_{x_b,x_a}$ traces a backtracking path
that is contractible to a point.

If we introduce arbitrarily some orientation of the edges, i.e.\ for
$e=\edge{a}{b}$ we may have, for example, $\vec{e}=(a,b)$ and
$\cev{e}=(b,a)$, then we can use the synonims
$\gamma_{\vec{e}}=\gamma_{x_a,x_b}$ and
$\gamma_{\cev{e}}=\gamma_{x_b,x_a}$.\footnote{Of course, bipartite
  graphs come with a canonical orientation, e.g.\ with geodesics going
  from the $A$-vertex to the $B$-vertex.}

%, that is, for example, the concatenation $\gamma_{x_a,x_b} \circ
% \gamma_{x_b,x_a}$ is a closed circuit homotopic to a point.

%% For the embedding of a hypergraph, we will require the following: for
%% each $\Gamma_A$, embedding of the hyperedge $A$ realized as the
%% geodesic convex hull of the set $\{x_v\}_{v \in A}$, and each ordered
%% lists $(x_1,\ldots,x_k)$ of distinct vertices in $A$, we require that
%% the concatenation of $\gamma_{x_1,x_2}$, $\gamma_{x_2,x_3}$, \ldots,
%% $\gamma_{x_k,x_1}$ is contractible.  This implies that the whole set
%% $\Gamma_A$ is homotopically contractible, and topologically a sphere,
%% and it does not contain any pair of conjugate points.  

First, to prevent hyperedges from wrapping around non-contractible cycles of
the manifold (such as winding around a cylinder or a handle), we
require that for any ordered list of distinct vertices in $A$, the
closed circuit formed by concatenating the boundary geodesic segments
is contractible in $\Omega$. This topological condition ensures that
each $\Gamma_A$ is homeomorphic to a closed ball, and 
$\partial \Gamma_A$ is homeomorphic to a sphere, so that $\Gamma_A$
acts locally as a genuine convex hull, akin to its flat $\mathbb{R}^d$
counterpart, and rules out global pathological wrappings at the level
of the elementary constituents of the (hyper-)graph.  Note that this
requirement is automatically satisfied in the case of an ordinary
graph.

Then, in the case of bipartite graphs, we choose to define the
projection of the embedding as follows:
\begin{defn}[Geodesic projection]
\label{def.projgeo}
Let $\bipg{G}{A}{B}{E}$ be a (weighted) bipartite graph, and let
$(\Omega,(X,Y),f)$ provide a geometric realisation of $(G,w)$, with
$X=\{x_a\}_{a\in A}$ and $Y=\{y_b\}_{b\in B}$. Let $\bar{G}=(A,E')$,
the projection of $G$, be as in Definition~\ref{def.proj}. Then, the
embedding of $\bar{G}$ is determined by the definition of the arcs
$\gamma_{e'}$, for $e' \in E'$.  
For $\vec{e}'=(a_1,a_2) \in E'$ coming from the pair
$\vec{e}_1=(a_1,b)$ and $\vec{e}_2=(a_2,b)$, we will set
$\gamma_{\vec{e}'}$ to be the shortest geodesic from $x_{a_1}$ to
$x_{a_2}$ having the same homotopy class as the concatenation
$\gamma_{\vec{e}_1} \circ \gamma_{\cev{e}_2}$ (i.e., the path going
from $x_{a_1}$ to $y_b$ and then from $y_b$ to $x_{a_2}$).
%% For $\vec{e}'=(a_1,a_2) \in E'$ coming from the pair
%% $\vec{e}_1=(a_1,b)$ and $\vec{e}_2=(a_2,b)$, we will set
%% $\gamma_{\vec{e}'}$ to be the shortest geodesic from $x_{a_1}$ to
%% $x_{a_2}$, with the same topology as $\gamma_{\vec{e}_1} \circ
%% \gamma_{\cev{e}_2}$ (i.e., the concatenation of $\gamma_{\vec{e}_1}$
%% and the reverse of $\gamma_{\vec{e}_2}$,
%% i.e.\ $\gamma_{x_{a_1},x_{a_2}}$ is the shortest geodesics with the
%% same topology of the concatenation of $\gamma_{x_{a_1},y_b}$ and
%% $\gamma_{y_b,x_{a_2}}$).
%
% , that is, such that $\gamma_{\vec{e}_1} \circ \gamma_{\cev{e}_2}
% \circ \gamma_{\cev{e}'}$ is a contractible closed curve.
\end{defn}
\noindent
Note that the curve $\gamma_{\vec{e}'}$ always exists, i.e.\ the set
of geodesics from $x_1$ to $x_2$ with the required topology is
non-empty, as it contains at least the minimum-length geodesic
obtained by steepest descent (for the functional $S[\gamma]=|\gamma|$)
from
$\gamma_{\vec{e}_1} \circ \gamma_{\cev{e}_2}$.

Also note that, with this description, a hypergraph obtained from a
projection automatically satisfies the condition given above.  Namely, we
required that for each $\Gamma_A$, embedding of the hyperedge $A$, and
each ordered lists $(x_1,\ldots,x_k)$ of distinct vertices in $A$, the
pointed cycle obtained as the concatenation of $\gamma_{x_1,x_2}$,
$\gamma_{x_2,x_3}$, \ldots, $\gamma_{x_k,x_1}$ is contractible.  But,
from our construction we know that
this cycle is homotopic to the concatenation of 
$\gamma_{x_1,y}$, $\gamma_{y,x_2}$, $\gamma_{x_2,y}$,
$\gamma_{y,x_3}$, \ldots, $\gamma_{x_k,y}$, $\gamma_{y,x_1}$, in this
order, that is essentially a star in which each arm is traversed
twice, so it is indeed contractible.

As an example of the subtlety of this definition, if $\Omega$ is a
portion of a cylinder, with
$(x^{(1)},x^{(2)})\sim(x^{(1)}+1,x^{(2)})$, then, in a graph $G$
containing the edges $e_1=\edge{x_1}{y}$ and $e_2=\edge{x_2}{y}$ with
$x_1=(0,a)$, $x_2=(2/3,b)$ and $y=(1/3,c)$, the geodesics
$\gamma_{x_1,y}$ and $\gamma_{x_2,y}$ are the obvious shortest
geodesics, 
$\gamma_{x_1,y}(t)=(t/3,(1-t)a+t c)_{t\in[0,1]}$
and $\gamma_{x_2,y}(t)=((1+t)/3,(1-t)c+t b)_{t\in[0,1]}$,
but the geodesic $\gamma_{x_1,x_2}$ appearing in the
embedding of $\bar{G}$ is the curve $\gamma(t)=(2t/3,(1-t)a+t b)$ 
instead of the shortest one 
$\gamma(t)=(1-t/3,(1-t)a+t b)$,
as depicted below:
\[
\if\faifig1  
\begin{tikzpicture}[scale=2.5]
\begin{scope}[dashed,decoration={
    markings,
    mark=at position 0.5 with {\arrow{>}}}
    ] 
    \draw[postaction={decorate}] (0,0)--(0,1);
    \draw[postaction={decorate}] (1,0)--(1,1);
\end{scope}
\node[label={270:{\small $x_1$}},inner sep=2.5pt,
circle,fill=white,draw
] (x1) at (.1666,.2) {}; 
\node[label={270:{\small $x_2$}},inner sep=2.5pt,
circle,fill=white,draw
] (x2) at (.83333,.4) {}; 
\node[label={90:{\small $y$}},inner sep=2.5pt,
rectangle,fill=gray!50,draw
] (y) at (.5,.7) {}; 
\draw[thick,dashed] (x1) -- (0,.3); 
\draw[thick,dashed] (x2) -- (1,.3); 
\draw[very thick,blue] (x1) -- node[sloped,below]{\rule{0pt}{0pt}$\gamma_{(x_1,x_2)}$\rule{0pt}{0pt}} (x2); 
\draw[very thick,red] (x1) -- node[sloped,above]{\rule{0pt}{0pt}$\gamma_{(x_1,y)}$} (y); 
\draw[very thick,red] (y) -- node[sloped,above]{\rule{0pt}{0pt}$\gamma_{(y,x_2)}$} (x2); 
\end{tikzpicture}
\else [...TikZ\ code...] \fi
\]
The main reason for this definition is that the following
Theorem~\ref{th:treenoncrossToro}, the analogue of
Theorem~\ref{th:treenoncross} for cylinders and tori, would not hold
if we did choose the na\"ive definition, namely that also the geodesic
$\gamma_{x_1,x_2}$ is just the shortest one. A counterexample to the
statement of the theorem when using this different definition is
provided in Appendix~\ref{app.torosbagliato}.

\begin{figure}[t]
\includegraphics[height=0.4\columnwidth]{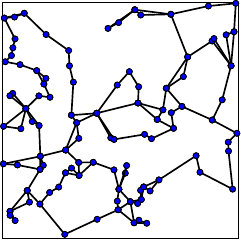}\quad
\includegraphics[height=0.4\columnwidth]{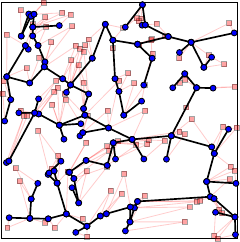}
%\quad\includegraphics[height=0.4\columnwidth]{monotorus.png}
\caption{\label{fig:mono}% 
Examples of embedded graphs $H_\calJ$ and $\bar{H}_\calJ$. In this
case, $\Omega$ is the (flat) unit square,
% we are in the standard monopartite (left) or bipartite (right)
% setting,
and the sets of points are just i.i.d.\ and uniform in $\Omega$.
Left: the monopartite standard setting,
% the complete graph $G=(V,E)=\calK_{2n+1}$ 
with $2n+1=101$ points,
% $\calJ=\{ \{v\} \}_{v \in V}$, as discussed in Example ??????, 
and $f(x)=x$, that is, $w_{u,v}=\|x_u-x_v\|$.  The graph
$H_\calJ$ is shown. Note that it has no bridges, according to
Lemma~\ref{lem:nobridge}, and is non-crossing, according to
Corollary~\ref{corr.1pnoncross}, proven below.
Right: the first bipartite standard setting
% the complete bipartite graph $G=\calK_{n+1,n}=(A\cup B,E)$ 
with
$n=|B|=|A|-1=100$ and $f(x)=x^2$.
$A$-vertices and $B$-vertices are in blue and red, respectively.  The
graph $\bar{H}_\calJ$ is shown in black and blue, while $H_\calJ$ is
shown in shaded red. Note that both $H_\calJ$ and $\bar{H}_\calJ$ are
trees, according to Theorem~\ref{thm:Jtree} and
Remark~\ref{rmk.projForIsFor}, and that $\bar{H}_\calJ$ is
non-crossing, according to Theorem~\ref{th:treenoncross}, proven
below.}
\end{figure}

Now the exact realization of a system with reservoir 
$G=(V \cup \{s\};E)$, with $|V|=n$, in terms of a graph $G^{\rm
  res}=(V\cup S \cup S' \cup V',E^{\rm res})$, with
$|S|=|V|=|S'|=|V'|=n$, described on page~\pageref{pg.reservoir} in
Section~\ref{sec.reservintro}, has also a geometric interpretation: we
can imagine that $G^{\rm res}$ is realised on a manifold without
boundary, consisting of two identical copies of $\Omega$, say $\Omega$
and $\Omega'$, with $\partial \Omega$ and $\partial \Omega'$ sewn
together in the natural way, and the involution $\iota$ exchanges
homologous points in $\Omega$ and $\Omega'$.\footnote{This is a
  classical construction in Differential Geometry, in particular when
  dealing with cobordism, and Poincar\'e Duality. If $\Omega$ is a
  differentiable manifold with boundary, its \emph{double} is obtained
  by gluing two copies of $\Omega$ together along their common
  boundary. Precisely, the double is $\Omega \times \{0,1\}/\sim$,
  where $(x,0)\sim(x,1)$ for all $x \in \partial \Omega$.}
The vertices in $V$ and $V'$ are in $\Omega$ and $\Omega'$,
respectively, as well as those in $S$ and $S'$. However, the latter
are on the two boundaries, which are sewn together, which justifies
the choice of having $w_{s_j,s'_j}=0$.  Under this construction, the
whole graph $G$ is embedded onto
$\Omega \cup \Omega'$, and the weights coincide with those of the
customary embedding for generic graphs and manifolds.

If we insist in giving a graphical representation also for the subgraph
$H^{\rm (l)}_{\calJ}$, defined in the generic symmetric setting, while
working in the setting of $G^{\rm res}$ for bipartite graphs with a
reservoir, we have that the $v_j$'s are inside $\Omega$, the $s_j$'s
are on $\partial \Omega$, and we have an extra vertex, $\sigma$,
attached to all and only the $s_j$'s. The most fruitful way to imagine
$\sigma$ is that it represents $\partial \Omega$ as a whole.

Thus, the collection of the geodesics $\gamma_{x_i,x_j}$ for
$\edge{v_i}{v_j}\in H_\calJ$, the geodesics $\gamma_{x_i}$, from $x_i$
to its nearest point $u(x_i)$ on $\partial \Omega$, for
$\edge{v_i}{s_i}\in H_\calJ$, and the boundary $\partial \Omega$,
constitutes the geometric embedding of the left part of $H_\calJ$.

We will use the notation $(\Omega,X,f)_{\rm res}$ to denote 
a setting in which the graph $G$ is a reservoir graph isomorphic to
$\calK_{n+1}$, with vertex set $V=\{v_1,\ldots,v_n\}$ (plus the
reservoir vertex $s$). 
% With abuse of notation, we will use the same
We will use $(\Omega,X,f)_{\rm sym}$
% notation 
for the extended construction, on $\Omega \cup \Omega'$, that
transforms the system with a reservoir into a special instance of
symmetric graph, as explained above. In this case the choice of
$\calJ$ is
$\calJ=\varnothing \cup \{ \{v_j,v'_j\} \}_{j=1,\ldots,n}$.

% QUESTO SETTING AL MOMENTO NON HA NOME (SAREBBE IL SECOND MONOPARTITE
% STD SETTING versione reservoir) PERCHE' NON SO SE SERVE DAVVERO.....

Note however a major difference in the case of bipartite graphs: while,
in the generic case, for $G$ bipartite and $\calJ$ as in
Theorem~\ref{thm:Jtree}, the graph $H_{\calJ}$ is a tree in which all
$B$-vertices (not in a set $U_j$) are of degree 2, now this property
holds for the $B$-vertices inside $\Omega$, and the $A$-vertices
inside $\Omega'$, so that the projection $\bar{H}_{\calJ}$ is
tree-like only if restricted to $\Omega$ (i.e., in
Theorem~\ref{thm:Jtreesym} the projected spanning graph is indeed
$\bar{H}^{\rm (l)}_{\calJ}$, or, almost equivalently, in the reservoir
setting of Section~\ref{sec.reservintro} the subgraph $H_{\calJ}$ is
the subgraph of the extended graph
$H^{\rm res} \subseteq G^{\rm res}$, and the graph which is a tree
rooted on $s$, with properties analogous to $\bar{H}_{\calJ}$, is
$\overline{(\calL H^{\rm res})}$, that is the projection of 
$\calL H^{\rm res}$).

In the bipartite case, the collection of the geodesics
$\gamma_{x_i,y_j}$ for $\edge{a_i}{b_j}\in H_\calJ$, the geodesics
$\gamma_{x_i}$, from $x_i$ to its nearest point $u(x_i)$ on 
$\partial \Omega$, for $\edge{a_i}{d_i}\in H_\calJ$, and the geodesics
$\gamma_{y_j}$, for $\edge{b_j}{c_j}\in H_\calJ$, and the boundary
$\partial \Omega$, constitutes the geometric embedding of the left
part of $H_\calJ$, which is a forest of rooted trees, where each tree
has a unique vertex on $\partial \Omega$. In other words, it is a
spanning tree (i.e.\ a tree visiting all vertices of $X\cup Y$) with
``wired boundary conditions''. 

Then, for the subgraph $\bar{H}^{\rm (l)}_{\calJ}$ we will perform the
projection as in the ordinary case, by contracting the
$B$-vertices. This accounts to replace pairs of geodesics
$\gamma_{x_{a_1},y_b}$, $\gamma_{x_{a_2},y_b}$, by the shortest
geodesic $\gamma_{x_{a_1},x_{a_2}}$ with the same topology of
$\gamma_{x_{a_1},y_b} \circ \gamma_{y_b,x_{a_2}}$.  The graph
$\bar{H}^{\rm (l)}_{\calJ}$, which is the most interesting one for
applications in Statistical Mechanics, is again a forest of rooted
trees, where each tree has a unique vertex on $\partial \Omega$. In
other words, it is a spanning tree on the pertinent graph (i.e.\ a
tree visiting all vertices of $X$) with ``wired boundary conditions''.

A remark is in order to deal with boundary vertices, images of the
vertex sets $S$ and $S'$.  If we have a geodesic to the boundary
$\gamma_{y_b}$, it goes from $y_b$ to some point $u(y_b)$ on 
$\partial \Omega$, image of a $C$-vertex and the point on the boundary
nearest to $y_b$. So if $\edge{a}{b}$ and $\edge{b}{c}$ are edges of
$H_{\calJ}$, we should add the shortest geodesic from $x_a$ to
$u(y_b)$ with the same topology of the concatenation of
$\gamma_{x_{a},y_b}$ and $\gamma_{y_b}$, in analogy with the customary
case.  If we have a geodesic to the boundary $\gamma_{x_a}$, then it
goes from $x_a$ to some point $u(x_a)$ on $\partial \Omega$, image of
a $D$-vertex.  In $H_\calJ$, this vertex is connected to $s'$, so in
our rule we should concatenate the geodesic from $x_a$ to $u(x_a)$,
and the one from $u(x_a)$ to $\iota u(x_a)$. However, in fact
$u(x_a)$ and $\iota u(x_a)$ are the same point of 
$\partial \Omega \equiv \partial \Omega'$, so that the concatenated
geodesic coincides with $\gamma_{x_a}$.

In finite-dimensional Statistical Mechanics, ensembles of spanning
trees (uniform, minimal,\ldots) have two natural boundary conditions:
``free'' and ``wired'' (see e.g.~\cite{BLPSforests,kytolaetc}).
Essentially, the Sections~\ref{sec.teoremaserio1} and
\ref{sec.teoremaserio1b} are devoted to set up theorems adapted to
these two cases, as the reservoir setting is a specialization of the
setting of Section~\ref{sec.teoremaserio1b}.

Similarly to what is done above, we will use the notation
$(\Omega,(X,Y),f)_{\rm res}$ to denote a graphical realization of a
bipartite graph with reservoir, and the notation
$(\Omega,(X,Y),f)_{\rm sym}$ for the construction of the associated
extended graph, that thus denotes a graphical realization of the
``second bipartite standard setting'', specialised to instances coming
from systems with reservoir.
\medskip

\noindent
The goal of this section is to identify choices of $\Omega$ and $f$
such that, within generic instances $(G,w)$ that admit a geometric
representation, the matching subgraphs $H_{\calJ}$ or their projection
$\bar{H}_{\calJ}$ have further properties w.r.t.\ the case of generic
$(G,w)$ that we have analysed above. In particular, for $d=2$, we will
try to identify situations in which the graphs $H_{\calJ}$ or
$\bar{H}_{\calJ}$ are guaranteed to be non-crossing.  Yet again, as
was the case in Section~\ref{sec.teoremaserio1}, we anticipate that
the most relevant results are obtained when the underlying graph $G$
is bipartite.

%%%%%%%%%%%%%%%%%%%%%%%%%%%%%%%%%%%%%%%%%%%%%%%%%%%%%%%
\section{\texorpdfstring{Non-crossing matching trees at $p=1$}{Non-crossing matching trees at p=1}}
%%%%%%%%%%%%%%%%%%%%%%%%%%%%%%%%%%%%%%%%%%%%%%%%%%%%%%%

\noindent
Let us consider first the $p=1$ case, which will be relatively simple.
We start with a simple remark:
\begin{rmk}
\label{rmk.diagop1}
For an arbitrary two-dimensional manifold $\Omega$ as above, given
$X=\{x_1,x_2\}$ and $Y=\{y_1,y_2\}$, in generic position, for the
embedding of $G=\calK_{2,2}$ with the two families of vertices given
by $X$ and $Y$, at $p=1$, the optimal matching $M_{\star}$ is always
non-crossing.
\end{rmk}
\noindent
% This remark holds both in the ordinary ($G=\calK_4$) case, and in the
% bipartite ($G=\calK_{2,2}$) case. Let us start from the latter.
Indeed, suppose that the optimal matching is
$M_{\star}=\{\edge{x_1}{y_1},\edge{x_2}{y_2}\}$, and that the geodesics
$\gamma_{x_1,y_1}$ and $\gamma_{x_2,y_2}$ cross, say at a
point $z$. As our arcs of geodesics $\gamma_e$ are those of minimal
length, it is clear that they do satisfy the triangle inequality,
i.e., $|\gamma_{x,y}|\leq |\gamma_{x,z}|+|\gamma_{z,y}|$, with
equality only if $z \in \gamma_{x,y}$. This implies that
\be
\begin{split}
|\gamma_{x_1,y_1}|
+|\gamma_{x_2,y_2}|
&=
|\gamma_{x_1,z}|+|\gamma_{z,y_1}|+|\gamma_{x_2,z}|+|\gamma_{z,y_2}|
% \\ &
\geq
|\gamma_{x_1,y_2}|+|\gamma_{x_2,y_1}|
\end{split}
\ee
so that the matching $M_1=\{\edge{x_1}{y_2},\edge{x_2}{y_1}\}$ has a smaller
weight than $M_{\star}$. We have a strict inequality whenever $z$ is
not on both the geodesics $\gamma_{x_1,y_2}$ and $\gamma_{x_2,y_1}$,
which is always the case under the condition of generic embedding.

Similarly, we get
\begin{rmk}
\label{rmk.diagop1mono}
For an arbitrary two-dimensional manifold $\Omega$ as above, given
$X=\{x_1,x_2,x_3,x_4\}$, in generic position, for the embedding of
$G=\calK_{4}$ with point configuration $X$, at $p=1$, of the three
possible matchings, the two matchings of smallest weight are 
always non-crossing.
\end{rmk}
\noindent
The argument is just as in Remark~\ref{rmk.diagop1}, repeated
twice.  If the geodesics
$\gamma_{x_1,x_2}$ and $\gamma_{x_3,x_4}$ cross, say at a
point $z$, we have
\begin{subequations}
\begin{align}
|\gamma_{x_1,x_2}|
+|\gamma_{x_3,x_4}|
&=
|\gamma_{x_1,z}|+|\gamma_{z,x_2}|+|\gamma_{x_3,z}|+|\gamma_{z,x_4}|
% \\ &
\geq
|\gamma_{x_1,x_3}|+|\gamma_{x_2,x_4}|
\\
|\gamma_{x_1,x_2}|
+|\gamma_{x_3,x_4}|
&=
|\gamma_{x_1,z}|+|\gamma_{z,x_2}|+|\gamma_{x_3,z}|+|\gamma_{z,x_4}|
\geq
|\gamma_{x_1,x_4}|+|\gamma_{x_2,x_3}|
,
\end{align}
\end{subequations}
so that, again because of the genericity condition, both matchings 
$\{\edge{x_1}{x_3},\edge{x_2}{x_4}\}$
and
$\{\edge{x_1}{x_4},\edge{x_2}{x_3}\}$
have a smaller weight than
$\{\edge{x_1}{x_2},\edge{x_3}{x_4}\}$.

Note that the hypothesis $p=1$ above is crucial: it is easy to
construct counterexamples (even just on $\bR^2$) for all other values of
$p \in \bR$.

As a result of the two remarks above, at $p=1$, pairs $\{e_1,e_2\}$ such that
their geodesics do intersect form a set of ``forbidden pairs'', in the
sense of Definition~\ref{def.forbpair}. Thus, as a consequence of the
general criterium of Lemma~\ref{lem.nopairF}, we get
\begin{corr}
\label{corr.1pnoncross}
For any $2$-dimensional connected differentiable
manifold with a connection $\Omega$, 
in the settings $(\Omega,X,f(x)=x)$ and $(\Omega,(X,Y),f(x)=x)$,
the embedded graph $H_{\calJ}$ is non-crossing.
\end{corr}
\noindent
Note that, at the aim of applying the crucial Lemma~\ref{lem.nopairF},
it is important that the standard monopartite and bipartite settings
are \emph{single-path} in the sense of
Definition~\ref{def.singlepath}. It is easy to construct
counterexamples to the corollary above, with the geometric embedding
of triples $(G,w,\calJ)$ when $\calJ$ is not single-path. For example,
when $G=\calK_4$, $X$ consists of the vertices of a square, and
$\calJ$ is the set of all pairs of points.

This corollary generalises to matching subgraphs $H_{\calJ}$ (and to
arbitrary manifolds) the well-known fact that, in the $p=1$ Euclidean
setting, optimal matchings $M_{\star}(G)$ are non-crossing.

%%%%%%%%%%%%%%%%%%%%%%%%%%%%%%%%%%%%%%%%%%%%%%%%%%%%%%%
\section{\texorpdfstring{Non-crossing projected matching trees at
    $p=2$}{Non-crossing projected matching trees at p=2}}
\label{sec.caso2dp2}
%%%%%%%%%%%%%%%%%%%%%%%%%%%%%%%%%%%%%%%%%%%%%%%%%%%%%%%

%-------------------------------------------------------
\subsection{Non-crossing criteria and generalities}
\label{ssec.nccgen}
%-------------------------------------------------------

Before delving into the more subtle question involving $\bar{H}_\calJ$,
let us first establish a slightly simpler criterium for a setting to
be guaranteed to be non-crossing.
We shall concentrate only on the first bipartite standard setting.
\begin{rmk}
\label{rmk.crossifpath}
Given a family $\calX=\{ (\Omega,(X,Y),f) \}$ of bipartite settings
with a geometric realization, closed under taking subsets (given
$X'\subseteq X$ and $Y'\subseteq Y$ with 
$|X|=|Y|+1$ and $|X'|=|Y'|+1$, 
if $(\Omega,(X,Y),f) \in \calX$, then
$(\Omega,(X',Y'),f) \in \calX$), then there exists a 
$(\Omega,(X,Y),f) \in \calX$ crossing iff there exists a 
$(\Omega,(X,Y),f) \in \calX$, with $|X|=n+1$, such that
$\bar{H}_\calJ$ is the path $(x_1,x_2,\ldots,x_{n+1})$, and the
geodesics $\gamma_{x_1,x_2}$ and $\gamma_{x_n,x_{n+1}}$ (given by the
prescription of Definition~\ref{def.projgeo}) do cross.
\end{rmk}
\noindent
Let us illustrate the validity of this remark. Of course, we need to
prove only one of the implications, as the other one is obvious.  

Say that $\gamma_{x_a,x_b}$ and $\gamma_{x_c,x_d}$ do cross.  As
$\bar{H}_\calJ$ is a spanning tree, there is a unique path connecting
one of the endpoints of $\edge{x_a}{x_b}$ to one of those of
$\edge{x_c}{x_d}$, without visiting the other two endpoints. Up to
renaming the vertices, we can say that the path connects $x_b$ to
$x_c$, and has length $\ell$. This implies that $\bar{H}_\calJ$
contains the path $(x_a,x_b,x'_1,\ldots,x'_{\ell-1},x_c,x_d)$, for
certain vertices $x'_j$, and that $H_\calJ$ contains the path
$(x_a,y'_0,x_b,y'_1,x'_1,\ldots,y'_{\ell-1},x'_{\ell-1},y'_{\ell},x_c,y'_{\ell+1},x_d)$,
for certain vertices $y'_j$.  
This path is also characterised as the
subset of edges $e$ of $H_\calJ$ such that the singleton sets of $X'$,
that is $\tilde{X}'=\{\{x_a\},\{x_b\},\{x'_1\},\ldots\}$, has
non-trivial intersection with $\calC(e)$ (while, for the other edges $e$,
either $\tilde{X}' \subseteq \calC(e)$ or $\tilde{X}' \cap \calC(e)=\varnothing$).

Now consider the triple
$(\Omega,(X',Y'),f)$ for
$X'=\{x_a,x_b,x'_1,\ldots,x'_{\ell-1},x_c,x_d\}$ and
$Y'=\{y'_0,y'_1,\ldots,y'_{\ell},y'_{\ell+1}\}$.  This triple is in
$\calX$. Furthermore, crucially, its associated graph $H_\calJ$
coincides with the path above, as implied by the characterisation via
$\tilde{X}' \cap \calC(e)$, and as a special case of
Lemma~\ref{lem.subtreeIsHJ}.  We deduce that $(\Omega,(X',Y'),f)$ has
the required characteristics, and we can conclude.

\medskip
\noindent
If $f(x)=x^p$ and $\Omega$ is a prortion of $\bR^d$, we can identify a group of
invariances of the problem: $(\bR^d,(X,Y),x^p)$ is crossing iff
$(\bR^d,(gX,gY),x^p)$ is crossing, where
$gX=\{g \circ x_1, g \circ x_2,\ldots\}$, and $g$ is a translation, or
rotation, or dilation (or a combination of them). As a result, we can
``standardise'' the sets $(X,Y)$, for example by setting that the
crossing of the two edges is at the origin, one of the crossing edges
is a portion of the horizontal axis, and one of its endpoints is at
unit distance from the origin.

A further invariance property arises when $p=2$.
The determination of the pairs $(X,Y)$ (with $|X|=n+1$) such that the
segments $[x_1,x_2]$ and $[x_n,x_{n+1}]$ do cross goes through the
determination of the optimal assignments for cost matrices which are
$n \times n$ minors of the matrix 
$W=(W_{ij})_{1\leq i \leq n+1; 1\leq j \leq n}$, with
$W_{ij}=\|x_i-y_j\|^2$. 
However, in light of the gauge invariance discussed in
Remark~\ref{rmk.gaugeinv}, we also have
\begin{corr}
The matching subgraph $H_{\calJ}$ associated to the bipartite problem
with cost matrix $W_{ij}=\|x_i-y_j\|^2$ coincides with the one
with cost matrix $W'_{ij}=-2 \langle x_{i}, y_{j}\rangle$.
% x_i \cdot y_j$
\end{corr}
\noindent
Indeed, this corresponds to choose $\lambda_i=|x_i|^2$ and
$\mu_j=|y_j|^2$.  The form of the matrix $W'$ is structurally simpler
than the one of $W$, as the entries are bilinear in the $x_i$'s and
$y_j$'s.  This implies that there is a larger group of invariances of
the problem: besides the previous transformations, now
$(\bR^d,(X,Y),x^2)$ is crossing also iff $(\bR^d,(BX,B^{-1}Y),x^2)$ is
crossing, where
$B$ is in $\mathrm{GL}(d,\bR)$.  On top of this, we can translate the
points in $X$, and the points in $Y$, separately, as (e.g.) the
transformation $y_j \to y_j + z$ corresponds to the gauge
transformation $\lam_i = 2 \langle x_i,z \rangle$ and the
transformation $x_i \to x_i + z$ corresponds to the gauge
transformation $\mu_j = 2 \langle z,y_j \rangle$.

This is a well-known fact in Optimal Transportation Theory. When the
cost is $\|x-y\|^2$, and the two measures are continuous and non-zero,
the convex potential $\phi(x)$ (appearing in Brenier's Theorem) and
its conjugate $\psi(y)$ satisfy the optimality relation
\be
\phi(x) + \psi(y) \le \|x-y\|^2
\ef;
\ee
and the equality occurs iff the infinitesimal mass in $x$ is
transported in $y$, which, in turns, occurs when $y = x -
\frac{1}{2}\nabla \phi(x)$.

An analogous theorem holds for the cost function 
$-2\langle x,y\rangle$, and the pair of conjugate potentials
$\phi^{\rm new}(x)=\phi(x)-\|x\|^2$ and 
$\psi^{\rm new}(y)=\psi(y)-\|y\|^2$, namely
\be
\phi^{\rm new}(x)+\psi^{\rm new}(y)
\le -2\langle x,y\rangle
\ef;
\ee
and the equality occurs iff the infinitesimal mass in $x$ is
transported in $y$, which, in turns, occurs when 
$y=-\frac{1}{2}\nabla \phi^{\rm new}(x)$.  With a version of the
Brenier's Theorem (and of the Monge--Ampére equation) in place for the
bilinear cost, the covariances discussed above are easily deduced also
in the setting of Optimal Transportation.

As a result, although this is not strictly needed in our forthcoming
analysis, we could ``standardise'' the sets $(X,Y)$ even more, for
example by setting that the crossing of the two edges is at the
origin, the two crossing edges are portions of the first two axis, and
one endpoint of each of these edges is at unit distance from the
origin.

In this section, from this point onward, we will restrict to the case
$p=2$, that will turn out to be the only case in which $\bar{H}_\calJ$
is guaranteed to be non-crossing for a class of 2-dimensional domains.

%-------------------------------------------------------
\subsection{The case of the plane}
\label{ssec.noncrossplane}
% \texorpdfstring{The case of the first bipartite standard setting, $p=2$ and $\Omega\subset \bR^2$}
% {The case of the first standard setting, p=2 and the plane}}
%-------------------------------------------------------
%
In this section we will make some use of a diagrammatic notation (with
grids of dots) explained in detail in
Appendix~\ref{app.altproofValid1st2nd}. This is done here in very
simple settings, the graphical explanations are not compulsory to the
understanding of the proof, are provided only as a visualisation aid,
and are also supplemented by an alternative graphical
representation. For these reasons, we postpone all the pertinent
definitions of our grids of dots to
Appendix~\ref{app.altproofValid1st2nd}, where its role becomes more
prominent.

We will start our analysis from the case in which $\Omega$ is a
portion of $\bR^2$, and the function $f$ is $f(x)=x^2$ (that is,
$p=2$).  We will consider embedded instances in the ``first bipartite
standard setting'' of Corollary~\ref{cor.setting1}, that is, the graph
$G$ is just the complete bipartite graph,
$G\equiv \calK_{n+1,n}=(A\cup B,E)$, with $A$ being the set of $n+1$
vertices, associated to the points in $X=\{x_a\}$, and $B$ the set of
$n$ vertices, associated to $Y=\{y_b\}$, we have 
$\calJ=\{ \{a\} \}_{a \in A}$, and we assume that the embedding is
generic.  Thus, $H_\calJ$ is a spanning tree on $G$, with all vertices
in $B$ of degree 2, and $\bar{H}_\calJ$ is a spanning tree on
$\bar{G}=\calK_{n+1}$. Both these graphs have a representation on
$\Omega$, where the $\gamma_e$'s are straight segments.

We will not need to specify $\Omega$ in what follows, as the datum of
$X$ and $Y$ is sufficient (and we can think that $\Omega$ is some
compact containing the convex hull of $X \cup Y$).

The goal of this section is to prove the following
Theorem~\ref{th:treenoncross}. Before doing this, we need to introduce
the main tools in the proof.

Let us introduce a symbol for minors.  Given a $n \times m$ matrix
$W$, and a set $I \subseteq [n]$ of rows, and $J \subseteq [m]$ of
columns, we denote by $W^{I|J}$ the $(n-|I|)\times(m-|J|)$ minor of
$W$ with the rows with indices in $I$, and the columns with indices in
$J$, removed.

Let us also introduce a symbol for a suitable alternating sum. For
$1 \leq k \leq n+1$, and $\pi \in \mathfrak{S}_n$
\be
W(k,\pi)
:=
\sum_{i=1}^n
\left(
W^{\{k\}|\varnothing}_{ii}
-
W^{\{k\}|\varnothing}_{i\,\pi(i)}
\right)
.
\ee
Now we are ready to define:
\begin{defn}
\label{def.validmat}
A matrix $W \in \bR(n+1,n)$ is 
\emph{valid}
% \emph{valid in the first sense} 
if, for all
$1 \leq k \leq n+1$, and all $\pi \in \mathfrak{S}_n$ distinct from
the identity,
\be
W(k,\pi)
<0
\ef.
\ee
\end{defn}
\noindent
The importance of this definition comes from the observation
\begin{prop}
\label{prop.crossifvalid}
Given an instance in the first bipartite standard setting, with
$n=|A|-1=|B|$, and costs $w_{a_i,b_j}$,
the tree $H_{\calJ}$ is the path
$(a_1, b_1, a_2, b_2,\ldots, a_{n}, b_{n}, a_{n+1})$ if and only if
% (minus) 
the matrix of costs $W$ (with $W_{ij}=w_{a_i,b_j}$) is valid.
%  in the first sense.
\end{prop}
\noindent
Indeed, the matrix $W$ is valid if and only if, for all $k$, the
optimal matching for the graph $G\setminx a_{k}$ is
\[
M_\star(G\setminx a_{k})=\{
\edge{a_1}{b_1},
\ldots,
\edge{a_{k-1}}{b_{k-1}},
\edge{a_{k+1}}{b_k},
\ldots,
\edge{a_{n+1}}{b_n}\}
\]
and (also from Remark~\ref{rmk.reconstr}) this is a necessary and
sufficient condition for $H_{\calJ}$ to be our path, so that (using
the shortcut $U=k$ for the color $U=\{a_k\}$) the edges 
$e_{j}=\edge{a_j}{b_j}$ have $\calC(e_j)=\{j+1,j+2,\ldots,n+1\}$, while the
edges $e'_{j}=\edge{a_j}{b_{j+1}}$ have $\calC(e'_j)=\{1,2,\ldots,j\}$.

Now consider $G$ as embedded on a two-dimensional manifold $\Omega$,
with $x_i=x_{a_i}$ and $y_i=y_{b_i}$. The proposition above implies
that there exist crossing configurations in $\Omega$ if and only if
there exist sets of points $(X,Y)$ such that $\gamma_{x_1,x_2}$ and
$\gamma_{x_n,x_{n+1}}$ do cross, and the associated matrix $W$ is
valid. If $\Omega$ is not simply-connected, the arcs of geodesics
$\gamma_{x_1,x_2}$ and $\gamma_{x_n,x_{n+1}}$ are those with the
topology of $\gamma_{x_1,y_1}\circ\gamma_{y_1,x_2}$ and
$\gamma_{x_n,y_{n}}\circ\gamma_{y_n,x_{n+1}}$.

We will need the following simple observation:
\begin{rmk}
\label{rmk.nuovaeq2n1n}
If $W \in \bR(n+1,n)$ is valid, then
$W_{21}+W_{nn}-W_{n1}-W_{2n}<0$.
\end{rmk}
\noindent
To see this, it suffices to take the sum of the two inequalities
$W(1,\pi_1)+W(n+1,\pi_2)$, with (in cycle notation)
\begin{align}
\pi_1&=(123\cdots n-1)(n)
\ef;
&
\pi_2&=(1)(n\,n-1\,n-2\cdots 2)
\ef.
\end{align}
In the diagrammatic notation of Appendix~\ref{app.altproofValid1st2nd}:
\be
\label{eq.innuovaeq}
\setlength{\unitlength}{10pt}
\thicklines
\raisebox{-27pt}{\begin{picture}(5,6)
  \linethickness{0.5pt}
  \multiput(0,0)(1,0){6}{\line(0,1){6}} 
  \multiput(0,0)(0,1){7}{\line(1,0){5}}
  \linethickness{1.5pt}
  \put(0,0){\line(0,1){6}}
  \put(0,0){\line(1,0){5}}
  \put(0,6){\line(1,0){5}}
  \put(5,0){\line(0,1){6}}
\put(0.5,4.5){\gogre{\line(1,0){4}}}
\put(0.5,1.5){\gogre{\line(1,0){4}}}
\put(0.5,1.5){\gogre{\line(0,1){3}}}
\put(4.5,1.5){\gogre{\line(0,1){3}}}
\dotbr{4}{0}{1}
\dotbr{0}{4}{4}
\end{picture}}
\;=\;
\raisebox{-27pt}{\begin{picture}(5,6)
\put(0,5){\goyel{\rule{50pt}{10pt}}}
  \linethickness{0.5pt}
  \multiput(0,0)(1,0){6}{\line(0,1){6}} 
  \multiput(0,0)(0,1){7}{\line(1,0){5}}
  \linethickness{1.5pt}
  \put(0,0){\line(0,1){6}}
  \put(0,0){\line(1,0){5}}
  \put(0,6){\line(1,0){5}}
  \put(5,0){\line(0,1){6}}
  \put(0,5){\line(1,0){5}}
  \put(0,6){\line(1,0){5}}
  \put(0,0){\line(0,1){6}}
%
% \put(0.5,6.5){\gogre{\line(1,0){5}}}
\put(0.5,1.5){\gogre{\line(1,0){3}}}
\put(0.5,1.5){\gogre{\line(0,1){3}}}
% \put(5.5,0.5){\gogre{\line(0,1){6}}}
\put(2.5,2.5){\gogre{\line(1,0){1}\line(0,-1){1}}}
\put(1.5,3.5){\gogre{\line(1,0){1}\line(0,-1){1}}}
\put(0.5,4.5){\gogre{\line(1,0){1}\line(0,-1){1}}}
\dotbr{0}{1}{4}
\dotbr{1}{2}{3}
\dotbr{2}{3}{2}
\dotbr{3}{0}{1}
\dotx{4}{0}
\end{picture}}
\;+\;
\raisebox{-27pt}{\begin{picture}(5,6)
\put(0,0){\goyel{\rule{50pt}{10pt}}}
  \linethickness{0.5pt}
  \multiput(0,0)(1,0){6}{\line(0,1){6}} 
  \multiput(0,0)(0,1){7}{\line(1,0){5}}
  \linethickness{1.5pt}
  \put(0,0){\line(0,1){6}}
  \put(0,0){\line(1,0){5}}
  \put(0,6){\line(1,0){5}}
  \put(5,0){\line(0,1){6}}
  \put(0,0){\line(1,0){5}}
  \put(0,1){\line(1,0){5}}
  \put(5,0){\line(0,1){6}}
\put(1.5,4.5){\gogre{\line(1,0){3}}}
\put(4.5,1.5){\gogre{\line(0,1){3}}}
\put(3.5,1.5){\gogre{\line(1,0){1}}}
\put(2.5,2.5){\gogre{\line(1,0){1}\line(0,-1){1}}}
\put(1.5,3.5){\gogre{\line(1,0){1}\line(0,-1){1}}}
\put(1.5,4.5){\gogre{\line(0,-1){1}}}
\dotbr{1}{4}{4}
\dotbr{2}{1}{3}
\dotbr{3}{2}{2}
\dotbr{4}{3}{1}
\dotx{0}{5}
\end{picture}}
\ee
%
% =======================================================
%
We can illustrate these inequalities through an
alternative graphical representation: we will draw a hypothetical
counterexample, that is, a path $H_\calJ$ in which the extremal edges
of its projection $\bar{H}_\calJ$ do cross, and describe an inequality
through the following formalism:
\begin{itemize}
\item one $A$-vertex, marked in green, corresponds to the row removed from $W$;
\item some edges of $\calK_{\ell+1,\ell} \supseteq H_\calJ$ are marked in orange, blue
  or green, but blue and green edges can occupy only edges of
  $H_\calJ$;
\item blue and orange edges form an alternating cycle;
\item blue and green edges, and the green vertex, form a monomer-dimer
  covering of $H_\calJ$.
\end{itemize}
Then, the sum of the costs of the blue edges, minus the sum of the costs
of the orange edges, must be non-positive. The green components are only
used to certify in a visual way that the blue edges can be completed
to a monomer-dimer configuration of the form described above.

For example, here on a path with $\ell=6$, we have:
\be
% gin{center}
\label{eq.3478Imgs1}
\if\faifig1
\makebox[0pt][c]{
% PRIMO (finito)

}
\else [...TikZ\ code...] \fi
\ee
% \end{center}
%
where the resulting graph on the RHS indeed reads
$\langle x_2-x_6,y_1-y_6\rangle<0$. Note that the diagram on the RHS
can \emph{not} be completed to a monomer-dimer configuration on
$H_{\calJ}$ with a single monomer (that is, our inequality does not
result from a single constraint of our family). Nonetheless, in
light of what is shown above, since it is a positive linear combination of
elementary constraints, it provides a legitimate inequality.

With these properties at hand, we are ready to prove the main theorem
of this section. When $p=2$ we already know that $H_{\calJ}$ is not
guaranteed to be non-crossing. Nonetheless, we will prove the
surprising fact:
\begin{thm}
\label{th:treenoncross}
When $\Omega$ is a portion of $\bR^2$, and $p=2$, in the first
bipartite standard setting the embedding of the tree $\bar{H}_{\calJ}$
is non-crossing.
\end{thm}
\begin{proof}
In light of Remark~\ref{rmk.crossifpath}, it is sufficient to prove
that there does not exist a pair $(X,Y)$, with $|X|=\ell+1$ and
$|Y|=\ell$, such that the tree $H_{\calJ}$ is the path
$(x_1,y_1,x_2,y_2,\ldots,x_\ell,y_\ell,x_{\ell+1})$, and the segments
$[x_1,x_2]$ and $[x_\ell,x_{\ell+1}]$ do cross.  Also, in light of
Proposition\ \ref{prop.crossifvalid}, we shall prove that there does
not exist a pair $(X,Y)$ such that the segments $[x_1,x_2]$ and
$[x_\ell,x_{\ell+1}]$ do cross, and the matrix $W=(W_{ij})$, with
$W_{ij}=-\langle x_{i}, y_{j}\rangle$, is valid.

Note that we can assume that $\ell \geq 3$, as for $\ell=2$ it is
clearly seen that the segments $[x_1,x_2]$ and $[x_2,x_3]$ cannot
cross at an internal point. Introduce the shorthands
$\xi_{a,b}\coloneqq x_a-x_b$ and $\eta_{a,b}\coloneqq y_a-y_b$.
Call $R$ the rotation counterclockwise by
90 degrees, that is $R(x,y)=(-y,x)$.
Also, for $k \geq 3$, and $u_j$'s
some non-zero vectors in the plane, let us adopt
the symbol $\prec$ in expressions as
\[
u_1 \prec u_2 \prec \cdots \prec u_k
\]
to state that the list of vectors is
in counterclockwise cyclic order, i.e.
\[
\arg(u_1)<\arg(u_2)<\ldots<\arg(u_k)<\arg(u_1)+2\pi
\ef.
\]
% Up to translation, we can assume that the segments $[x_1,x_2]$
% and $[x_{\ell},x_{\ell+1}]$ cross at the origin. 
A simple application of this notation is the equivalence of the conditions:
\be
\langle u,v \rangle > 0
\qquad \Longleftrightarrow \qquad
R^{-1} v \prec u \prec R v
\ef.
\ee
The crossing condition states that the polygon with vertices
$(x_1,x_{\ell+1},x_2,x_{\ell})$ is a convex quadrilater (so that the
diagonals do cross), which means that one of the following must be
true (depending if the sequence of vertices above is in clockwise or
counterclockwise order)
\begin{subequations}
\label{eqs.4357856187main}
\begin{gather}
\xi_{1,\ell+1} \prec \xi_{\ell+1,2} \prec \xi_{2,\ell} \prec \xi_{\ell,1}
\ef;
\\
%% \xi_{1,\ell+1} \succ \xi_{\ell+1,2} \succ \xi_{2,\ell} \succ \xi_{\ell,1}
\xi_{\ell,1} \prec \xi_{2,\ell}  \prec \xi_{\ell+1,2}  \prec \xi_{1,\ell+1}  
% \xi_{1,\ell} \prec \xi_{\ell,2} \prec \xi_{2,\ell+1} \prec \xi_{\ell+1,1}
\ef.
\end{gather}
\end{subequations}
The validity conditions
% (in the first sense) 
take a specially simple
form when the stability is checked against a transposition.
In our list of $(k,\pi)$'s, we have $3\binom{\ell}{2}$ transpositions,
with rows and columns $\big((i,j),(i,j)\big)$, $\big((i,j+1),(i,j)\big)$, 
and $\big((i+1,j+1),(i,j)\big)$. If the transposition involves the
rows and columns $\big((i,j),(k,h)\big)$, the corresponding inequality
takes the simple form
\be
\langle \xi_{i,j},\eta_{k,h} \rangle >0
\ef.
\ee
We have three such transpositions involving the variables
$\{x_1,x_2,x_{\ell},x_{\ell+1}\}$
and $\{y_1,y_{\ell}\}$, namely those with
$\big((i,j),(k,h)\big) \in \{
\big((1,\ell),(1,\ell)\big), 
\big((1,\ell+1),(1,\ell)\big), 
\big((2,\ell+1),(1,\ell)\big)\}$. 
In the diagrammatic notation of
Appendix~\ref{app.altproofValid1st2nd}, these are
\begin{align}
\label{eq.876987627}
% 1
\setlength{\unitlength}{10pt}
\raisebox{-27pt}{
}
&
\end{align}
or, in our alternative graphical representation,
\be
% gin{center}
\label{eq.3478Imgs1b}
\if\faifig1
\makebox[0pt][c]{
% UNO
%
% ----------------
}
\else [...TikZ\ code...] \fi
\ee
%% We can illustrate these inequalities through an alternative graphical
%% representation: we will draw a hypothetical counterexample, that is, a
%% path $H_\calJ$ in which the extremal edges of its projection
%% $\bar{H}_\calJ$ do cross, and describe an inequality through the
%% following formalism:
%% \begin{itemize}
%% \item one $A$-vertex is marked in green, and corresponds to the row
%%   removed from $W$;
%% \item some edges of $\calK_{\ell+1,\ell} \supseteq H_\calJ$ are marked in orange, blue
%%   or green, but blue and green edges can occupy only edges of
%%   $H_\calJ$;
%% \item blue and orange edges form an alternating cycle;
%% \item blue and green edges, and the green vertex, form a monomer-dimer
%%   covering of $H_\calJ$.
%% \end{itemize}
%% Then, the sum of the costs of blue edges, minus the sum of the costs
%% of the orange edges, must be negative.\footnote{It must be
%%   non-positive in general, and strictly negative under our hypothesis
%%   of generic weights.} The green components are only used to certify
%% in a visual way that the blue edges can be completed to a
%% monomer-dimer configuration.
Furthermore, we have the expression discussed in
Remark~\ref{rmk.nuovaeq2n1n}, that, although being a combination of
two inequalities for $W(k,\pi)$'s, takes the form of a transposition,
namely $\big((2,\ell),(1,\ell)\big)$.
So overall we get the four inequalities
\begin{subequations}
\label{eqs.4357856187bmain}
\begin{align}
\langle
\xi_{1,\ell},
\eta_{1,\ell}
\rangle
&>0
\ef;
&
\langle
\xi_{1,\ell+1},
\eta_{1,\ell}
\rangle
&>0
\ef;
\\
\langle
\xi_{2,\ell},
\eta_{1,\ell}
\rangle
&>0
\ef;
&
\langle
\xi_{2,\ell+1},
\eta_{1,\ell}
\rangle
&>0
\ef;
\end{align}
\end{subequations}
which can also be stated as
\begin{subequations}
\label{eqs.4357856187cmain}
\begin{align}
R^{-1} \eta_{1,\ell} &\prec
\xi_{1,\ell+1} \prec 
R \eta_{1,\ell}
\ef;
&
R \eta_{1,\ell} &\prec
\xi_{\ell+1,2} \prec 
R^{-1} \eta_{1,\ell}
\ef;
\\
R^{-1} \eta_{1,\ell} &\prec
\xi_{2,\ell} \prec 
R \eta_{1,\ell}
\ef;
&
R \eta_{1,\ell} &\prec
\xi_{\ell,1} \prec
R^{-1} \eta_{1,\ell}
\ef.
\end{align}
\end{subequations}
These inequalities are not compatible with the two possibilities in
(\ref{eqs.4357856187main}),\footnote{More generally, the statements
$c_- \prec a_i \prec c_+$ for all $i$, 
$c_+ \prec b_j \prec c_-$ for all $j$, 
and $a_{i_1} \prec b_{j_1} \prec a_{i_2} \prec b_{j_2}$ for some
$(i_1,i_2,j_1,j_2)$ are incompatible.}
so we have found a contradiction. This
completes the proof.

Another way of concluding is by observing that the function
$\phi(x)=\langle x,\eta_{1,\ell} \rangle$ defines an order over the
four points $\{x_1,x_2,x_\ell,x_{\ell+1}\}$, and that 
the inequalities (\ref{eqs.4357856187bmain}) mean that 
$\{x_1,x_2\}$ are `smaller' than $\{x_\ell,x_{\ell+1}\}$, thus there
must exist a line
% hyperplane 
orthogonal to $\eta_{1,\ell}$ that separates them, so that the
corresponding segments cannot cross.
\end{proof}
\noindent
Two alternative proofs, of which the second one is more long but
possibly more transparent at the level of the underlying mechanisms,
are given in Appendix~\ref{app.altproofValid1st2nd}.

%-------------------------------------------------------
\subsection{The case of the cylinder and the torus}
%% \texorpdfstring{The case of the first bipartite standard setting,
%%     $p=2$ and $\Omega$ a cylinder or a torus}{The case of p=2 and the
%%     cylinder or the torus}}
\label{ssec.noncrossToro}
%-------------------------------------------------------
%
Now, let us consider the case in which $\Omega$ is a (flat) cylinder,
with periodicity $\w$ (that is, $\Omega=\bR^2/\sim$, with 
$x\sim x+\w$), or a (flat) torus, with periodicities $(\w_1,\w_2)$
(that is, $\Omega=\bR^2/\sim$, with $x\sim x+\w_1$ and $x\sim x+\w_2$).

We shall prove an analogue of Theorem~\ref{th:treenoncross}:
\begin{thm}
\label{th:treenoncrossToro}
When $\Omega$ is a cylinder or a torus, and $p=2$, in the first
bipartite standard setting the embedding of the tree $\bar{H}_{\calJ}$ is
non-crossing.
\end{thm}
\begin{proof}
The proof uses many ideas and concepts already appearing in the proof
of Theorem \ref{th:treenoncross}, that will not be repeated here.  Let
us assume by contradiction that $H_\calJ$ is the path
$P=(x_1,y_1,x_2,\ldots,y_{\ell},x_{\ell+1})$, and that the segments
$[x_1,x_2]$ and $[x_{\ell},x_{\ell+1}]$ (with the same topology of the
paths $(x_1,y_1,x_2)$ and $(x_{\ell},y_{\ell},x_{\ell+1})$,
respectively, where $[x_i,y_i]$ and $[x_{i+1},y_i]$ are the shortest
segments on $\Omega$ with these endpoints) do cross.  Call $x_0$ the
intersection of these two segments. Thus the path
$C=(x_0,x_2,x_3,\ldots,x_{\ell},x_0)$, where the steps $[x_i,x_{i+1}]$
with $2\leq i \leq \ell-1$ are constructed with the same prescription
above, is an oriented cycle on $\Omega$, and it has a winding number
$k \in \bZ$, on a cylinder (that is, if the segments of the path are
followed by continuity, on $\bR^2$, then $C$ has endpoints separated
by the vector $k \w$), or a winding number $(k_1,k_2) \in \bZ^2$, on a
torus (that is, if the segments of the path are followed by
continuity, on $\bR^2$, then $C$ has endpoints separated by the vector
$k_1 \w_1+k_2 \w_2$). 
% In principle, we could 
Let us discuss more in detail the case of the torus, which is more
general (we can imagine the cylinder as a torus with one periodicity
vector long enough so that the shortest geodesics never use it).

Call $\hat{x}_i$ and $\hat{y}_j$ a set of coordinates in $\bR^2$,
equivalent to the $x_i$'s and $y_j$'s (i.e., on the torus,
$\hat{x}_i=x_i+h_1(i)\, \w_1+h_2(i)\, \w_2$ for some $h_1(i)$ and
$h_2(i)$ in $\bZ$), chosen in the following way:
\begin{itemize}
\item $\hat{x}_1$ is the point in $\bR^2$ with coordinates equal to
  those of $x_1$;
% , equivalent to $x_1$, nearest  to the origin;
\item $\hat{y}_1$ is the point in $\bR^2$, equivalent to $y_1$, such
  that the segment $[\hat{x}_1,\hat{y}_1]$ has the same length as
  $\gamma_{x_1,y_1}$;
\item $\hat{x}_2$ is the point in $\bR^2$, equivalent to $x_2$, such
  that the segment $[\hat{x}_2,\hat{y}_1]$ has the same length as
  $\gamma_{x_2,y_1}$, and so on.
\end{itemize}
This choice correspond to require that the path $\hat{P}$ on $\bR^2$ 
(interpreted as the natural covering $\Omega \times \bZ^{\times 2}$)
is the `lifting' of the path $P$ on $\Omega$, and in particular it
has segments of the same lengths of those of $P$. Furthermore, 
the statement above on the cycle $C$ means that, for the set of
coordinates $\hat{x}_i$ and $\hat{y}_j$, the values 
$(k_1,k_2) \in \bZ^{\times 2}$ are those such that there exist 
$t,s\in \;]0,1[\,$ with
\be
x_0 = (1-t) \hat{x}_1 +t \hat{x}_2 = s \hat{x}_\ell +(1-s) \hat{x}_{\ell+1}
-(k_1 \w_1+k_2 \w_2)
\ef.
\ee
% (we will call $\hat{x}_0$ the position above).
Now, suppose that $(k_1,k_2)=(0,0)$, that is, the cycle $C$ is
contractible on $\Omega$. Then we can conclude from Theorem
\ref{th:treenoncross}. Indeed, let us compare the costs
$w_{ij}=|\gamma_{x_i,y_j}|^2$ for the instance on $\Omega$, with the
costs $\hat{w}_{ij}=|\gamma_{\hat{x}_i,\hat{y}_j}|^2$ for the instance
on (a portion of) $\bR^2$, with coordinates $\hat{x}_i$ and
$\hat{y}_j$. We have that $w_{ij}\leq \hat{w}_{ij}$ in general, and
that $w_{ij}= \hat{w}_{ij}$ if $i=j$ or $i=j+1$, that is, if
$[x_i,y_j]$ is one of the segments appearing in the path $P$.  This
implies that the matrix $W$ for the instance on
$\Omega$  is valid only if the matrix
$\hat{W}$ for the instance on $\bR^2$ is valid, but we know from
Theorem \ref{th:treenoncross} that there are no valid matrices of
costs $\hat{W}$
% $W_{ij}=\hat{W}_{ij}-|\hat{x}_i|^2-|\hat{y}_j|^2$ 
such that $[\hat{x}_1,\hat{x}_2]$ and
$[\hat{x}_{\ell},\hat{x}_{\ell+1}]$ do cross. In other words, it is
harder to find a counterexample, consisting of a contractible path, on
a torus, than it is to find a counterexample \emph{tout court} on the
plane (where the very notion of a contractible path is trivial, and we
have already proven that the latter does not exist).
So we can assume that $(k_1,k_2)\neq (0,0)$.
Now, if 
$\mathrm{gcd}(k_1,k_2)\neq 1$ (or, in the cylinder, $k \neq \pm 1$),
the cycle $C$ has a self-intersection, thus there exists a value
$\ell'<\ell$ such that the restriction of the instance to
$\{x_1,\ldots,x_{\ell'+1}\}$ and $\{y_1,\ldots,y_{\ell'}\}$ is
crossing, so we can exclude this case by requiring that the value of
$\ell$ is minimal for having a crossing. Thus, up to reversing the
path, or performing a reflection of the domain, on the cylinder we can
assume that $k=1$. Similarly, on the torus, up to applying a change of
basis in $\mathrm{SL}(2,\bZ)$, we can assume that $(k_1,k_2)=(1,0)$.

Now, for the case of the torus with basis $(\w_1,\w_2)$, we can choose
a set of coordinates $\tilde{x}_i$ and $\tilde{y}_j$ on the cylinder
with periodicity $\w_1$, similarly to what we have done above for the
plane (that is, we lift $P$ in the torus to a $\tilde{P}$ on the
cylinder, that is the $\Omega \times Z$ cover obtained from the
unwrapping of the periodicity under translations by $\w_2$). Then we
have that the costs $w_{ij}=|\gamma_{x_i,y_j}|^2$ for the instance on
the torus, and the costs
$\tilde{w}_{ij}=|\gamma_{\tilde{x}_i,\tilde{y}_j}|^2$ for the instance
on the cylinder, satisfy $w_{ij}\leq \tilde{w}_{ij}$ in general, and
$w_{ij}= \tilde{w}_{ij}$ if $i=j$ or $i=j+1$, that is, if $[x_i,y_j]$
is one of the segments appearing in the path $P$.
% that is, if $w_{ij}$ is one of the weights appearing in the path. 
This implies that the matrix $W$ of costs for the instance on the
torus with $(k_1,k_2)=(1,0)$ is valid only if the matrix $\tilde{W}$
of costs for the instance on the cylinder with $k=1$ is valid. As a
result, if we can prove that there are no crossing paths with winding
number $k=1$ on the cylinder, and a valid matrix of costs, we have
implicitly also solved the case of the torus.

Thus our aim is to prove that, in $\bR^2$, there does not exist an
instance with $X=\{\hat{x}_1,\ldots,\hat{x}_{\ell+1}\}$,
$Y=\{\hat{y}_1,\ldots,\hat{y}_{\ell}\}$, such that
there exist $t,s\in \;]0,1[\,$ and
\be
\label{eq.8964624675}
(1-t) \hat{x}_1 +t \hat{x}_2 = s \hat{x}_\ell +(1-s) \hat{x}_{\ell+1}
-\w
% = 0
\ee
(or, in other words, the segments $[\hat{x}_1,\hat{x}_2]$ and
$[\hat{x}_{\ell}-\w,\hat{x}_{\ell+1}-\w]$ cross),
% at the origin), 
and the matrix of costs $W$ for the associated instance on the
cylinder (with periodicity $\w$ and positions $x_i$'s and $y_j$'s)
% $x_i=(x_i^{(1)},x_i^{(2)})=(\mathrm{mod}(\hat{x}_i^{(1)},1),\hat{x}_i^{(2)})$
% and similarly for $y_j$, 
is valid.

Introduce the $(\ell+1)\times \ell$ matrix $K_{ij}=\argmin_k
|\hat{x}_i-\hat{y}_j-k \w|^2$.  Note that $|K_{i+1\,j}-K_{ij}|\leq 1$
and $|K_{i\,j+1}-K_{ij}|\leq 1$ (but we do not use this fact), and
that, by construction of the $\hat{x}_i$'s and $\hat{y}_j$'s,
$K_{ii}=K_{i+1\,i}=0$, that is
\be
K=
\left(
\begin{array}{rcccl}
%{p[10pt]p[10pt]p[10pt]p[10pt]p[10pt]}
0 & \ast & \ast & \ast & \cdots \\
0 &    0 & \ast & \ast & \cdots \\
\ast & 0 &    0 & \ast & \cdots \\
&& \vdots &&\\
\cdots & \ast & \ast & 0    & 0 \\
\cdots & \ast & \ast & \ast & 0 
\end{array}
\right)
\ee
We have
\begin{align}
\hat{W}_{ij}
&=
|\hat{x}_i-\hat{y}_j|^2
\geq
|\hat{x}_i-\hat{y}_j-K_{ij} \w|^2
= W_{ij}
\ef,
\end{align}
and
$\hat{W}_{ii}-W_{ii}=\hat{W}_{i+1\,i}-W_{i+1\,i}=0$.

Furthermore, let $\tilde{K}$ be the matrix
\be
\tilde{K}=
\left(
\begin{array}{rcccl}
%{p[10pt]p[10pt]p[10pt]p[10pt]p[10pt]}
 0 &  1 &  1 &  1 & 1 \\
 0 &  0 &  1 &  1 & 1 \\
-1 &  0 &  0 &  1 & 1 \\
-1 & -1 &  0 &  0 & 1 \\
-1 & -1 & -1 &  0 & 0 \\
-1 & -1 & -1 & -1 & 0 
\end{array}
\right)
\ee
(or, at our purposes, any other matrix with
$\tilde{K}_{ii}=\tilde{K}_{i+1\,i}=0$, 
$\tilde{K}_{1\,\ell}=\tilde{K}_{2\,\ell}=1$
and
$\tilde{K}_{\ell\,1}=\tilde{K}_{\ell+1\,1}=-1$),
and define $\tilde{W}$ as
\begin{align}
\tilde{W}_{ij}
&=
|\hat{x}_i-\hat{y}_j-\tilde{K}_{ij} \w|^2
\geq
|\hat{x}_i-\hat{y}_j-K_{ij} \w|^2
= W_{ij}
\ef.
\end{align}
Note that, again, 
$\tilde{W}_{ii}-W_{ii}=\tilde{W}_{i+1\,i}-W_{i+1\,i}=0$.
These constructions imply in particular that $W$ is valid only if
$\tilde{W}$ is valid, so at our purposes it is enough to show that 
$\tilde{W}$ is never valid.

So we have reached an equivalent formulation of the problem, purely in
the plane.  Dropping $(\hat{\cdot})$'s for simplicity of notation,
what we will actually prove is that, in $\bR^2$, there does not exist
an instance with $X=\{{x}_1,\ldots,{x}_{\ell+1}\}$,
$Y=\{{y}_1,\ldots,{y}_{\ell}\}$,
%% $X=\{\hat{x}_1,\ldots,\hat{x}_{\ell+1}\}$,
%% $Y=\{\hat{y}_1,\ldots,\hat{y}_{\ell}\}$, 
and a non-zero vector $\w$, such that
there exist $t,s\in \;]0,1[\,$ and
\be
\label{eq.8964624675toro}
(1-t) {x}_1 +t {x}_2 = s {x}_\ell +(1-s) {x}_{\ell+1}
-\w
% = 0
\ee
(or, in other words, the segments $[{x}_1,{x}_2]$ and
$[{x}_{\ell}-\w,{x}_{\ell+1}-\w]$ cross), and the matrix $\tilde{W}$,
obtained with the choice of shifts $\tilde{K}$, is valid.

The case $\ell=1$ is clearly excluded, and the (simpler) case $\ell=2$
is treated separately in Appendix~\ref{app.casiL3}.  So we will assume
that $\ell \geq 3$.

For $(i,j)\in \{(1,\ell), (1,\ell+1),(2,\ell), (2,\ell+1)\}$
% \{1,2\}\times\{\ell,\ell+1\}$ 
define
$\tilde{\xi}_{ij}=\xi_{ij}+\w={x}_i-{x}_j+\w$, and for 
$(i,j)\in \{(1,\ell-1),(1,\ell),(2,\ell)\}$ define
$\tilde{\eta}_{ij}=\eta_{ij}+\w={y}_i-{y}_j+\w$.
Then the validity of $\tilde{W}$ requires in particular that the four
inequalities associated to the transpositions involving the rows and
columns $\big((i,j),(k,h)\big)$, with $k=1$, $h=\ell$, $i\in\{1,2\}$
and $j\in\{\ell,\ell+1\}$ (that are the four transpositions appearing
in the proof of Theorem~\ref{th:treenoncross}). What turns out is that
the corresponding inequalities coincide with those in equations
(\ref{eqs.4357856187bmain}), with $\xi$'s and $\eta$'s replaced by
$\tilde{\xi}$'s and $\tilde{\eta}$'s. Remark that this fact is true in
particular because, crucially, the values $\tilde{K}_{12}$ and
$\tilde{K}_{\ell+1\,\ell-1}$ do not appear in the evaluation of these
inequalities.  In general the condition coming from the transposition
$\{(i,k),(j,h)\} \to \{(i,h),(j,k)\}$ has the form
\be
\label{eq.4976275}
|x_i-y_k-\tilde{K}_{ik} \w|^2
+
|x_j-y_h-\tilde{K}_{jh} \w|^2
-
|x_i-y_h-\tilde{K}_{ih} \w|^2
-
|x_j-y_k-\tilde{K}_{jk} \w|^2
<0
\ef.
\ee
Now, in our four cases, we always have
\be
\begin{pmatrix}
\tilde{K}_{ik} & \tilde{K}_{ih} \\
\tilde{K}_{jk} & \tilde{K}_{jh} 
\end{pmatrix}
=
\begin{pmatrix}
0&1\\
-1&0
\end{pmatrix}
\ee
that is a matrix of the form $M_{ij}=a_i+b_j+c$,
so that we can write (\ref{eq.4976275}) as
\be
\begin{split}
0
&>
|x_i-y_k|^2
+
|x_j-y_h|^2
-
|x_i-y_h+\w|^2
-
|x_j-y_k-\w|^2
\\
&=
|x_i-y_k|^2
+
|(x_j-\w)-(y_h-\w)|^2
-
|x_i-(y_h-\w)|^2
-
|(x_j-\w)-y_k|^2
\\
&=
-2
\langle 
x_i-x_j+\w,
y_k-y_h+\w
\rangle
=
-2
\langle 
\tilde{\xi}_{ij},\tilde{\eta}_{kh}
\rangle
\ef,
\end{split}
\ee
as claimed.
Also, while in the proof of Theorem \ref{th:treenoncross}  we
ask that the segments $[x_1,x_2]$ and $[x_{\ell},x_{\ell+1}]$ cross at
the origin, which can be stated as the fact that there exist 
$t,s\in \;]0,1[\,$ such that
\be
\label{eq.8964624675ggg}
(1-t) x_1 +t x_2 = s x_\ell +(1-s) x_{\ell+1}
\ee
that in turn can be restated into the existence of $t,s\in \;]0,1[\,$ such that
\be
\label{eq.54429765}
(1-t+s)\,\xi_{1,\ell+1} 
+t \,\xi_{2,\ell+1} 
- s \,\xi_{1,\ell} = 0
\ef,
\ee
the crossing condition in this setting, i.e.\ the equation
(\ref{eq.8964624675toro}) above, can be restated~as
\begin{equation*}
\begin{split}
0
&=
(1-t) \hat{x}_1 +t \hat{x}_2 - s \hat{x}_\ell -(1-s) \hat{x}_{\ell+1}
+\w
\\
&=
(\hat{x}_1-\hat{x}_{\ell+1}+\w)
-t ((\hat{x}_1-\hat{x}_{\ell+1})-(\hat{x}_2-\hat{x}_{\ell+1}))
-s ((\hat{x}_1-\hat{x}_\ell)-(\hat{x}_1-\hat{x}_{\ell+1}))
\\
&=
(1-t+s)\,\tilde{\xi}_{1,\ell+1}
+t\,\tilde{\xi}_{2,\ell+1}
-s\,\tilde{\xi}_{1,\ell}
\ef,
\end{split}
\end{equation*}
that is equation (\ref{eq.54429765}) with the $\xi$'s replaced by the
$\tilde{\xi}$'s. In other words, both the requirements on cyclic
orderings (\ref{eqs.4357856187main}) and (\ref{eqs.4357856187cmain})
hold, up to adopting the shifted expressions. So, this produces a
contradiction by the same mechanism that we exploited in the previous
proof.
\end{proof}

%-------------------------------------------------------
\subsection{The case of the M\"obius Strip and the Klein Bottle}
%-------------------------------------------------------
%
As well known, the Teichm\"uller space of both the M\"obius Strip and
the Klein Bottle is the positive half-line, encoding the aspect ratio
$\rho=|\w_2|/|\w_1|$ of a rectangle with base $\w_1$ and height
$\w_2$, i.e.\ a fundamental region of the two varieties, where the
vertical sides are open (in the M\"obius Strip) or identified (in the
Klein Bottle), while the horizontal sides are identified up to a
reflection. For example, a typical crossing configuration on the Klein
Bottle (if it existed) would be as in
Figure~\ref{fig.crossKleinBot}.

\begin{figure}[t]
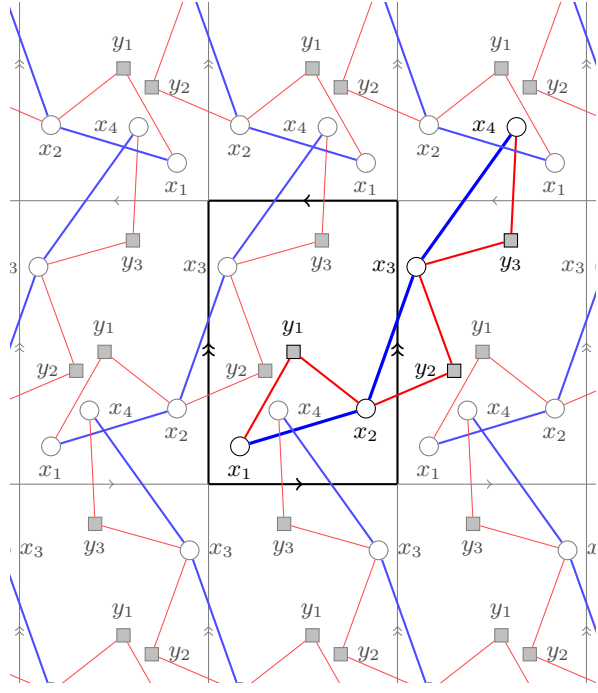

\[
\if\faifig1  

\else [...TikZ\ code...] \fi
\]
\caption{\label{fig.crossKleinBot}%
Example of potential crossing configuration on the Klein Bottle. Here
$\rho=|\w_2|/|\w_1|=1.5$.}
% and $(k_1,k_2)=(1,1)$.}
\end{figure}

We will provide analogues of Theorems \ref{th:treenoncross} and
\ref{th:treenoncrossToro} for these surfaces.  The idea of proof is
based on the following remark.
\begin{lem}
\label{lem.2cover}
Let $\Omega^{\rm c}$ be a manifold as above. Assume that $\Omega^{\rm c}$ is a
2-covering of a manifold $\Omega$, and that $g:\Omega^{\rm c} \to \Omega^{\rm c}$ is
an isomorphism that exchanges the two copies of $\Omega$. Let
$(X,Y)$ be a generic-weight configuration on $\Omega$
%  such that $H_\calJ$ is the path $P=(x_1,y_1,x_2,\ldots,y_{\ell},x_{\ell+1})$,
and call
$x'_i=g\circ x_i$, and $y'_i=g\circ y_i$.  If $(X,Y)$ is a crossing
configuration on $\Omega$, then there exist crossing
configurations in $\Omega^{\rm c}$.
\end{lem}
\begin{proof}
We construct a crossing configuration $(\tilde{X},\tilde{Y})$ on
$\Omega^{\rm c}$ from the configuration $(X,Y)$.
% Up to taking a smaller configuration, we can assume that
Say that $|Y|=|X|-1=\ell$.
We choose
$\tilde{X}=\{x_1,\ldots,x_{\ell+1},x'_1,\ldots,x'_{\ell+1}\}$ and
$\tilde{Y}=\{y_1,\ldots,y_{\ell},y'_1,\ldots,y'_{\ell},\tilde{y}\}$,
where the position $\tilde{y}$ is still to be determined, while, for
the other positions, we choose the copy of $x_1$ in $\Omega^{\rm c}$
arbitrarily, and the copy of the other points so that
$x_i$ and $x'_i$ are in different copies, as well as $y_j$ and $y'_j$,
and, for $e=\edge{x_i}{y_j}\in H_\calJ$,
that the geodesic $\gamma_{x_i,y_j}$
is shorter than $\gamma_{x_i,y'_j}$.
%% The fact that $(X,Y)$ is a crossing configuration on $\Omega$
%% implies that, on $\Omega^{\rm c}$, 
%% $\gamma_{x_1,x_2}$ intersects $\gamma_{x'_{\ell},x'_{\ell+1}}$ and 
%% $\gamma_{x'_1,x'_2}$ intersects $\gamma_{x_{\ell},x_{\ell+1}}$.

Call $\lam_i$ and $\mu_j$ the canonical simultaneous gauge parameters
for the pair $(X,Y)$. As $g$ is a isomorphism, they are also the gauge
parameters for the pair $(X',Y')$, and, as have chosen the copies of
$x_i$'s and $y_j$'s in the natural way, it must be that the quantities
$W_{i,\ell+j}-\lam_i-\mu_j$, associated to the cost of the geodesics
$\gamma_{x_i,y'_j}$, and $W_{\ell+1+i,j}-\lam_i-\mu_j$, associated to
the cost of the geodesics $\gamma_{x'_i,y_j}$, are positive.

Call $i^* \in \{1,\ldots,\ell+1\}$ the index realising the maximum of
the $\lam_i$'s. Let $c(t):[0,1]\to \Omega^{\rm c}$ be a continuous curve
connecting $x_{i^*}$ to $x'_{i^*}$.  Consider the last column of the
matrix $\tilde{W}$, when we set $\tilde{y}=c(t)$, after the gauge
transformation, namely
\be
\begin{split}
&\big( \tilde{W}_{i,2\ell+1}-\lam_i
\big)_{1 \leq i \leq 2\ell+2}
=\big(
f(|\gamma_{x_1,c(t)}|)-\lam_1,
\ldots,
\\
& \qquad
f(|\gamma_{x_{\ell+1},c(t)}|)-\lam_{\ell+1},
f(|\gamma_{x'_1,c(t)}|)-\lam_1,
\ldots,
f(|\gamma_{x'_{\ell+1},c(t)}|)-\lam_{\ell+1}
\big)
%% =\big(
%% f(\|x_1-\gamma(t)\|),
%% \ldots,
%% f(\|x_{\ell+1}-\gamma(t)\|),
%% f(\|x'_1-\gamma(t)\|),
%% \ldots,
%% f(\|x'_{\ell+1}-\gamma(t)\|)
%% \big)
\ef.
\end{split}
\ee
Note that the interval $[0,1]$ is partitioned into a collection of
intervals $A_1$,\ldots,$A_k$, where
$\argmin(\tilde{W}_{i,2\ell+1})\subseteq \{1,\ldots,\ell+1\}$, and a
collection of intervals $B_1$,\ldots,$B_k$, where
$\argmin(\tilde{W}_{i,2\ell+1})\subseteq \{\ell+2,\ldots,2\ell+2\}$.
From our choice of curve, we know that $t=0$ is the left
endpoint of $A_1$, because $\argmin(\tilde{W}_{i,2\ell+1})=\{i^*\}$ in
this case, and $t=1$ is the right endpoint of $B_k$, because
$\argmin(\tilde{W}_{i,2\ell+1})=\{\ell+1+i^*\}$. So there are
$2k-1\geq 1$ points in $[0,1]$ which are endpoints of an interval
$A_i$ and a $B_j$ (namely, the $k$ right endpoints of $A_i$, which are
left endpoints of $B_i$, and the $k-1$ left endpoints of $A_{i+1}$,
which are right endpoints of $B_i$). Any of these points is a good
choice for $\tilde{y}$. Indeed, if we subtract from the column of
$(\tilde{W}'_{i,2\ell+1}-\lam_i)$ the minimum of the column, we must
get two zeroes, one at a position $i^*_1$ in the range
$\{1,\ldots,\ell+1\}$ and one at a position $i^*_2$ in the range
$\{\ell+2,\ldots,2\ell+2\}$ (in general there may be more than two
zeroes, but if the weights are generic the number of zeroes must be
exactly two).  Our choice of gauge is Hungarian, and the subgraph
$H_{\rm gauge}$ associated to the set of zeroes is the union of
$H_\calJ$ for $(X,Y)$, of $g \circ H_\calJ$, and of the two edges
$\edge{x_{i^*_1}}{\tilde{y}}$ and
$\edge{x'_{i^*_2-\ell-1}}{\tilde{y}}$.  As this graph is a tree, it
must coincide with the matching subgraph $\tilde{H}_{\calJ}$ of the
instance $(\tilde{X},\tilde{Y})$. Thus, the projection
$\overline{\tilde{H}_{\calJ}}$ consists of the two copies of the
projection, $\overline{H_{\calJ}}$ and $g\circ \overline{H_{\calJ}}$,
plus the geodesics connecting $x_{i^*_1}$ to $x'_{i^*_2-\ell-1}$ with
the same topology of
$\gamma_{x_{i^*_1},\tilde{y}}\circ \gamma_{\tilde{y},x'_{i^*_2-\ell-1}}$.
In particular, it must be crossing.
%% , that is 
%% %
%% $\gamma_{x_1,x_2}$ intersects $\gamma_{x'_{\ell},x'_{\ell+1}}$ and 
%% $\gamma_{x'_1,x'_2}$ intersects $\gamma_{x_{\ell},x_{\ell+1}}$.
% $[x_1,x_2]$ intersects $[x'_{\ell},x'_{\ell+1}]$ and $[x'_1,x'_2]$
% intersects $[x_{\ell},x_{\ell+1}]$.
\end{proof}
\noindent
The importance of this lemma in the context of this section, of
course, comes from the fact that a 2-covering of a M\"obius Strip is a
finite cylinder of base $2 \w_2$ and height $\w_1$, and a 
2-covering of a Klein Bottle is a torus with $\tau=2i \rho$.

We shall prove
\begin{thm}
\label{th:treenoncrossKleinB}
If $\Omega$ is a M\"obius Strip or
a Klein Bottle, and $p=2$, then in the first bipartite standard
setting there are no crossing configurations.
\end{thm}
\begin{proof}
We will treat only the case of the Klein Bottle, which is more
general. The proof is by contradiction. Suppose that $(X,Y)$ is a crossing
configuration. 
Just as in the proof of Theorem~\ref{th:treenoncrossToro},
call $x_0$, and consider the cycle
$C=(x_0,x_2,x_3,\ldots,x_{\ell},x_0)$ on $\Omega$. Let 
$\Omega^{\rm c}$ be the 2-covering of $\Omega$. If the lift of $C$ to
$\Omega^{\rm c}$ is still a closed curve, then we would have a
crossing configuration on the torus, which is excluded by
Theorem~\ref{th:treenoncrossToro}. Otherwise, we can apply the
construction of Lemma~\ref{lem.2cover}, to get a configuration of size
$2\ell+1$ that is crossing on the torus, again something which is
excluded by Theorem~\ref{th:treenoncrossToro}. As we have reached a
contradiction, we can conclude that there are no crossing
configurations on the Klein Bottle.
\end{proof}

%-------------------------------------------------------
\subsection{The case of the reservoir setting}
\label{ssec.noncrossreservoir}
%-------------------------------------------------------

\noindent
The goal of this section is to prove an analog of
Theorem~\ref{th:treenoncross} in the reservoir bipartite standard
setting, embedded on a domain with boundary $\Omega$, and where the
boundary $\partial \Omega$ plays the role of the reservoir --- that
is, the left-restriction of instances of symmetric bipartite graphs,
embedded on the double manifold
$\Omega \cup_{\partial \Omega} \Omega'$ --- where all edges
$\edge{a_i}{b_j}$ within one copy of the domain, and all reservoir
edges $\edge{a_i}{d_i}$ and $\edge{b_i}{c_i}$ between a vertex and its
reservoir counterpart, are present (as described on
page~\pageref{pg.reservoir} in Section~\ref{sec.reservintro}).

Our theorem will hold under the hypothesis that $p=2$ and $\Omega$ is
a compact of $\bR^2$ (or, in the generalization of the forthcoming
Section \ref{sec.ktuplesGeom}, a compact of~$\bR^d$).

Let us start with a remark, that simplifies the proof of our theorem,
and essentially states that we are free to include all the edges
$\edge{a_i}{d_j}$ and $\edge{b_i}{c_j}$ between ordinary $A$-vertices
and reservoir $B$-vertices (and vice versa), also with distinct
indices:
\begin{prop}
\label{prop.addresedgestoothers}
The bipartite reservoir standard setting embedded on $\Omega$, when
$p=2$ and $\Omega \subset \bR^d$,
% a portion of $\bR^d$, 
is equivalent to the analogous setting in which also the edges
$\edge{x_j}{v_k}$ and $\edge{y_j}{u_k}$ for $j\neq k$ are allowed.
\end{prop}
% \noindent
% This happens because, 
\begin{proof}
The underlying mechanism is that, in the extended graph, the edges
$\edge{x_j}{v_k}$ and $\edge{y_j}{u_k}$ are never taken by the optimal
matching $M_\star(G)$, as an argument of alternating cycles will show.
% , with $G$ a symmetric graph as in the bipartite reservoir standard setting. 
Indeed, suppose by contradiction that (e.g.)  the edges
$\edge{x_{j_2}}{v_{j_1}}$ is taken by $M_\star$.  Then $v_{j_2}$ is
not matched to $x_{j_2}$, so it is either matched to its mirror image
$\iota v_{j_2}$, or to some other vertex $x_{j_3}$. Thus, taking all
the edges of the form above, and concatenating them in the natural
way, we have some finite lists of edges, that either produces a cyclic
sequence, as in
\[
\{\edge{x_{j_2}}{v_{j_1}},\edge{x_{j_3}}{v_{j_2}},\ldots,\edge{x_{j_\ell}}{v_{j_{\ell-1}}},\edge{x_{j_1}}{v_{j_\ell}}\}
\]
or a path ending in a sym-edge, as in
\[
\{\edge{x_{j_2}}{v_{j_1}},\edge{x_{j_3}}{v_{j_2}},\ldots,\edge{x_{j_\ell}}{v_{j_{\ell-1}}},\edge{v_{j_\ell}}{\iota
v_{j_\ell}}\}
\]
In both cases, we have semi-alternating cycles
%  matchings of smaller weight 
contradicting the optimality hypothesis on the matching. In the first
case, swapping the cycle accounts to take instead the edges
\[
\{\edge{x_{j_1}}{v_{j_1}},\edge{x_{j_2}}{v_{j_2}},\ldots,\edge{x_{j_{\ell-1}}}{v_{j_{\ell-1}}},\edge{x_{j_\ell}}{v_{j_\ell}}\}
\]
while in the second case we take
\[
\{\edge{x_{j_2}}{v_{j_2}},\edge{x_{j_3}}{v_{j_3}},\ldots,\edge{x_{j_\ell}}{v_{j_{\ell}}},\edge{v_{j_1}}{\iota
v_{j_1}}\}
\]
(where we only listed the edges swapped on the left part of the graph,
and the sym-edges, while the modifications on the right part are
immediately deduced by symmetry).
\end{proof}
\noindent
We will produce two proofs of the main theorem of this section, one of
which is presented here, and the second one is postponed to
Appendix~\ref{ssec.2ndproofNonCrossReserv}. The second proof, although
more lengthy, has the advantage of using only
Theorem~\ref{th:treenoncross} in the plane as a black box, plus
tedious but conceptually simple arguments of continuity.
Conversely, for the proof presented here we shall construct again, in
the present setting, the analogs of the
inequalities~(\ref{eqs.4357856187bmain}). We will do this alongside
the proof of the main theorem, which is the following:
\begin{thm}
\label{th:treenoncrosssym}
When $\Omega \subset \bR^2$, and $p=2$, in the bipartite reservoir
standard setting the embedding of the tree $\bar{H}_{\calJ}^{\rm (l)}$
is non-crossing.
\end{thm}
\begin{proof}
The proof is by contradiction, so we assume that there are two edges
in $\bar{H}_\calJ$, $e=\edge{a_i}{a_j}$ and $e'=\edge{a_k}{a_\ell}$, that
when embedded in $\Omega$ cross each other (note that for each edge, one of the two
endpoints may be a reservoir point, for example we could have
$e=\edge{c_i}{a_j}$).

Recall that, in the bipartite reservoir standard setting, the graph
$\bar{H}_{\calJ}$ is a forest in which each component has exactly one
reservoir vertex (which we call the \emph{root} of the corresponding
tree). Thus, there are two possible cases:
\begin{itemize}
\item $e$ and $e'$ belong to the same component;
\item $e$ and $e'$ belong to distinct components.
\end{itemize}
Let us start by considering the first case. So we have a single tree
$T_\alpha$, component of $H_\calJ$,
rooted at the boundary of $\Omega$, with two or more crossing edges.
Consider the instance of the bipartite standard setting (not the
reservoir case) consisting only of this tree, with the small caveat
that if the tree is rooted at a $c_j$ vertex, this vertex is included
as an $A$-vertex of the instance; while if it is rooted at a $d_j$
vertex, this vertex is included as a $B$-vertex, and the corresponding
$c_j$ vertex is included as an $A$-vertex of the instance,
embedded in the same position as $d_j$.

If the tree $T_\alpha$ is a portion of the matching subgraph $H_\calJ$
in the reservoir setting, it is \emph{a fortiori} the whole matching
subgraph for the instance of the ordinary setting we have
constructed. Indeed, the reservoir setting includes all potential
choices of matchings of the ordinary setting, plus other matchings
obtained by using the reservoir (i.e., the bridge edges of the double
manifold) at places other than the root of the tree. (This reasoning
is somewhat analogous to the argument used in
Section~\ref{ssec.noncrossToro}, showing that a counterexample on the
torus with support cycle of winding number $(0,0)$ would imply a
counterexample in the plane.)

For the second case we cannot just reduce to an ordinary bipartite
instance, and we shall instead reduce to a setting that is slightly in
the fashion of a reservoir setting, where only the two roots of the
two trees $T_\alpha$ and $T_\beta$ have an associated reservoir
vertex.

This would lead us to develop an adapted version of the notion of
valid matrix w.r.t.~Definition \ref{def.validmat}, and a new version
of Proposition \ref{prop.crossifvalid}. Or, alternatively, we could
proceed as in the alternate proof of
Appendix~\ref{ssec.2ndproofNonCrossReserv}, to produce a combinatorial
tool that `connects' the two reservoir points in a way that mimics an
ordinary standard bipartite setting.

In fact, we choose to do neither of the two things: we leave the
(simple) appropriate generalization of Proposition
\ref{prop.crossifvalid} to the reader, and avoid the grid notation
used (for example) in (\ref{eq.innuovaeq}) (which becomes a bit
cumbersome in presence of reservoir vertices), and only adopt the
graph notation, as (for example) in (\ref{eq.3478Imgs1b}). In this
case the reservoir vertices are easily visualized: they are drawn on
the boundary of the example domain, and they can be either white (that
is, present in the sub-instance $G_U$) or green (that is, absent from
the sub-instance $G_U$), with no constrain on their total number in a
diagram. We will use this notation to explain, through a generic
example, why analogs of the inequalities~(\ref{eqs.4357856187bmain})
hold also in this case (recall that in the ordinary setting the whole
construction and analysis of valid matrices was aimed at the sole
purpose of deriving these inequalities, and then the rest of the proof
followed from algebraic reasonings based on them).

As before, we will construct alternating cycles with red and blue
edges, where the blue edges are on $H_\calJ$. The new rules, adapted
to the mechanisms underlying the reservoir setting, are that we have
up to one green vertex that is not a reservoir vertex, and as many
green edges on $H_\calJ$ as desired, but such that blue and green
edges, plus green vertices, determine a monomer-dimer configuration on
$H_\calJ$.  As in the ordinary case, three of our four inequalities
are derived directly from one choice of edge coloring, while the last
inequality is obtained from the ``sum'' of two diagrams.
% (that is, we have two diagrams, each one producing an inequality of
% the form $X_i\leq 0$, and our inequality is $X_1+X_2\leq 0$).

In order to mimic the notation for the ordinary setting, we choose to
label the vertices ``as if'' the two roots are connected, that is, the
vertices of the crossing edges are $\edge{x_1}{x_2}$ and
$\edge{x_\ell}{x_{\ell+1}}$, where we have $\ell+1$ $A$-vertices
overall on the path connecting the crossing edges on $H_\calJ$, and
for the rest the indices are consecutive along the two paths.  Also,
in order to produce an example encompassing both mechanisms for the
two types of rootings at the reservoir, we make one of the trees
rooted at a $c_i$ vertex, and one rooted at a $d_j$ vertex.

The first three inequalities, analogs of those in
(\ref{eq.876987627}), are shown in the top row of
Figure~\ref{fig.seidise2T}, while the last inequality is derived from
a combination of two elementary diagrams, as illustrated in the bottom
row of Figure~\ref{fig.seidise2T}.

\begin{figure}
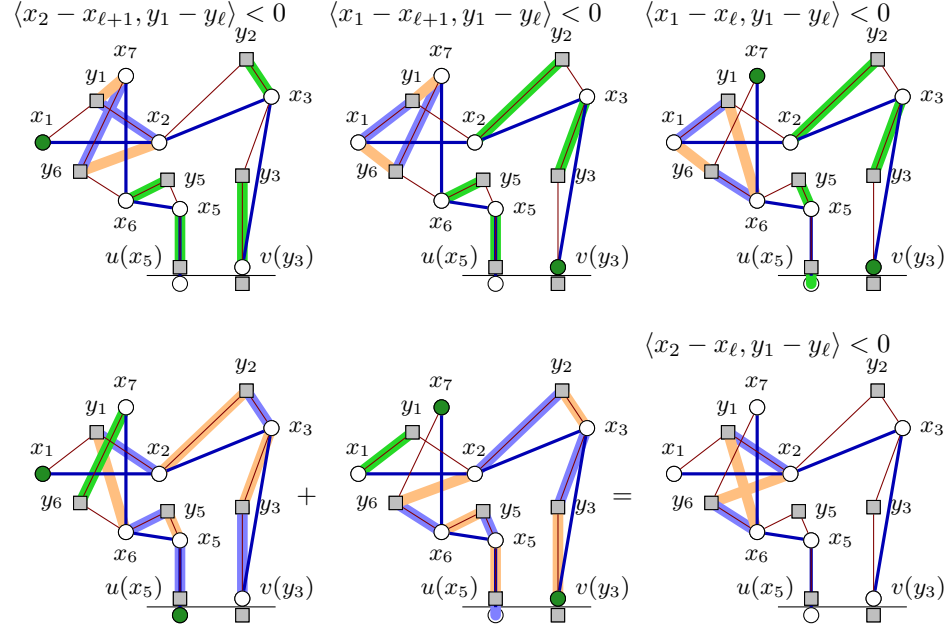

% \begin{align*}
% gin{center}
% \label{eq.3478Imgs2}
\setlength{\unitlength}{.33\textwidth}

%\end{align*}
\caption{\label{fig.seidise2T}The diagrams illustrating the validity
  of the inequalities (\ref{eqs.4357856187bmain}) also in the
  reservoir setting, and for edges in distinct components. Here $\ell=6$.}
\end{figure}

We recognize that, although via a slightly different combinatorial
reasoning, we obtain the very same inequalities as in
(\ref{eqs.4357856187bmain}). Thus, the subsequent algebraic steps do
not need to be repeated, and we can conclude immediately.
\end{proof}

%%%%%%%%%%%%%%%%%%%%%%%%%%%%%%%%%%%%%%%%%%%%%%%%%%%%%%%
% 
\section{\texorpdfstring{Non-crossing projected matching hypertrees at
    $p=2$}{Non-crossing projected matching hypertrees at p=2}}
% {The hyper-Assignment Problem}
% \texorpdfstring{A $k$-to-1 variant}{A k-to-1 variant}}
\label{sec.ktuplesGeom}
%%%%%%%%%%%%%%%%%%%%%%%%%%%%%%%%%%%%%%%%%%%%%%%%%%%%%%%

\noindent
In the previous section we have established (among other things) that,
for costs induced by the geometric embedding on a portion of $\bR^2$
and exponent $p=2$, the projected matching tree $\bar{H}_\calJ$ is
always non-crossing, that is, its edges never (properly) cross each
other, when represented as straight segments between the pairs of
$A$-vertices matched to the same $B$-vertex.  We also mentioned that
the analogous notion in higher dimension is not of interest, as
generically pairs of segments in dimension 3 or higher are not
coplanar.

Nonetheless, in Section \ref{sec.ktuples} we have introduced a variant
of the problem, namely the hyper-Assignment Problem, where we have
integer parameters $\gb_j \geq 1$ (the ordinary case corresponds to
all $\gb_j=1$) and the projected matching subgraph is a hypertree with
hyperedges $A_{(j)}$ of size $\gb_j+1$, instead that just an ordinary
tree.  As a consequence, its hyperedges, in their natural embedding
for a geometric realisation in $\bR^d$, consist of polytopes whose
dimension is generically $\min(d,\gb_j)$, and in particular, when
$d\geq \gb_j$, are generically $(\gb_j)$-dimensional simplices.
Figure~\ref{fig.cactus} shows an example in $d=3$, with all $\gb_j$
equal to $3$, so that all the hyperedges are in fact tetrahedra.

\begin{figure}[tb!]
\[
\includegraphics[scale=0.85]{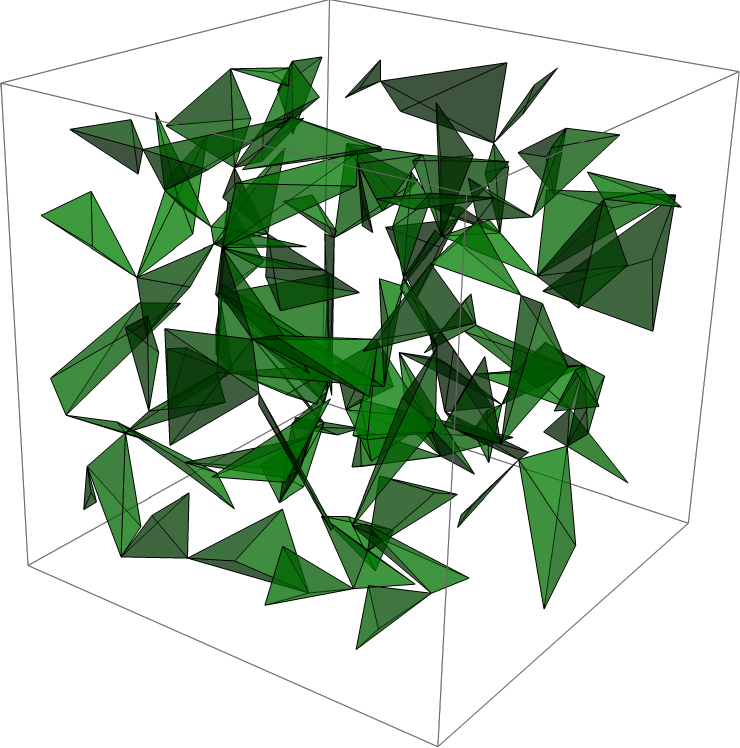}
\]
\caption{\label{fig.cactus}Example of matching hypergraph
  $\bar{H}_\calJ$ at $d=3$ and $p=2$, in the uniform case $\ga_i=1$
  and $\gb_j=3$ (so that in the tree $H_\calJ$ all the $B$-vertices
  have degree $4$, and the projection is a $4$-uniform hypertree,
  whose embedding is a ``cactus of tetrahedra''). The tetrahedra do
  not cross, as a special case of Theorem~\ref{th:treenoncrossHyper}.}
\end{figure}

Thus, it becomes interesting to investigate whether, in this variant
and when $d \leq 2\max_j \gb_j$, the projected spanning (hyper)tree is
almost-surely non-crossing (while the answer is trivially positive,
because the subspaces containing two hyperedges are almost surely
transversal, whenever $d>2 \max_j \gb_j$, as well as it was the case
for ordinary assignment and $d>2$). In this section we answer
positively to this question. In particular we prove the following
\begin{thm}
\label{th:treenoncrossHyper}
When $\Omega$ is a portion of $\bR^d$, and $p=2$, in the 
hyper-assignment standard setting
% first bipartite standard setting 
the embedding of the hypertree $\bar{H}_{\calJ}$
is non-crossing.
\end{thm}
\noindent
Before providing the proof, let us outline more clearly the setting
and recall some of the main features from the previous section.

We are in the case in which $\Omega$ is a portion of $\bR^d$, and the
cost function $f$ is $f(x)=x^2$.  We consider embedded instances in
the ``hyper-assignment standard setting'' of Theorem
\ref{thm.setting1hyper}, that is the hyper-assignment generalization
of the ``first bipartite standard setting'' of
Corollary~\ref{cor.setting1}, that is, the graph $G$ is just the
complete bipartite graph,
$G\equiv \calK_{n+1,m}=(A\cup B,E)$, with $A$ being the set of $n+1$
vertices, associated to the points in $X=\{x_a\}$, and $B$ the set of
$m$ vertices, with parameters $\gb_j \in \bN^+$ such that 
$\sum_j \gb_j = n$, associated to $Y=\{y_{b_j}\}$, we have
$\calJ=\{ \{a\} \}_{a \in A}$, and we assume that the embedding is
generic.  Thus, $H_\calJ$ is a spanning tree on $G$, with vertices
$b_j \in B$ of degree $\gb_j+1$, and $\bar{H}_\calJ$ is a 
% $(k+1)$-uniform
spanning hypertree on $\bar{G}=\calK_{n+1}$, with $m$ hyperedges. Both
these graphs have a representation on $\Omega$, in the first case the
$\gamma_e$'s are straight segments, while in the second case the
$\gamma_{A_{(j)}}$'s are convex polytopes with up to $\gb_j+1$
vertices, and, if $d\geq \gb_j$ for all $j$, are just
$\gb_j$-dimensional simplices.

We say that $\bar{H}_\calJ$ is \emph{non-crossing} if any two embedded
hyperedges either share a single $A$-vertex, because they are incident
in the hypertree, or they have empty intersection.

Let us call $\gamma_j\equiv \gamma_{A_{(j)}}$ the hyperedge resulting
from the projection of $b_j$, that is, the subset of $A$ of
cardinality $\gb_j+1$ consisting of all the vertices $a_i$ which are
adjacent to $b_j$ in $H_{\calJ}$. The condition of being non-crossing
means that, if 
$\gamma_{j}=\{a_i,a_{i_1},\ldots,a_{i_{\gb_j}}\}$ and
$\gamma_{j'}=\{a_i,a_{i'_1},\ldots,a_{i'_{\gb_{j'}}}\}$, that is, the
two hyperedges are incident, then the equation 
\be
\label{eq.hyperconvcomb1}
t_1 x_{i_1} + \cdots + t_k x_{i_k}
+\Big(1-\sum_\alpha t_\alpha\Big) x_i
% +(1-t_1-\cdots-t_k) x_i
=
s_1 x_{i'_1} + \cdots + s_k x_{i'_k}
+\Big(1-\sum_\alpha s_\alpha\Big) x_i
% +(1-s_1-\cdots-s_k) x_i
\ee
has only the trivial solution $t_\alpha=s_\beta=0$ in the cartesian
product of simplices $t_\alpha, s_\beta \in [0,1]$, 
$t_1+\cdots+t_k\leq 1$ and
$s_1+\cdots+s_k\leq 1$.
Similarly, if
$\gamma_{j}=\{a_{i_0},a_{i_1},\ldots,a_{i_{\gb_j}}\}$ and
$\gamma_{j'}=\{a_{i'_0},a_{i'_1},\ldots,a_{i'_{\gb_{j'}}}\}$,
%% $\gamma_{j}=\{a_{i_0},a_{i_1},\ldots,a_{i_k}\}$ and
%% $\gamma_{j'}=\{a_{i'_0},a_{i'_1},\ldots,a_{i'_k}\}$, 
with all indices distinct, that is, the two hyperedges are not
incident, then the equation
\be
\label{eq.hyperconvcomb2}
t_1 x_{i_1} + \cdots + t_k x_{i_k}
+\Big(1-\sum_\alpha t_\alpha\Big) x_{i_0}
%+(1-t_1-\cdots-t_k) x_{i_0}
=
s_1 x_{i'_1} + \cdots + s_k x_{i'_k}
+\Big(1-\sum_\alpha s_\alpha\Big) x_{i'_0}
%+(1-s_1-\cdots-s_k) x_{i'_0}
\ee
has no solution in the cartesian product of simplices above.

We will not need to specify $\Omega$ in what follows, as the datum of
$X$ and $Y$ is sufficient (and we can think that $\Omega$ is some
compact containing the convex hull of $X \cup Y$).

We need an analogue of Remark \ref{rmk.crossifpath}, on the minimal
form of a crossing counterexample. In this case this reads
\begin{rmk}
\label{rmk.crossifpathHyper}
There exists a crossing projected matching hypertree in a generic
instance in the standard hyper-assignment setting, with its natural
geometric realization in $\bR^d$, if and only if there exists the
datum of $\gb'+\gb''+\ell+1$ points
$X=\{x_0,\ldots,x_\ell,x'_1,\ldots,x'_{\gb'},x''_1,\ldots,x''_{\gb''}\}$,
with $\ell \geq 0$, $\ell+2$ points
$Y=\{y',y'',y_1,\ldots,y_{\ell}\}$, and parameters $\gb(j')=\gb'$,
$\gb(j'')=\gb''$, and all other $\gb_j$'s equal to 1, such that the
matching tree $\bar{H}_{\calJ}$ consists of the path
$(y',x_0,y_1,x_1,\ldots,y_{\ell},x_{\ell},y'')$ connecting the two
star graphs in which $y'$ is connected to the $x'_j$'s and $y''$ is
connected to the $x''_j$'s, and the equation (\ref{eq.hyperconvcomb1})
on the pair $\{x_0,x'_1,\ldots,x'_{\gb'}\}$ and
$\{x_0,x''_1,\ldots,x''_{\gb''}\}$ (for $\ell=0$), or equation
(\ref{eq.hyperconvcomb1}) on the pair $\{x_0,x'_1,\ldots,x'_{\gb'}\}$
and $\{x_{\ell},x''_1,\ldots,x''_{\gb''}\}$ (for $\ell>0$), are not
satisfied.
\end{rmk}
\noindent
The structure of the graph $H_\calJ$
% and $\bar{H}_\calJ$ 
is the following:
\begin{center}

\end{center}
The reasons for the validity of this remark are analogous to the ones
already analysed for Remark \ref{rmk.crossifpath}, and will not be
repeated here. Essentially, if the geometric embedding of a certain
instance in the standard hyper-assignment setting has a crossing
matching hypertree, where the polytopes associated to $y'$ and $y''$
are crossing, then the analogous statement holds for the sub-instance
in which we only keep the two pertinent polytopes, and the vertices on
the path in the tree that connects them, because the restriction of an
optimal hyper-assignment to a subset of the vertices that corresponds
to a subtree of $H_\calJ$ is optimal for the reduced instance.
An illustration of an instance of this type is provided in
Figure~\ref{fig.exRmkcrossifpathHyper}.

\begin{figure}
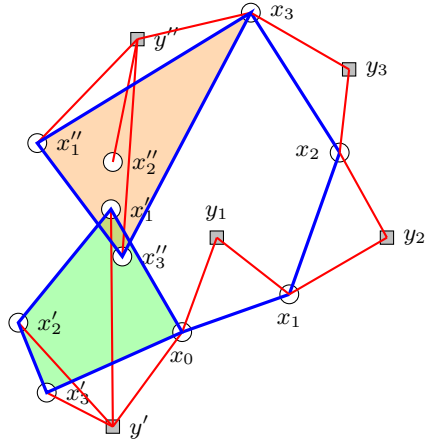

\[
\if\faifig1  

\else [...TikZ\ code...] \fi
\]
\caption{\label{fig.exRmkcrossifpathHyper}%
Example of potential crossing sub-configuration in the setting of 
Remark~\ref{rmk.crossifpathHyper}, here with $\gb'=\gb''=3$ (that is, the
polygons are the convex hull of up to 4 points) and $\ell=3$ (that is,
the two polygons are connected by a path of length 3). We omitted to
depict the splitting of the $k$ copies of $y'$ and $y''$, and
implicitly used the result of Lemma~\ref{lem.hypermatchNOiji}.}
\end{figure}

\begin{proof}[Proof of Theorem \ref{th:treenoncrossHyper}]
As was the case also in the previous section, we will prove Theorem
\ref{th:treenoncrossHyper} by contradiction on the existence of a
structure as the one depicted in the remark above.  In turn, this
structure corresponds to a situation in which the matrix of costs $W$,
for the list of vertices
$(x'_1,\ldots,x'_k,x_0,\ldots,x_\ell,x''_1,\ldots,x''_k)$ in this
order, and $(y',y_1,\ldots,y_\ell,y'')$ in this order, admits a
Hungarian gauge of the form (empty entries correspond to non-negative
values)
\[
W=
% \begin{pmatrix}
\left(
\begin{array}{c|ccc|c}
0      & && & \\
\vdots & && & \\
0      & && & \\
\hline
0 & 0 &&    & \\
  & 0 & 0 & & \\
  && 0 & 0  & \\
  &&& 0     & 0 \\
\hline
& && & 0 \\
& && & \vdots \\
& && & 0 
\end{array}
\right)
% \end{pmatrix}
\]
In other words, that's the same pattern of the matrix appearing in
the ordinary case, if we take a minor by dropping any subset of
$\gb'-1$ rows among the first $\gb'$, and
$\gb''-1$ rows among the last $\gb''$.

Similarly to what was done in Section \ref{sec.caso2dp2}, we can
subtract $|x_i|^2$ from the rows of $W$ and $|y_j|^2$ from the
columns, and have that the statement above holds also for the Gram
matrix of the point configurations (up to an overall factor $-2$), a
situation that has been analysed at length in that section, for the
case in which the matrix is of size $(n+1)\times n$, and has zeroes on
the $(i,i)$ and $(i+1,i)$ entries. That is, for all the sub-configurations
of the form
$X_{h,k}=\{x'_h,x_0,\ldots,x_\ell,x''_k\}$
and $Y=\{y',y_1,\ldots,y_\ell,y''\}$, for $(h,k)\in[\gb']\times[\gb'']$.
Let us perform the analysis for the case $\ell \geq 1$, as the case
$\ell=0$ is a simpler variant.
This implies the inequalities (\ref{eqs.4357856187bmain}), specialised
to the instances above, that read (here $\ell$ is a parameter, while
$1 \leq h \leq \beta'$ and $1 \leq k \leq \beta''$)
\begin{subequations}
\label{eqs.4357856187bmainhyper}
\begin{align}
\langle
x_0-x_\ell
,
y'-y''
\rangle
&
%=:\delta
> 0
\ef;
&
\langle
x'_h-x_\ell
,
y'-y''
\rangle
&
%=:u_h
>0
\ef;
\\
\langle
x_0-x''_k
,
y'-y''
\rangle
&
%=:v_k
>0
\ef;
&
\langle
x'_h-x''_k
,
y'-y''
\rangle
&
%=u_h+v_k-\delta
>0
\ef.
\end{align}
\end{subequations}
In order to conclude, observe that the function 
$\phi(x)=\langle x,y'-y'' \rangle$ defines an order over the
$\gb'+\gb''+2$ points 
%
% $\{x'_h,x_0,x_\ell,x''_k\}_{1\leq h,k\leq k}$, 
$\{x'_h\}_{1\leq h\leq \gb'}
\cup
\{x''_k\}_{1\leq k\leq \gb''}
\cup
\{x_0,x_\ell\}$
and that the inequalities (\ref{eqs.4357856187bmainhyper}) mean that
the $k+1$ points $\{x'_h\}_{1\leq h\leq \gb'} \cup x_0$ are `smaller' than the $k+1$ points
$\{x''_{k}\}_{1\leq k\leq \gb''} \cup x_\ell$, thus there must exist a hyperplane orthogonal to
$y'-y''$ that separates them, and this hyperplane certifies that the convex hulls
of the neighbours of $y'$ and of $y''$ 
cannot cross. The argument for $\ell=0$ is only slightly different. In
this case we have an order over the $\gb'+\gb''+1$ points
$\{x'_h\}_{1\leq h\leq \gb'}
\cup
\{x''_k\}_{1\leq k\leq \gb''}
\cup
x_0$,
%
%% $\{x'_h,x_0,x''_k\}_{1\leq h,k\leq k}$, 
and the analogues of the inequalities (\ref{eqs.4357856187bmainhyper})
mean that the $\gb'$ points $\{x'_h\}$ are `smaller' than $x_0$, that
in turn is smaller than the $\gb''$ points $\{x''_{k}\}$, thus the
hyperplane orthogonal to $y'-y''$ passing through $x_0$ separates the
convex hulls of the neighbours of $y'$ and of $y''$
% $\{x'_h,x_0\}$ and of $\{x_0,x''_{k}\}$ 
except for the vertex~$x_0$.
\end{proof}

%%%%%%%%%%%%%%%%%%%%%%%%%%%%%%%%%%%%%%%%%%%%%%%%%%%%%%%
\section{Hitchcock's generalization of the Assignment Problem}
% Assignment with capacities
\label{sec.AssCapa}
%%%%%%%%%%%%%%%%%%%%%%%%%%%%%%%%%%%%%%%%%%%%%%%%%%%%%%%

\noindent
In this appendix we introduce and study a generalization of the
Assignment Problem, which is also a subclass of the general
Linear Programming, due to Hitchcock \cite{Hitchcock,Ford1962-gm,P-743}.
In his own words (rephrased to match our notations), the problem is
formulated as follows:
\textit{Suppose there are $n$ sources $\{a_1,\ldots,a_n\}$ for a
  commodity, with $\ga_i$ units of supply at $a_i$, and $m$ sinks
  $\{b_1,\ldots,b_m\}$ for the commodity, with the demand $\gb_j$ at
  $b_j$. If $W_{ij}$ is the unit cost of shipment from $a_i$ to $b_j$,
  find a flow that satisfies demand from supplies and minimizes the
  flow cost.}

While the Assignment Problem, in its historical Monge description,
emerges as the algorithmic problem underlying the Monge--Kantorovich
Optimal Transport Problem when both measures are empirical measures
composed of $n$ Dirac masses of weight $1/n$, the generalization
introduced here still deals with empirical measures, but relaxes the
constraint of uniform masses. A first step in the direction of the
Hitchcock Problem was presented in the hyper-Assignment Problem of
Section \ref{sec.ktuples} --- where we have $kn$ sources with $\ga_i=1$
and $n$ sinks with $\gb_j=k$ --- though the present framework is
substantially more general.

It turns out that an efficient algorithmic solution of this problem
has been found by Ford and Fulkerson (see
e.g.\ \cite[sec.~3.1]{Ford1962-gm}). In their algorithm, one closely
follows the ideas of the classical Hungarian Algorithm, except that
one step, that is relatively simple in the latter, in the former must
be replaced with a run of the algorithm (also due to Ford and
Fulkerson) for the Maximum Flow problem, on a related graph where the
$\ga_i$'s and $\gb_j$'s are the capacities of certain edges. For this
reason, in this section we call \emph{Assignment Problem with
  Capacities} the problem depicted above.
% the \emph{Assignment Problem with Capacities}, that is a generalization
% of the Assignment Problem, while still being 

The reason for studying this generalization is that, in contrast to
ordinary Assignment where the optimal configuration is a perfect
matching on $\cK_{n,n}$, it will turn out that here, more generally,
the optimal configurations have support on spanning forests, and,
under suitable generality hypotheses (not satisfied by ordinary
Assignment), a spanning tree of~$\cK_{n,m}$.  Thus, it is natural to
ask whether this tree, obtained as the optimal configuration of a
single instance of the Hitchcock's Assignment Problem, has any
relation with the matching subgraph $H_{\calJ}$ studied in the
previous sections, obtained as the union of the optimal assignments
$M_U$ of several small variations $U\in\calJ$ of one given instance of
the ordinary Assignment Problem. As we will see later on (in
Corollary~\ref{corr.HnewIsHold_ass}), this is indeed the case,
although the connection is (surprisingly) far from evident.

Furthermore, also for the hyper-Assignment variant of the problem
(that is, in the hyper-assignment standard setting of Theorem
\ref{thm.setting1hyper}), the matching subgraph $H_{\calJ}$ associated
to our problem is related to the optimal transportation of a single
instance of the Hitchcock's Assignment Problem, with suitably tuned
capacities. This generalization is the content of
Corollary~\ref{corr.HnewIsHold_hass}).  However, we believe that the
two results mentioned above are optimal in a strong sense, leaving
virtually no room for further broad extensions. Indeed, a careful
inspection of the proofs --- supported by explicit counterexamples to
the main lemma (Proposition \ref{prop.treeOnlyIfStar}) under slightly
modified hypotheses --- suggests that this connection between matching
subgraphs and the support structures of the Hitchcock problem is
maximal within this framework, indicating that our formulation hits a
rigid structural boundary.
\medskip

\noindent
As used already several times along this paper, the ordinary
Assignment Problem on an arbitrary (bipartite) graph $G$ can be
studied by replacing edges $e=\edge{a_i}{b_j}$ absent from $G$ with
edges having very large cost, so that we can always suppose that the
ambient graph is the complete bipartite graph $\cK_{n,n}$, and the
datum of an instance of the Assignment Problem can be encoded solely
by a
\emph{matrix of costs} $W=(W_{ij})_{1\leq i,j\leq n}$, with 
$W_{ij} \in \bR$ (and, in the hyper-Assignment Problem, 
solely
by a
matrix of costs $W=(W_{ij})_{1\leq i \leq kn, 1\leq j\leq n}$).

Analogously, the datum of an instance of the Assignment Problem with
Capacities is given by a triple $(W;\va,\vb)$, where
$W=(W_{ij})_{1\leq i\leq n\,,\,1\leq j\leq m}$, with $W_{ij} \in \bR$,
is the \emph{matrix of costs}, and $\va=(\ga_{i})_{1\leq i\leq n}$ and
$\vb=(\gb_{j})_{1\leq j\leq m}$, with $\ga_{i}, \gb_{j} \in \bR^+$,
are the
\emph{vectors of (row and column) capacities}. We have the constraint
that $M:=\sum_i \ga_i=\sum_j \gb_j$ (this parameter is the
normalization of the two measures in the Optimal Transportation
Problem, that is, the `total mass' to be transported), and the case of
ordinary Assignment is when the $\ga_i$'s and $\gb_j$'s are all equal.

A $n \times m$ matrix $T$ is a \emph{transport matrix}
for the capacities $(\va,\vb)$ if $T_{ij}\geq 0$ for all $i$ and $j$, and
\begin{align}
\sum_{j=1}^m T_{ij} &= \alpha_i \qquad \forall \ 1\leq i\leq n
\ef;
&
\sum_{i=1}^n T_{ij} &= \beta_j  \qquad \forall \ 1\leq j\leq m
\ef.
\end{align}
We denote by 
$\mathcal{T}(\va,\vb)$
% (\alpha, \beta)$ 
the set of all transport matrices for the capacities~$(\va,\vb)$.
In a transportation problem where 
$\mu_1(z)=\frac{1}{M} \sum_{i=1}^n \ga_i \delta(z-x_i)$ and
$\mu_2(z)=\frac{1}{M} \sum_{j=1}^m \gb_j \delta(z-y_j)$,
and the cost for moving a unit of mass from $x_i$ to $y_j$ is
$W_{ij}$, the quantity $T_{ij}$ represents the portion of the mass in
$x_i$ that is moved to $y_j$ under the transportation plan $T$. So,
provided that we can find a metric such that the $W_{ij}$'s can be
interpreted as the transportation costs, this variant of the
Assignment is also natural in the framework of Optimal Transport (and
this justifies the name we give to matrices~$T$). However, at this
level we deal with a problem of Combinatorial Optimization on graphs,
with no underlying geometry.

The study of the combinatorial structures associated to optimal
transport matrices, and in particular of its relation with the
matching subgraphs $H_{\calJ}$, goes through the introduction of a
useful notion:
\begin{defn}[Support of a transport matrix]
For $T \in \cT(\va,\vb)$ as above, and calling $A=\{a_i\}$ and
$B=\{b_j\}$ the two vertex classes of $\cK_{n,m}$, the 
\emph{support graph} $H(T)$ is the spanning subgraph of $\cK_{n,m}$
such that
$\edge{a_i}{b_j} \in E(H)$ iff~$T_{ij}>0$.
\end{defn}
\noindent
If all the $\ga_i$'s and $\gb_j$'s are non-zero, the graphs
$H\subseteq \cK_{n,m}$
obtained as $H(T)$ have no singletons, that is, all vertices have
degree at least 1. We have the following property
(see e.g.\ \cite[Sec.\ 21.6]{schrijver2003combinatorial})
\begin{lem}
\label{lem.HTforest}
If $H(T)$ is a forest, then $T$ is uniquely identified by $H$, $\va$
and $\vb$,
that is, if $T',T'' \in \cT(\va,\vb)$,
$H$ is a forest and $H(T')=H(T'')$, then $T'=T''$. 
\end{lem}
\begin{proof}
We can prove this fact by induction on $n+m$. 
If $n=m=1$ the statement is obvious, as $T_{11}=\ga_1=\gb_1$ is
the unique possibility. 

Since singletons correspond to capacities $\ga_i$ or $\gb_j$ equal to
zero, they play a trivial role (we can remove the entry from the
vector of capacities and the corresponding row or column from the
matrix of costs, yielding an equivalent problem with a smaller value
of $n+m$); therefore, we can assume that there are no singletons.

Then, let $H$ be a forest, with no singletons. So $H$ has at least two
leaves. Up to relabeling the vertices, and possibly swapping the role
of $\va$ and $\vb$ and taking the transpose of $W$, we can assume that
$a_1$ is a leaf, adjacent to $b_1$.  Since $a_1$ has degree 1 in
$H(T)$, we must have $T_{1j}=0$ for all $j \geq 2$. The row-sum
constraint then forces $T_{11}=\ga_1$ (which implicitly requires
$\ga_1 \leq \gb_1$).  The graph $H'=H\setminx \edge{a_1}{b_1}$ is a
forest, and $H'=H(T')$ if and only if $H=H(T)$, where
$T'\in \cT(\va',\vb')$ is the matrix $T$ with the first row removed, 
$\va'=(\ga_2,\ga_3,\ldots,\ga_n)$ and
$\vb'=(\gb_1-\ga_1,\gb_2,\ldots,\gb_m)$. (If $\gb_1=\ga_1$, we can
also remove the first column of $T'$, and shorten the vector $\gb$.)
As we are in the same setting from where we started, but with
% one edge less,
a smaller size $n+m$, our induction is complete.
\end{proof}
\noindent
Remark that, if $T'$ and $T''$ are transport matrices for the
capacities $(\va,\vb)$,
then also $\lambda T'+(1-\lambda)T''$, with 
$\lambda \in [0,1]$, is a transport matrix for the same capacities.
Therefore the set $\mathcal{T}(\va,\vb)$ of all transport matrices for
the capacities $(\va,\vb)$ is convex (and, as the conditions for being
a transport matrix are linear equalities and inequalities, it is in
fact a convex polytope).

Of course, if $E'=E(H(T'))$ and $E''=E(H(T''))$,
then for all $0<\lam<1$ we have
\be
\label{eq.287658354}
E(H(\lam T'+(1-\lam)T''))=E' \cup E''
\ef.
\ee
In the case of the ordinary Assignment Problem (say, of size $n$), the
extremal points of this convex set are given by the $n!$ permutation
matrices $T^{(\pi)}_{ij}=\delta_{j\,\pi(i)}$, and in particular their
support graphs are the set of perfect matchings on $\cK_{n,n}$.

This fact generalizes in the following
\begin{lem}
\label{lem.ExtrIfForest}
For $T \in \cT(\va,\vb)$ as above, $T$ is extremal iff $H(T)$ is a
forest.
\end{lem}
\begin{proof}
Let us start by proving that if $T$ is not extremal, then $H(T)$ is
not a forest.
% The `if' part 
This is a corollary of the Lemma \ref{lem.HTforest} above, and
equation (\ref{eq.287658354}). Indeed, if $T$ is not extremal in
$\cT$, there exist $T' \neq T''$ in $\cT$ and $0<\lam<1$ such that
$T=\lambda T' + (1 - \lambda) T''$, so, from equation
(\ref{eq.287658354}), we would have a whole 1-dimensional open
interval within $\cT$, namely the transport matrices 
$\rho T' + (1 - \rho) T''$ for $\rho \in\; ]0,1[$, with the
same support graph as $H(T)$, thus, in light of Lemma
\ref{lem.HTforest} above, $H(T)$ cannot be a forest.

Now we shall prove that if $T$ is extremal, then $H(T)$ is a forest.
We give a proof by contradiction.  If $T$ is extremal and $H(T)$ is
not a forest, then we have a cycle $C \subseteq H \subseteq \cK_{n,m}$
(thus of even length).  Root and orient the cycle $C$
arbitrarily. Then the matrix $T'$ such that
\be
T'_{ij}
=
\left\{
\begin{array}{ll}
T_{ij}+(-1)^k \eps & \textrm{$\edge{a_i}{b_j}$ is the $k$-th edge of $C$;}
\\
T_{ij}             & \edge{a_i}{b_j}\not\in C;
\end{array}
\right.
\ee
for an interval 
$\eps \in[-\eps_1,\eps_2]$ for some $\eps_1,\eps_2>0$,
is still a transport matrix for $(\va,\vb)$,
So $T$ is not extremal.
% Call $\eps$ the smallest entry $T_{ij}$ for $\edge{a_i}{b_j}\in C$, and 
\end{proof}
\noindent
Note that, despite what could be guessed, it is not the case that all
extremal points $T^{(i)} \in \cT(\va,\vb)$ have associated support
graphs $H^{(i)}=H(T^{(i)})$ which are forests with the same number of
components.
For example, for $\va=(2,2)$ and $\vb=(1,2,1)$, there are four extremal
matrices, of which two with 2 components, and two with a single component
(spaces stand for zeroes):
\begin{align*}
&
\begin{pmatrix}
{} & 2 &   \\
 1 &   & 1  
\end{pmatrix}
&&
\begin{pmatrix}
 1 &   & 1 \\
{} & 2 &    
\end{pmatrix}
&&
\begin{pmatrix}
 1 & 1 &   \\
{} & 1 & 1  
\end{pmatrix}
&&
\begin{pmatrix}
{} & 1 & 1 \\
 1 & 1 &    
\end{pmatrix}
\end{align*}
and correspond to the four vertices of the square $(a,b)\in [0,1]^2$,
where an arbitrary transport matrix is parametrized as
\be
\begin{pmatrix}
a   & 2-a-b & b   \\
1-a & a+b   & 1-b    
\end{pmatrix}
\ef.
\ee
The Lemma \ref{lem.ExtrIfForest} above implies the following simple
corollary (that usually goes under the name of 
\emph{total unimodularity}, or \emph{integrality property of the 
Hitchcock--Koopmans polytope}):
\begin{corr}
\label{corr.Tinteg}
If the capacities $\ga_i$ and $\gb_j$ are integer-valued, also the extremal
transport matrices $T \in \cT(\va,\vb)$ are integer-valued.
\end{corr}
\begin{proof}
The subgraph of $H(T)$ consisting of the edges $\edge{a_i}{b_j}$ where
$T_{ij}-\lfloor T_{ij} \rfloor >0$ has no vertex of degree 1. So
either it is the empty graph (and all the $T_{ij}$'s are integers), or
it contains a cycle, and the latter is excluded by
Lemma~\ref{lem.ExtrIfForest}.
\end{proof}
\noindent
We can introduce a version of Definition \ref{def.genwei1} of generic
weights adapted to this variant of the problem, namely
\begin{defn}[Generic weights and capacities] 
\label{def.genwei1cap}
  The cost matrix $W$ is \emph{generic} if, for all even-length
  cycle $C=(e_1,e_2,\dots,e_{2\ell})$ in $\cK_{n,m}$ (with edges
  labeled in cyclic order), we have that the alternating sum
  $\walt(C)$ (defined in equation (\ref{eq:altsum})) is not vanishing.
  The vectors of capacities $(\va,\vb)$ are \emph{generic} if there
  exists no pair $(I,J)$, with 
$\emptyset \subsetneq I \subsetneq [n]$ and
$\emptyset \subsetneq J \subsetneq [m]$,
such that $\sum_{i \in I}\ga_i=\sum_{j \in J}\gb_j$.
\end{defn}
\noindent 
It is easily seen that, under the sole assumption that the capacities
$(\va,\vb)$ are generic, irrespectively of $W$, all the support graphs
$H(T)$ are spanning connected subgraphs of $\cK_{n,m}$, and in
particular, all the $H(T)$'s which are forests are in fact
trees. Indeed, if we had a non-trivial connected component
$H_{\ell}\subset H$, this would have sets of vertices
$A_{\ell}=\{a_i\}_{i\in I}$ and $B_{\ell}=\{b_j\}_{j\in J}$, and the
sets of indices $I$ and $J$ would violate the genericity hypothesis.
Of course the capacities $(\va,\vb)$ for ordinary Assignment are not
generic, and actually all pairs $(I,J)$ with $1\leq |I|=|J|\leq n-1$
violate the genericity hypothesis (and a similar statement holds for
the hyper-Assignment Problem).  Nonetheless, in the Lebesgue measure
on $\bR^{2n-1}$ for capacities, generic capacities are of course
dense, and the set of non-generic capacities is a finite union of
hyperplanes of codimension~1, so that, similarly as what is the case
for generic weights, certain general facts on the problem with
capacities can be deduced by continuity from the corresponding
statements under the assumption of generic capacities.

Now that we have described in some detail the space $\cT(\va,\vb)$ of
possible transport matrices for a given set of capacities, we shall
introduce the cost function $\cW(T)=\cW(T;W)$:\footnote{This function
  is also $\tr T W^{\rm T}$, but we prefer to avoid this notation as
  Linear Algebra is not the appropriate framework here.}
%;\va,\vb)
\be
\cW(T;W)=\sum_{\substack{
1\leq i \leq n \\
1\leq j \leq m }}
T_{ij} W_{ij}
\ef.
\ee
Just as in the ordinary Matching and Assignment Problems described in
the previous sections, we shall concentrate on the optimal transport
matrices $T_{\star}(W;\va,\vb)$, that is, the matrices that minimize
$\cW(T;W)$ inside $\cT(\va,\vb)$, and in the value
$\cW_{\star}(W;\va,\vb)=\cW(T_{\star}(W;\va,\vb);W)$ of the optimal
cost.

Another notion that generalises naturally from the ordinary Assignment
Problem
% , also the Assignment Problem with Capacities has a 
is ``gauge invariance''. This is not surprising, as in fact the gauge
parameters are the dual variables in Linear Programming duality, and
the Hitchcock's Assignment Problem is within this more general
framework. Also, just like gauge transformations for ordinary
Assignment (that is, operations on the dual variables) are used
extensively in the Hungarian Algorithm, gauge transformations for this
problem are used extensively in the Ford and Fulkerson algorithm for
this problem.  So we have

\begin{rmk}[Gauge Invariance with capacities]
\label{rmk.gaugeinvcapac}
Let $W$ be a $n \times m$ real matrix, 
$(\va,\vb)$ two vectors of capacities,
and
$\bm{\lambda}=(\lambda_1,\ldots,\lambda_n)$,
$\bm{\mu}=(\mu_1,\ldots,\mu_m)$ two vectors. Define
$W'_{ij}=W_{ij}-\lambda_i-\mu_j$ and
%. Then, calling
% $\cW(T)$ the cost of the transport matrix $T$ for the matrix of costs
% $W$, $\cW'(T)$ the cost for the matrix of costs $W'$, and
$c=\sum_{i=1}^n \ga_i \lambda_i+\sum_{j=1}^m \gb_j \mu_j$. We have
\be
\cW(T,W') = \cW(T,W)-c
\ee
and in particular $T$ is optimal for $W$ iff it is optimal for $W'$
(and, if $W$ is generic, $T_\star(W)=T_\star(W')$).
\end{rmk}
\noindent
Again, a gauge such that $W'_{ij}\geq 0$ for all entries and
$W'_{ij}=0$ on all entries such that $(T_{\star})_{ij}>0$ will be called
a \emph{Hungarian gauge}, and provides a certificate of optimality for
$T_{\star}$.

Call $G=G(W;\bm{\lambda},\bm{\mu})$ the graph such that
$\edge{a_i}{b_j} \in G$ iff $W'_{ij}=0$. The statement above reads as
the fact that $(\bm{\lambda},\bm{\mu})$ is a Hungarian gauge for the
instance $(W;\va,\vb)$ iff 
$W'_{ij}\geq 0$ for all entries, and
there exists $T \in \cT(\va,\vb)$ such that
$H(T;W) \subseteq G(W;\bm{\lambda},\bm{\mu})$, and in this case the
optimal transport matrices are all and only the transport matrices
with the property above, and the optimal cost is
$\cW_{\star}(W;\va,\vb)=\sum_i \lam_i \ga_i+\sum_j \mu_j \gb_j$.

If both $W$ and $(\va,\vb)$ are generic, then both $H(T_{\star})$ and
$G$ are spanning trees, so they must coincide. In this case the
Hungarian gauge is unique up to the obvious invariance
$(\bm{\lambda},\bm{\mu})\to(\bm{\lambda}+s,\bm{\mu}-s)$
(which, in passing, implies that we
have $n+m-1$ independent gauge parameters).
% and $H(T_{\star})=G$. 
Indeed, as the support graph $H(T_{\star})$ is a spanning tree, it has
$n+m-1$ edges and no cycles, and having $W'_{ij}=0$ on all the edges
of $H$ is a necessary condition for the gauge to be Hungarian.
Imposing all these identities accounts for $n+m-1$ linear relations,
all independent (because linear relations in equations of this kind
come from alternating cycles in the graph, but $H$ is a tree). As this
coincides with the number of independent gauge parameters, the
dimension counting is coherent with our claim of existence and
unicity. For example we could set $\lambda_1=0$, and solve the
equations associated to the edges of the tree iteratively in an order
compatible with the partial ordering given by rooting the tree $H$ on
$a_1$. That is, if $\edge{a_i}{b_j}$ is in $H$ and $\lam_i$ has been
determined, we can set $\mu_j=W_{ij}-\lam_i$, while if $\mu_j$ has
been determined, we can set $\lam_i=W_{ij}-\mu_j$. In fact, the
$\lam_i$'s and $\mu_j$'s can be determined in any (total) order
compatible with the partial ordering given by the distance from the
root $a_1$ on the tree~$H$.
% $G(W;\bm{\lambda},\bm{\mu})$.
%% and it is not possible that both $\lam_i$ and $\mu_j$ have been
%% already determined by other equations, because in a rooted tree every
%% edge has exactly one of the endpoints on the branch of the tree
%% containing the root.  
Once that all the equalities $W'_{ij}=0$ for
edges in $H$ have been set, the inequalities $W'_{ij}>0$ for edges not
in $H$ are implied by the optimality of $T$ 
(saying that $W'_{ij}\geq 0$) and the genericity of $W$ (saying that 
$W'_{ij}\neq 0$).

Also, if $W$ is generic (so that $T_{\star}$ is unique) but
$(\va,\vb)$ is not necessarily generic (so that $H=H(T_{\star})$ may
be a forest with $k>1$ components), the set of Hungarian gauges (up to
the trivial parameter) is a convex polytope of dimension
$k-1$. Indeed, convexity is obvious, the fact that it is a polytope
follows from the fact that the defining properties consist of linear
equalities and inequalities, and the dimension comes from counting
variables and equalities in the associated linear system, and the fact
that $H$ is a forest (that is, it is acyclic) implies that the linear
relations are independent.  In fact, this poytope coincides with the
\emph{Spanning Tree Polytope} of the graph consisting of the edges
$\edge{a_i}{b_j}$ such that $W'_{ij}=0$ for some Hungarian gauge (see
\cite{Edmonds1965}, or \cite[Sec.\ 50.4]{schrijver2003combinatorial}).

As a result, for $(\va,\vb)$ not generic,
% (so that $H(T_{\star})$ may be a forest with several components), 
for all $G$ a spanning tree containing $H(T_{\star})$ there exists a
choice of Hungarian gauge $(\bm{\lambda},\bm{\mu})$ such that
$G=G(W;\bm{\lambda},\bm{\mu})$, which in fact is a vertex of the
polytope depicted above. On the other side, a gauge associated to an
internal point of the polytope has exactly $G=H(T_{\star})$.
\medskip

\noindent
As we have seen, under the assumption of generic weights (so that we
have unicity of $T_\star$), for what concerns the graph properties of
$H(T_\star)$, the case of generic capacities and the case of the
`ordinary' Assigment Problem, with $n=m$ and $\ga_i=\gb_j=1$, are
rather different. In the first case, we have that $H$ is a spanning
tree, and in the second case it is a matching, that is a forest with
$n$ components, in which all components are isolated edges. So it is
natural to guess that we can classify the quadruples $(n,m;\va,\vb)$
according to three possible behaviors:
\begin{enumerate}
\item $H$ is guaranteed to have one single component, 
\item $H$ is guaranteed to have a number of components determined
  \emph{a priori} in terms of $n$ and $m$ and
scaling linearly with $n$
\item the number of components of $H$ is not determined \emph{a
  priori}
\end{enumerate}
Let us remark that the special case of the Hitchcock Problem discussed
in Section~\ref{sec.ktuples} under the name of hyper-Assignment
Problem, that is when we work on $\cK_{n,m}$, the $\ga_i$'s are all 1
and the $\gb_j$'s are $m$ positive integers summing up to $n$, falls
in the second case of the classification, as the extremal transport
matrices have entries in $\{0,1\}$ and correspond to graphs which are
collections of star graphs, where the $B$-vertex $b_j$ is incident to
exactly $\gb_j$ $A$-vertices, that are all leaves, so we have exactly
$m$ components, and the set of extremal transport matrices is in
bijection with elements in 
$\mathfrak{S}_n \slash \bigtimes_j \mathfrak{S}_{\gb_j}$.

We will not perform a full classification, and we will restrict our
attention to a special subclass of quadruples, those in which the
$\ga_i$'s are all equal, as well as the $\gb_j$'s, that is (up to
rescaling all capacities), for $p$ and $q$ two coprime positive
integers, the case in which the underlying graph is $\cK_{pn,qn}$, and
$\ga_i=q$ for all $1\leq i \leq pn$ and
$\gb_j=p$ for all $1\leq j \leq qn$.
Define the graph $L_{pq}$ as the bipartite graph, spanning on
$\cK_{p,q}$, such that
$(ij)\in L_{pq}$ iff
$c_{ij}=|\{(i-1)q+1,\ldots,iq\}\cap\{(j-1)p+1,\ldots,jp\}|>0$, that is
the graph induced by the following construction (here illustrated for
$(p,q)=(7,4)$)
\begin{gather*}
\setlength{\unitlength}{10pt}

% \else [...TikZ\ code...] \fi
\end{gather*}
In particular $L_{11}$ is a single edge, and $L_{p1}$ is the star
graph with a $B$-vertex in the central hub, and all the leaves being
$A$-vertices. This construction shows that, for all $p$ and $q$
coprime (and $n=1$), the set of extremal transport matrices with
capacities as above is non-empty. Note that it is not the case that
all trees $T$ which are associated to extremal transport matrices for
the pair $(p,q)$ are equal to $L_{pq}$ up to isomorphism, for example,
for $(p,q)=(4,3)$ we have the two graphs
\begin{align*}
&
%
% \else [...TikZ\ code...] \fi
\end{align*}
(of which the one on the left corresponds to $L_{43}$).

Simple arguments show that, with the choice of sizes and capacities
above, the number of components of $H$ must be an integer between 1
and $n$, and each component contains $pm$ $A$-vertices and $qm$
$B$-vertices, for some $1\leq m \leq n$.  In particular, as the case
$m=1$ is always possible, we will be in the case (1) of the
classification if and only if $n=1$, that is, the numbers of
rows and columns are coprime. When $n>1$, it remains to
determine under which choices of $(p,q)$ we are in the case (2) of our
proposed classification, and under which choices we are in case
(3). We shall prove the following:
\begin{prop}
\label{prop.treeOnlyIfStar}
Let $p>q\geq 1$ be as above. Let $W$ be a generic matrix of costs, of
size $(np)\times(nq)$, for some $n\geq 1$, and let the capacities be
$\ga_i=q$ for all $i$, and $\gb_j=p$ for all $j$.  Then, if $q=1$, $H(T_\star)$
is guaranteed to be a forest with $n$ components, all isomorphic to
the star graph
$L_{p1}$, while if $q>1$ the number of components of $H(T_\star)$ may
take any value between 1 and $n$.
\end{prop}
\begin{proof}
From Lemma \ref{lem.ExtrIfForest} above, and the fact that, when the
weights are generic, the optimal transport matrix is unique, we see
that $T_\star$ must be an extremal point in the polytope of possible
transport matrices. Conversely, for any spanning forest $H$ such that
there exists a transportation matrix $T$ with support $H$, there
exists also a matrix of costs for which $T$ is optimal (it suffices to
take $W_{ij}=0$ if $\edge{a_i}{b_j}\in H$ and $1$ otherwise, or, if we
insist on the genericity of the weights, a small perturbation of this
choice).

In other words, it suffices to understand the properties of the
extremal points in the polytope of possible transport matrices (i.e.,
for short, \emph{extremal transport matrices}), which is now a fact
merely of graph theory, that is independent from the Assignment
Problem. That is, we should determine for which pairs $(p,q)$ it is
the case that all forests $F \subseteq \cK_{pn,qn}$ equipped with a
positive-integer-valued function on the edges,
$c_{\edge{a_i}{b_j}}\geq 1$, such that $\sum_j c_{\edge{a_i}{b_j}}=q$
for all $i$ and $\sum_i c_{\edge{a_i}{b_j}}=p$ for all $j$, must be a
collection of bipartite graphs of size $(p,q)$.
% $L_{pq}$ with $c_{ij}$'s as above.

The positive claim, for $q=1$, is trivially seen to be true.  Indeed,
from Corollary \ref{corr.Tinteg} we know that in this case all the
$A$-vertices must have degree 1, which easily implies that all the
components must be star graphs.

Then, for the negative case when $q>1$, we will prove the stronger
fact that there exist extremal transport matrices whose support graph
is a spanning tree, for all $n\geq 1$ (which implies that, for all
$n$, and all integer partitions $\lam$ of $n$, there exist 
extremal transport matrices whose components have sizes
$(p+q)\lam_i$).

In order to do so, we shall produce three auxiliary graphs. Recall
that, from the B\'ezout lemma, there exist values $a$ and $b$ such
that $pa-qb=1$, and $0<a<q$, $0<b<p$.  As a result, the values
$a'=q-a$ and $b'=p-b$ are such that $pa'-qb'=-1$, and $0<a'<q$,
$0<b'<p$. Our three graphs are produced by performing constructions
analogue to the one for $L_{pq}$, but truncated to the left or right
portion of the segment:
% with an offset at the endpoints. Namely, we have
\begin{itemize}
\item a graph 
$L_{pq}^{\rm r}$
% $L_{pq}^{01}$, 
(r stands for ``root''),
where we partition the segments of length
  $pa$ (for $A$-vertices) and $qb$ (for $B$-vertices), with an offset
  of $0$ on the left (and thus of $1$ on the right). This produces a
  graph with one ``out-leg'', that is a half-edge of capacity 1 attached
  to an $A$-vertex.
\item a graph 
$L_{pq}^{\rm i}$
%$L_{pq}^{-12}$, 
(i stands for ``internal node''),
where we have the same segments as above, but the leftmost cell of
each segment is treated differently and produced an extra leg.
 %% partition the segments of
 %%  length $pa$ (for $A$-vertices) and $qb$ (for $B$-vertices), with an
 %% offset of $-1$ on the left (and thus of $2$ on the right).  This
produces a graph with two ``out-legs'', and one ``in-leg'', that is a
half-edge of capacity 1 attached to a $B$-vertex.
\item a graph 
$L_{pq}^{\rm l}$
% $L_{pq}^{-10}$, 
(l stands for ``leaf''),
where we partition the segments of length
  $pa'$ (for $A$-vertices) and $qb'$ (for $B$-vertices), with an offset
  of $-1$ on the left (and thus of $0$ on the right). This produces a
  graph with one ``in-leg''.
\end{itemize}
These three graphs are illustrated below, for our case example
$(p,q)=(7,4)$, for which we have 
$7\cdot 3-4\cdot 5=1$ and $7\cdot 1-4\cdot 2=-1$.
% --------- root
\begin{center}
\setlength{\unitlength}{10pt}

% \else [...TikZ\ code...] \fi
\end{center}
Note that the graph $L_{pq}^{\rm i}$ is connected whenever $q>1$,
because the capacity of the left-most internal edge (not a leg) is
exactly $q-1$ (and this is a crucial difference with the case where
$L_{pq}=L_{p1}$ is a star graph).

These three graphs have schematically the form
\[
%
\]
and the rule is that white-disk--legs must be paired with
gray-square--legs. 

Now, consider the set of planar binary trees, rooted at a leaf
(notoriously counted by Catalan numbers). These trees have $m$
internal vertices, $m+1$ non-root leaves and one root leaf, and can be
realised by assemblying the elementary units $L_{pq}^{\rm r,i,l}$
accordingly. Due to the values $a$, $b$, $a'$ and $b'$ from the
B\'ezout Lemma, the resulting graph is a tree with $(m+1)p$ white-disk
vertices and $(m+1)q$ gray-square vertices. This proves the existence
of extremal transportation matrices of all sizes for all cases of $p$
and $q$ coprime with $p>q>1$, and all sizes $m+1\geq 1$, as was to be
proven.\footnote{In fact, up to simple reasonings in light of the
  theory of Otter trees, we can deduce that there are exponentially
  many of them for large~$m$.}
\end{proof}

\noindent
Now, let us recall a rather evident subadditivity property of the
Optimal Transport Problem (namely, the subadditivity of the optimal
cost seen as a function of the unnormalized marginal measures):
\begin{rmk}[Subadditivity]
For a given value of $n$ and $m$, and a given $n \times m$ matrix of
costs $W$, let $T^{(1)}$ be a transport matrix for 
$(\va^{(1)}, \vb^{(1)})$ and $T^{(2)}$ a transport matrix for
$(\va^{(2)}, \vb^{(2)})$. Then $T^{(1)} + T^{(2)}$ is a transport
matrix for $(\va^{(1)} + \va^{(2)}, \vb^{(1)} + \vb^{(2)})$.
In particular, as, if 
$T^{(1)}$ is optimal for $(W;\va^{(1)}, \vb^{(1)})$
and
$T^{(2)}$ is optimal for $(W;\va^{(2)}, \vb^{(2)})$,
the matrix
$T^{(1)} + T^{(2)}$ may or may not be optimal for
$(W;\va^{(1)} + \va^{(2)}, \vb^{(1)} + \vb^{(2)})$, we have
\be
\label{eq.subadd1}
\begin{split}
&
\cW_{\star}
(W;\va^{(1)}, \vb^{(1)})
+
\cW_{\star}
(W;\va^{(2)}, \vb^{(2)})
=
\cW
(T^{(1)};W)
+
\cW
(T^{(2)};W)
\\
&\quad
=\cW
(T^{(1)}+T^{(2)};W)
\geq 
\cW_{\star}
(W;\va^{(1)} + \va^{(2)}, \vb^{(1)} + \vb^{(2)})
\ef;
\end{split}
\ee
with equality if and only if $T^{(1)}+T^{(2)}$ is optimal,
and more generally
  \begin{equation}
    \sum_{k = 1}^\ell \mathcal{W}_{\star} \left( W, \va^{(k)} , \vb^{(k)} 
    \right) \geqslant 
\mathcal{W}_{\star} \left(W, 
\va^{(1)}+\cdots+\va^{(\ell)}
, 
\vb^{(1)}+\cdots+\vb^{(\ell)}
\right)
% \mathcal{W}_{\star} \left( W, \sum_{k = 1}^l \va^{(k)}, \sum_{k = 1}^l \vb^{(k)} \right)
\ef. 
\label{sub}
  \end{equation}
\end{rmk}
\noindent
This suggests to introduce the definition
\begin{defn}[Perfect decomposition]
Given $\ell+1$ pairs of capacity vectors,
$(\va, \vb)$ and $\{(\va^{(k)}, \vb^{(k)})\}_{1 \leq k \leq \ell}$,
we say that the $\ell$-tuple $\{(\va^{(k)}, \vb^{(k)})\}_{1 \leq k \leq \ell}$
is a \emph{decomposition} of $(\va, \vb)$ if
\begin{align}
\va &= \sum_{k = 1}^\ell \va^{(k)}
\ef;
&
\vb &= \sum_{k = 1}^\ell \vb^{(k)}
\ef.
\end{align}
We say that it is a \emph{perfect decomposition} if the inequality
(\ref{sub}) holds with equality, and write
\be
\label{eq.3876584}
\{(\va^{(k)}, \vb^{(k)})\}_{1 \leq k \leq \ell}
\xleftarrow{\mathrm{p.d.}}
(\va, \vb)
\ef.
\ee
\end{defn}
\noindent
This notion has a form of associativity, as we have the equation
(\ref{eq.3876584}) above iff
\begin{gather}
\{
(\va^{(1)}+\va^{(2)}, \vb^{(1)}+\vb^{(2)}), (\va^{(3)}, \vb^{(3)}),
\ldots, (\va^{(\ell)}, \vb^{(\ell)})\}
\xleftarrow{\mathrm{p.d.}}
(\va, \vb)
\\
\label{eq.3876584b}
\{(\va^{(1)}, \vb^{(1)}), (\va^{(2)}, \vb^{(2)})\}
\xleftarrow{\mathrm{p.d.}}
(\va^{(1)}+\va^{(2)},\vb^{(1)}+\vb^{(2)})
\end{gather}
Furthermore, it has a notion of ``linearity'': if the equation
(\ref{eq.3876584b}) above holds, then
\be
\label{eq.3876584c}
\{(\lam_1\va^{(1)}, \lam_1\vb^{(1)}),
(\lam_2\va^{(2)}, \lam_2\vb^{(2)})\}
\xleftarrow{\mathrm{p.d.}}
(\lam_1\va^{(1)}+\lam_2\va^{(2)}, \lam_1\vb^{(1)}+\lam_2\vb^{(2)})
\ee
for all $\lam_1, \lam_2 \in\bR^+$.
A corollary of the fact seen above, that the extremal transport
matrices are those with acyclic support graphs, is the following:
\begin{corr}
\label{corr.perfOIforest}
  For $W$ a generic matrix of costs, the decomposition 
$\{(\va^{(k)}, \vb^{(k)})\}$ of $(\va, \vb)$ is perfect only if
  the forests $F^{(k)}=H(T_{\star}(W, \va^{(k)}, \vb^{(k)}))$,
  support of the optimal transport matrices for the capacities
  $(\va^{(k)}, \vb^{(k)})$, are such that $F:=F^{(1)} \cup \cdots \cup
  F^{(\ell)}$ is a forest.
\end{corr}
\begin{proof}
Indeed, observe that $F=H(T)$, with 
$T=\sum_{k=1}^{\ell}T_{\star}(W,\va^{(k)},\vb^{(k)})$.
% $T=T_{\star}(W,\va^{(1)},\vb^{(1)})+\cdots+T_{\star}(W,\va^{(\ell)},\vb^{(\ell)})$.
%
Then, from Lemma \ref{lem.ExtrIfForest} we know that if 
$F$ is not a forest, then $T$ is not optimal, and thus (\ref{sub}) is
a strict inequality.
\end{proof}
\noindent
Note that the other direction of the implication is not true. For example,
for $n=m=2$, we have a counterexample by choosing (again, spaces in
$T$ stand for zeroes)
\be
\label{eq.3876538}
\begin{array}{rlrlrl}
W
&\hspace*{-2.5mm}=
\begin{pmatrix}
-1 & 0 \\
0 & 0
\end{pmatrix}
&
(\va,\vb)&\hspace*{-2.5mm}=((1,3),(1,3))
&
T_{\star}&\hspace*{-2.5mm}=
\begin{pmatrix}
1 & \pz  \\
\pz  & 3
\end{pmatrix}
\\
&&
(\va^{(1)},\vb^{(1)})&\hspace*{-2.5mm}=((0,2),(1,1))
&
T^{(1)}&\hspace*{-2.5mm}=
\begin{pmatrix}
\pz & \pz \\
1 & 1
\end{pmatrix}
\rule{0pt}{17pt}
\\
&&
(\va^{(2)},\vb^{(2)})&\hspace*{-2.5mm}=((1,1),(0,2))
&
T^{(2)}&\hspace*{-2.5mm}=
\begin{pmatrix}
\pz  & 1 \\
\pz  & 1
\end{pmatrix}
\rule{0pt}{17pt}
%% \\
%% T^{(1)}+T^{(2)}=
%% \begin{pmatrix}
%% 0 & 1 \\
%% 1 & 2
%% \end{pmatrix}
\end{array}
\ee 
Also, the other direction does not hold even if we add the further
hypothesis that the $\ga_i^{(k)}$'s and $\gb_j^{(k)}$'s are all
non-zero, as shown by a larger counter-example in
Figure~\ref{fig.controesABpos}.

\begin{figure}[t]
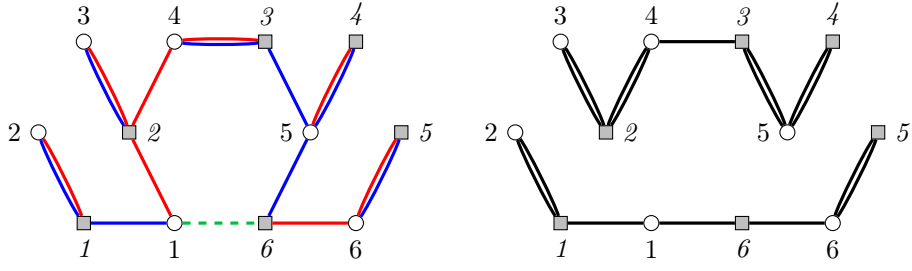

\begin{center}
\if\faifig1  
%
\else [...TikZ\ code...] \fi
\end{center}
\caption{\label{fig.controesABpos}% 
  Counterexample to the possibility that $F^{(1)}\cap F^{(2)}$ being a
  forest may imply that 
  $\{(\va^{(1)}, \vb^{(1)}), (\va^{(2)}, \vb^{(2)})\}$ is a perfect
  decomposition, under the further assumption that the $\ga_i^{(k)}$
  and $\gb_j^{(k)}$ are all nonzero.  Left: red and blue denote
  $F^{(1)}$ and $F^{(2)}$, all the costs of edges in $F^{(1)}\cup
  F^{(2)}$ are zero, the cost of the green dashed edge is negative,
  while edges not shown have very large positive costs.  Right: the
  support graph of the optimal configuration for the sum of the two
  capacities.  The multiplicity of an edge denote the corresponding
  (integer) entry of the transport matrix.}
 %% The crucial fact is that
 %%  the central hexagon has monochromatic edges of both colors and both
 %%  parities, invalidating all arguments on alternating cycles.
\end{figure}

We find that the problem of understanding perfect decompositions is
important, and we analyse it in what follows. Let us give a first
characterization of perfect decompositions in terms of the graph
$G(W;\bm{\lambda},\bm{\mu})$:
\begin{lem}
\label{lem.perfdeco1st}
The decomposition $\{(\va^{(k)}, \vb^{(k)})\}_{1 \leq k \leq \ell}$ is
a perfect decomposition for $(\va, \vb)$ and the matrix of costs $W$
if and only if there exists a Hungarian gauge $(\bm{\lambda},\bm{\mu})$
for $(\va, \vb)$, and
optimal assignments $T^{(k)}=T_{\star}(W;\va^{(k)}, \vb^{(k)})$, such
that $H(T^{(k)})\subseteq G(W;\bm{\lambda},\bm{\mu})$ for all 
$1 \leq k \leq \ell$.
\end{lem}
\begin{proof}
Of course, the notion of perfect decomposition is gauge invariant,
i.e.\ the $\ell$-tuple
$\{(\va^{(k)}, \vb^{(k)})\}_{1 \leq k \leq \ell}$ is a perfect
decomposition for $(\va, \vb)$ and the matrix of costs $W$, if and
only if it is perfect for $(\va, \vb)$ and
$W'=\{W_{ij}-\lam_i-\mu_j\}$. As, in a Hungarian gauge
for $(W;\va,\vb)$, we have
$\cW(T_{\star};W')=0$, the only possibility to have
$\cW(T_{\star};W')=\sum_k \cW(T^{(k)};W')$ is that the gauge
$(\bm{\lambda},\bm{\mu})$ is a Hungarian gauge also for all the
instances $(W;\va^{(k)}, \vb^{(k)})$, which is equivalent to our
claim.
\end{proof}
\noindent
It would be interesting to find a characterization that does not
require the detection of a Hungarian gauge.  In light of
Corollary \ref{corr.perfOIforest} (but also of the counterexample in
equation (\ref{eq.3876538}) and Figure \ref{fig.controesABpos}), and
of the intuition that the matching subgraph introduced in the previous
sections may be related to the present framework, we have been led to
the identification of the following fact:
\begin{thm}
\label{thm.critForestCapa}
  Let $\ell, n, m$ be three integers such that $n \geq m,\ell$, let
  $W$ be a generic matrix of costs of size $(n+1)\times m$,
  $(\va,\vb)$ two vectors of capacities with all $\ga_i$ in the range
  $\{1,\ldots,\ell\}$ and $\gb_j$ non-zero integers, multiples of
  $\ell$, such that $\sum_j \gb_j=n\ell$, and
  $\{(\va^{(k)},\vb^{(k)})\}_{1\leq k \leq \ell}$ a decomposition
  of $(\va,\vb)$ of a special form:
%% with the following properties:
%% \begin{itemize}
%% \item
there exists $i_k \in [n+1]$, such that
$\ga^{(k)}_{i}=1-\delta_{i,i_k}$, while $\gb_j^{(k)}=\gb_j/\ell$ for
all $j$ and $k$.\footnote{In other words, each subproblem
  $(W;\va^{(k)},\vb^{(k)})$ is a hyper-Assignment Problem on a
  subgraph $\cK_{n,m} \subset \cK_{n+1,m}$, with all $\ga_i$ equal to
  1, as in Section~\ref{sec.ktuples}.}  Then the decomposition is perfect if
and only if the optimal hyper-assignments 
$M^{(k)}=H(T_{\star}(W, \va^{(k)}, \vb^{(k)}))$, support of the
optimal transport matrices for the capacities $(\va^{(k)},\vb^{(k)})$,
are such that $F:=M^{(1)} \cup \cdots \cup M^{(\ell)}$ is a forest.
\end{thm}
\begin{proof}
% Observe 
Note that the $i_k$'s are not all equal, otherwise we would have
$\ga_{i_1}=0$.  Furthermore, we can assume w.l.o.g.\ that the $i_k$'s
are all distinct, because two sets of capacities that are proportional
have a trivial role in a decomposition, in light of the linearity
property of equation (\ref{eq.3876584c}). In other words, we may make
a proof by contradiction, and consider a counterexample of minimal
size, where the order is given by the pair $(\ell,n)$ (that is,
$(\ell,n)\preceq (\ell',n')$ if $\ell\leq \ell'$ and, when
$\ell=\ell'$, we also have $n \leq n'$).  By what we said above, we
deduce that on a minimal counterexample the $i_k$'s must be all
distinct.

Again, let us define $F=H(T)$, with 
$T=\sum_k T_{\star}(W,\va^{(k)},\vb^{(k)})$.
The `only if' part is the statement of Corollary
\ref{corr.perfOIforest}, so we pass directly to the `if' part.

In parallel to the original instance, we could consider also the
framework of the hyper-Assignment Problem where the $B$-vertices with
$\gb_j=k \ell$ are replaced by $k$-tuples of vertices, with costs
perturbed infinitesimally (that is, the ``splitting'' procedure
discussed in Section \ref{sec.ktuples}). This framework allows to
prove easily, as always using the generality hypothesis on the costs,
that the symmetric difference of optimal hyper-assignments,
$M^{(k)}\symdif M^{(k')}$, must consist of a single alternating path
with endpoints $a_{i_k}$ and $a_{i_{k'}}$.

We introduce a notion of `colors', and sets $\calC(e)$, similar to
those introduced in equation (\ref{eq.2876587653}) for the study of
matching subgraphs.
More precisely, we associate to the edges $e$ of $\calK_{n+1,m}$
%=\edge{a_i}{b_j}$
% of $F$ the (non-empty) 
the sets $\calC(e) \subseteq \{1,\ldots,\ell\}$ such that 
$k\in \calC(e)$ iff
% $T^{(k)}_{ij}>0$, i.e.,  if 
$e\in M^{(k)}$.
% , while we set $\calC(e)=\emptyset$ if $e\not\in E(F)$.  

Note that there exists no edge $e=\edge{a_i}{b_j}$ such that 
$\calC(e) = \{1,\ldots,\ell\}$, otherwise we could drop $a_i$ from our
set of $A$-vertices, decrease the value of $\gb_j$ by one, and get a
counterexample with the same value of $\ell$ and a smaller value of
$n$, violating the minimality hypothesis.

Now, suppose that $F$ is a forest with two or more components. 
As $M^{(k)}\symdif M^{(k')}$ contains an alternating path
with endpoints $a_{i_k}$ and $a_{i_{k'}}$ (and in fact coincides with
this path), we deduce that the $\ell$ vertices 
$\{a_{i_k}\}_{1 \leq k \leq \ell}$ are all in the same component.
Thus, any other non-trivial component would contain an edge $e$ with 
$\calC(e) = \{1,\ldots,\ell\}$, an eventuality that we have just
excluded. Thus $F$ is in fact a tree.

Now, we know that there exists a gauge
$(\bm{\lambda},\bm{\mu})$ 
%% with 
%% % $G(\bm{\lambda},\bm{\mu})\supseteq F$,
%% $G(\bm{\lambda},\bm{\mu})=G$, and thus so 
such that $W'_{ij}=0$ on all the
edges of $F$ (and thus, in all the edges which are in some of the
$M^{(k)}$'s). This gauge is obtained
by solving all the equations $W_{ij}-\lam_i-\mu_j=0$, for
$\edge{a_i}{b_j}\in E(G)$, and setting $\lam_1=0$. We shall prove
that this gauge is Hungarian for the instance $(W;\va,\vb)$
(and all the $(W;\va^{(k)},\vb^{(k)})$ as well), thus
implying that $T$ is optimal.  In other words, we have to prove that,
for all edges $e=\edge{a_i}{b_j}$ not in $F$,
$W'_{ij}\geq 0$ (and in fact, from the generality hypothesis, that
$W'_{ij}>0$). We do this by contradiction, and assume that $e$ is an
edge with $W'_{ij}<0$.
As $F$ is a tree, adding $e$ to $F$ produces a unicyclic graph, with
cycle $C=(e,e_1,\ldots,e_{2h-1})$ of even length (and $h\geq 2$). We
claim that
$\Gamma:=\calC(e_1)\cap\calC(e_3)\cap\cdots\cap\calC(e_{2h-1})\neq\emptyset$.
Suppose that this is the case, and let $k$ be an element of $\Gamma$.
% in this set. 
Then
%, if $W'_{ij}<0$, 
we would get a contradiction with the optimality
of $M^{(k)}$, from Proposition \ref{prop.swapcyc} and the fact that
$\walt(C)=-W'_{ij}>0$, and conclude.

So we are only left with proving our claim that
$\Gamma \neq \emptyset$.  Say that $k$ is such that 
% $i_k \not\in \calC(e)$.
$i_k \neq i$, where $i$ is the index of the endpoint $a_i$ of the edge
$e$.  Such an index must exist, because the $i_k$'s are not all equal.
% and $\calC(e)$ is not the full set of indices.
Then there exists (exactly) one edge $e'\in F \setminx e$ incident to
$a_i$ and with $k\in\calC(e')$.  Now, 
$\calC(e') \subsetneq \{1,\ldots,\ell\}$, and there must exist an
element $k' \not\in \calC(e')$.  So there is a subgraph of $F$
consisting of the edges $e$ such that $\calC(e)\cap\{k,k'\}$ has
cardinality 1, which has a connected component containing $a_i$. 
On the splitted instance,
%% Our reasoning above (on the perturbed instance)
%  allows to select a path on 
%% implies that 
this subgraph is a path alternating in colors $k$ and
$k'$, going through $a_i$ and with endpoints $a_{i_k}$ and
$a_{i_{k'}}$. Going back to the merged instance, this path remains a
simple path (i.e., the merging procedure does not produce
self-intersections), in light of Lemma~\ref{lem.hypermatchNOiji}.

Let us repeat the argument on $b_j$. There must be exactly $\gb_j$
distinct edges
% $e''$ 
incident to $b_j$ such that their sets $\calC$ contain $k$.
% with $k\in\calC(e'')$. 
Let us take one of these, and call it $e''$. Suppose that
$k' \not\in \calC(e'')$.  Again, in this case there would exist an
alternating path $P''$, in colors $k$ and $k'$, going through
$b_j$. But this second path has the same endpoints as $P'$, so these
two paths must coincide, and the portion of this path between $a_i$
and $b_j$ must be
$C \setminx e$, and, as a result, one among $k$ and $k'$ is in
$\Gamma$. If instead $k' \in \calC(e'')$, let $k''$ be some element
not in $\calC(e'')$ (this must exist again because no edge contains
all colors), and define $P''$ as the path, alternating in colors $k$
and $k''$, going through $b_j$. The paths $P'$ and $P''$ share an
endpoint, because the endpoints of $P''$ are $a_{i_k}$ and
$a_{i_{k''}}$. Call $P'''$ the unique open path of
$M^{(k')}\symdif M^{(k'')}$, thus with endpoints $a_{i_{k'}}$ and
$a_{i_{k''}}$. 

In fact, as by assumption $F$ is a tree, and all pairs of our three
paths share one endpoint, their union must be either one path, or a
$Y$-shaped diagram, similar to the one appearing in the proof of our
crucial Lemma~\ref{lem:tripla}.  

It is easily seen that, for a $Y$-shaped graph
% with leaves $x_1$, $x_2$ and $x_3$, 
any two vertices $v_1$ and $v_2$ must be contained in at least one of
the three paths connecting the leaves pairwise.

Applying this fact to the vertices $a_i$ and $b_j$, we deduce that at
least one among the paths $P'$, $P''$ or $P'''$ must contain both vertices,

Thus the path 
$C \setminx e$ is a portion of one of the paths $P'$, $P''$ or $P'''$,
and thus is alternating in two among our three colors $k$, $k'$ and
$k''$, so that one of the three colors must be in~$\Gamma$.
\end{proof}

\noindent
Note that in the theorem above we have relaxed the hypotheses at the
best of our possibilities. For example, even just in the setting of
the ordinary Assignment Problem, if we have an underlying graph
$\cK_{n+2,n}$ instead of $\cK_{n+1,n}$, and for the rest essentially
the same hypotheses, one can easily identify a counterexample:
\be
\label{eq.3876538b}
\begin{split}
% \begin{array}{rlrlrl}
W
&=
\begin{pmatrix}
-1 & 0 \\
0 & 0 \\
0 & 2 \\
1 & 0
\end{pmatrix}
\qquad
\begin{array}{rl}
(\va,\vb)&\hspace*{-2.5mm}=((1,2,2,1),(3,3)) \\
(\va^{(1)},\vb^{(1)})&\hspace*{-2.5mm}=((1,0,1,0),(1,1)) \\
(\va^{(2)},\vb^{(2)})&\hspace*{-2.5mm}=((0,1,1,0),(1,1)) \\
(\va^{(3)},\vb^{(3)})&\hspace*{-2.5mm}=((0,1,0,1),(1,1))
\end{array}
\\
T_{\star}&=
\begin{pmatrix}
1 & \pz \\
\pz & 2 \\
2 & \pz \\
\pz & 1
\end{pmatrix}
\quad
T^{(1)}=
\begin{pmatrix}
\pz & 1\\
\pz & \pz \\
1 & \pz \\
\pz & \pz 
\end{pmatrix}
\rule{0pt}{29pt}
\quad
T^{(2)}=
\begin{pmatrix}
\pz & \pz \\
\pz & 1 \\
1 & \pz \\
\pz & \pz 
\end{pmatrix}
\rule{0pt}{29pt}
\quad
T^{(3)}=
\begin{pmatrix}
\pz & \pz \\
1 & \pz \\
\pz & \pz \\
\pz & 1 
\end{pmatrix}
\rule{0pt}{29pt}
\end{split}
\ee
and we verify immediately that $T_\star \neq T^{(1)}+T^{(2)}+T^{(3)}$
(for example, they differ in the top-left entry).

Now we are ready to show the connection between our two problems:
\begin{corr}
\label{corr.HnewIsHold_ass}
Let $W$ be a matrix of costs $(n+1)\times n$, and consider the
`standard setting' of Corollary \ref{cor.setting1}. Call $H_\calJ$ the
resulting matching subgraph.  Now consider the instance $(W;\va,\vb)$
of the Hitchcock's Assignment Problem, with $W$ as above, $\ga_i=n$
for all $i$ and $\gb_j=n+1$ for all $j$, call $T_{\star}$ the optimal
transport matrix, and $H_\star=H(T_{\star})$ its support.  Then 
$H_\calJ = H_\star$, and, more precisely
$(T_{\star})_{ij}=|\calC(\edge{a_i}{b_j})|$, where the $\calC(e)$'s
are the sets associated to the construction of the matching subgraph
(i.e.\ the set of colors $U\in \calJ$ such that $e \in M_U$).
\end{corr}
\noindent
and its hyper-Assignment generalization in the sense of
Section~\ref{sec.ktuples}:
\begin{corr}
\label{corr.HnewIsHold_hass}
Let $(W;\va,\vb)$ be an instance of the Hitchcock's Assignment
Problem, with $W$ of size $(n+1)\times m$, of the form discussed in
Theorem~\ref{thm.critForestCapa} above, with $\ell=n+1$ and
$i_k=k$.\footnote{In particular $\ga_i=n$ for all $i$ and
  $\gb_j=(n+1)\tilde{\gb}_j$ for some positive integers
  $\tilde{\gb}_j$ such that $\sum_{j=1}^m \tilde{\gb}_j=n$, note that
  these vectors of capacities are generic.}  Then the optimal
transport matrix $T_{\star}$ has support graph $H(T_{\star})$ that
coincides with the matching subgraph $H_{\calJ}$ for the
`hyper-assignment standard
setting' of
%  Corollary \ref{cor.setting1} and its generalization of
Theorem \ref{thm.setting1hyper}, namely
$\calJ=\{ \{a_i\} \}_{1 \leq i \leq n+1}$.
% of Section \ref{sec.teoremaserio1}, 
More precisely $(T_{\star})_{ij}=|\calC(\edge{a_i}{b_j})|$, where
the $\calC(e)$'s are the collection of subsets discussed at the end of 
Section~\ref{sec.ktuples}.
\end{corr}
\begin{proof}
Indeed, the matrix $T$ defined by $T_{ij}=|\calC(\edge{a_i}{b_j})|$ is
the sum of the transport matrices associated to the optimal matchings
of the various subproblems on $G|_U$, as in the setting of
Theorem~\ref{thm.critForestCapa}, and in particular
% calling $T=\sum_{k=1}^{n+1} T_{\star}(W;\va^{(k)},\vb^{(k)})$, we have
$H_{\calJ}=H(T)$.  Now, from Theorem~\ref{thm:Jtree} we know that the
graph $H_{\calJ}$ is a tree.
%, that was called $H_{\calJ}$ in Section \ref{sec.teoremaserio1}.
As the hypotheses of
Theorem~\ref{thm.critForestCapa} are satisfied,
$T=T_{\star}(W;\va,\vb)$, so that
$H_{\calJ}=H(T_{\star})$.
\end{proof}

\begin{rmk}
Now we have a simple explanation of the fact, determined in
Section~\ref{sec.teoremaserio1}, that all $B$-vertices have degree~2
in $H_{\calJ}$. In fact, if $\min_j(\gb_j)>\max_i(\ga_i)$, for any 
$T \in \calT(\va,\vb)$ the graph $H(T)$ has all $B$-vertices of degree
at least $2$. If $W$ and $(\va,\vb)$ are generic, $H(T)$ is a spanning
tree of $\cK_{n,m}$, so it has $n+m-1$ edges. And, if $n=m+1$ the
number of edges is exactly twice the number of $B$-vertices.
\end{rmk}

\begin{rmk}
Contrarily to Theorem~\ref{thm:Jtree}, that is very rigid, in this
setting there seems to be some space for generalizations.  A typical
situation that may be interesting to consider is when $n=m+k$, with
$k$ ``finite'' and $m$ ``large''. In such a case it is still easy to
produce vectors $(\va,\vb)$ that are generic, and the number of
$B$-vertices not of degree 2 is finite. For example, for
$(n,m)\to(2n+1,2n-1)$ and all $\ga_i$'s equal to $2n-1$, and the
$\gb_j$'s equal to $2n+1$, all $B$-vertices have degree 2, except for
a single vertex that has degree 3. This provides a sensible notion of
projected graph $\bar{H}$, that now is a spanning (hyper-)tree with
up to $k$ hyperedges that are not ordinary edges. These hyperedges
subdivide the tree into up to $2k-1$ portions. The main difference
with the case $k=1$ is that we do not have a general recipe to
describe a perfect decomposition in this situation (that is, we do not
have an analogue of Theorem~\ref{thm.critForestCapa} for $n-m>1$, note
in particular the counterexample in equation~(\ref{eq.3876538b})). 
\end{rmk}

\noindent
The setting of the remark above will be described in more detail
elsewhere, for the time being we shall content ourselves of just
providing one example, in Figure~\ref{fig.exCapa2tri}.

\begin{figure}[t]
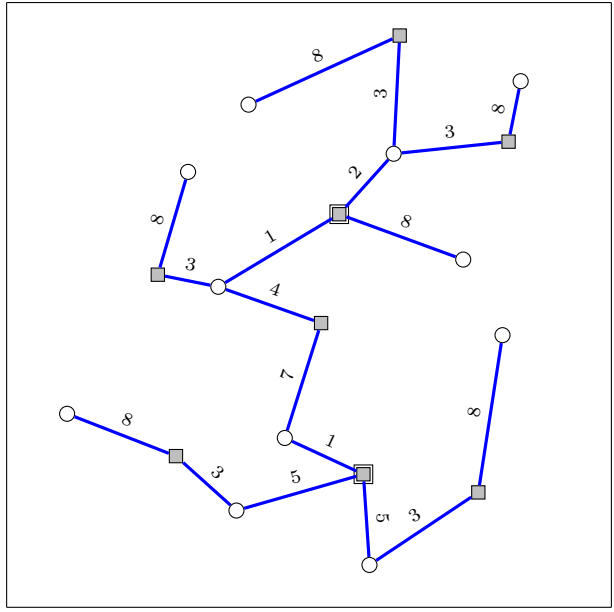

\begin{center}
\if\faifig1

\else [...TikZ\ code...] \fi
\end{center}
\caption{\label{fig.exCapa2tri}% 
  Example of an instance of the Hitchcock's version of the Assignment
  Problem, with capacities. The $a_i$'s are circles, while the $b_j$'s
  are squares. Here $n=11$, $m=8$, $\ga_i=8$ for all $i$ and
  $\gb_j=11$ for all $j$. The costs $W_{ij}$ are given by the square
  distance among the vertices of different classes. For what we have
  seen above, the support graph $H(T_{\star})$ must be a tree (because
  the capacities are generic), and all $B$-vertices must have degree
  2, except for two vertices of degree 3 (as is the case here), or one
  vertex of degree 4, because $\min_j \gb_j>\max_i\ga_i$ (so that
  there are no $B$-vertex leaves), and the total number of edges is
  $n+m-1=2m+2$ (so $\sum_j (\deg b_j-2)=2$).  These two vertices have
  been marked in the drawing. The values of $T_{ij}$ are shown next to
  the corresponding edges.}
\end{figure}

%%%%%%%%%%%%%%%%%%%%%%%%%%%%%%%%%%%%%%%%%%%%%%%%%%%%%%%
\appendix
\renewcommand{\sectionname}{}
% \begin{appendices}
%%%%%%%%%%%%%%%%%%%%%%%%%%%%%%%%%%%%%%%%%%%%%%%%%%%%%%%

%%%%%%%%%%%%%%%%%%%%%%%%%%%%%%%%%%%%%%%%%%%%%%%%%%%%%%%
\section{\texorpdfstring{More on the non-crossing property at $p=2$}{More on the non-crossing property at p=2}}
%%%%%%%%%%%%%%%%%%%%%%%%%%%%%%%%%%%%%%%%%%%%%%%%%%%%%%%

\noindent
In this section we provide some supplementary material to
Sections~\ref{sec.caso2dp2} and~\ref{sec.ktuplesGeom}. Although the
proofs of all relevant theorems have already been provided in the main
body of this manuscript, we find it worthwhile to describe some
further structural properties and alternative proofs. While these may
be longer and less direct, they offer a more explicit illustration of
certain mechanisms and could serve as inspiration for future
developments.

%-------------------------------------------------------
\subsection{Further properties of valid matrices}
\label{app.altproofValid1st2nd}
%-------------------------------------------------------

\noindent
In this section we discuss some more properties of the family of
$(n+1)\times n$ matrices $W$ which are valid in the sense of
Definition~\ref{def.validmat}. We recall that, in light of
Proposition~\ref{prop.crossifvalid}, these matrices play a prominent
role in the proof of Theorem~\ref{th:treenoncross} and its variants.

We start with a simple observation, in the form of a sufficient
condition for validity.
\begin{rmk}
\label{rmk.validifdiagzeroes}
% There exists a simple criterium for $W$ to be a valid matrix: 
If for each $i$ we have $W_{ii}=W_{i+1\,i}<W_{ki}$ for all 
$k\neq i,i+1$, we have that, after the gauge transformation with
$\lam_i=0$ and $\mu_i=W_{ii}$, $W$ is a matrix with zeroes on the
diagonals of the minors $W^{\{1\}|\varnothing}$ and
$W^{\{n+1\}|\varnothing}$, and strictly positive elsewhere, so that in
particular the minor $W^{\{1\}|\varnothing}$ has strictly-positive
entries in its top triangular part, and $W^{\{n+1\}|\varnothing}$ has
strictly-positive entries in its bottom triangular part, this implying
that the $W(1,\pi)$'s and $W(n+1,\pi)$'s are strictly positive for all
$\pi$'s distinct from the identity, and thus that $W$ is valid.
\end{rmk}
\noindent
A more remarkable property of valid matrices is that a smaller set of
inequalities are sufficient to imply validity. In other words, we will
produce a convenient necessary and sufficient condition for validity.

Definition~\ref{def.validmat} describes \emph{valid} matrices.  Let us
say, within this section, that matrices satisfying the conditions of
this definition are \emph{valid in the first sense}, and introduce the
variant:
\begin{defn}
\label{def.valid2nd}
A matrix $W \in \bR(n+1,n)$ is \emph{valid in the second sense} if
%, for $k=1$ and $k=n+1$, and 
for all single-cycle permutations\footnote{We call
\emph{single-cycle} a permutation with a unique non-singleton cycle.}
$\pi \in \mathfrak{S}_n$ 
% with a unique non-singleton cycle,
%
\begin{align}
W(1,\pi)&>0,
&
W(n+1,\pi)&>0.
%% \sum_{i=1}^n
%% \left(
%% W^{\{k\}|\varnothing}_{ii}
%% -
%% W^{\{k\}|\varnothing}_{i\,\pi(i)}
%% \right)
\end{align}
\end{defn}
\noindent
Then, as our choice of names suggest,
\begin{thm}
\label{th:valid12}
The matrix $W$ is valid in the first sense iff it is valid in the second sense.
\end{thm}
Of course, validity in the first sense implies validity in the second
sense. Also, it is obvious that validity in the first sense is
equivalent to the same notion, where only single-cycle permutations
%  with a unique non-singleton cycle 
are considered, as a permutation with a different
cycle type is obtained by swapping the cycles one by one.

If we adopt a cycle notation for permutations, and the symbol 
``${\rm s}$'' to denote the completion of a list of non-singleton
cycles to a permutation (e.g.,
$\mathfrak{S}_9 \ni \pi=((164)(37){\rm s})=((164)(2)(37)(5)(8)(9))$),
the observation above reads
\be
W(k,(\gamma_1\,\gamma_2\,\ldots\,\gamma_c\, {\rm s}))
=
W(k,(\gamma_1\, {\rm s}))
+
W(k,(\gamma_2\, {\rm s}))
+ \cdots
+
W(k,(\gamma_c\, {\rm s}))
,
\ee
so that the positivity of the summands on the RHS implies the
positivity of the LHS.  

We will provide two proofs. A first one, given below, is shorter, but
more ``abstract''. A more explicit and constructive proof is given
later on.
% in Appendix~\ref{app.altproofValid1st2nd}.

\begin{proof}[First proof of Theorem \ref{th:valid12}]
As we have seen in the proof of Proposition~\ref{prop.crossifvalid},
in light of Remark~\ref{rmk.reconstr} we get that, if $W$ is valid in
the first sense, then the edges $e_{j}=(x_j, y_j)$ have
$\calC(e_j)=\{j+1,j+2,\ldots,n+1\}$, while the edges 
$e'_{j}=(x_j, y_{j+1})$ have $\calC(e'_j)=\{1,2,\ldots,j\}$.

Validity in the second sense gives the apparently weaker facts that
the edges $e_{j}=(x_j, y_j)$ have
$n+1 \in \calC(e_j)$, while the edges 
$e'_{j}=(x_j, y_{j+1})$ have $1\in \calC(e'_j)$. However this already
implies that the whole path $(x_1,y_1,x_2,y_2,\ldots,x_{n+1})$ is in
$H_\calJ$, and, as we know from the main Theorem \ref{thm:Jtree} that
$H_\calJ$ is a tree, we can infer that the whole sets $\calC(e_j)$ and
$\calC(e'_j)$ are as in the case of validity in the first
sense. Indeed, if this were not the case, there would be an index
$2\leq i \leq n$, and an index $j \geq i$, such that $i \not\in
\calC(e'_j)$ (w.l.o.g.\ up to reversing the path). That is, in
$M_\star(G_i)$ the vertex $x_j$ is matched to some $y_k$, with $k \neq
j+1$, while in $M_\star(G_1)$ the vertex $x_j$ is matched to
$y_{j+1}$.  Now consider the path passing through $x_j$, alternating
in colors $1$ and $i$, and thus containing the chain
$(y_{j+1},x_j,y_k)$. As this path cannot be a cycle, it must be an
open chain with endpoints $x_1$ and $x_i$, and this would produce a
graph with a cycle when including also the path alternating in colors
$1$ and $n+1$ we started from.
\end{proof}

%% Here we provide an alternate proof of Theorem \ref{th:valid12}, on
%% page~\pageref{th:valid12}. Let the basic remarks in the first
%% paragraph of the original proof be understood also now.

\begin{proof}[Second proof of Theorem \ref{th:valid12}]
Our second proof strategy is as follows: 
for all $1<k<n+1$, and all single-cycle permutations 
$\pi \in \mathfrak{S}_n$, we shall exhibit a list of pairs
$\{(k_i,\pi_i)\}$ such that
\be
\label{eq.98537295}
W(k,\pi)
=
\sum_i W(k_i,\pi_i)
\ef.
\ee
This induces a directed graph in the set of of pairs $(k,\pi)$, where
the $\pi$'s are single-cycle permutations, by putting an oriented edge
from $(k,\pi)$ to $(k',\pi')$ if $(k',\pi')$ appears on the RHS of the
equation for $(k,\pi)$. Pairs $(1,\pi)$ and $(n+1,\pi)$ are sinks of
this graph, as they are the pairs associated to the definition of
validity in the second sense. At our aims it is sufficient to prove
that this directed graph is in fact a poset.

We shall establish the existence of a function $\ell:(k,\pi)\to \bN$
such that
$\ell(k_i,\pi_i)\leq \ell(k,\pi)$, and then provide a suitable control
of the cases where the inequality is strict.

\emph{A priori} an equally valid proof strategy would have been to
have pairs $(k_i,\pi_i)$ on the RHS of the decomposition above where
the $\pi_i$'s are not necessarily single-cycle, and then, if
$(k',\gamma_1\,\ldots\,\gamma_c\,{\rm s})$ is on the RHS of the
equation for $(k,\pi)$, put a directed edge from $(k,\pi)$ to all of
the $(k',\gamma_j\,{\rm s})$. Nonetheless, it will come out that, in
our construction, on the RHS of (\ref{eq.98537295}) we will always
have at most two pairs, and that the $\pi_i$'s will be single-cycle
permutations.

Our function of interest is the ``non-trivial length'' of the
cycle of $\pi$, that is 
$\ell(\pi=(\gamma\,{\rm s}))=\ell(\gamma)$, and the non-trivial length
$\ell(\gamma)$ of a cycle $\gamma$ is calculated as follows: after
renaming the elements of the cycle in the canonical way (e.g.,
$(1,6,4,7)\to(1,3,2,4)$), $\ell(\gamma)$ is the number of $i$ such
that $\pi(i) \neq i \pm 1$.
Cycles of length 2 have non-trivial length $\ell((12))=0$, while the
only cycles with non-trivial length 1 are those of the form
$(123\cdots n)$ and their inverses, for $n \geq 3$.

We will prove separately that our claim holds for single-cycle
permutations $\pi$ with cycle of length 2, that is, with non-trivial
length 0, then, for single-cycle permutations $\pi$ with cycle of
length larger than 2, and non-trivial length $\ell$, we will produce
permutations $\pi_i$ where the cycle $\gamma_i$ has at most
$\ell(\gamma_i)=\ell$, and, provably, ultimately strictly smaller than
$\ell$.

We shall adopt a graphical representation for pairs $(k,\pi)$. We will
represent the $+1$ and $-1$ summands in the alternating sum
$W(k,\pi)$ as blue and red dots (respectively) in a rectangular
matrix, and the singletons of $\pi$ with a (smaller) black dot, while
the entries of $W$ not entering the sum will be left empty. In fact,
singletons and empty cells have the same role (namely, no role at all
in the expression for $W$), and the mark on the singletons is there
only to help visualising the whole permutation~$\pi$.

Also, in order to help the eye, we will also draw the lines between
blue and red dots in the same row and column, which will constitute
the cycle of $\pi$.  We will use thick lines separating the cells to
highlight the natural block decomposition of the matrix associated to
$\pi$, and a yellow strip to denote the empty row. We will
occasionally use light yellow to denote rows and columns that are left
empty once the singletons are removed.

Our equalities will turn into equalities for ``sums of diagrams''. As
diagrams are shortcuts for linear combinations of weights, sums of
diagrams must be performed with according rules. In our construction,
it will be sufficient to adopt
the graphical convention that a ``blue dot'' plus a
``red dot'' makes an empty cell, or, if on the diagonal of the main
blocks, a singleton. In describing generic situations at arbitrary
size, we will use white disks to emphasize that a cell is forced to be
empty, and light gray to denote cells whose content is unspecified.

As a first trivial observation, note that singletons adjacent to the
missing rows can be moved freely, that is, in formulas
% , and using a cycle notation for permutations,
$W(k,\cdots(k)\cdots)=W(k+1,\cdots(k)\cdots)$, and in diagrams
\be
\setlength{\unitlength}{10pt}
\thicklines
\raisebox{-27pt}{
}
\ee
so we can standardize singletons in a run adjacent to the missing row
by putting them after the missing row (and just omit singletons not in
a run adjacent to the missing row).  Of course, if, by a sequence of
these operations, we can transform a pair $(k,\pi)$ so that the
missing row is in position $1$ or $n+1$, this pair is connected to a
sink of our abstract digraph, so that, for the remaining pairs, a
generic standardised pair has the form
\[
\setlength{\unitlength}{10pt}
\thicklines
\raisebox{-27pt}{%
}
\]
For cycles of length 2, if we restrict to long-cycle permutations, we
have a single equation
\be
W(2,(12))=W(1,(12))+W(3,(12))
,
\ee
encoded by the graphical identity
\be
\setlength{\unitlength}{10pt}
\thicklines
\raisebox{-12pt}{%
}
\ee
When we have singletons adjacent to the missing row,
% , adopting a cycle notation for permutations,
% with singletons in which only the non-singleton cycles
the appropriate combination becomes
\be
W(2,(1n)(2)(3)\cdots)
=
W(1,(123\cdots n))
+
W(n+1,(1n\,n-1\cdots 2))
\ee
that is, in graphical notation
\be
\setlength{\unitlength}{10pt}
\thicklines
\raisebox{-32pt}{%
}
\ee
This identity is exceptional w.r.t.\ what follows, as the non-trivial
length for the pair $(k,\pi)$ on the LHS is $\ell=0$, while on the RHS
we have two pairs, with $k=1$ and $k=n+1$, and $\ell_i=1>\ell$. In all
the following identities, the values of $\ell_i$ for pairs on the RHS
will never be larger than the one on the LHS.

The mechanism above is the main mechanism of the proof for dealing
with singletons:
% next to the missing row: 
whenever, in absence of singletons next to the missing row, we add and
subtract a blue and a red dot as 
\be
\label{eq.276476534ns}
\setlength{\unitlength}{10pt}
\thicklines
\raisebox{-22pt}{
}
\ee
(the meaning of green lines is explained later on),
then when we also have singletons we add and subtract a ``ladder'' of
blue/red dot pairs, that is
\be
\label{eq.276476534ws}
\setlength{\unitlength}{10pt}
\thicklines
\raisebox{-27pt}{%
}
\ee
Now, for longer cycles, we shall distinguish three cases, schematised below:
\begin{align}
\textrm{type $A_+$:\ }&
\setlength{\unitlength}{10pt}
\thicklines
\raisebox{-22pt}{%
}
\end{align}
that is, let $(h+1,h)$ be the first blue cell in the second
  diagonal block, after a run of $h-k$ singletons. Then 
\begin{description}
\item[type $A_+$] the cell $(h+1,k-1)$ is red, and the cell $(k-1,h)$
  is not red.
\item[type $A_-$] the cell $(k-1,h)$ is red, and the cell $(h+1,k-1)$
  is not red.
\item[type $B$] neither the cell $(h+1,k-1)$ not the cell $(k-1,h)$ are red.
\end{description}
Note that it is not possible that $(h+1,k-1)$ and $(k-1,h)$ are both
red, as otherwise we would have a cycle of length 2, while here we are
considering cycles of length at least 3.

In the case $A_+$, in absence of singletons in $\pi$, we have, schematically,
\be
\setlength{\unitlength}{10pt}
\thicklines
\raisebox{-32pt}{
}
\ee
that is, $W(k,\pi)$ is the sum of $W(k-1,\pi_1)$, where $\pi_1$ is the
transposition $(k-1,k)$, and $W(k+1,\pi)$ (remark: indeed $\pi_2=\pi$
in this situation). Now, the diagram for the pair $(k+1,\pi)$ has the
cell $(k,k+1)$ not red (because the red cell in the same row is in
$(k,k-1)$), so it is not of type $A_-$.  For $A_-$ the situation is
symmetric (w.r.t.\ rotating the diagram 180 degrees): we get a new
pair $(k-1,\pi)$, whose diagram is not of type $A_+$.  Thus, iterated
applications of these rules to diagrams of type $A_+$ or $A_-$ either
ultimately lead to $k=n+1$ or $k=1$, respectively, or give a diagram
of type $B$, with a cycle of the same non-trivial length of the one we
started from (actually, with the very same cycle we started from).

In presence of singletons, the situation is similar, except that now
$\pi_1$ is a long cycle of non-trivial length $1$, while $\pi_2$ is
possibly not equal to $\pi$, but it has the same non-trivial length
(this situation illustrates once again the importance of defining
non-trivial length, instead of trying to devise an induction based on
the ordinary length of the cycles):
\be
\setlength{\unitlength}{10pt}
\thicklines
\raisebox{-27pt}{
}
\ee
%
%% The examples at $n=3$ shown above are of type $A_-$ (and the pairs
%% omitted because evinced by symmetry are of type $A_+$).
These diagrams are a good choice for the explanation of the meaning of
the green lines: on the LHS, the two tones of green denote the two
portions of the cycle of $\pi$ connecting the two highlighted blue
cells, and everything that is not drawn is left as is in $\pi_2$,
which, in particular, has the same set of rows contributing to the
non-trivial length. Note that at the end of these paths some
``dangling edges'' are drawn, going towards parts of the diagram that
haven't been disclosed so far. One edge goes ``west'', while one goes
``south''.  The latter choice is taken only for helping visualisation,
but at this point the first edge may go west or east, and the second
one may go south or north. Besides the graphical representation, this
does not change the algebraic construction of the pairs, so we do not
need a case analysis at this point.

Now repeat the construction for type $B$. In particular, consider the
non-singleton cycle of $\pi$, represented by a green line in our
diagrams, constituted of the concatenation of two open chains: the
path $P_1$, reaching the blue cell $(k-1,k-1)$ horizontally, and the
blue cell $(h+1,h)$ vertically, and the path $P_2$, reaching the blue
cell $(k-1,k-1)$ vertically, and the blue cell $(h+1,h)$ horizontally.

We will construct an equation of the form
$W(k,\pi)=W(k-1,\pi_1)+W(h+1,\pi_2)$, where $\pi_1$ and $\pi_2$ are
single-cycle permutations, with non-trivial length at most $\ell-1$.
The construction is exactly the one illustrated in equations
(\ref{eq.276476534ns}) and (\ref{eq.276476534ws}) (again with the
\emph{caveat} on the fact that the two dangling edges may go west or
east, and south or north, respectively).  The set of rows in the
diagram $(k,\pi)$ contributing to the non-trivial length $\ell$ is
split between $\pi_1$ and $\pi_2$, so that
$\ell(\pi)=\ell(\pi_1)+\ell(\pi_2)$. This is true because $\pi$ is of
type $B$, so that the situation in which rows $k-1$ or $h+1$ do not
contribute to $\ell(\pi)$, but do contribute to $\ell(\pi_2)$ or
$\ell(\pi_1)$, respectively, is excluded (this would happen if $h>k$,
and $(k,\pi)$ were not of type $B$). Also, again because $\pi$ is of
type $B$, we know that $\ell(\pi_1)$ and $\ell(\pi_2)$ are at least 1
(as, in order be 0, they should be a cycle of length 2, and this would
happen if $h=k$, and if $\pi$ were not of type $B$). Thus
$\ell(\pi_1),\ell(\pi_2)<\ell(\pi)$, and we have an induction.
\end{proof}

\noindent
A generic example, in absence of singletons, is shown below:
\be
\setlength{\unitlength}{10pt}
\thicklines
\raisebox{-47pt}{
}
\ee
\medskip
\noindent
The characterization given above is the most relevant result of this
appendix. Nonetheless, a simple sufficient condition for validity that
is sometimes useful is based only on consecutive transpositions (not
necessarily along the diagonal). Indeed we have
\begin{prop}
\label{prop.onlytransp}
If $W_{i\,j}+W_{i+1\,j+1}-W_{i+1\,j}-W_{i\,j+1}<0$ for all $i,j$, then
the matrix $W$ is valid.
\end{prop}
\begin{proof}
% \noindent
By Theorem \ref{th:valid12}, in order to see this it suffices to show
that a matrix $W$ with the property above must be valid in
the second sense. As a result, it is sufficient to show that, for all
square matrices $W$, and all permutations $\pi$, calling
\be
\cW(\pi)=\sum_i (W_{ii}-W_{i\,\pi(i)})
\ee
and $T_{ij}=W_{i\,j}+W_{i+1\,j+1}-W_{i+1\,j}-W_{i\,j+1}$, there exists
$c_{\pi;\,i,j}\geq 0$ such that
\be
\cW(\pi)=\sum_{i,j} c_{\pi;\,i,j} T_{ij}
\ef.
\ee
In fact, these coefficients are related to a classical notion in
Combinatorics.  Consider a diagram associated to a permutation, where
we put a blue dot on a diagonal cell $(i,i)$, that is not a singleton,
and a red dot on a cell $(i,j=\pi(i))$ that is not a singleton.
Define the \emph{height function} of a permutation as the
integer-valued function on the dual cells of the diagram above, valued
0 on the boundary, consistent with the local rules:
\begin{align*}
% B1
\setlength{\unitlength}{10pt}
\thicklines
\raisebox{-12pt}{
}
&
\end{align*}
(in fact, the last two rules are sufficient to determine the height
function univocally). This function is the 
needed quantity $c_{\pi;\,i,j}$. In formulas:
\be
c_{\pi;\,i,j}=\left\{
\begin{array}{ll}
\#\{k\;|\;k \leq i,\; \pi(k)\geq j+1\}
& j \geq i
\\
\rule{0pt}{13pt}%
\#\{k\;|\;k \geq i+1,\; \pi(k)\leq j\}
& j \leq i
\end{array}
\right.
\ee
\end{proof}
\noindent
An example on a permutation consisting of a unique cycle is the following:
\[
\setlength{\unitlength}{10pt}
\thicklines
\raisebox{-44.5pt}{%
}
\]

%-------------------------------------------------------
\subsection{Alternative proof of the non-crossing property 
% for $p=2$ and $\Omega \subset \bR^2$, 
in the bipartite standard setting}
\label{ssec.2ndproofNonCross}
%-------------------------------------------------------
\noindent
Now we use the results of the previous section, to provide an
alternate proof of Theorem~\ref{th:treenoncross}, that is crucially
based on the equivalence between validity in the first and second
sense established above in Theorem~\ref{th:valid12}.

\begin{proof}[Second proof of Theorem~\ref{th:treenoncross}]
The proof will be by contradiction. Let assume that a pair $(X,Y)$ as in the
statement of the theorem does exist, and
that $\ell$ is the minimum value such that this is possible.

Recall that, by gauge invariance, we can use the matrix
$W_{ij}=-\langle x_i,y_j\rangle$ instead of the original matrix of
costs $W_{ij}=|x_i-y_j|^2$.  Due to the invariances discussed at the
end of Section\ \ref{ssec.nccgen}, and the generality hypothesis, we
can also fix three (general) positions determined by the $x_j$'s. We
choose to fix $x_2$, $x_\ell$, and the intersection of the line going
through the segment $[x_1, x_2]$ and the line going through the
segment $[x_{\ell}, x_{\ell+1}]$. Thus we can always assume, without
loss of generality, that
\begin{align}
\label{eq.2876873}
x_1 &= (a,0),
%^\intercal,
&
x_2 &= (-1,0),
%^\intercal,
&
x_\ell &= (0,-1),
%^\intercal,
&
x_{\ell+1} &= (0,b)
%^\intercal
\end{align}
for some $a,b\in\bR$. As we have assumed (by contradiction) that 
the two segments above do cross, and because of the generality
hypothesis,
% The crossing hypothesis implies that segment $( x_1, x_2)$
% intersects the segment $( x_\ell, x_{\ell+1})$ and therefore
we have that both $a$ and $b$ are strictly positive.
% $a>0$ and $b>0$. 
%% On the other hand, the non-degeneracy condition implies that the
%% matrix $W$ on such a path is valid.
That is, a typical configuration (if the theorem did not hold, and the
subgraphs $H_\calJ$ and $\bar{H}_\calJ$ were as depicted) would be the
following:
\[
\if\faifig1  

\else [...TikZ\ code...] \fi
\]
The validity conditions (in the first sense) take a specially simple
form when the stability is checked against a transposition, namely, by
using the shorthands
$\xi_{a,b}\coloneqq x_a-x_b$ and $\eta_{a,b}\coloneqq y_a-y_b$,
the conditions coming from transpositions have the form
\begin{align}
\langle \xi_{i,j},\eta_{i,j}\rangle
&>0,
&
\langle \xi_{i,j+1},\eta_{i,j}\rangle
&>0,
&
\langle \xi_{i+1,j+1},\eta_{i,j}\rangle
&>0.
\end{align}
%
% the validity of the matrix implies a collection of inequalities. 
There exist seven inequalities of this form which involve the
quantities $\{\eta_{i,j}\}$ with $(i,j)=(1,\ell-1)$, $(1,\ell)$ or
$(2,\ell)$,
% $\{ \eta_{i,j}\}_{(i,j)=(1,\ell-1),(1,\ell),(2,\ell)}$.
% and $\{\xi_{i,j}\}_{i=1,2;\,j=\ell,\ell+1}$, 
namely, arranging them in a grid according to the indices of the
$\xi$'s and the $\eta$'s,
\begin{subequations}
\label{eqss.xiij}
\begin{align}
\label{eqs.xi1l}
% \langle \eta_{1,2},\xi_{1,2}\rangle&>0,&
%\langle \eta_{\ell-1,\ell},\xi_{\ell,\ell+1}\rangle>0,\\
%\langle \eta_{2,\ell-1},\xi_{2,\ell}\rangle>0,\\
\langle 
\xi_{1,\ell},\eta_{1,\ell}
\rangle &> 0, &
&&
\langle 
\xi_{2,\ell+1},\eta_{1,\ell}
\rangle &> 0, &
\langle 
\xi_{1,\ell+1},\eta_{1,\ell}
\rangle &> 0,
\\
\label{eqs.xi1lm1}
\langle 
\xi_{1,\ell},\eta_{1,\ell-1}
\rangle &> 0, &
\langle 
\xi_{2,\ell},\eta_{1,\ell-1}
\rangle &> 0,
\\
\label{eqs.xi2l}
&&
\langle 
\xi_{2,\ell},\eta_{2,\ell}
\rangle &> 0, &
\langle 
\xi_{2,\ell+1},\eta_{2,\ell}
\rangle &> 0
\ef.
\end{align}
\end{subequations}
Rewriting the $\xi_{i,j}$'s in light of (\ref{eq.2876873}) gives
\begin{align}
\xi_{1,\ell} &= (a,1),
%^\intercal,
&
\xi_{2,\ell} &= (-1,1),
%^\intercal,
&
\xi_{2,\ell+1} &= (-1,-b),
%^\intercal,
&
\xi_{1,\ell+1} &= (a,-b).
%^\intercal.
\end{align}
We shall read the inequalities (\ref{eqss.xiij}) as linear
inequalities for the $\eta_{i,j}$'s, in terms of the $\xi_{k,m}$'s
seen as parameters. Each inequality constrains some $\eta_{i,j}$ to be
contained inside some half-plane (with the origin on the boundary),
determined by a $\xi_{k,m}$. As a result, each of our $\eta_{i,j}$'s
will be contained in a certain cone (with the vertex at the origin).

For example, the condition 
$\langle \xi_{1,\ell+1},\eta_{1,\ell}\rangle>0$ expresses the
optimality of the perfect matching $M$ containing the edges 
$(x_1, y_1)$ and $(x_{\ell+1}, y_{\ell})$ (among others), with respect
to the pairing $M'$ that has the same edges as $M$, except for the two
edges above, that are replaced by $(x_{\ell+1}, y_1)$ and
$(x_{1},y_{\ell})$, in the graph $G_U$ where $U=\{a_k\}$, for some 
$2 \leq k \leq \ell$. This condition forces the vector $\eta_{1,\ell}$
to have positive projection along $\xi_{1,\ell+1}$, constraining it in
a half-plane.

The last of inequalities (\ref{eqs.xi1l}) constrains $\eta_{1,\ell}$
to be in the half-plane with positive projection on $(a,-b)$. As both
$a$ and $b$ are strictly-positive, this half-plane is contained inside
the three quadrants $(x>0 \vee y<0)$.  For $k \geq 3$, and $u_j$'s
some non-zero vectors in the plane, let us adopt
the symbol $\prec$ in expressions as
\[
u_1 \prec u_2 \prec \cdots \prec u_k
\]
to state that the list of vectors is
in counterclockwise cyclic order, i.e.
\[
\arg(u_1)<\arg(u_2)<\ldots<\arg(u_k)<\arg(u_1)
\ef.
\]
The statement above reads, in these notations,
\be
\label{eq.54w64w643}
(-1,0) \prec  \eta_{1,\ell} \prec (0,1)
\ee
The other two inequalities (\ref{eqs.xi1l}) give, depending on the
sign of $ab-1$,
\begin{subequations}
\begin{align}
\label{eq.54w64w6431}
(1,-a) & \prec  \eta_{1,\ell} \prec (b,-1)
&&
ab>1
\\
\label{eq.54w64w6432}
(-b,1) & \prec  \eta_{1,\ell} \prec (-1,a)
&&
ab<1
\end{align}
\end{subequations}
The case $ab=1$ would be produced by a non-generic instance.
%
% ----------------------------------
%    .....FIGURA CON I CONI.....
% ----------------------------------
% 
\begin{figure}
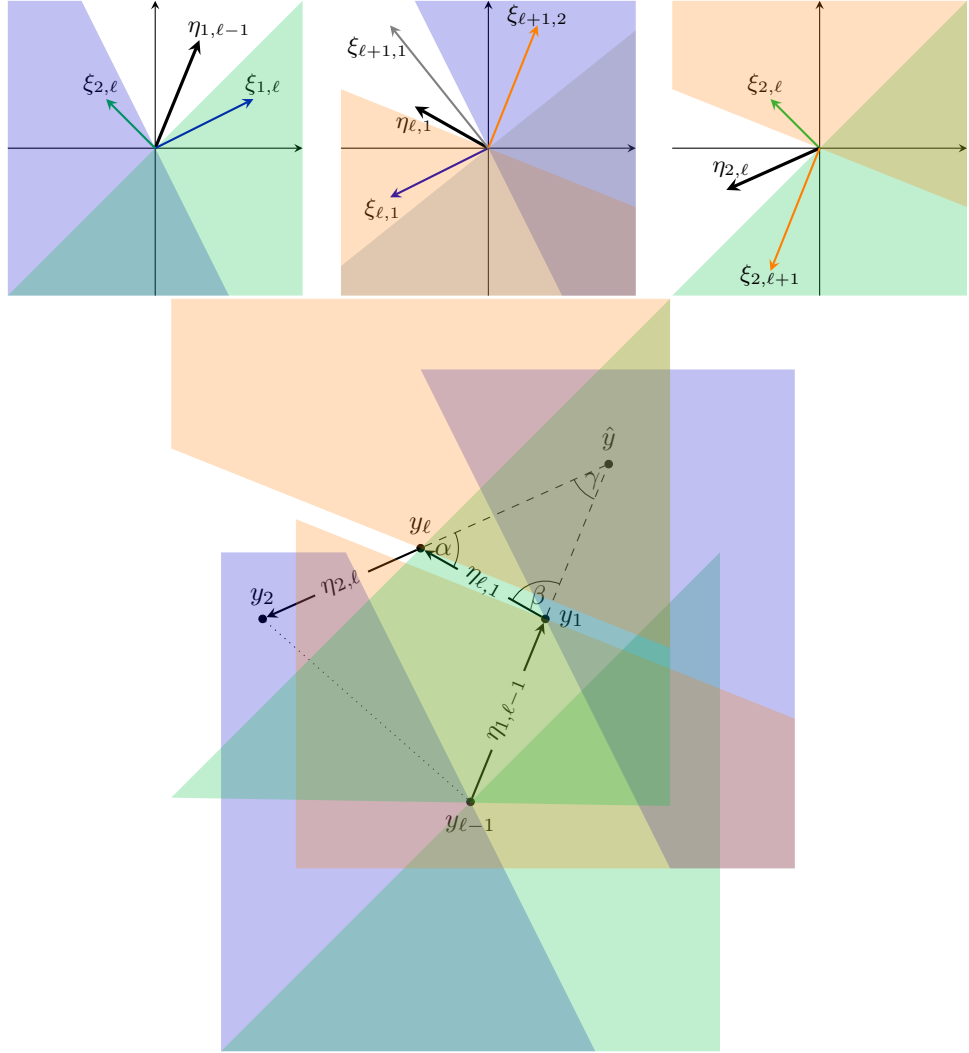

\if\faifig1  
\begin{center}
%gather*}
% fig 1
\makebox[0pt][c]{

\end{center}
%gather*}
\else [...TikZ\ code...] \fi
\caption{\label{fig:prova}Top: geometric construction for the proof of
  Theorem~\ref{th:treenoncross}. Determination of the allowed cones
  for the vectors $\eta_{1,\ell-1}$, $\eta_{\ell,1}$ and
  $\eta_{2,\ell}$, in terms of the vectors $\xi_{i,j}$.  Bottom: the
  cones above are combined in a single plot, showing the generic shape
  of the polygon with vertices $\{y_{\ell-1},y_1,y_{\ell},y_2\}$ (in
  this order).}
\end{figure}

The possibility $ab<1$ is ruled out by the fact that
(\ref{eq.54w64w643}) and (\ref{eq.54w64w6432}) are incompatible, so we
must have $ab>1$, and $\eta_{1,\ell}$ contained in the cone determined
by (\ref{eq.54w64w6431}), that is, taking the opposite vectors (and
using the obvious fact $\eta_{i,j}=-\eta_{j,i}$),
\be
(-b,1)
\prec \eta_{\ell,1} \prec 
(-1,a) 
.
\ee
Then, more easily, the two inequalities in (\ref{eqs.xi1lm1}) imply
\be
(-1,a)
\prec \eta_{1,\ell-1} \prec 
(1,1) 
,
\ee
while those in (\ref{eqs.xi2l}) imply
\be
(-1,-1)
\prec \eta_{2,\ell} \prec 
(-b,1)
,
\ee
(see Figure~\ref{fig:prova} for an illustration).
Combining these facts gives
\be
\label{eq.289764378}
(-1,-1) \prec
\eta_{2,\ell} \prec 
(-1,a) \prec
\eta_{\ell,1} \prec 
(-b,1) \prec
\eta_{1,\ell-1} \prec
(1,1) 
\ef.
\ee
In other words, calling $\alpha$ the angle between $\eta_{2,\ell}$ and
$\eta_{\ell,1}$, $\beta$
the angle between 
$\eta_{\ell,1}$ and $\eta_{1,\ell-1}$, $\gamma'$
the angle between $\eta_{1,\ell-1}$ and $(1,1)$ and
$\gamma-\gamma'$
the angle between 
$(-1,-1)$ and $\eta_{2,\ell}$,
we have $\alpha, \beta, \gamma >0$ and $\alpha+\beta+\gamma=\pi$, that
is, $\{\alpha, \beta, \gamma\}$ is the set of angles of a triangle.

This fact can be rephrased into a geometric construction:
call $\hat{y}$ the intersection of the line passing through
$(y_{\ell-1},y_1)$ and the line passing through
$(y_{\ell},y_2)$, then the drawing in the bottom part of Figure~\ref{fig:prova}
is generic, that is, the polygon
% consider the SPEZZATA CHIUSA.... $Q$ 
with vertices (in cyclic order) 
$(y_{\ell-1}, y_{1}, y_{\ell}, y_2)$ is a convex quadrilater, and the
line passing through $y_{1}$ and $y_{\ell}$ separates it from
$\hat{y}$. The fact that $(y_{\ell-1}, y_{1}, y_{\ell}, y_2)$ is a
convex quadrilater, in turn, implies that its diagonals (namely, the
segments $(y_{1},y_2)$ and $(y_{\ell-1},y_{\ell})$) do cross.

As a consequence, the new sets 
$X'=(x'_1,\ldots, x'_{\ell})=(y_1,\ldots,y_{\ell})$
and 
$Y'=(y'_1,\ldots,y'_{\ell-1})=(x_2,\ldots, x_{\ell})$
are such that the segments
$[x'_1, x'_2]$ and
$[x'_{\ell-1}, x'_{\ell}]$ strictly cross.
On the other hand, the new matrix
$W'$ (with 
$W'_{ij}=-\langle x'_{i}, y'_{j}\rangle=-\langle y_{i}, x_{j-1}\rangle$)
% $W^{\rm (new)}$ 
is the transpose of the minor of $W$ obtained by dropping the first
and last row, that is
$W'=(W^{\{1,\ell+1\},\varnothing})^\intercal
=(W^\intercal)^{\varnothing,\{1,\ell+1\}}$. 
We claim that $W'$ is valid. Essentially, this is due to the fact that
we can use the notion of validity in the second sense, that is
preserved by the operation above (of dropping the first and last row,
and taking the transpose).

More explicitly, as $W$ is valid in the first sense, is in particular
also valid in the second sense, that reads
\begin{align}
\sum_j
(W^{\{1\},\varnothing}_{jj}-W^{\{1\},\varnothing}_{j\,\pi(j)})&>0
,
\\
\sum_j
(W^{\{\ell+1\},\varnothing}_{jj}-W^{\{\ell+1\},\varnothing}_{j\,\pi(j)})
&>0
,
\end{align}
for all single-cycle permutations $\pi \in \mathfrak{S}_{\ell}$. 
In particular, this is true
for $\pi$'s that have a singleton in the first or last position, that
implies
\begin{align}
\sum_j
(W^{\{1,\ell+1\},\{\ell\}}_{jj}-W^{\{1,\ell+1\},\{\ell\}}_{j\,\pi(j)})
&>0
,
\\
\sum_j
(W^{\{1,\ell+1\},\{1\}}_{jj}-W^{\{1,\ell+1\},\{1\}}_{j\,\pi(j)})
&>0
,
\end{align}
for all single-cycle permutations $\pi \in \mathfrak{S}_{\ell-1}$. 
Of course, the inverse of single-cycle permutations are 
single-cycle permutations, thus, writing (for $\sigma=\pi^{-1}$)
\[
\sum_j (W_{jj}-W_{j\,\pi(j)})
=
\sum_j (W_{jj}-W_{j\,\sigma^{-1}(j)})
=
\sum_j (W_{jj}-W_{\sigma(j)\,j})
\]
we get
\begin{align}
\sum_j
((W^\intercal)^{\{\ell\},\{1,\ell+1\}})_{jj}
-((W^\intercal)^{\{\ell\},\{1,\ell+1\}})_{j\,\pi(j)})
&>0
,
\\
\sum_j
((W^\intercal)^{\{1\},\{1,\ell+1\}})_{jj}
-((W^\intercal)^{\{1\},\{1,\ell+1\}})_{j\,\pi(j)})&>0
,
\end{align}
that is
\begin{align}
\sum_j
((W')^{\{\ell\},\varnothing})_{jj}
-((W')^{\{\ell\},\varnothing})_{j\,\pi(j)})
&>0
,
\\
\sum_j
((W')^{\{1\},\varnothing})_{jj}
-((W')^{\{1\},\varnothing})_{j\,\pi(j)})&>0
.
\end{align}
This is exactly the condition for $W'$ to be valid in the second
sense. However, by Theorem~\ref{th:valid12} (which plays a crucial role
at this point), this implies that $W'$ is also valid in the first
sense, that is, it is valid \emph{tout court}.

In conclusion, the pair $(X',Y')$ provides an instance containing a
crossing path, of smaller length than the one provided by $(X,Y)$. As
we assumed that the value of $\ell$ was minimal for these properties
to hold, we have reached a contradiction, so that we can conclude that
no crossing configurations do exist.
\end{proof}

%-------------------------------------------------------
\subsection{Alternative proof of the non-crossing property 
% for $p=2$ and $\Omega \subset \bR^2$, 
in the bipartite reservoir standard setting}
\label{ssec.2ndproofNonCrossReserv}
%-------------------------------------------------------
\noindent 
Here we provide a second proof of Theorem \ref{th:treenoncrosssym}.
% restricted to the case in which $\Omega$ is simply connected. 
The main idea is that we can approximate the spectrum of a reservoir
instance
% $\{W_{(G,w)}(M)\}$ 
with the spectrum of a suitably constructed instance of the bipartite
standard setting, and then just apply
Theorem~\ref{th:treenoncross}. The new instance comes with a parameter
$N$, denoting the number of extra vertices in the construction, and
the distance between the two spectra is bounded (in $L_\infty$ norm)
by a function of $N^{-1}$ converging to zero for $N\to\infty$. As in
particular the hypothesis of generic weights implies that in every
subproblem $G_U$ the difference of cost between $M_U$ and the second
best matching is a finite quantity (independent from $N$), and as the
list $\calJ$ is finite, we can then conclude by continuity.

Before passing to the proof, we explain the approximation procedure
sketched above.
\begin{rmk}
\label{rmk.facciorese}
Let $\Omega_0$ be a flat region of $\bR^2$, and 
$\Omega_r \subseteq \Omega_0$ a connected subset. Call
$\Omega=\Omega_0 \setminx \Omega_r$ and $\partial \Omega$ the boundary
of $\Omega$. Let the cost function $f(x)$ be a strictly monotone
function with the property
$\lim_{x\to 0}f(x)=0$ (a fact that is true in particular in the case
$p=2$, i.e.\ $f(x)=x^2$).  Then, for all $(X,Y)$ generic
configurations of points in $\Omega$, there exists sets $(X_r,Y_r)$ of
points in $\Omega_r$ (described below) such that the matching subgraph
$H$ of the standard bipartite reservoir setting
$(\Omega,(X,Y),f)_{\rm res}$ coincides with the matching subgraph $H'$
of the standard bipartite setting 
$(\Omega_0,(X\cup X_r,Y\cup Y_r),f)$, restricted to $\Omega$.
\end{rmk}
\begin{proof}
Say that $|X|=n$ and $|Y|=m$.  Let us call $u_i=u(x_i)$ the points on
$\partial \Omega$ nearest to $x_i \in X$, and similarly $v_j=u(y_j)$
for $y_j\in Y$. Let $T$ be a configuration of curves inside $\Omega_r$,
with the topology of a tree, containing $\{u_i\}\cup\{v_j\}$, and such
that its Hausdorff dimension is~1. For $N\geq 0$, we construct $X_r$
and $Y_r$ as follows:
\begin{itemize}
\item all $u_i$'s are added to $Y_r$, and all $v_j$'s are added to
  $X_r$;
\item further $N$ points $u'_a \in T$, chosen evenly spaced, are added
  to $Y_r$, and further $N+1$ points $v'_a \in T$, chosen evenly spaced,
  are added to $X_r$.
\end{itemize}
Now consider the standard setting 
$(\Omega_0,(X\cup X_r,Y\cup Y_r),f)$.
% and call $H'$ its matching subgraph. 
This setting has a relatively large set
$\calJ_0=\{\{x_i\}\}_{1\leq i \leq n}\cup
\{\{v_j\}\}_{1\leq j \leq m}\cup
\{\{v'_a\}\}_{1\leq a \leq N+1}$,
to be compared with the subset 
$\calJ=\{\{x_i\}\}_{1\leq i \leq n}\cup\{\emptyset\}$ for the setting
$(\Omega,(X,Y),f)_{\rm res}$.  Choose one element $\{v'_a\}$
arbitrarily, and call $\calJ'\subseteq \calJ_0$ the set
$\calJ'=\{\{x_i\}\}_{1\leq i \leq n}\cup\{\{v'_a\}\}$.  Our precise
claim is that, for $N$ large enough, the restriction of $H'$ to
$\Omega$ coincides with $H$, and furthermore, for all
$e\in H$, the set $\calC(e)$ in the reservoir setting coincides with
the restriction to $\calJ'$ of the set $\calC'(e)$ in the ordinary
setting, with the identification of $\emptyset \in \calJ$ with
$\{v'_a\}\in \calJ'$.

Indeed, from the fact that the instance is generic, we know that for
each $U \in \calJ$, in the reservoir setting, the difference of cost
between $M_U$ and any other matching in its ensemble is some
$\delta_U$ strictly positive. So there exists a value
$\delta=\min_{U\in \calJ} \delta_U >0$. For every $U \in \calJ$, we
can construct some matching $M'_U$ in the larger configuration of points
$(X\cup X_r,Y\cup Y_r)$, by making it coincide with $M_U$ on $(X,Y)$,
and then by completing it with the matching of smaller cost on the
remaining vertices. All these vertices are on $T$. So it is easy to
see that $0\leq w(M'_U)-w(M_U)\lesssim N f(1/N)$. On the other side,
for any matching $M'$ on $(X\cup X_r,Y\cup Y_r)$
that does not coincide with $M_U$ on $(X,Y)$, we have
$w(M')-w(M_U)\geq \delta$. So, by our hypothesis on $f$, our claim
follows.
\end{proof}
\noindent
Now we are ready to provide our proof.
%% Our main theorem is:
%% \begin{thm}
%% \label{th:treenoncrosssym}
%% When $\Omega$ is a simply connected domain with boundary of $\bR^2$, and $p=2$, in
%% the bipartite reservoir standard setting the embedding of the tree
%% $\bar{H}_{\calJ}^{\rm (l)}$ is non-crossing.
%% \end{thm}
\begin{proof}[Second proof of Theorem~\ref{th:treenoncrosssym}]
As was the case in the main proofs for the cylinder and other flat surfaces, 
the main idea is to reduce to the case of the standard setting,
on a domain that is a portion of the plane.

A fact used in this proof is that, for $\Omega$ to contain all
(geodesic) segments $[x_i,y_j]$ (and their reservoir analogs),
essentially it must contain the convex hull $\tilde{\Omega}$ of the
reservoir vertices $\{u(x_i)\}\cup\{v(y_j)\}$ (and in fact the
instances over $\Omega$ and $\tilde{\Omega}$ have the same matrix of
costs). So we can assume w.l.o.g.\ that $\Omega$ is convex, which in
turn implies that it is simply connected (and topologically a disk),
that is, $\partial \Omega$ is topologically a circle, and in particular
it is connected.

%% when $\Omega$ is simply
%% connected. Indeed, in dimension 2, $\Omega$ is simply connected iff
%% %
%% $\partial \Omega$ is connected. We will first discuss this situation,
%% and then explain how to pass to the general case towards the end of
%% the proof.

The graph $H_{\calJ}$ determined by the instance
$(\Omega,(X,Y),f)_{\rm res}$
% $H_{\calJ}^{\rm (l)}$
%  and $\bar{H}_{\calJ}^{\rm (l)}$ 
consists of a collection of components $H_{\alpha}$, each of them
being a tree.  With reference to the notation in
Remark~\ref{rmk.bridgerootC}, and using that statement, we know that
each of these components has a unique sym-edge $e_\alpha$.  In the
reservoir case, this edge is embedded with coincident endpoints, and
on the boundary (although, for clarity of exposition, in our pictures
we will split these endpoints).  If the endpoint of $e_\alpha$ on the
left is a $B$-vertex $d_{\alpha}$ in position $v_\alpha$, then
$\calC(e_\alpha)=C_{\alpha}$. Let us now construct a suitable instance
$(\Omega,(X',Y'),f)$ of the first bipartite standard setting.

We choose to take all the vertices in the left part of $H_\alpha$,
plus an $A$-vertex in a position $\tilde{v}_\alpha$ coinciding with
$v_{\alpha}$ and $\iota v_\alpha=v'_{\alpha}$ (or, in our pictures,
infinitesimally near to $v_{\alpha}$ and $v'_{\alpha}$).  

The set $\calJ$ in the new instance is determined as by the
prescription of the first bipartite standard setting. In particular,
recalling that (e.g.\ in Remark \ref{rmk.bridgerootC}) we have called
$C_{\alpha}$ the set of colors associated to the component $\alpha$,
we can identify the new set $\calJ$ with 
$C_\alpha \cup \{\varnothing\}$, where the extra color, $\varnothing$,
is associated to the $A$-vertex $\tilde{v}_\alpha$, and has a role
analogous to the color $U=\varnothing$ in the original symmetric
instance. This construction is depicted on the top part of
Figure~\ref{fig.nocrosssym5}.

We claim that the graph $H_{\calJ}$ of this new instance coincides
with a graph $\tilde{H}_{\alpha}$, that is the restriction of
$H_{\alpha}$ to its edges on the left part, and the edge embedded as
the segment $[v_\alpha,\tilde{v}_{\alpha}]$ of zero (or infinitesimal)
length.

Indeed, and also in light of
Proposition~\ref{prop.addresedgestoothers}, the optimality of
$H_{\calJ}$ in the new instance is determined by a subset of the
inequalities implying the optimality of $H_{\alpha}$ in the original
instance, except for the possibility in the new instance that
$\tilde{v}_{\alpha}$ (which is an $A$-vertex) is connected to some
$y_j \neq y_\alpha$ for some color $U_k$. As we are in the standard
bipartite setting, the set $U_k$ corresponds to the removal of a
single vertex $x_k \in H_{\alpha}$.  It can be seen that the only
interesting case is when the path on the tree connecting $x_k$ and
$v_\alpha$ goes through~$y_j$.  Refer again to
Figure~\ref{fig.nocrosssym5},~top, for an illustration.

Call $W$ the alternating sum of the path on $H_{\alpha}$ (and on
$\tilde{H}_{\alpha}$ as well), from $y_j$ to $y_{\alpha}$, and
$2W_\alpha$ the difference of cost (in the symmetric setting) between
$\cW(M_\star(G_{U_k}))$ and the optimal matching among those containing
the edge $[y_j,u_j]$.  In the reservoir setting the weight of the
sym-edges is zero.
% $w_{e_\alpha}, w_{e_j} = \cO(f(\eps))$.
Then, the weight for connecting $y_j$ to its reservoir point $u_j$ on
$\partial \Omega$ is some quantity
$w=f(|\gamma_{y_j}|)=f(|\gamma_{y_j,u_j}|)$. The fact that $H_\alpha$
is the component in the original instance, and in particular that the
edge $[y_j,u_j]$ is not in the embedding of $H_\calJ$, implies
\be
\label{eq.39876894756}
2W_\alpha = 2W-2 f(|\gamma_{y_j}|)
%+\cO(f(\eps))
<0
\ef.
\ee
In the new instance, the weight of the edge
$[v_{\alpha},\tilde{v}_{\alpha}]$ is again just zero,
% $\cO(f(\eps))$, 
and the weight of the edge $[y_{j},\tilde{v}_{\alpha}]$ is given by
$f(|\gamma_{y_j,\tilde{v}_{\alpha}}|)$, so, calling $\tilde{W}$ the
difference of cost between what we claim being the optimal matching,
and any other matching such that the edge $[y_j,\tilde{v}_{\alpha}]$
is in the embedding of $H_\calJ$, the missing inequalities for the
validity of the claimed matching subgraph $H_\calJ$ are
\be
\tilde{W} = W-f(|\gamma_{y_j,\tilde{v}_{\alpha}}|)
% +\cO(\eps,f(\eps))
<0
\ef.
\ee
These inequalities are implied by
the inequalities (\ref{eq.39876894756}) above, and the inequalities
implied by the definition of the reservoir setting, namely
$|\gamma_{y_j}|=\min_{x \in \partial \Omega}|\gamma_{y_j,x}|\leq
|\gamma_{y_j,\tilde{v}_{\alpha}}|$ for all 
$\tilde{v}_{\alpha} \in \partial \Omega$.
By Theorem~\ref{th:treenoncross} we conclude that, when $f(x)=x^2$,
there are no crossing pairs of edges which are in the same
component~$H_\alpha$.  This completes the argument when we consider
the possible crossings among edges in the same component $H_{\alpha}$
of the forest, and the left endpoint of the sym-edge of $H_{\alpha}$
is a $B$-vertex.

\begin{figure}
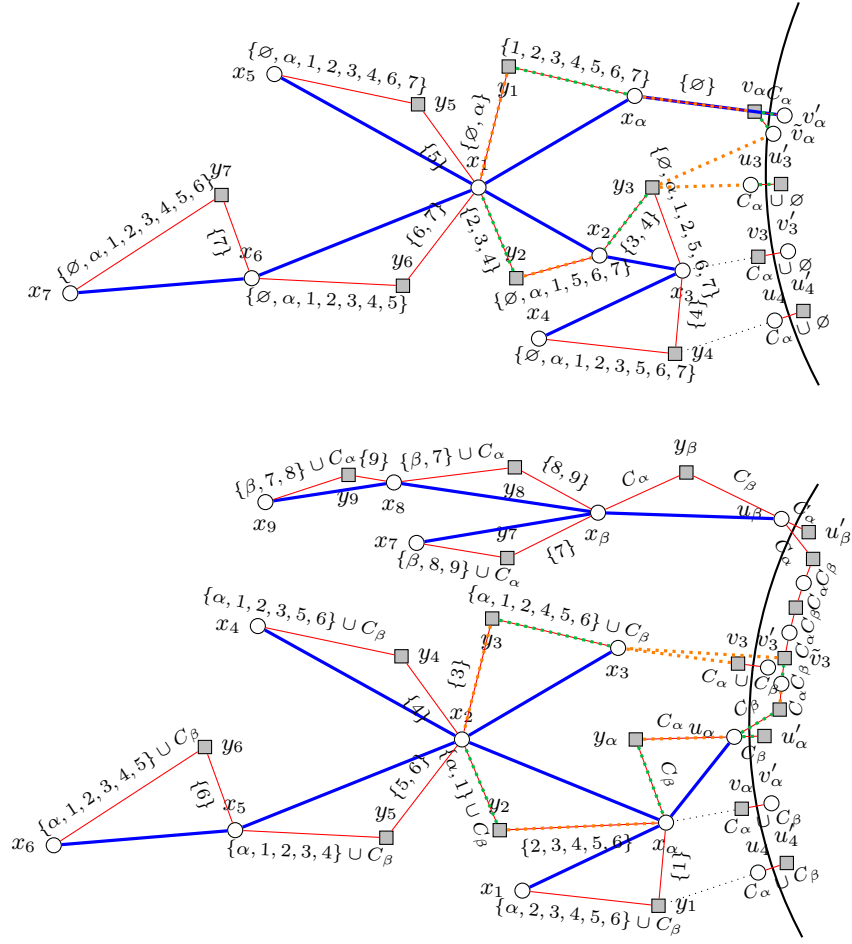

\begin{center}
\if\faifig1
%  UN ALBERO

\else [...TikZ\ code...] \fi
\end{center}
\caption{\label{fig.nocrosssym5}% 
Top: example of tree
  $H_{\alpha}$. Here $C_{\alpha}=\{\alpha,1,2,\ldots,7\}$. For
  clarity, the vertices on $\partial \Omega$ have been moved slightly
  off.  The auxiliary instance in $\bR^2$ has matching subgraph
  $H_\calJ$ coinciding with $H_\alpha$, expect for the extra vertex
  $\tilde{v}_\alpha$. Here $H_\alpha$ and $H_\calJ$ have been drawn
  one on top of the other. The proof argument states that the
  newly-introduced vertex $\tilde{v}_\alpha$ cannot be connected to a
  $y_j$ (here $y_3$). Indeed, the cycle in the symmetric setting
  $(\ldots,y'_3,u'_3,u_3,y_3,x_2,y_2,x_1,y_1,x_\alpha,v_\alpha,v'_\alpha,x'_\alpha,\ldots)$
  has some alternating cost $2W_\alpha>0$, while taking the edge
  $(y_3,\tilde{v}_{\alpha})$ in $H_\calJ$ would require to swap a
  cycle with alternating cost $\tilde{W}>W+\cO(\eps^2)$, where the
  difference is
  $\tilde{W}-W=f(|\gamma_{y_3,u_3}|)-f(|\gamma_{y_3,\tilde{v}_{\alpha}}|)$. These
  alternating cycles are drawn in dotted orange and green.
Bottom: the analogous construction when we consider two components
$H_\alpha$ and $H_\beta$, and we connect their root vertices on
$\partial \Omega$ by a path $P$.%
}
\end{figure}

If the endpoint of $e_\alpha$ on the left is an $A$-vertex
$a_{\alpha}$ in position $u_\alpha$, then 
$\calC(e_\alpha)=\calJ \setminx C_{\alpha}$.  Nonetheless, an argument
very similar to the one above does hold. Now, in the construction of
an instance in the first bipartite standard setting, we choose to take
just all the vertices in the left part of $H_\alpha$.  The graph
$H_{\calJ}$ of this new instance coincides with the graph
$\tilde{H}_{\alpha}$, that is the restriction of $H_{\alpha}$ to its
edges on the left part, in light of a similar argument than the one
given above, in particular involving the same inequalities.  Again, by
Theorem~\ref{th:treenoncross} we conclude that, when $f(x)=x^2$, there
are no crossing pairs of edges which are in the same
component~$H_\alpha$.

Now, we should exclude the possibility that there are crossing pairs
of edges which are in distinct components $H_\alpha$ and $H_\beta$.
The main ideas are similar to those used in the previous analysis, and
the slight difference between the two cases in which the left endpoint
of the sym-edge is an $A$-vertex or a $B$-vertex works in a similar
way. So we shall describe in detail only one case of the three that
are possible, say, the case in which the left endpoints of both
$e_\alpha$ and $e_\beta$ are $A$-vertices.

At this point, we shall use Proposition~\ref{prop.addresedgestoothers}
and Remark~\ref{rmk.facciorese}, the latter with choice of $\Omega_r$ given
by the portion of $\partial \Omega$ between $u_{\alpha}$ and
$u_{\beta}$ (say, in CCW order from $\ga$ to $\gb$).

According to this paradigm, we construct an instance of the first bipartite standard setting
in which the embedded graph consists of the restriction of
$H_{\alpha}$ and $H_{\beta}$ to their left parts, plus a chain $P$ of
$\sim 1/\eps$ alternating $A$-vertices and $B$-vertices, at Euclidean
distance $\eps$ one from the other, connecting $u_\alpha$ to
$u_\beta$, along
$\partial \Omega$.  

As in Remark~\ref{rmk.facciorese}, it is convenient to consider the
sets $\calC(e)$ restricted to the subset
$C_\alpha \cup C_\beta =: \calJ' \subseteq \calJ$. Indeed, while the
set $\calJ$ of the standard bipartite setting for this family (in
$\eps$) of configurations changes with $\eps$, the set $\calJ'$ is
fixed once and for all, so it is conceivable that not only the
matching subgraph $\bar{H}_{\calJ}(\eps)$ may have a limit for 
$\eps \to 0$, but also the sets $\calC'(e)=\calC(e)\cap \calJ'$ do.

% the most convenient choice for the collection $\calJ$ is to set
%
% $\calJ = C_\alpha \cup C_\beta$.

We claim that, when $\lim_{x \to 0} f(x)/x = 0$, for some
$\eps=\eps(X,Y,f)>0$ the matching subgraph $H_\calJ$ of this instance
of the first bipartite standard setting consists of the left parts of
$H_{\alpha}$ and $H_{\beta}$, and the path $P$, where the edges $e$ of
$P$ have alternately $\calC(e)\cap\calJ'=C_\alpha$ and $C_\beta$.
%% \[
%% \includegraphics[scale=.12]{nocrosssym3.png}
%% \]
Similarly as above, the optimality of $H_{\calJ}$ in the new instance
is determined by a subset of the inequalities as the optimality of
$\tilde{H}_{\alpha}$ and $\tilde{H}_{\beta}$ in the original instance,
except for the possibility in the new instance that some point 
$x \in P$ is connected to some $y_j$ or that some point $y \in P$ is
connected to some $x_j$.

The two cases are similar (and of course $\ga$ and $\gb$ play a
symmetric role). Let us say that $y \in P$ is connected to 
$x_j\in H_{\alpha}$.  Call $W$ the alternating sum of the path on
$\tilde{H}_{\alpha}$ from $x_j$ to $u_{\alpha}$.  We have that
% the weight of the original sym-edge is $w_{e_\alpha}=\cO(\eps^2)$.
% Then,
the cost for connecting $x_j$ to its reservoir point $v_j$ on
$\partial \Omega$ is
$f(|\gamma_{y_j}|)=f(|\gamma_{y_j,v_j}|)$. The fact that
$H_\alpha$ is a separate component of the original instance, rooted at
$u_{\ga}$, implies
\be
\label{eq.39876894756b}
2 W_{\alpha}
=
2W-2 f(|\gamma_{y_j}|)
% +\cO(\eps^2)
<0.
\ee
This construction is illustrated in Figure~\ref{fig.nocrosssym5},~bottom.

In the new instance, we have a similar expression for the difference
of cost for connecting $x_j$ to some point $\tilde{v}_j$ on the path.
In fact, it can be seen that, as the points on the path are equally
spaced (at distance $\eps$), the modification in the cost due to the
modified connection of the points along the path would give a
difference of cost of order at most $f(\eps)$. However, it is simpler
to just use the paradigm of Remark~\ref{rmk.facciorese}, and write
that the difference of costs is some quantity $\tilde{W}$:
\be
\tilde{W}\leq W-f(|\gamma_{x_j,\tilde{v}_j}|)+R_U(\eps)
%+\cO(\ell^2/N^2)
\ef,
\ee
where $R_U(\eps)$ is some function, depending only on $U$,
such that $\lim_{\eps \to 0} R_U(\eps)=0$,
again according to Remark~\ref{rmk.facciorese}. And, as $|\calJ'|$ is
finite, there exists a function $R(\eps)\leq R_U(\eps)$ for all $U$.
That is, for the inequalities (\ref{eq.39876894756b}) we get
\be
\label{eq.39876894756c}
W_{\alpha}
\leq R(\eps)
\ef.
\ee
Furthermore, the generality hypothesis in fact implies the stronger fact 
\be
\label{eq.39876894756d}
W_{\alpha}
\leq R(\eps)-\delta
\ef,
\ee
with $\delta=\min_U \delta_U$ defined as in Remark~\ref{rmk.facciorese}.

We can conclude that there exists a value $\eps>0$ small enough to
make all equations (\ref{eq.39876894756d}) to hold simultaneously. As
this corresponds to a finite instance in the standard bipartite
setting on a portion of the plane, when $f(x)=x^2$ the configuration
must be non-crossing.
\end{proof}

%%%%%%%%%%%%%%%%%%%%%%%%%%%%%%%%%%%%%%%%%%%%%%%%%%%%%%%
\section{More on the analysis of when the matching subgraph is non-crossing}
\label{app.crossingGene}
%%%%%%%%%%%%%%%%%%%%%%%%%%%%%%%%%%%%%%%%%%%%%%%%%%%%%%%

\noindent
Also this section is devoted to the non-crossing property of projected
matching subgraphs $\bar{H}_\calJ$.  Here we take the task of
providing some technical simple facts (like the starting point of our
induction proofs provided in the main body of the manuscript), and a
list of negative results, implied by counterexamples, that go in the
direction of classifying \emph{all} manifolds $\Omega$ and values $p$
where the non-crossing property holds.

%-------------------------------------------------------
\subsection{Proof that there are no crossing configurations of small size}
% with $|X|=4$ and $|Y|=3$ on the plane or the cylinder}
\label{app.casiL3}
%-------------------------------------------------------

\noindent
In this section we prove the base case
($\ell=2$) of the induction for 
Theorem~\ref{th:treenoncrossToro},
%\ref{th:treenoncross} and 
that is, the claim
%, in the proof Theorem \ref{th:treenoncrossToro}, 
that there exist no paths $P=(x_1,y_1,\ldots,y_{\ell},x_{\ell+1})$ on
the cylinder such that the matrix $\tilde{W}$ is valid, thus
completing the proof of the theorem.

In passing, we will compare this proof to the case $\ell=3$, for which
the argument is essentially equivalent to the one for generic $\ell$,
provided in the main body of the proof (but it may be instructive to
see it explicitly in its simplest incarnation).  As a corollary of our
calculation, we will also derive a proof for the analogous case
$\ell=3$ for Theorem \ref{th:treenoncross}, that there exist no paths
$P=(x_1,y_1,\ldots,y_{\ell},x_{\ell+1})$ on the plane such that the
matrix $W$ is valid. Again this is not strictly necessary, because in
the case of the plane the induction works down to the base case
$\ell=2$, which is trivially established, but again we find
instructive to have a direct explanation of the non-crossing property
of the plane in the first non-trivial setting at~$\ell=3$.

We will use the graphical notation of the previous section to describe
the inequalities that guarantee that a matrix of costs is valid.  In
the case of the cylinder, with periodicity vector $\w$, we will also
adopt the shortcuts $x_{j_{\rm L}}=x_j-\w$ and $x_{j_{\rm R}}=x_j+\w$
(and similarly for $y$), where the choice of notation stands for left-
and right-translation. Furthermore, we use $R$ for the operator that
applies a rotation by 90 degrees to vectors, $R(x,y)=(-y,x)$.

Let us start with the case $\ell=2$ on the cylinder. Let us assume
w.l.o.g.\ that $\w$ is parallel to the horizontal axis, and oriented
towards the right, that $[x_1,x_2]$ and 
$[x_{2_{\rm L}},x_{3_{\rm L}}]$ cross at the origin, and that $x_2$
has negative ordinate (so that $x_1$ and $x_3$ must have positive
ordinate). In this case we have 
\be
\label{eq.487528763a}
R^{-1}\w \prec \xi_{2,1} \prec \xi_{3_{\rm L},1} \prec \xi_{3,2} \prec
R\w 
\ef.
\ee
Furthermore, we have that the image of $y_1$ nearest to $x_2$ must be
on its left, and the image of $y_2$ nearest to $x_2$ must be
on its right, more precisely
\begin{align}
\langle
(y_1-x_1), \w 
\rangle
&< \frac{|\w|^2}{2} 
\ef;
&
\langle
(x_2-y_1), \w 
\rangle
&< \frac{|\w|^2}{2}
\ef;
\end{align}
so that, in particular,
\be
\langle
\eta_{1,2_{\rm L}} , \w 
\rangle
>0
\ee
that can also be restated as
\be
\label{eq.487528763b}
R^{-1}\w \prec \eta_{1,2_{\rm L}} \prec R\w 
\ef.
\ee
Combining (\ref{eq.487528763a}) and (\ref{eq.487528763b}) gives, in
particular
\be
\label{eq.4576582764}
\langle
\xi_{1,3_{\rm L}},
\eta_{1,2_{\rm L}}
\rangle
<0
\ef.
\ee
Now, for what concerns the inequalities for the validity of the
matrix $\tilde{W}$, we have
\be
\tilde{K}=
\left(
\begin{array}{rcccl}
%{p[10pt]p[10pt]p[10pt]p[10pt]p[10pt]}
 0 &  1 \\
 0 &  0 \\
-1 &  0 
\end{array}
\right)
\ee
Recall from Section \ref{ssec.noncrossToro} that the inequalities with
a simple form are those corresponding to
transpositions $\{(i,k),(j,h)\} \to \{(i,h),(j,k)\}$, and such that
the corresponding minor of $\tilde{K}$ has the ``affine'' form
\be
\begin{pmatrix}
K_{ik} & K_{ih} \\
K_{jk} & K_{jh} 
\end{pmatrix}
=
\begin{pmatrix}
a&a+b\\
a+c&a+b+c
\end{pmatrix}
\ef.
\ee
In this case we have a single inequality of this form, namely the one
associated to the diagram
\[
\setlength{\unitlength}{10pt}
\thicklines
\raisebox{-12pt}{\begin{picture}(2,3)
\put(0,1){\goyel{\rule{20pt}{10pt}}}
  \linethickness{0.5pt}
  \multiput(0,0)(1,0){3}{\line(0,1){3}} 
  \multiput(0,0)(0,1){4}{\line(1,0){2}}
  \linethickness{1.5pt}
  \put(0,0){\line(0,1){3}}
  \put(0,0){\line(1,0){2}}
  \put(0,3){\line(1,0){2}}
  \put(2,0){\line(0,1){3}}
  \put(0,1){\line(1,0){2}}
  \put(0,2){\line(1,0){2}}
  \put(1,0){\line(0,1){3}}
\put(0.5,2.5){\gogre{\line(1,0){1}}}
\put(0.5,0.5){\gogre{\line(1,0){1}}}
\put(0.5,0.5){\gogre{\line(0,1){2}}}
\put(1.5,0.5){\gogre{\line(0,1){2}}}
\dotbr{0}{1}{2}
\dotbr{1}{0}{0}
\end{picture}}
\]
that gives
\be
\label{eq.4576582764n}
\langle
\xi_{1,3_{\rm L}},
\eta_{1,2_{\rm L}}
\rangle
>0
\ef.
\ee
As this is in contrast with (\ref{eq.4576582764}), we can conclude.

In the case $\ell=3$, on the cylinder, the crossing condition states
that the polygon with vertices $(x_1,x_{4_{\rm L}},x_2,x_{3_{\rm L}})$
is a convex quadrilater (so that the diagonals do cross), which means
that one of the following must be true (depending if the vertices
in the list above are in clockwise or counterclockwise order)
\begin{subequations}
\label{eqs.4357856187}
\begin{gather}
\xi_{1,4_{\rm L}} \prec \xi_{4_{\rm L},2} \prec \xi_{2,3_{\rm L}} \prec \xi_{3_{\rm L},1}
\ef;
\\
\xi_{1,3_{\rm L}} \prec \xi_{3_{\rm L},2} \prec \xi_{2,4_{\rm L}} \prec \xi_{4_{\rm L},1}
\ef.
\end{gather}
\end{subequations}
The matrix $\tilde{K}$ is
\be
\tilde{K}=
\left(
\begin{array}{rcccl}
%{p[10pt]p[10pt]p[10pt]p[10pt]p[10pt]}
 0 &  1 &  1 \\
 0 &  0 &  1 \\
-1 &  0 &  0 \\
-1 & -1 &  0 
\end{array}
\right)
\ef.
\ee
Now we can find six diagrams corresponding to transpositions of the
simplest form, of which four have the further property of not involving
the coordinate $y_2$, namely
\begin{align}
% 1
\setlength{\unitlength}{10pt}
\thicklines
\raisebox{-17pt}{%
}
\end{align}
The inequalities associated to these four diagrams are
\begin{subequations}
\label{eqs.4357856187b}
\begin{align}
\langle
\xi_{1,3_{\rm L}},
\eta_{1,3_{\rm L}}
\rangle
&>0
\ef;
&
\langle
\xi_{1,4_{\rm L}},
\eta_{1,3_{\rm L}}
\rangle
&>0
\ef;
\\
\langle
\xi_{2,3_{\rm L}},
\eta_{1,3_{\rm L}}
\rangle
&>0
\ef;
&
\langle
\xi_{2,4_{\rm L}},
\eta_{1,3_{\rm L}}
\rangle
&>0
\ef;
\end{align}
\end{subequations}
which can also be stated as
\begin{subequations}
\label{eqs.4357856187qwerty}
\begin{align}
R^{-1} \eta_{1,3_{\rm L}} &\prec
\xi_{1,4_{\rm L}} \prec 
R \eta_{1,3_{\rm L}}
\ef;
&
R \eta_{1,3_{\rm L}} &\prec
\xi_{4_{\rm L},2} \prec 
R^{-1} \eta_{1,3_{\rm L}}
\ef;
\\
R^{-1} \eta_{1,3_{\rm L}} &\prec
\xi_{2,3_{\rm L}} \prec 
R \eta_{1,3_{\rm L}}
\ef;
&
R \eta_{1,3_{\rm L}} &\prec
\xi_{3_{\rm L},1} \prec
R^{-1} \eta_{1,3_{\rm L}}
\ef.
\end{align}
\end{subequations}
These inequalities are not compatible with the two possibilities in
(\ref{eqs.4357856187}), so we have found a contradiction.

In the case of the plane, the proof is repeated \emph{verbatim}, by
dropping all $(\cdot)_{\rm L}$ and $(\cdot)_{\rm R}$ subscripts.

%-------------------------------------------------------
\subsection{Crossing configurations for the shortest geodesics
  prescription}
\label{app.torosbagliato}
%-------------------------------------------------------

\noindent
In this section we illustrate the importance of adopting the choice in
Definition~\ref{def.projgeo} for embedding the arcs of geodesics in
the projection procedure $H_\calJ \to \bar{H}_{\calJ}$, in the
bipartite setting, in order to have the non-crossing property at
$p=2$, as proven in Theorem~\ref{th:treenoncrossToro}.

We will consider an instance with 4 $A$-vertices and 3
$B$-vertices. Our torus
% $\mathbb{T}$ 
has basis vectors 
$(\omega_1,\omega_2)=\big( (3,0),(-1,2) \big)$. Our points are
\begin{align*}
x_1&=(-\eps,0);&
x_2&=(1-2\eps,\eps^2);&
x_3&=(2-2\eps,\eps^2);&
x_4&=(\eps,0);
\end{align*}
\begin{align*}
y_1&=(-\eps,1)=(1-\eps,-1);&
y_2&=(1,1)=(2,-1);&
y_3&=(2+\eps,1)=(\eps,-1).
\end{align*}
For $\eps$ small enough, the edges $[x_i,y_i]$ have cost
$W_{i\,i}=1+\cO(\eps)$, by taking the image on the torus such that
$y_i-x_i=(0,1)+\cO(\eps)$, while the edges $[x_{i+1},y_i]$ have cost
$W_{i+1\,i}=1+\cO(\eps)$, by taking the image on the torus such that
$y_i-x_{i+1}=(0,-1)+\cO(\eps)$. For the other pairs of indices, the
precise choice of the image depends on the value of $\eps$, but in
fact is of little importance to our purposes, for all entries except
$W_{13}$ and $W_{41}$, because the costs are valued $2+\cO(\eps)$.  To
be definite, the corresponding matrix of costs, for a range of $\eps$,
reads
\be
W=
\ee
that is, up to a gauge transformation (in the sense of
Remark~\ref{rmk.gaugeinv}) with vectors
\begin{align}
\bm{\lambda} &= 
(1, 1 + 3 \eps^2 + \eps^4, 1 + 7 \eps^2 + \eps^4, 1);
&
\bm{\mu} &= 
(0, -\eps^2, 0);
\end{align}
the matrix
\be
W'=%
.
\ee
This matrix is the actual matrix of costs, and is valid, in the range
$0\leq\eps\leq \frac{\sqrt{2}-1}{2}$. This instance, for the value
$\eps=1/5$, is shown in Figure~\ref{fig.ceToro}.

\begin{figure}[t]
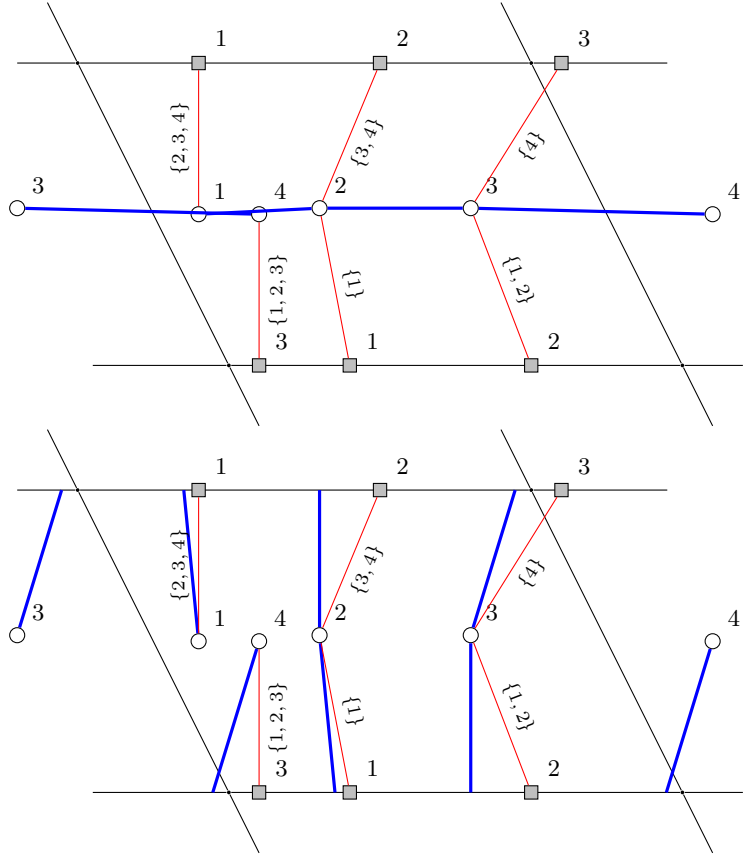

\begin{center}
\if\faifig1  
%
\else [...TikZ\ code...] \fi
\\
\if\faifig1  
%
\else [...TikZ\ code...] \fi
\end{center}
\caption{\label{fig.ceToro}% 
Top: our example, with tree $\bar{H}_\calJ$ constructed with the
na\"ive prescription, of $\gamma_{\{x_i,x_j\}}$ being the shortest
geodesic. The curves $\gamma_{\{x_1,x_2\}}$ and $\gamma_{\{x_3,x_4\}}$
do cross. Bottom: the same example, now with tree $\bar{H}_\calJ$ constructed with the
prescription that $\gamma_{\{x_i,x_j\}}$ is the shortest
geodesic with the same topology of the concatenation of the
corresponding curves $\gamma_{\{x_i,y_k\}}$ and
$\gamma_{\{y_k,x_j\}}$. Now $\bar{H}_\calJ$ is non-crossing in this
example (and in general, by Theorem~\ref{th:treenoncrossToro}).}
\end{figure}

%-------------------------------------------------------
\subsection{Optimality of the hypotheses for the matching subgraph to
  be non-crossing}
%% {\texorpdfstring{Optimality of the hypotheses for
%%     $\bar{H}_\calJ$ to be non-crossing}
%% {Optimality of the hypotheses for the matching subgraph to be non-crossing}}
% \section{Classification of manifolds such that $\bar{H}_\calJ$ is non-crossing}
% Two classes of counterexamples}
\label{app.crossingPneq2}
%-------------------------------------------------------

\noindent
Theorems~\ref{th:treenoncross}, \ref{th:treenoncrossToro} and
\ref{th:treenoncrossKleinB} establish the fact that, in the first
bipartite standard setting, the matching subgraph $\bar{H}_{\calJ}$ is
non-crossing when $p=2$ and $\Omega$ is the plane, or more generally a
manifold that is flat everywhere (namely, a cylinder, or a torus, or a
M\"obius strip or a Klein bottle).  Overall, the proof may look not
quite complicated. However we believe that this is a subtle and
fragile property. 

In order to support this claim, in this section we provide negative
results on the most evident ways we can think to relax our hypotheses
for the theorems above to hold. First, we show that if we relax
certain inequalities between real quantities from strict to weak,
there are large families of counterexamples already on the plane.
Then, we pass to the proof that the theorems above exhaust the list of
possible manifolds with the non-crossing property, by providing a
family of crossing configurations in all manifolds with a curvature
(either distributed or concentrated on conical
singularities). Finally, we show that the non-crossing property does
not hold even just in the plane, for any other function $f(x)=x^p$
with $p\neq 2$.

Let us first drop the hypothesis of generic weights, by allowing that
$x_1$ and $x_{\ell+1}$ are in the same position. This will add some
degeneracy of costs, and imply that certain cycles of even length have
vanishing alternating sums. We will still require that, besides the
cycles which have this property for obvious reasons, no other cycles
have vanishing alternating sums. Furthermore, we have defined a
crossing configuration as one such that there exist four distinct
indices, here $(1,2,\ell,\ell+1)$, such that the segments $[x_1,x_2]$
and $[x_{\ell},x_{\ell+1}]$ cross at an internal point, that is, there
exist $0<t<1$ and $0<s<1$ such that
\be
(1-t) x_1 +t x_2 = s x_\ell +(1-s) x_{\ell+1}
\ef.
\ee
Now we relax this requirement, by letting $0 \leq t \leq 1$ and 
$0 \leq s \leq 1$.  Of course, with our choice, it will just be
$s=t=0$.  Finally, we will say that a matrix $W$ is weakly-valid if
the inequalities in Definition~\ref{def.validmat} are turned into weak
inequalities.

In this larger setting, let $(x_1,x_2,\ldots,x_{\ell})$ be the
vertices of a convex polygon $P$. Up to translation, we can assume
that the origin is (strictly) inside $P$. Then, again up to
translation, let $Y$ be a set of points such that $y_i$ is on the
segment $[x_i,x_{i+1}]$. Then, as we know from
Section~\ref{sec.caso2dp2}, the matrix of costs is valid iff the
matrix with entries $W_{ij}=-\langle x_i,y_j \rangle$ is valid. Let
$d_i$ be the distance from the line passing through $x_i$ and
$x_{i+1}$ to the origin. We have that $W_{ii}=W_{i+1\,i}=-d_i^2$, and
that $W_{ki}>W_{ii}$ for all other values of $k$, except for the fact
that $W_{1\ell}=-d_{\ell}^2$
% W_{\ell\ell}$ 
and 
$W_{\ell+1,1}=-d_{1}^2$.
%W_{11}$.
By Remark~\ref{rmk.validifdiagzeroes}, we have that $W$ is weakly valid.

\medskip
\noindent
Let us now pass to consider other possible manifolds, or other values
of $p$ (curiously, the construction of counterexamples in these two
very different situations is very similar).

First of all, it is easy to get convinced by numerics of the fact that
our theorems for having non-crossing configurations $\bar{H}_{\calJ}$
fail when $p$ is taken different from 2, and that the analogous facts
for configurations $H_{\calJ}$ fail when $p$ is taken different from
1.  A plot is shown in Figure~\ref{fig.numeCrossN4}. The resulting
exponents for the vanishing of the probability of crossing as a
function of $|\delta p|$ (1 and 3 for $H$ at $p=1$ and for $\bar{H}$
at $p=2$, respectively) is expected to be independent from the choice
of domain (and may be of some theoretical interest). Of course, these
data are provided just as an illustration, as we provide explicit
counterexamples.

\begin{figure}
\setlength{\unitlength}{0.4\textwidth}
\begin{picture}(2,1.02)
\put(0,0){\includegraphics[width=0.8\textwidth]{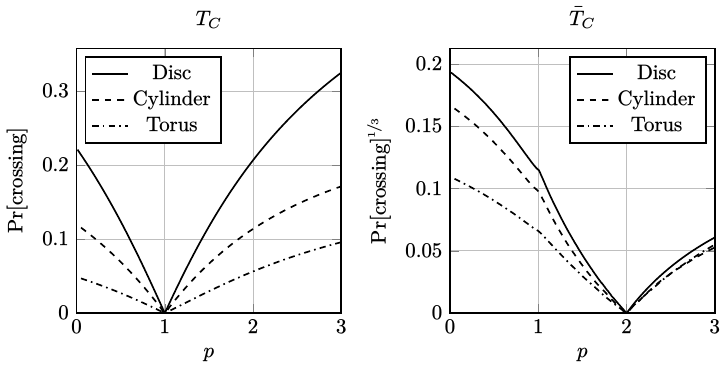}}
%plot.pdf
\put(.58,.94){\makebox[0pt][c]{\colorbox{mywhite}{$H_\calJ$}}}
\put(1.6,.94){\makebox[0pt][c]{\colorbox{mywhite}{$\bar{H}_\calJ$}}}
\end{picture}
\caption{\label{fig.numeCrossN4}% 
  Estimated probability of observing a crossing configuration
  $H_\calJ$ (left) and a crossing configuration $\bar{H}_\calJ$
  (right) in an ensemble of graphs generated by uniformly choosing
  configurations $(X,Y)$, with $|X|=|Y|+1=4$, on the unit disc, the
  unit cylinder with aspect ratio $1$, and the unit torus with aspect
  ratio $1$, shown as a function of the exponent $p>0$ in the cost
  function.  The probability for $H_\calJ$ to be crossing is
  numerically found to scale as $\Theta(|p-1|)$ in a neighborhood of
  $p=1$, whereas the probability for $\bar{H}_\calJ$ to be crossing is
  numerically found to scale as $\Theta(|p-2|^3)$ in a neighborhood of
  $p=2$.  Each data point is obtained averaging on $10^9$ independent
  samples.}
\end{figure}

Both in the case with points on a portion of the plane and $p\neq 2$, and with
points on a curved manifold (and possibly $p=2$), we will use
cartesian coordinates. In the second case, it will be understood that
we have points within a sufficiently small region, so that a single
chart for local parametrization suffices.
We choose the set of points
\begin{align}
\label{eq.XYposCross}
X
&=\left\{
(-1,-\theta),(2,2\theta),(4,0),(-8,0)
% \frac{\eta^2}{2+\eta+\eta^2}
\right\}
\ef;
&
Y
&=
\{(0,1),(0,2),(0,4)\}
\ef.
\end{align}
We will use these points in several proofs of existence of crossing
configurations, adapting the choice of $\theta$.
In the case of $p\neq 2$, we will set $p=2+2\eta$, and choose
%% In the case with
%% curvature, we will say that the curvature is a continuous function,
%% valued $\eta$ (O MENO ETA????) at the origin of the chart.
% Introduce the shortcut
\be
\label{eq.deftheta}
\theta = 
\frac{\eta}{8(4+\eta)}
%\frac{1}{2}\frac{\eta^2}{2+\eta+\eta^2}
\ef.
\ee
Note that the vertices in $X$ are close to the horizontal axis when
$\eta$ is close to zero, and the vertices in $Y$ are on the vertical
axis. Whenever $\theta \neq 0$ (that is, whenever $\eta \neq 0$), the
segments $[x_1,x_2]$ and $[x_3,x_4]$ cross at the origin, so that a
sufficient condition for proving that the configuration is crossing is
to show that either the matrix $W_{ij}=f(|x_i-y_j|)$ is valid, or the
matrix $W'_{ij}=W_{i\,4-j}$ is valid. As we will see, one or the other
case will occur, depending on the sign of $\eta$. So, in light of
Theorem~\ref{th:valid12} and Proposition~\ref{prop.onlytransp}, it is
sufficient to check that the six quantities $T_{ij}$'s (for $i=1,2,3$
and $j=1,2$) are all positive for $\eta>0$, and all negative for
$\eta<0$.  The plot in Figure \ref{fig.plotTij} shows evidence of this
fact. A formal proof is given in the remainder of this appendix.
\begin{figure}
\[
% 0 0 500 313
\setlength{\unitlength}{100pt}
\begin{picture}(2.5,1.55)
\put(0,0){\includegraphics[scale=.5]{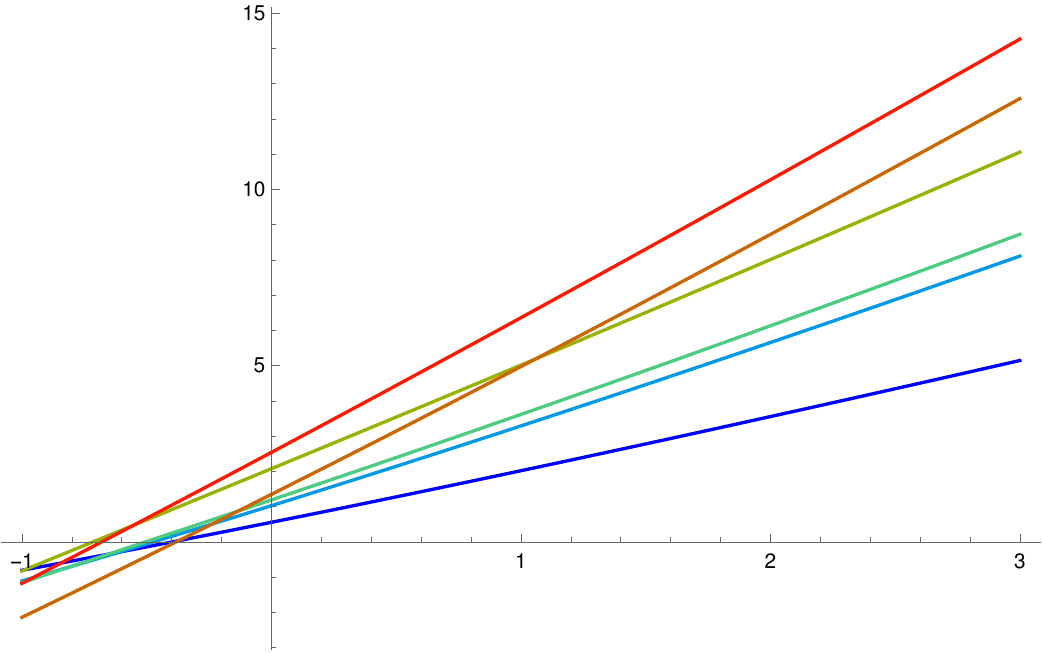}}
\put(2.5,.15){$\eta$}
\put(0.,1.35){$\displaystyle{\ln\left(\frac{T_{ij}}{\eta(\eta+1)}\right)}$}
\end{picture}
\]
\caption{\label{fig.plotTij}%
Plot of the six functions $T_{ij}$. We plot
$\ln(T_{ij}/(\eta(\eta+1)))$, in the range $-1\leq \eta \leq 3$, in
order to remove the zeroes at $\eta=0$ and $\eta=-1$ and improve
visualization. We use a logarithmic plot because for large values of
$\eta$ the quantity $T_{ij}$ is dominated by
$(|x_{i+1}|^2+|y_{i+1}|^2)^\eta$, so it grows exponentially as
$\exp[\eta \ln(|x_{i+1}|^2+|y_{i+1}|^2)]$.
The colors go from blue to red, for the list of
$(ij)$'s in lexicographic order.}
\end{figure}

Let us first establish the useful (simple) property
\begin{prop}
\label{prop.ineqa1a2b1b2simple}
For $0< a_1 < a_2$ and $0< b_1 < b_2$, and $c>0$, the expression
\be
X=\frac{(a_1+b_1)^c+(a_2+b_2)^c-(a_1+b_2)^c-(a_2+b_1)^c}{c(c-1)}
\ee
is positive.
\end{prop}
\begin{proof}
We shall see the quantity $X$ as a function $X(a_2,b_2)$ from
$[a_1,\infty[\,\times[b_1,\infty[\,$ to $\bR$, 
depending on the parameters $a_1,b_1,c>0$. We have that
$X(a_1,b_2)=0$, and that $X$ is monotone in $a_2$, because
\be
\begin{split}
\frac{\mathrm{d}}{\mathrm{d} a_2}X(a_2,b_2)&=
\frac{(a_2+b_2)^{c-1}-(a_2+b_1)^{c-1}}{c-1}
\ef{}
\end{split}
\ee
is positive at sight.
\end{proof}
\noindent
This fact is a preparation for the slightly more subtle
\begin{prop}
\label{prop.ineqa1a2b1b2}
For $0< a_1 \leq a_2$ and $0\leq b_1 \leq b_2$, and $c>0$,
%, and $\rho^{-1}\leq a_1/b_1 \leq \rho$,
the expression
\begin{align}
X&=\frac{((a_1+b_1)(1+s\,\delta))^c+(a_2+b_2)^c-(a_1+b_2)^c-(a_2+b_1)^c}{c(c-1)}
\ef;
\\
\delta
&=
(c-1)\frac{(a_2-a_1)(b_2-b_1)}{a_2 b_2}
\ef;
\end{align}
is positive for $s$ in an interval
containing $0$ as an internal point, namely
\begin{align}
\label{eq.928757863}
|s (c-1)|
&\leq 
\frac{a_2 b_2}{(a_2-a_1)(b_2-b_1)}
\ef;
&
s&>-\frac{a_1 b_1}{(a_1+b_1)^2}
\ef.
\end{align}
\end{prop}
\begin{proof}
Again, we shall see the quantity $X$ as a function $X(a_2,b_2)$
depending on the parameters $\{a_1,b_1,c,s\}$. We have that
$X(a_1,b_2)=0$, and that 
$\frac{\mathrm{d}}{\mathrm{d} a_2}X(a_2,b_1)=
\frac{\mathrm{d}}{\mathrm{d} a_2}0=0$, so that
\be
X(a_2,b_2)=\int_{a_1}^{a_2} \dx a \int_{b_1}^{b_2} \dx b\,
\frac{\mathrm{d}^2}{\mathrm{d} a\,\mathrm{d} b}X(a,b)
\ee
This second derivative is the expression
\begin{align}
\frac{\mathrm{d}^2}{\mathrm{d} a\,\mathrm{d} b}X(a,b)
&=
(a+b)^{c-2} 
\left(
1
+
s
\frac{a_1 b_1
% (a+b)^2
}{(a b)^{2}}
\frac{(a_1+b_1)^c}{(a+b)^{c-2}}
(1+s\delta)^{c-2} (1+cs\delta) 
\right)
\ef.
\end{align}
Of course, a sufficient condition implying $X\geq 0$ is that the
expression above is non-negative in the whole range
$(a,b)\in[a_1,a_2]\times[b_1,b_2]$.

If $s\geq 0$ and $c\geq 1$ we have $s \delta>0$, and our quantity is
positive at sight. 
% If $s\geq 0$ and $0<c<1$ 
If $s< 0$ and $c\geq 1$ 
we have $s \delta<0$, and we have two
conditions:
\begin{gather}
\label{eq.3765872}
1+s\delta>0
\ef;
\\
\label{eq.3765872bis}
1+s
\frac{a_1 b_1 (a+b)^2}{(a b)^{2}}
\frac{(a_1+b_1)^c}{(a+b)^c}
(1+s\delta)^{c-2} (1+cs\delta) 
>0
\ef.
\end{gather}
The first condition is implied by the first inequality in (\ref{eq.928757863}).
We rewrite the second condition as
\be
\label{eq.5689775672}
% \left|
s
\frac{1+cs\delta}{1+s\delta}
\frac{a_1 b_1 (a+b)(a_1+b_1)}{(a b)^{2}}
\left(\frac{(a_1+b_1)(1+s\delta)}{a+b}\right)^{c-1}
% \right|
>-1
\ef.
\ee
If $cs\delta<-1$ then $s(1+cs\delta)$ is positive, as well as all the
other factors, and we are left only with the condition
(\ref{eq.3765872}). Otherwise we need the expression on the LHS of
(\ref{eq.5689775672}) to be valued in $[-1,0]$. As the factor
$\frac{1+cs\delta}{1+s\delta}$ is in $[0,1]$ in this case, we can drop
it, and write the sufficient conditions
\begin{align}
\label{eq.5689775672b}
-s
\frac{a_1 b_1 (a+b)(a_1+b_1)}{(a b)^{2}}
&\leq 1
&
\frac{(a_1+b_1)(1+s\delta)}{a+b}
&\leq 1
\ef.
\end{align}
These conditions should be verified in the whole range
$(a,b)\in[a_1,a_2]\times[b_1,b_2]$.
The second condition is always verified, and the first condition 
is more tight when $a\to a_1$ and $b\to b_1$, where it reads
\be
\label{eq.2987637865}
-s
\frac{(a_1+b_1)^2}{a_1 b_1}
\leq 1
\ee
that is the second inequality in (\ref{eq.928757863}).

Now, if $s\geq 0$ and $0<c<1$ we have $s \delta \leq 0$, and 
again $1+s\delta>0$ is guaranteed only in light of
the first inequality in (\ref{eq.928757863}) (now with the other
determination of the absolute value).
%% reads $s<\frac{1}{c-1}\big(\frac{\lam}{\lam-1}\big)^2$. 
Once that this condition is satisfied, we have that $1+cs\delta \geq
0$, thus all the factors in (\ref{eq.3765872bis}) are positive, and we
are done.

If $s<0$ and $0<c<1$ we have $s \delta > 0$. In this case we have
$1<1+cs\delta<1+s \delta$, and again we can drop the factor
$\frac{1+cs\delta}{1+s\delta}$.
%<1+s(c-1)<1-s$. 
Now there is a case analysis
depending if the quantity in (\ref{eq.3765872bis}) appearing with
exponent $c-1<0$ is larger or smaller than 1. To stay simple, we can
use the fact $1+s \delta>1$
% <1-s$, 
and write
% restrict our analysis to the range $s>1-\lam$, for which we have the
% condition
\be
-s
\frac{a_1 b_1 (a_1+b_1)(a+b)}{(a b)^{2}}
<
\frac{a_1+b_1}{a+b}
% \min\left(1,\frac{(1-s)(a_1+b_1)}{a+b}\right)
\ef.
\ee
% One case gives back the inequality (\ref{eq.2987637865}). The second case 
This gives
\be
-s
\frac{a_1 b_1(a+b)^2}{(a b)^{2}}
<
1
%-s
\ef.
\ee
Although this inequality is tighter than the first inequality in
(\ref{eq.5689775672b}), for all values of $a$ and $b$, in fact again
this condition is more tight when $a\to a_1$ and $b\to b_1$,
where it becomes identical to~(\ref{eq.2987637865}).
This concludes our analysis.
\end{proof}

\noindent
Now we are ready to analyse the six expressions $T_{ij}$ needed to
ensure the validity of our cost matrix. The quantities $T_{31}$ and
$T_{32}$ are specially simple, as in this case 
Proposition~\ref{prop.ineqa1a2b1b2simple} is enough to conclude. For
example, the pertinent minor of $W$ giving $T_{31}$ is
\be
\begin{pmatrix}
17^{\eta} & 20^{\eta} \\
65^{\eta} & 68^{\eta} 
\end{pmatrix}
\ee
and the evaluation of $T_{31}=17^{\eta}+68^{\eta}-20^{\eta}-65^{\eta}$
is a special case of the expression $X$ in the proposition, namely
$T_{31}=c(c-1)X$ for $c=\eta+1$, and
$(a_1,a_2;b_1,b_2)=(4^2,8^2;1^2,2^2)$.
The case $T_{32}$ is identical, now with
$(a_1,a_2;b_1,b_2)=(4^2,8^2;2^2,4^2)$.

The four other quantities can be written in a similar way as
$T_{ij}=c(c-1)X$ for $c=\eta+1$, in terms of the more general
expressions $X$ appearing in Proposition~\ref{prop.ineqa1a2b1b2}, for
suitable choices of $s$ and $(a_1,a_2;b_1,b_2)$, and in this case we
have to choose the value of $\theta$, as a function of $\eta$, in
order to ensure that the conditions (\ref{eq.928757863}) for $s$ are
satisfied.  The whole matrix is
\be
M=
\left(
\begin{array}{ccc}
 ((1+\theta)^2+1)^{\eta} & ((2+\theta)^2+1)^{\eta} & ((4+\theta)^2+1)^{\eta} \\
 ((1-2 \theta)^2+4)^{\eta} & ((2-2 \theta)^2+4)^{\eta} & ((4-2 \theta)^2+4)^{\eta} \\
 17^{\eta} & 20^{\eta} & 32^{\eta} \\
 65^{\eta} & 68^{\eta} & 80^{\eta} \\
\end{array}
\right)
\ee
This gives the table
\[
\begin{array}{|c|cccc|c|}
\cline{2-6}
\multicolumn{1}{c|}{} & 
\rule{0pt}{11pt}\raisebox{-5pt}{\rule{0pt}{11pt}}%
a_1&a_2&b_1&b_2 & (a_1+b_1)s\delta
\\
\hline
\rule{0pt}{11pt}%
T_{11} &
1+4\theta+\theta^2 & 4-8\theta+4\theta^2 & 1+4\theta & 4 & -6\theta \\
T_{12} &
1+8\theta+\theta^2 & 4-16\theta+4\theta^2 & 4+8\theta & 16 & -12\theta \\
T_{21} & 
4-8\theta+4\theta^2 & 16 & 1 & 4 & 4 \theta \\
\raisebox{-5pt}{\rule{0pt}{11pt}}%
T_{22} & 
4-16\theta+4\theta^2 & 16 & 4 & 16 & 8 \theta \\
\hline
\end{array}
\]
Each line determines four constraints.
First, the values of $\theta$ must be compatible with the hypotheses
of the proposition, that $a_1<a_2$ and $b_1<b_2$. This gives, overall,
\be
-0.645751\ldots = 2 - \sqrt{7} < \theta < 4 - \sqrt{15} = 0.127017\ldots
\ee
Then, the identity
$(a_1+b_1)s\delta=k\theta$, together with the 
first of conditions (\ref{eq.928757863}), gives
$|\theta|<\frac{1}{|k|}(a_1+b_1)$. Combining the four rows, this gives
\be
|\theta|<7-\sqrt{47}=0.14434\ldots
\ee
Finally,
the second of conditions (\ref{eq.928757863}) gives that
$s=\frac{k\theta}{(a_1+b_1)\delta}>-\frac{a_1 b_1}{(a_1+b_1)^2}$, that
is
$\sign(k \eta) \theta>-|\eta|\frac{a_1 b_1 (a_2-a_1)(b_2-b_1)}{|k|(a_1+b_1)
  a_2 b_2}$. Combining the
four rows, this gives
\be
-0.107969\ldots < 
\frac{\theta}{|\eta|} < 0.0376064\ldots
\ee
where the bounds are roots of certain polynomials of degree 4 and 6.
It turns out that our choice (\ref{eq.deftheta}) is on the safe side,
as $1/8<0.127017\ldots$ and $1/32<0.0376064\ldots$
% \end{proof}

Now we shall consider the case in which $p=2$, but $\Omega$ is not a
flat manifold. This case will be considerably simpler, as the presence
of curvature breaks the scale covariance of the problem, so it is
sufficient to exhibit a family of configurations that are crossing in
the limit in which they are scaled to a small neighbourhood of a point
with non-zero curvature, and some small parameter is adjusted
accordingly. In other words, this account to consider a family of
configurations with points $x_i$ and $y_j$ with norms of order 1, but
in a manifold $\Omega$ with constant infinitesimal curvature, as a
sphere with large radius, or a hyperbolic plane with infinitesimal
negative curvature.

% $\bu \bv \bx \by \bz \bw \bphi$

In the second case, if we adopt the Poincar\'e Disk Model to encode
the positions of the points, we have the general formula
\cite{AndersonHyperbolicGeometry}
% https://en.wikipedia.org/wiki/Poincaré_disk_model#Lines_and_distance
\be
\label{eq.75343275676pre}
d(\bu,\bv)
=
\frac{1}{\sqrt{\eta}}
\arccosh
\left(
1+
\frac{\frac{\eta}{2} \|\bu-\bv\|^2}
{(1-\frac{\eta}{4}\|\bu\|^2)(1-\frac{\eta}{4}\|\bv\|^2)}
\right)
\ee
where $\eta$ encodes the curvature parameter.  In the limit of small
curvature, and keeping only the leading correction, we get
\be
\label{eq.75343275676}
d(\bu,\bv)
=
\|\bu-\bv\| 
\left(
1+\frac{\eta}{24} 
\Big(
3(\|\bu\|^2+\|\bv\|^2)-\|\bu-\bv\|^2
\Big)
+\cO(\eta^2)
\right)
\ee
As we work at $p=2$, we get
\be
\label{eq.75343275676sq}
w_{\bu,\bv}
=d(\bu,\bv)^2
=
\|\bu-\bv\|^2
\left(
1+\frac{\eta}{12} (3(\|\bu\|^2+\|\bv\|^2)-\|\bu-\bv\|^2)+\cO(\eta^2)
\right)
\ee
In the case of positive curvature, we consider points on a sphere of
radius $R$, using the standard orthogonal projection onto the tangent
plane at the origin (or equivalently, an equatorial plane). Let $\bu$
and $\bv$ be the coordinates of the projected points. The exact
geodesic distance $d(\bu,\bv)$ can be obtained from the spherical law
of cosines (the fundamental relation underlying the Haversine
formula). Specifically, the points on this plane 
are lifted to the sphere as 
$\tilde{\bu} = (\bu,\sqrt{R^2-\Vert{}\bu\Vert{}^2})$ and
$\tilde{\bv} = (\bv,\sqrt{R^2-\Vert{}\bv\Vert{}^2})$, yielding the
exact relation for the angular distance $\theta =d(\bu,\bv)/R$:
\be
\cos\left(\frac{d(\bu,\bv)}{R}\right) =
\frac{\tilde{\bu}\cdot\tilde{\bv}}{R^2} = \frac{\bu\cdot\bv +
  \sqrt{R^2-|\bu|^2}\sqrt{R^2-|\bv|^2}}{R^2}
\ef.
\ee
To analyze the small-curvature regime, we set $\eta = 1/R^2$ and
expand the square roots and the arccosine for large $R$ (or
equivalently, small $\eta$). Expanding to order $\mathcal{O}(\eta^2)$
first the radicals, and then the $\arccos(1-x)$ function,
\be
\label{eq.75343275676sqSph}
w_{\bu,\bv} = d(\bu,\bv)^2 = |\bu-\bv|^2\left(1-\frac{\eta}{12}\Big(3(|\bu|^2+|\bv|^2)-|\bu-\bv|^2\Big)+\cO(\eta^2)\right)
\ef.
\ee
This shows that, up to identifying $\eta$ with $1/R^2$, the cases of
positive and negative curvature share the same formula, up to the
choice of sign for~$\eta$.

Now, the $4\times 3$ matrix of weights $W_{ij}$, as a function of
$\eta$ and $\theta$, is determined by using the formula
(\ref{eq.75343275676sq}), and the set of positions
(\ref{eq.XYposCross}).  Performing an expansion for $\eta$ and
$\theta$ small, and keeping only the terms of order $1$ complexively
in $\eta$ and $\theta$, gives certain linear expressions for the
$T_{ij}$'s, namely
\be
\begin{pmatrix}
T_{11} & T_{12} \\
T_{21} & T_{22} \\
T_{31} & T_{32}
\end{pmatrix}
=
\begin{pmatrix}
3 \eta - 6 \theta & 12 \eta - 12 \theta \\ 
12 \eta + 4 \theta & 48 \eta + 8 \theta \\ 
48 \eta & 192 \eta
\end{pmatrix}
\ee
Note that, again, with our choice of points the transpositions
$T_{31}$ and $T_{32}$ are simpler than the others, and they have the
appropriate sign depending on the sign of $\eta$, regardless of the
value of $\theta$. For the other four transpositions we get that,
when $\eta>0$, all $T_{ij}$'s are positive if
\be
-3\eta < \theta < \frac{1}{2} \eta
\ee
while when $\eta<0$ all $T_{ij}$'s are negative if
\be
\frac{1}{2} \eta < \theta < -3\eta 
\ef.
\ee
As these intervals are non-empty, we have proven that on curved
two-dimensional manifolds also at $p=2$ there are in general crossing
configurations.

We should now investigate if there are crossing configurations on
manifolds which are almost everywhere flat, and have conical
singularities. In the case of a conical singularity with positive
curvature, we can encode a neighbourhood of a singularity with angle
$2\pi-2\alpha$ by stating that the positions are in the portion of the
plane with $\arg(\bu)=\arctan(u_2/u_1) \not\in [-\alpha,\alpha]$, and
that the points $(r \cos \alpha,r \sin \alpha)$ and 
$(r \cos \alpha,-r \sin \alpha)$ are identified.

Now we choose as set of positions
\begin{align}
X&=
\{
(-1,1),(-1,-1),(-1-\eps,0),(-1+\eps,0)
\}
\ef;
\\
Y
&=
\{
(r \cos \alpha,r \sin \alpha),
(-1-\eps/2,-1/2),
(0,0)
\}
\ef.
\end{align}
In light of Remark~\ref{rmk.validifdiagzeroes}, the matrix of costs
is valid at sight, for all columns, for generic
$r$ and $\eps<1$, except for the first column, where we have the
condition
\be
(r \cos \alpha+1)^2+(r \sin \alpha-1)^2
<
(r \cos \alpha+1-\eps)^2+(r \sin \alpha)^2
\ef.
\ee
In particular, if we choose $r=1/\sin\alpha$, we get
\be
(1/\tan \alpha+1)^2
<
(1/\tan \alpha+1-\eps)^2+1
\ef,
\ee
which is true, for example, if 
$\eps=\frac{1}{2} \frac{\tan \alpha}{1+\tan \alpha}$.

When we have conical singularities with negative curvature, say with
angle $2\pi+2\alpha$ with $\alpha>0$, then our manifold does not
satisfy our generality hypotheses on the transversality of geodesics,
as all pairs of points with the angular distance exceeding $\pi$ along
both potential traversal directions around the apex
% with polar coordinates $(r_1,\phi_1)$ and $(r_2,\phi_2)$, with
% $|\phi_2-\phi_1|>\pi$ in both angular directions,
have geodesics passing through the singularity.
% $\pi<\phi_2-\phi_1<\pi+2\alpha$

Nonetheless, we shall illustrate the fact that also in this case
% , for $\alpha<\pi/2$, 
there exist generic instances that are crossing.

The case $\alpha>\pi$ is specially simple. In this case, using polar
coordinates, a counterexample is given by the configuration (calling
$\theta=\frac{\pi+\alpha}{2}$)
\begin{align}
X&=\big(
(4,\theta),(3,-\theta),(2,-\theta-\eps),(1,\theta+\eps)
\big)
\ef;
&
Y&=\big(
(1,0),(2,0),(3,0)
\big)
\ef.
\end{align}
When $\eps$ is non-zero, the path $(x_1,x_2,x_3,x_4)$ is crossing. The
matrix of costs, before and after a gauge transformation,
neglecting the perturbation $\eps$ of the angles, reads
\begin{align}
W&=
\begin{pmatrix}
25 & 36 & 49 \\
16 & 25 & 36 \\
 9 & 16 & 25 \\
 4 &  9 & 16
\end{pmatrix}
\ef;
&
W'&=
\begin{pmatrix}
0 & 2 & 6 \\
0 & 0 & 2 \\
2 & 0 & 0 \\
6 & 2 & 0
\end{pmatrix}
\ef.
\end{align}
As the entries of $W'$ out of the main diagonals are at least $2$,
there must exist a non-empty interval of $\eps$ such that the
corresponding matrix of costs is valid. Another way of seeing this
goes through the calculation of the matrix of transpositions,
according to the sufficient condition described in Proposition
\ref{prop.onlytransp}. In this case we get:
\be
\label{eq.2775895}
T=\begin{pmatrix}
2 & 2 \\
2 & 2 \\
2 & 2 
\end{pmatrix}
\ef.
\ee
When $0<\alpha \leq \pi$, and we keep the same configuration of points
above, the matrix of costs is modified because the distance between
$x_i$'s and $y_j$'s is not given by the sum of the radia, squared, but
it rather involves the angle $\theta$, namely
$d(x_i,y_j)^2=x_i^2+y_j^2-2 x_i y_j \cos(\theta)$. The calculations of
$W$, $W'$ and $T$ go on almost indentically, and we find a matrix $T$
that coincides with (\ref{eq.2775895}) up to an overall factor 
$-\cos(\theta) \in \;]\rule{.5pt}{0pt}0,1]$. 
This allows to conclude, again in light of Proposition
\ref{prop.onlytransp}, that there are crossing configurations also for
the range $0<\alpha \leq \pi$ of conical singularities, thus
completing our analysis.

%%%%%%%%%%%%%%%%%%%%%%%%%%%%%%%%%%%%%%%%%%%%%%%%%%%%%%%
\addtocontents{toc}{\protect\setcounter{tocdepth}{-2}}
\section*{Acknowledgement}
\addtocontents{toc}{\protect\setcounter{tocdepth}{2}} % Riporta a 2 per le sezioni successive
% \acknowledgements
%%%%%%%%%%%%%%%%%%%%%%%%%%%%%%%%%%%%%%%%%%%%%%%%%%%%%%%

\noindent 
We are grateful to Massimiliano Gubinelli for a stimulating discussion
on the connection between the Euclidean Random Assignment Problem and
$\text{SLE}_2$. His skepticism regarding the absence of loops in the
matching subgraph $H_\mathcal{J}$ prompted us to seek a rigorous
proof; what initially seemed a routine application of cycle swapping
arguments ultimately led to the deeper structural analysis presented here.

%% We are grateful to Massimiliano Gubinelli for a stimulating discussion
%% that originally sparked this work. The project began as an attempt to
%% investigate the connection between the Euclidean Random assignment
%% Problem and $\text{SLE}_2$, during which his insightful skepticism
%% regarding the absence of loops in what would have later become the
%% ``matching subgraph'' $H_\calJ$ prompted us to seek a rigorous
%% proof. What initially was argued to be a straightforward application
%% of classical arguments of cycle swapping ultimately led to the deeper
%% structural analysis presented in this paper.

%% I am deeply grateful to Massimiliano Gubinelli for a stimulating
%% conversation that originally inspired this work. The initial
%% motivation stemmed from my attempt to convince him that the paths
%% underlying this problem should be governed by SLE$_{2}$. During
%% that discussion, I took it for granted that these paths must form a
%% spanning tree, assuming the impossibility of loops. He insightfully
%% questioned this assumption, prompting me to promise either a formal
%% proof or a numerical simulation in support of my claim. To my
%% surprise, establishing this result proved to be far more subtle
%% than anticipated: while I initially believed standard alternating
%% path swapping arguments would suffice, a much finer
%% analysis—relying on the careful combination of semi-alternating
%% paths developed here—was ultimately required.

The work of A.S.\ is or has been supported by the ANR projects
\emph{DIMERS} (ANR-18-CE40-0033), \emph{COMBIN\'E}
(ANR-19-CE48-0011) and \emph{COMETA-GAE}
(ANR-25-CE48-0602).

%%%%%%%%%%%%%%%%%%%%%%%%%%%%%%%%%%%%%%%%%%%%%%%%%%%%%%%
\bibliographystyle{unsrt}
\bibliography{biblio}
%%%%%%%%%%%%%%%%%%%%%%%%%%%%%%%%%%%%%%%%%%%%%%%%%%%%%%%

\end{document}